\documentclass[a4paper,twoside,leqno]{article}
\usepackage{amsmath,amsthm,amssymb,amscd}
\usepackage{ascmac}
\usepackage[mathscr]{eucal}
\usepackage[all,cmtip]{xy}
\usepackage{tikz}
\usetikzlibrary{decorations.pathmorphing}
\usetikzlibrary{decorations.markings}
\usepackage{enumerate}
\usepackage{graphicx}
\usepackage{verbatim}
\usepackage{version}
\usepackage{color}
\newenvironment{NB}{
\color{red}{\bf NB}. \footnotesize
}{}
\newenvironment{NB2}{
\color{blue}{\bf NB2}. \footnotesize
}{}
\excludeversion{NB}
\excludeversion{NB2}

\usepackage[
pdftex,
bookmarks=true,
colorlinks=true,
unicode,
pdfnewwindow=true]{hyperref}

\usepackage{scalerel}
\usepackage{stmaryrd}
\usepackage{youngtab}

\newenvironment{aenume}{%
  \begin{enumerate}%
  }{\end{enumerate}}

\usepackage{a4wide}

\DeclareSymbolFont{symbolsC}{U}{pxsyc}{m}{n}
\SetSymbolFont{symbolsC}{bold}{U}{pxsyc}{bx}{n}
\DeclareFontSubstitution{U}{pxsyc}{m}{n}
\DeclareMathSymbol{\medcirc}{\mathbin}{symbolsC}{7}

\newtheorem{Theorem}[equation]{Theorem}
\newtheorem{Corollary}[equation]{Corollary}
\newtheorem{Lemma}[equation]{Lemma}
\newtheorem{Proposition}[equation]{Proposition}

\theoremstyle{definition}
\newtheorem{Definition}[equation]{Definition}
\newtheorem{Example}[equation]{Example}

\newtheorem{Conjecture}[equation]{Conjecture}

\theoremstyle{remark}
\newtheorem{Remark}[equation]{Remark}

\newtheorem*{Claim}{Claim}

\numberwithin{equation}{section}

\newcommand{\thmref}[1]{Theorem~\ref{#1}}
\newcommand{\secref}[1]{\S\ref{#1}}
\newcommand{\lemref}[1]{Lemma~\ref{#1}}
\newcommand{\propref}[1]{Proposition~\ref{#1}}
\newcommand{\corref}[1]{Corollary~\ref{#1}}
\newcommand{\subsecref}[1]{\S\ref{#1}}

\newcommand{\remref}[1]{Remark~\ref{#1}}

\newcommand{\MR}[1]{}
\newcommand{\defeq}{\overset{\operatorname{\scriptstyle def.}}{=}}
\newcommand{\CC}{{\mathbb C}}
\newcommand{\ZZ}{{\mathbb Z}}
\newcommand{\RR}{{\mathbb R}}
\newcommand{\AAA}{{\mathbb A}}
\newcommand{\proj}{{\mathbb P}}
\newcommand{\CP}{\proj}

\newcommand{\SL}{\operatorname{\rm SL}}
\newcommand{\SU}{\operatorname{\rm SU}}
\newcommand{\GL}{\operatorname{GL}}

\newcommand{\grpSp}{\operatorname{\rm Sp}}
\newcommand{\algsl}{\operatorname{\mathfrak{sl}}} 
\newcommand{\su}{\operatorname{\mathfrak{su}}}
\newcommand{\gl}{\operatorname{\mathfrak{gl}}}

\newcommand{\g}{{\mathfrak g}}

\newcommand{\Spec}{\operatorname{Spec}\nolimits}

\newcommand{\End}{\operatorname{End}}
\newcommand{\Hom}{\operatorname{Hom}}

\newcommand{\Ker}{\operatorname{Ker}}
\newcommand{\Coker}{\operatorname{Coker}}
\newcommand{\Ima}{\operatorname{Im}}

\newcommand{\rank}{\operatorname{rank}}

\newcommand{\tr}{\operatorname{tr}}

\newcommand{\id}{\operatorname{id}}

\newcommand{\codim}{\mathop{\text{\rm codim}}\nolimits}
\newcommand{\diag}{\mathop{\text{\rm diag}}\nolimits}
\newcommand{\wA}{\vphantom{j^{X^2}}\smash{\overset{\circ}{\vphantom{\vstretch{0.7}{A}}\smash{\mathbb A}}}} 
\newcommand{\bA}{\vphantom{j^{X^2}}\smash{\overset{\bullet}{\vphantom{\vstretch{0.7}{A}}\smash{\mathbb{A}}}}} 
\newcommand{\wZ}{\vphantom{j^{X^2}}\smash{\overset{\circ}{\vphantom{\vstretch{0.7}{A}}\smash{Z}}}} 
\newcommand{\bu}{{\mathbf u}} 
\newcommand{\tn}{{\text{new}}}

\newcommand{\dslash}{/\!\!/} 
\newcommand{\bv}{{\mathbf v}} 
\newcommand{\bw}{{\mathbf w}} 
\newcommand{\bM}{{\mathbf M}} 

\newcommand{\vout}[1]{\operatorname{o}(#1)}
\newcommand{\vin}[1]{\operatorname{i}(#1)}

\newcommand{\cM}{{\mathcal M}} 
\newcommand{\fM}{{\mathfrak M} 
}

\newcommand{\GV}{\mathcal G}

\newcommand{\bN}{\mathbf N}
\newcommand{\cR}{\mathcal R}
\newcommand{\cK}{\mathcal K}
\newcommand{\cO}{\mathcal O}
\newcommand{\Gr}{\mathrm{Gr}}
\newcommand{\intsys}{\varpi}

\newcommand{\bz}{\mathbf z}
\newcommand{\sfu}{{\mathsf{u}}}
\newcommand{\sfy}{{\mathsf{y}}}

\newcommand{\aff}{\mathrm{aff}}
\newcommand{\hf}{\hfill}
\makeatletter
\newdimen\y@inside
\y@inside=\y@boxdim
\advance\y@inside by -8\y@linethick
\definecolor{halfgray}{gray}{0.5}
\newcommand{\rf}{\color{halfgray}
\rule[-.2\y@inside]{\y@inside}{\y@inside}
}
\makeatother
\newcommand{\graysquare}{\color{halfgray}\blacksquare}
\newcommand{\WW}{{
\CC}}
\newcommand{\xl}{{
x}}
\newcommand{\xp}{{
\sigma}}
\newcommand{\bigzero}{\makebox(0,0){\text{\huge0}}}

\newcommand{\rvline}{\hspace*{-\arraycolsep}\vline\hspace*{-\arraycolsep}}
\usepackage{multirow}
\usepackage{textgreek}

\title{\textbf{Cherkis bow varieties and Coulomb branches of quiver gauge theories of affine type $A$}}
\author{Hiraku Nakajima\thanks{Research Insitute for Mathematical Sciences, Kyoto University, Kyoto 606-8502, Japan}  and Yuuya Takayama${}^*$}
\date{\empty}
\begin{document}
\maketitle

\begin{abstract}
    We show that Coulomb branches of quiver gauge theories of affine
    type $A$ are Cherkis bow varieties, which have been introduced as
    ADHM type description of moduli space of instantons on the
    Taub-NUT space equivariant under a cyclic group action.
\end{abstract}

\section{Introduction}

We have two purposes of this paper. The first is to initiate
algebro-geometric study of Cherkis bow varieties $\cM$ of affine type
$A$, which have been introduced as ADHM type description of moduli
spaces of $\mathrm{U}(n)$-instantons on the Taub-NUT space equivariant
under a cyclic group $\ZZ/\ell \ZZ$-action \cite{MR2824478} (see also
\cite{MR2525636,MR2721657}). Our interest lies in the Uhlenbeck
partial compactification of instanton moduli spaces.
They are algebraic varieties with hyper-K\"ahler structure on their
regular loci, with an $\SU(2)$-action rotating three complex
structures. They include quiver varieties of affine type $A$
introduced by the first-named author \cite{Na-quiver} as special
cases, 
but more general. We introduce quiver description of arbitrary bow
varieties (\thmref{thm:quiver_desc}). See
Figure~\ref{fig:quiverBow}. The original definition involves Nahm's
equation, which is an ordinary differential equation, while quiver
description involves matrices and algebraic equations. Hence the
latter is more suitable for algebro-geometric questions. For example,
quiver description enabels us to show that bow varieties are normal in 
many cases. 
The proof of the quiver description of a bow variety is a simple
combination of known results \cite{MR769355,MR987771,Takayama}. See
paragraphs after \thmref{thm:quiver_desc} for more detail.
Nevertheless we think that it is important, as it is easier to
handle. We believe that it has applications to other problems on bow
varieties.
%
%
%
It is also expected that bow varieties are related to representation
theory of affine Lie algebras of type $A$, as for quiver varieties,
but in a better way. We will study stratifications, semismall
resolutions, torus fixed points, etc for bow varieties, as they have
played important roles in relation between quiver varieties and
representation theory.

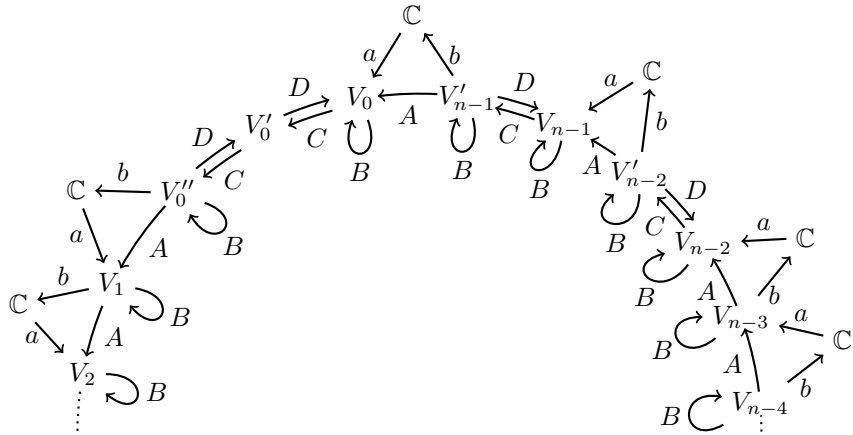
\begin{figure}[htbp]
    \centering
%
\begin{tikzpicture}[scale=1.5,
arcnode/.style 2 args={                
            decoration={
                        raise=#1,             
                        markings,   
                        mark=at position 0.5 with { 
                                    \node[inner sep=0] {#2};
                        }
            },
            postaction={decorate}
}
] 
    \draw[thick,dotted] (-1:3) arc (-1:3:3);
    \node (Vn-4) at (5:3) {$\smash{V_{n-4}}\vphantom{\bigcirc}$};
    \draw[thick,->,arcnode={7pt}{$B$}] (Vn-4) to
    [loop above,out=210,in=170,distance=12pt] (Vn-4);
    \node (Wn-3) at (12.5:3.8) {$\CC$};
    \node (Vn-3) at (20:3) {$\smash{V_{n-3}}\vphantom{\bigcirc}$};
    \draw[thick,->,arcnode={7pt}{$B$}] (Vn-3) to
    [loop,out=220,in=180,distance=12pt] (Vn-3);
    \draw[thick,->,arcnode={8pt}{$A$}] (7:3) arc (7:18:3);
    \draw[thick,->] (Vn-4) -- (Wn-3) node[midway,below] {$b$};
    \draw[thick,->] (Wn-3) -- (Vn-3) node[midway,above] {$a$};
    \node (Wn-2) at (27:3.8) {$\CC$};
    \node (Vn-2) at (34:3) {$\smash{V_{n-2}}\vphantom{\bigcirc}$};
    \draw[thick,->,arcnode={7pt}{$B$}] (Vn-2) to 
    [out=235,in=195,distance=12pt] (Vn-2);
    \draw[thick,->,arcnode={8pt}{$A$}] (22:3) arc (22:31:3);
    \draw[thick,->] (Vn-3) -- (Wn-2) node[midway,below] {$b$};
    \draw[thick,->] (Wn-2) -- (Vn-2) node[midway,above] {$a$};
    \node (V'n-2) at (50:3) {$\smash{V'_{n-2}}\vphantom{\bigcirc}$}; 
    \draw[thick,->,arcnode={7pt}{$B$}] (V'n-2) to 
    [out=270,in=220,distance=12pt,looseness=10] (V'n-2);
    \draw[thick,->,arcnode={8pt}{$C$}] (38:2.95) arc (38:45:2.95);
    \draw[thick,<-,arcnode={-8pt}{$D$}] (38:3.05) arc (38:45:3.05);
    \node (Wn-1) at (57.5:3.8) {$\CC$};
    \node (Vn-1) at (65:3) {$\smash{V_{n-1}}\vphantom{\cdot}$};
    \draw[thick,->,arcnode={7pt}{$B$}] (Vn-1) to 
    [out=260,in=220,distance=12pt,looseness=100] (Vn-1);
    \draw[thick,->,arcnode={8pt}{$A$}] (55:3) arc (55:60:3);
    \draw[thick,->] (V'n-2) -- (Wn-1) node[midway,right] {$b$};
    \draw[thick,->] (Wn-1) -- (Vn-1) node[midway,above] {$a$};
    \node (V'n-1) at (82:3) {$\smash{V'_{n-1}}\vphantom{\bigcirc}$}; 
    \draw[thick,->,arcnode={7pt}{$B$}] (V'n-1) to 
    [out=285,in=245,distance=12pt,looseness=10] (V'n-1);
    \draw[thick,->,arcnode={8pt}{$C$}] (70:2.95) arc (70:77:2.95);
    \draw[thick,<-,arcnode={-8pt}{$D$}] (70:3.05) arc (70:77:3.05);
    \node (W0) at (91:3.7) {$\CC$};
    \node (V0) at (100:3) {$\smash{V_0}\vphantom{\bigcirc}$};
    \draw[thick,->,arcnode={7pt}{$B$}] (V0) to 
    [out=290,in=250,distance=12pt,looseness=10] (V0);
    \draw[thick,->,arcnode={8pt}{$A$}] (87:3) arc (87:97:3);
    \draw[thick,->] (V'n-1) -- (W0) node[midway,right] {$b$};
    \draw[thick,->] (W0) -- (V0) node[midway,left] {$a$};
    \node at (118:3) {$\smash{V'_0}$};
    \draw[thick,->,arcnode={8pt}{$C$}] (105:2.95) arc (105:113:2.95);
    \draw[thick,<-,arcnode={-8pt}{$D$}] (105:3.05) arc (105:113:3.05);
    \node (V''0) at (135:3) {$\smash{V''_0}\vphantom{\bigcirc}$};
    \draw[thick,->,arcnode={7pt}{$B$}] (V''0) to 
    [out=340,in=300,distance=12pt,looseness=10] (V''0);
    \draw[thick,->,arcnode={8pt}{$C$}] (122:2.95) arc (122:130:2.95);
    \draw[thick,<-,arcnode={-8pt}{$D$}] (122:3.05) arc (122:130:3.05);
    \node (W1) at (144.5:3.7) {$\CC$};
    \node (V1) at (154:3) {$\smash{V_1}\vphantom{\bigcirc}$};
    \draw[thick,->,arcnode={7pt}{$B$}] (V1) to 
    [out=350,in=310,distance=12pt,looseness=10] (V1);
    \draw[thick,->,arcnode={8pt}{$A$}] (138:3) arc (138:151:3);
    \draw[thick,->] (V''0) -- (W1) node[midway,above] {$b$};
    \draw[thick,->] (W1) -- (V1) node[midway,left] {$a$};
    \node (W2) at (162:3.7) {$\CC$};
    \node (V2) at (170:3) {$\smash{V_2}\vphantom{\bigcirc}$};
    \draw[thick,->,arcnode={7pt}{$B$}] (V2) to 
    [out=360,in=320,distance=12pt,looseness=10] (V2);
    \draw[thick,->] (V1) -- (W2) node[midway,above] {$b$};
    \draw[thick,->] (W2) -- (V2) node[midway,left] {$a$};
    \draw[thick,->,arcnode={8pt}{$A$}] (158:3) arc (158:167:3);
    \draw[thick,dotted] (173:3) arc (173:180:3); 
\end{tikzpicture}
\caption{Quiver description of a bow variety}
    \label{fig:quiverBow}
\end{figure}

The second purpose is to show that bow varieties also include Coulomb
branches of framed quiver gauge theories of affine type $A$ again as
special cases. 

Let us recall what are Coulomb branches briefly.
Let $G$ be a complex reductive group and $\bM$ its symplectic
representation.
The Coulomb branch $\mathcal M_C$ of a $3d$ $\mathcal N=4$ SUSY gauge
theory associated with $(G,\bM)$ is a hyper-K\"ahler manifold, and had
been studied by physicists for many years. But its mathematical
definition was unclear until the first named author proposed a
rigorous approach in \cite{2015arXiv150303676N}.
The subsequent paper \cite{main} written with Braverman and Finkelberg
gives
a mathematical definition of $\mathcal M_C$ as an affine algebraic
variety, under the assumption that $\bM$ is of form $\bN\oplus\bN^*$.
The companion paper \cite{blowup} by the same authors discusses
$\mathcal M_C$ for a framed/unframed quiver gauge theory of type
$ADE$. In particular, it was shown that $\mathcal M_C$ is isomorphic
to a slice in the affine Grassmannian of the corresponding group or
its generalization (generalized slice).

In this paper we consider the Coulomb branch of a framed quiver gauge
theory of affine type $A_{n-1}$ ($n\ge 1$).
We have two dimension vectors $\underline{\bv} =
(\bv_0,\dots,\bv_{n-1})$, $\underline{\bw} = (\bw_0,\dots,\bw_{n-1})$
and consider the gauge theory with the group $G = \prod \GL(\bv_i)$
and the representation $\bN = \bigoplus
\Hom(\CC^{\bv_i},\CC^{\bv_{i+1}})\oplus\bigoplus \Hom(\CC^{\bw_i},
\CC^{\bv_i})$.
Our goal is to identify $\mathcal M_C$ with Cherkis bow variety
$\cM$. 
Here $\ell$ is the level of the affine weight $\lambda := \sum
\bw_i\Lambda_i$. Other discrete data (e.g., instanton number) are
determined by $\bv_i$, $\bw_i$.
A little more precisely, we assume $\dim V_i = \dim V'_i = \dim V''_i
= \cdots$ (the \emph{balanced condition}), and the dimension of $W_i$
is encoded in the quiver as the number of vector spaces with primes
between $V_i$ and $V_{i+1}$ in Figure~\ref{fig:quiverBow}. Hence
$\ell$ is the number of arrows for $C$ (also for $D$), and $n$ is the
number of arrows for $A$ (and $a$, $b$).

This answer was known in the physics literature
\cite{MR1451054,MR1454291,MR1454292,MR1636383} for various cases. The
general answer as a bow variety with the balanced condition was
informed to us by Cherkis in a private communication.
See \cite{2011PhRvD..83l6009C} for a physical derivation.
Therefore the main result of this paper could be regarded as a yet
another confirmation of correctness of the proposal
\cite{2015arXiv150303676N}.

Our strategy to identify the Coulomb branch $\mathcal M_C$ with a bow
variety $\cM$ is the same as in \cite{blowup}:
The Coulomb branch $\mathcal M_C$ has a morphism $\varpi\colon
\mathcal M_C\to \AAA^{\underline{\bv}} := \prod \AAA^{\bv_i}/\mathfrak
S_{\bv_i}$, where $\bv_i = \dim V_i$. It satisfies the
\emph{factorization} property, i.e., $\mathcal M_C$ is isomorphic to a
product of lower dimensional Coulomb branches, locally around $x$ if
$\varpi(x)$ is a sum of disjoint configurations.
We define the corresponding \emph{factorization map} $\Psi\colon
\cM\to \AAA^{\underline{\bv}}$ for the bow variety, and prove the
factorization property. We then define a birational isomorphism
$\Xi^\circ\colon \cM_C\overset{\approx}{\dasharrow}\cM$ over
$\AAA^{\underline{\bv}}$ and show that it is biregular up to
codimension $2$.
Then we use the criterion in \cite[Th.~5.26]{main} to conclude that
$\Xi^\circ$ extends to the whole space.

The morphism $\varpi$ is an integrable system, i.e., components are
Poisson commuting and general fibers are algebraic torus. For $n=1$,
it is the cyclotomic, rational or trigonometric Calogero-Moser system
depending on $\dim W$, but probably new one for $n > 1$. We only use
it to identify Coulomb branches with bow varieties in this paper, but
it is certainly worthwhile to pursue its study.

\begin{Remark}\label{rem:dominant_quiver}
    In \cite{2015arXiv150303676N}, it was conjectured that $\mathcal
    M_C$ is a quiver variety of an affine type $A_{\ell-1}$ introduced
    in \cite{Na-quiver}, under the assumption that $\mu := \sum_i
    \bw_i \Lambda_i - \bv_i \alpha_i$ is \emph{dominant}.
    Recall that a quiver variety of affine type $A_{\ell-1}$ is a moduli
    space of $\operatorname{U}(n)$-instantons on $\RR^4$ equivariant
    under the cyclic group $\ZZ/\ell\ZZ$.
    The Taub-NUT space and $\RR^4$ are isomorphic as holomorphic
    symplectic manifolds but have different hyper-K\"ahler metrics.
    We will also prove this conjecture from our main result, but the
    bow variety is regarded as a \emph{correct} answer for $\mathcal
    M_C$ as
    \begin{enumerate}
          \item it covers also non-dominant cases,
          \item the conjectural hyper-K\"ahler metric is correct,
          \item several important properties of $\mathcal M_C$ can be
        checked for bow varieties rather naturally, but not for quiver
        varieties.
    \end{enumerate}
    As for 1, for example when $\ell = 1$, the dominant condition is
    satisfied only when $\sum \bv_i\alpha_i$ is a multiple of the
    primitive imaginary root.
    For 3, the factorization map $\Psi$ is constructed
    as a bow variety, but its description in a quiver variety is not
    obvious.

    The first named author learned the subtle difference between
    quiver and bow varieties after reading \cite{2015arXiv150304817B}.
\end{Remark}

In order to apply \cite[Th.~5.26]{main}, we need to check that $\cM$
is normal and all the fibers of $\Psi$ have the same dimension. These
conditions are delicate, and we use the quiver description in essential way. 

The quiver is divided into two types of pieces, \emph{triangle}
involving $A$, $B$, $a$, $b$, and \emph{two-way} involving $C$,
$D$.
When there are no two-way parts (corresponding to the case $\bw=0$),
the quiver is called the \emph{chainsaw} quiver considered in
\cite{fra}. The normality and the calculation of dimension of fibers
of $\Psi$ for triangles were already established in \cite{fra}. In
fact, the identification of the Coulomb branch in this case was
already given in \cite{blowup} by using \cite{fra}.
The remaining two-way parts are quiver varieties of type $A$ in the
sense of \cite{Na-quiver,Na-alg}. The proof of the normality and
calculation of dimension were given in \cite{CB,CB:normal} (in fact,
in much more general situations).
Our proof of the normality of bow varieties are basically consequences
of these facts.

Recall that the balanced condition is $\dim V_i = \dim V'_i = \cdots$.
When we assume the \emph{cobalanced}\footnote{This terminology is due
  to Sergey Cherkis.} condition that dimensions of vector spaces
connected by an arrow $A$ in Figure~\ref{fig:quiverBow} are the same,
the bow variety is isomorphic to a quiver variety $\fM$ introduced in
\cite{Na-quiver} of an affine type $A_{n-1}$. This is what we have
mentioned at the beginning.

There is an isomorphism between two bow varieties where adjacent
triangle and two-way parts are swapped where dimension of the middle
vector space is changed appropriately. See
\propref{prop:HW-trans}. For example, consider two parts for $V'_0$,
$V''_0$, $V_1$ in Figure~\ref{fig:quiverBow}. We can exchange triangle
and two-way parts. Vector spaces $V'_0$, $V_1$ are unchanged, and the
dimension of the new middle vector space becomes $\dim V_1 + \dim V'_0
+ 1 - \dim V''_0 $. It appeared as Hanany-Witten transition in
\cite{MR1451054}. This is a powerful tool, which does not have analog
in the context of quiver varieties. As applications, we prove
\begin{itemize}
      \item Under the dominance assumption mentioned in
\remref{rem:dominant_quiver}, a bow variety satisfying the balanced
condition is isomorphic to one satisfying the cobalanced condition
by successive application of Hanany-Witten transitions.

  \item We give a new proof of \cite{Na-quiver,MR2130242} saying a
quiver variety of type $A$ is isomorphic to the intersection of
Slodowy slice and a nilpotent orbit or its Springer resolution.

  \item We determine the symplectic leaves of the Coulomb branch
$\cM_C$. (This is new even for finite type $A$.)
\end{itemize}

We have the \emph{duality} of bow varieties exchanging triangle and
two-way parts without changing dimensions of vector spaces $V$'s.
It does not induces an isomorphism of bow varieties, but it originates
from the $\SL(2,\ZZ)$-duality in the superstring theory, and closely
related to the $3d$ mirror symmetry. (See \cite{MR1451054}.) Therefore
we expect two bow varieties exchanged by the duality have some deep
relations.
For example, the balanced and cobalanced conditions are exchanged
under the duality. Thus a quiver variety $\fM$ and a Coulomb branch
$\cM_C$ are exchanged by this process. When the dominance condition is
satisfied, it gives an exchange of two quiver varieties. Chasing
changes of dimensions under successive applications of Hanany-Witten
transitions, we will find that pairs of weights $(\lambda,\mu)$ are
replaced by their `transpose' as generalized Young diagrams, as
expected from the relation between the level-rank duality and quiver
varieties. (See \cite[\S A]{Na-branching}.)

The quiver description of a bow variety can be understood as
hamiltonian reduction of a finite dimensional symplectic manifold.
\begin{NB}
  It is an open subset in a Poisson manifold, but one can also use the
  Hurtubise slice as a symplectic manifold.
\end{NB}%
One can consider the corresponding quantum hamiltonian reduction
following \cite{fra} for the case $\bw = 0$. We conjecture that it is
isomorphic to the quantized Coulomb branch when the balanced condition
is satisfied.


\begin{Remark}
    As we have remarked above, bow varieties are moduli spaces of
    $\operatorname{U}(n)$-instantons on the Taub-NUT space equivariant
    under the $\ZZ/\ell\ZZ$-action, though mathematically rigorous
    proof of the completeness is still lacking as far as the authors
    know. Let us give other moduli theoretic description, but with
    more algebro-geometric flavor, known for special classes. However
    different stability conditions were imposed (except in
    \cite{Takayama}), thus the moduli spaces are \emph{not} exactly
    the same as bow varieties.

    \begin{itemize}
          \item Chainsaw quiver varieties \cite{fra} are moduli spaces
        of framed parabolic sheaves on $\CP^1\times\CP^1$. Here
        parabolic structures are put on $\{0\}\times\CP^1$, while
        framing is put on $\{\infty\}\times\CP^1\cup
        \CP^1\times\{\infty\}$.
        With the stability condition (C-S1),(C-S2) below, they are
        open loci of framed parabolic locally free sheaves
        \cite{Takayama}.

          \item If the bow diagram has both triangle and two-way parts
        only once, the varieties describe moduli spaces of framed rank
        $1$ stable perverse coherent sheaves on the blowup
        $\widehat{\CP}^2$ of $\CP^2$ at the origin \cite{perv1}. Here
        the framing is put on the line at infinity. This quiver
        description was originated in \cite{King-phd}.

          \item If there is only one two-way part, the quiver is
        called the \emph{dented} chainsaw quiver \cite{MR3134906}. The
        varieties describe moduli spaces of framed parabolic sheaves
        on $\widehat{\CP}^2$.

      \item The \emph{rift} quiver introduced in \cite{MR3134906} is a
        variant of our quiver. The varieties describe moduli spaces of
        $\ZZ/\ell\ZZ$-equivariant framed parabolic sheaves on
        $\widehat{\CP}^2$, where $\ZZ/\ell\ZZ$ acts on
        $\widehat{\CP}^2$ by
        $([z_0:z_1:z_2],[z:w])\mapsto ([z_0:\zeta z_1:z_2],
        [z:\zeta^{-1}w])$.    
    \end{itemize}
        Here $\widehat{\CP}^2$ is
        $\{ ([z_0:z_1:z_2],[z:w]) \in \CP^2\times\CP^1\mid z_1 w = z_2
        z\}$, the line at infinity is $z_0 = 0$, and paraoblic structures are defined on the line $z=0$. Based on these examples, we expect that bow varieties are isomorphic to Uhlenbeck-type partial compactification of moduli spaces of $\ZZ/\ell\ZZ$-equivariant framed parabolic locally free sheaves on $\widehat{\CP}^2$. The Taub-NUT space is embedded into $\widehat{\CP}^2$ as an open subvariety $z_0\neq 0$, $z\neq 0$, which is isomorphic to $\CC^2$.
\end{Remark}

The paper is organized as follows.
In \secref{sec:bow-varieties} we review the definition of Cherkis bow
varieties. We also explain the quiver description.
In \secref{sec:constituents} we study properties of triangle and
two-way parts of bow varieties.
In \secref{sec:basic-properties} we study stratification of bow
varieties and semismall morphisms from a bow variety to another with
different continuous parameters (stability conditions).
In \secref{sec:symplectic-structure} we study symplectic structure on
a bow variety. In particular, we show that it is given by a reduction
of a finite dimensional Poisson manifold.
In \secref{sec:Coulomb} we show that the Coulomb branch $\mathcal M_C$
is isomorphic to the bow variety $\cM$.
In \secref{sec:hanany-witten} we study Hanany-Witten transition as an
isomorphism between two bow varieties. We derive several applications
mentioned above.

\subsection*{Acknowledgments}

It is clear that the authors owe a lot to Sergey Cherkis. They are
grateful to him for his explanation on bow varieties over years.
The first-named author also thanks
A.~Braverman
and
M.~Finkelberg
for useful discussions.
We further thank
M.~Finkelberg
for pointing out mistakes on flavor symmetry in an earlier version.

The research of H.N.\ is supported by JSPS Kakenhi Grant Numbers
23224002, 
23340005, 
24224001, 
25220701. 


\begin{NB}
By the series of the recent works of [N, BFN], a mathematical definition of Coulomb branches of $3$-dimensional $\mathcal{N}=4$ gauge theories has been achieving.
In this paper, we consider theories whose Higgs branches are given as type $A$ quiver varieties, that is, $(G, \mathbb{M}) = (\GL(\bv), \mathbb{M}(\bv,\bw))$.
And our goal is to construct Coulomb branches of these theories in another way of [N, BFN].

In physics, these theories are studied by [HW] in the finite type $A$ case, and by [dBHOO] in the affine type $A$ case. 
Particularly, diagrams in [HW] is closely related to quiver varieties, which give Higgs branches, for example
\begin{align*}
\begin{xy}
(40,14)*{\text{the Higgs branch of vacua of $U(n)$ gauge theory with $r$ flavors}},
(-3,3)*{0},
(5,3)*{n},
(15,3)*{n},
(25,3)*{\cdots},
(35,3)*{n},
(45,3)*{n},
(53,3)*{0},
(25,-14)*{\underbrace{\hspace{35mm}}_{r}},
(75,5)*{0},
(90,5)*{\CC^n},
(90,-10)*{\CC^r},
(105,5)*{0},
\ar @{-} (0,0);(50,0)
\ar @{-} (0,10);(0,-10)
\ar @{.} (10,10);(10,-10)
\ar @{.} (20,10);(20,-10)
\ar @{.} (30,10);(30,-10)
\ar @{.} (40,10);(40,-10)
\ar @{-} (50,10);(50,-10)
\ar (60,0);(70,0)
\ar (78,6);(87,6)
\ar (87,4);(78,4)
\ar (93,6);(102,6)
\ar (102,4);(93,4)
\ar (89, 2);(89,-7)
\ar (91,-7);(91,2)
\end{xy}
\end{align*}
Cherkis justified this relation by using bow varieties; in order to obtain Higgs branches, NS5-branes (solid vertical lines) are displaced into quivers and D5-branes (dashed vertical lines) are into moduli spaces of solutions of Nahm's equations:
\begin{align*}
\begin{xy}
(-3,3)*{m_1},
(3,3)*{m_2},
(20,0)*{\CC^{m_1}},
(35,0)*{\CC^{m_2}},
(52,3)*{m},
(58,3)*{m},
(75,0)*{\CC^m},
(82.5,0)*{\bullet},
(90,0)*{\CC^m},
(95,0)*{=},
(100,3)*{\CC^m},
(107.5,0)*{\times},
(115,3)*{\CC^m},
(107.5,-10)*{\CC},
\ar @{-} (0,10);(0,-10)
\ar @{-} (-5,0);(5,0)
\ar (10,0);(15,0)
\ar (24,1);(31,1)
\ar (31,-1);(24,-1)
\ar @{.} (55,10);(55,-10)
\ar @{-} (50,0);(60,0)
\ar (65,0);(70,0)
\ar @{~} (77,0);(88,0)
\ar @{~} (102,0);(113,0)
\ar (106.5,-1);(106.5,-8)
\ar (108.5,-8);(108.5,-1)
\end{xy}
\end{align*}
Remark that quiver varieties consist of these quivers and moduli spaces of solutions of Nahm's equations.
\begin{align*}
\begin{xy}
(0,5)*{0},
(15,5)*{\CC^n},
(15,-10)*{\CC^r},
(30,5)*{0},
(35,5)*{=},
(40,5)*{0},
(55,5)*{\CC^n},
(62.5,5)*{\bullet},
(70,5)*{\CC^n},
(75,5)*{\cdots},
(80,5)*{\CC^n},
(87.5,5)*{\bullet},
(95,5)*{\CC^n},
(110,5)*{0},
(75,0)*{\underbrace{\hspace{30mm}}_{r}},
\ar (3,6);(12,6)
\ar (12,4);(3,4)
\ar (18,6);(27,6)
\ar (27,4);(18,4)
\ar (14, 2);(14,-7)
\ar (16,-7);(16,2)
\ar (43,6);(52,6)
\ar (52,4);(43,4)
\ar @{~} (57,5);(68,5)
\ar @{~} (82,5);(93,5)
\ar (98,6);(107,6)
\ar (107,4);(98,4)
\end{xy}
\end{align*}

On the other hand, the mirror of a theory given by Hanany-Witten diagram is obtained by exchanging NS5-branes and D5-branes, and mirror symmetry states
\begin{align*}
\mathcal{M}^A_{\mathrm{Coulomb}} = \mathcal{M}^B_{\mathrm{Higgs}}, \ \ \ \mathcal{M}^A_{\mathrm{Higgs}} = \mathcal{M}^B_{\mathrm{Coulomb}}.
\end{align*}
Summarizing these things, we can construct Coulomb branches of theories whose Higgs branches are given by quiver varieties, by exchanging quivers and moduli spaces of solutions of Nahm's equations;
\begin{align*}
\begin{xy}
(0,0)*{\text{theory A}},
(63,0)*{\mathcal{M}^A_{\mathrm{Higgs}} = \text{a quiver variety}},
(0,-20)*{\text{theory B}},
(72,-20)*{\mathcal{M}^B_{\mathrm{Higgs}} = \text{a bow variety} = \mathcal{M}^A_{\mathrm{Coulomb}}},
(25,8)*{\text{NS5}\leadsto \text{quiver}},
(25,4)*{\text{D5}\leadsto \text{Nahm}},
(25,-12)*{\text{NS5}\leadsto \text{quiver}},
(25,-16)*{\text{D5}\leadsto \text{Nahm}},
(-10,-10)*{\text{NS5}\leftrightarrow\text{D5}},
(65,-10)*{\text{quiver}\leftrightarrow\text{Nahm}},
\ar (10,0);(40,0)
\ar @{<->} (0,-5);(0,-15)
\ar @{<->} (50,-5);(50,-15)
\ar (10,-20);(40,-20)
\end{xy}
\end{align*}
And, we check that these Coulomb branches indeed satisfy the mathematical definition of [N, BFN].
\end{NB}%

\section{Bow varieties}\label{sec:bow-varieties}

Let us review the definition of a bow variety of affine type
$A_{\ell-1}$ ($\ell\ge 1$) \cite{MR2824478} and give its quiver
description. We will regard a bow variety as an affine algebraic
variety with holomorphic symplectic form on its regular locus. Then
its quiver description does not loose any information.
Therefore readers who are not interested in the original definition
can safely skip the first part \subsecref{subsec:original}, and regard
the quiver description in the second part \subsecref{subsec:quiver} as
the definition. We will use the original definition only in \subsecref{subsubsec:homological-degree}.

Let us first explain the data which are common in both original
definition and the quiver description. It is a bow diagram.

We take an affine Dynkin quiver of type $A_{\ell-1}$ with the
anticlockwise orientation. We understand $A_0$ is the Jordan
quiver.
We replace vertices of quiver by \emph{wavy lines} 
\( 
\begin{tikzpicture}
    \draw[decorate, decoration = {snake, segment
        length=2mm, amplitude=.4mm}] 
    (0,0) -- (1,0);
\end{tikzpicture}
\)
possibly with $\times$ on them. A wavy line without $\times$ is
allowed. 
Let $\mathcal I$ be the set of wavy lines.
We call $\times$ a \emph{$\xl$-point}. Let $\Lambda$ denote the set of $\xl$-points. (It was called a $\lambda$-point in \cite{MR2824478}, but we keep $\lambda$ for a weight.)
We orient wavy lines in the anticlockwise direction.
Wavy lines are denoted by $\xp$, $\xp'$, \dots, etc and we identify
them with $0$, $1$, \dots, $\ell-1$, or elements in $\ZZ/\ell\ZZ$. For
a later purpose, we index them in the clockwise order.
We also index the arrow by $0$, $1$, \dots, $\ell-1$, where $i$th
arrow is between the end of $i$th wavy line and the start of the $(i+1)$th wavy line.

Let $n$ be the number of $\xl$-points. We index them as $0$, $1$,
\dots, $n-1$ in the anticlockwise order. The number $n$ will play the
role of \emph{level} of affine weights. When bow varieties will be
identified with Coulomb branches of gauge theories, level and rank are
exchanged.

A \emph{segment} of a wavy line is a part which is cut out by
$\xl$-points.
%
\begin{NB}
Let $r_i$ denote the number of $\times$ on the $i$th wavy line.
We index $\times$ by a pair $(i,\alpha)$ with $0\le i\le k$,
$\alpha=1,\dots, r_i$. Segments are indexed by $(i,\alpha)$ with
$\alpha=0,\dots,r_i$.
\end{NB}%
An $\xl$-point is denoted by $\xl$, and a segment by
$\zeta$. Let $\vout{\zeta}$, $\vin{\zeta}$ denote the starting and
ending points of $\zeta$ respectively, i.e., 
\(
\vout{\zeta}
\begin{tikzpicture}
    \draw[decorate, decoration = {snake, segment
        length=2mm, amplitude=.4mm},->] 
    (0,0) -- (2.5em,0) node[midway,above] {$\zeta$};
\end{tikzpicture}
\vin{\zeta}.
\)
Here $\vout{\zeta}$, $\vin{\zeta}$ are either end points of wavy lines
or $\xl$-points.
An oriented edge is denoted by $h$, and its starting and ending points
are denoted also by $\vout{h}$, $\vin{h}$, i.e.,
\(
\vout{h}
\begin{tikzpicture}
    \draw[->] (0,0) -- (2.5em,0) node[midway,above] {$h$};
\end{tikzpicture}
\vin{h}
.\)
Here $\vout{h}$, $\vin{h}$ are end points of wavy lines. There should
be no fear of confusion for using the same notation for end points of
segments and edges.

We assign a nonnegative integer $R(\zeta)$ to each segment $\zeta$ of
wavy lines. These are discrete data for a bow variety. We put
$R(\zeta)$ next to $\zeta$.
Figure~\ref{fig:Bow}(a) is the bow diagram for the example in
Figure~\ref{fig:quiverBow} with $R(\zeta) = \dim V_i$, $\dim V_i' =
\bv_i$, $\bv_i'$ etc.
A simplified version of the diagram suffices for most of our purposes:
we replace an arrow by $\boldsymbol\medcirc$ and remove waves from lines. See Figure~\ref{fig:Bow}(b).

We understand $\bv_i$ as a part of the bow diagram. Considering
$(\bv_0,\bv'_0,\dots)$ as a vector, we call it the \emph{dimension
  vector} of the bow diagram.
\begin{NB}
    Added on Apr.13.
\end{NB}%

For our later purpose, it is better to index $\xl$-points as $0$,
$1$, \dots, $n-1$.
\begin{NB}
Then one-dimensional vector spaces $\WW_\xl$ are denoted by $\WW_0$, $\WW_1$, \dots, $\WW_{n-1}$. 
\end{NB}%
We take the $j$th $\xl$-point and index vector spaces on segments
as $V_j$, $V_j'$, \dots, anticlockwise until we meet the next
$\xl$-point. 
See Figure~\ref{fig:quiverBow}.
\begin{NB}
    If we need to put too many primes, we will consider how to write.
\end{NB}%

The bow diagram originates from a brane configuration in type IIB
string theory \cite{MR1451054}. Oriented edges $\to$ correspond to
NS $5$-branes, while $\xl$-points correspond to Dirichlet
$5$-branes. Integers $R(\zeta)$ represent the number of Dirichlet
$3$-branes running between either type of branes. The configuration is
drawn as Figure~\ref{fig:brane}, where radial full lines are Dirichlet
$5$-branes, radial dotted lines are NS $5$-branes, and circular
lines are Dirichlet $3$-branes. See \cite{MR1454291} for affine cases.
Note that the lines for NS and Dirichlet 5-branes are opposite of
\cite{MR1451054,MR1454291}, as bow varieties are Coulomb branches of
dual theories by \cite{2011PhRvD..83l6009C}. We will not (and could
not) use branes, hence this difference will not be important later.

\begin{figure}[htbp]
    \centering
\begin{equation*}
\begin{tikzpicture}
    [scale=0.9, snake circle/.style = {thick, decorate, decoration={snake,segment
        length=2mm, amplitude=.4mm}}]
    \node (0,0) {$\vdots$};
    \draw[snake circle]
    (3,0) arc (0:10:3);
    \node at (5:3.5) {$\bv_{n-4}$};
    \node at (10:3) {$\boldsymbol\times$};
    \draw[snake circle] (10:3) arc (10:20:3);
    \node at (14:3.5) {$\bv_{n-3}$};
    \node at (20:3) {$\boldsymbol\times$};
    \draw[snake circle] (20:3) arc (20:30:3);
    \node at (24:3.5) {$\bv_{n-2}$};
    \draw[thick,->] (30:3) arc (30:45:3);
    \draw[snake circle] (45:3) arc (45:55:3);
    \node at (48:3.5) {$\bv'_{n-2}$}; 
    \node at (55:3) {$\boldsymbol\times$};
    \draw[snake circle] (55:3) arc (55:65:3);
    \node at (58:3.5) {$\bv_{n-1}$};
    \draw[thick,->] (65:3) arc (65:80:3);
    \draw[snake circle] (80:3) arc (80:90:3);
    \node at (82:3.5) {$\bv'_{n-1}$}; 
    \node at (90:3) {$\boldsymbol\times$};
    \draw[snake circle] (90:3) arc (90:105:3);
    \node at (98:3.5) {$\bv_0$};
    \draw[thick,->] (105:3) arc (105:120:3);
    \draw[snake circle] (120:3) arc (120:135:3);
    \node at (125:3.5) {$\bv'_0$};
    \draw[thick,->] (135:3) arc (135:150:3);
    \draw[snake circle] (150:3) arc (150:160:3);
    \node at (155:3.5) {$\bv''_0$}; 
    \node at (160:3) {$\boldsymbol\times$};
    \draw[snake circle] (160:3) arc (160:170:3);
    \node at (165:3.5) {$\bv_1$};
    \node at (170:3) {$\boldsymbol\times$};
    \draw[snake circle] (170:3) arc (170:179.5:3);
    \node at (175:3.5) {$\bv_2$};
\end{tikzpicture}
%
\begin{tikzpicture}[scale=0.9]
    \node (0,0) {$\vdots$};
    \draw[thick] (3,0) arc (0:15:3);
    \node at (5:3.5) {$\bv_{n-4}$};
    \node at (15:3) {$\boldsymbol\times$};
    \draw[thick] (15:3) arc (15:30:3);
    \node at (20:3.5) {$\bv_{n-3}$};
    \node at (30:3) {$\boldsymbol\times$};
    \draw[thick] (30:3) arc (30:45:3);
    \node at (35:3.5) {$\bv_{n-2}$};
    \node at (45:3) {$\boldsymbol\medcirc$};
    \draw[thick] (45:3) arc (45:60:3);
    \node at (50:3.5) {$\bv'_{n-2}$}; 
    \node at (60:3) {$\boldsymbol\times$};
    \draw[thick] (60:3) arc (60:75:3);
    \node at (65:3.5) {$\bv_{n-1}$};
    \node at (75:3) {$\boldsymbol\medcirc$};
    \draw[thick] (75:3) arc (75:90:3);
    \node at (82:3.5) {$\bv'_{n-1}$}; 
    \node at (90:3) {$\boldsymbol\times$};
    \draw[thick] (90:3) arc (90:108:3);
    \node at (100:3.5) {$\bv_0$};
    \node at (108:3) {$\boldsymbol\medcirc$};
    \draw[thick] (108:3) arc (108:126:3);
    \node at (118:3.5) {$\bv'_0$};
    \node at (126:3) {$\boldsymbol\medcirc$};
    \draw[thick] (126:3) arc (126:144:3);
    \node at (135:3.5) {$\bv''_0$}; 
    \node at (144:3) {$\boldsymbol\times$};
    \draw[thick] (144:3) arc (144:162:3);
    \node at (155:3.5) {$\bv_1$};
    \node at (162:3) {$\boldsymbol\times$};
    \draw[thick] (162:3) arc (162:179.5:3);
    \node at (170:3.5) {$\bv_2$};
\end{tikzpicture}
\end{equation*}
\caption{(a) A bow diagram of affine type $A_{\ell-1}$ and (b) its
  simplified version}
    \label{fig:Bow}
\end{figure}
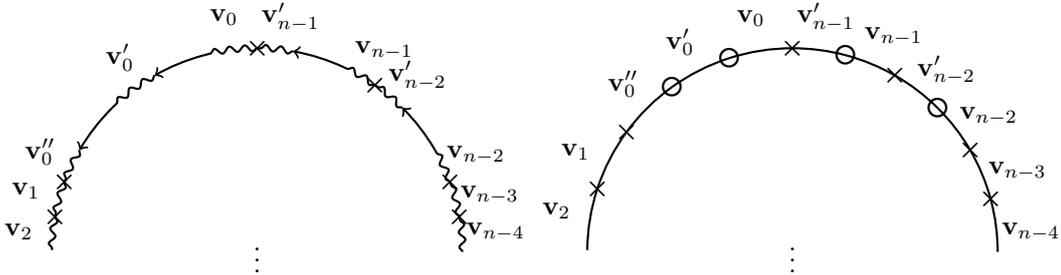

\begin{figure}[htbp]
    \centering
\begin{tikzpicture}[scale=0.8]
    \node (0,0) {$\vdots$};
    \foreach \r in {2.6,2.8,3.2,3.4}
    \draw [domain=0:15] plot ({\r*cos(\x)}, {\r*sin(\x)});
    \draw[dotted] (7.5:2.9) -- ++(7.5:0.2);
    \node at (6:4) {$\bv_{n-4}$};
    \draw[thick] (15:2) -- ++(15:2);
    \foreach \r in {2.5,2.7,3.3,3.5}
    \draw [domain=15:30] plot ({\r*cos(\x)}, {\r*sin(\x)});
    \draw[dotted] (22.5:2.9) -- ++ (22.5:0.2);
    \node at (21:4) {$\bv_{n-3}$};
    \draw[thick] (30:2) -- ++(30:2);
    \foreach \r in {2.6,2.8,3.2,3.4}
    \draw [domain=30:45] plot ({\r*cos(\x)}, {\r*sin(\x)});
    \draw[dotted] (37.5:2.9) -- ++ (37.5:0.2);
    \node at (36:4) {$\bv_{n-2}$};
    \draw[dashed] (45:2) -- ++(45:2);
    \foreach \r in {2.5,2.7,3.3,3.5}
    \draw [domain=45:60] plot ({\r*cos(\x)}, {\r*sin(\x)});
    \draw[dotted] (52.5:2.9) -- ++(52.5:0.2);
    \node at (51:4) {$\bv'_{n-2}$}; 
    \draw[thick] (60:2) -- ++(60:2);
    \foreach \r in {2.6,2.8,3.2,3.4}
    \draw [domain=60:75] plot ({\r*cos(\x)}, {\r*sin(\x)});
    \draw[dotted] (67.5:2.9) -- ++ (67.5:0.2);
    \node at (67:4) {$\bv_{n-1}$}; 
    \draw[dashed] (75:2) -- ++(75:2);
    \foreach \r in {2.5,2.7,3.3,3.5}
    \draw [domain=75:90] plot ({\r*cos(\x)}, {\r*sin(\x)});
    \draw[dotted] (82.5:2.9) -- ++(82.5:0.2);
    \node at (82:4) {$\bv'_{n-1}$}; 
    \draw[thick] (90:2) -- ++(90:2);
    \foreach \r in {2.6,2.8,3.2,3.4}
    \draw [domain=90:105] plot ({\r*cos(\x)}, {\r*sin(\x)});
    \draw[dotted] (97.5:2.9) -- ++ (97.5:0.2);
    \node at (98:4) {$\bv_0$}; 
    \draw[dashed] (105:2) -- ++(105:2);
    \foreach \r in {2.5,2.7,3.3,3.5}
    \draw [domain=105:135] plot ({\r*cos(\x)}, {\r*sin(\x)});
    \draw[dotted] (120:2.9) -- ++(120:0.2);
    \node at (120:4) {$\bv'_0$};
    \draw[dashed] (135:2) -- ++(135:2);
    \foreach \r in {2.6,2.8,3.2,3.4}
    \draw [domain=135:150] plot ({\r*cos(\x)}, {\r*sin(\x)});
    \draw[dotted] (142.5:2.9) -- ++(142.5:0.2);
    \node at (143:4) {$\bv''_0$}; 
    \draw[thick] (150:2) --  ++(150:2);
    \foreach \r in {2.5,2.7,3.3,3.5}
    \draw [domain=150:165] plot ({\r*cos(\x)}, {\r*sin(\x)});
    \draw[dotted] (157.5:2.9) -- ++(157.5:0.2);
    \node at (158:4) {$\bv_1$};
    \draw[thick] (165:2) --  ++(165:2);
    \foreach \r in {2.6,2.8,3.2,3.4}
    \draw [domain=165:180] plot ({\r*cos(\x)}, {\r*sin(\x)});
    \draw[dotted] (172.5:2.9) -- ++(172.5:0.2);
    \node at (175:4) {$\bv_2$};
\end{tikzpicture}
    \caption{Brane configuration}
    \label{fig:brane}
\end{figure}
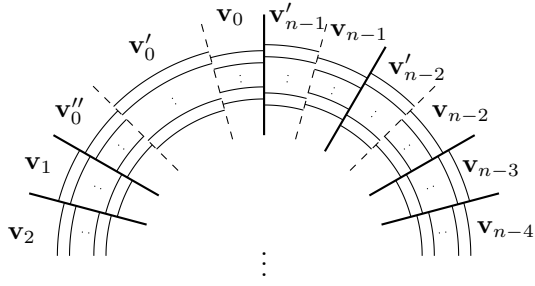

\begin{NB}
Let $n$ be the total number of $\times$. We index $\times$ from $0$ to
$n-1$. We use another indexing of segments based on $\times$. We take
the $j$th $\times$ and number segments as $(j,0)$, $(j,1)$, \dots,
anticlockwise until we meet the next $\times$.
We later take a vector space of dimension $R_i^\alpha$ for each
segment, but we use this new index, hence denote it by $V_j^p$.
See Figure~\ref{fig:Bow}.
\end{NB}%

We choose a parameter $\nu_\xp = (\nu^1_\xp,\nu^2_\xp,\nu^3_\xp)\in
\sqrt{-1}\RR^3$ for each wavy line (or arrow) $\xp$. In the quiver
description, we consider $\RR^3 = \RR\times\CC$ and denote it by
$\nu_\xp = (\nu^\RR_\xp,\nu^\CC_\xp)$ with $\nu^1_\xp =
\sqrt{-1}\nu^\RR_\xp$, $\nu_\xp^2 + \sqrt{-1}\nu^3_\xp = \nu^\CC_\xp$.
When we drop $\xp$, it means a vector, e.g., $\nu \in
(\sqrt{-1}\RR^3)^{\mathcal I}$.
When we use $\nu^\RR_\xp$ in the quiver description, we assume $\nu^\RR_\xp\in\mathbb Q$.

\subsection{Original definition}\label{subsec:original}

We further add the following data.

Let $\xp$ be a wavy line. We assign an interval
$[\xl_{\xp,L},\xl_{\xp,R}]$.
We assume those intervals are mutually disjoint.
Let $r_\xp$ be the number of $\times$ on $\xp$. 
We choose $\xl_{\xp,L} < \xl^1_\xp < \dots <
\xl^{r_\xp}_\xp < \xl_{\xp,R}$ corresponding to
$\xl$-points in $\xp$ in the anticlockwise order.
We set $\xl^0_\xp = \xl_{\xp,L}$, $\xl^{r_\xp+1}_\xp
= \xl_{\xp,R}$.
We now regard $\Lambda\subset \RR$. We identify a segment $\zeta$ with
the closed interval $[\vout{\zeta}, \vin{\zeta}]$ in $\RR$.

We assign one dimension complex vector space $\WW_\xl$ with a
hermitian inner product for each $\xl\in\Lambda$.
When there is no fear of confusion, we simply denote it by $\CC$,
e.g., as in Figure~\ref{fig:quiverBow}.
In addition, we choose a hermitian vector bundle $E_\zeta$
of rank $R(\zeta)$ for each segment $\zeta$.

Those data satisfy the matching condition at $\xl$-points.  Let us
suppose two segments $\zeta^-$, $\zeta^+$ meet at $\xl$. We suppose
$\vin{\zeta^-}=\xl=\vout{\zeta^+}$, i.e.,
\(
\begin{tikzpicture}[baseline=0pt]
    \draw[decorate, decoration = {snake, segment
        length=2mm, amplitude=.4mm},->] 
    (0,0) -- (1,0);
    \node at (0.5,0) {$\boldsymbol\times$};
    \node at (0.2,0.3) {$\zeta^-$};
    \node at (0.95,0.3) {$\zeta^+$};
    \node at (0.5,0.3) {$\xl$};
\end{tikzpicture}
.\)
Let us consider the difference $\Delta R(\xl) := |R(\zeta^-) -
R(\zeta^+)|$ of ranks of two vector bundles defined over $\zeta$,
$\zeta'$.
We suppose the fiber of the smaller (or equal to) rank vector bundle
is a subspace of the fiber of the bigger (or equal to) rank vector
bundle, i.e.,
\begin{equation*}
    \begin{split}
    E_{\zeta^-}|_{\xl} \subset E_{\zeta^+}|_{\xl} &
    \text{ if } R(\zeta^-) \le R(\zeta^+), \\
    E_{\zeta^-}|_{\xl} \supset
    E_{\zeta^+}|_{\xl} &
    \text{ if } R(\zeta^-) \ge R(\zeta^+).
    \end{split}
\end{equation*}
Furthermore we endow a structure of an irreducible $\su(2)$
representation on the orthogonal complement
$(E_{\zeta^-}|_{\xl})^\perp$ or $(E_{\zeta^+}|_{\xl})^\perp$. Its
dimension is $\Delta R(\xl)$. We thus have three endomorphisms
$\rho_1$, $\rho_2$, $\rho_3$ on the orthogonal complement satisfying
$\rho_1 = [\rho_2,\rho_3]$ and its cyclic permutation. We have
$\mathrm{U}(1)$-choices of identifications with the $\su(2)$
representation, which we consider as parts of data.
(In the quiver description below, we break symmetries between
$1$,$2$,$3$ and take a triangular decomposition $\mathfrak
n\oplus\mathfrak h\oplus\mathfrak n^-$ of $\algsl(2) =
\su(2)\otimes\CC$. Then the identification will be identified with an
isomorphism from the highest weight space to $\WW_x$.)
When there is no fear of the confusion, we drop the subscript $\zeta$
from $E_\zeta$. For example, we do not need to specify $\zeta$ for a
fiber of $E_\zeta$ at a point $\notin\Lambda$, or a point
$\xl\in\Lambda$ with $\Delta R(\xl) = 0$.

Now a \emph{bow solution} consists of the following data:
\begin{enumerate}
      \item A hermitian connection $\nabla$ and skew hermitian
    endomorphisms $T_1$,$T_2$,$T_3$ on the vector bundle $E_\zeta$
    over the open segment $(\vout{\zeta},\vin{\zeta})$
    satisfying Nahm's equation
\begin{equation*}
    \frac{\nabla}{ds} T_1 + [T_2, T_3] = 0,
    \quad\text{etc.}
\end{equation*}
\begin{NB}
We write $\nabla^\zeta$, $T_1^\zeta$, $T_2^\zeta$, $T_3^\zeta$ when we
want to emphasize $\zeta$.
\end{NB}%

  \item A pair of linear maps 
\begin{equation*}
    B^{LR} \colon E|_{\vout{h}} \to E|_{\vin{h}}, \qquad
    B^{RL} \colon E|_{\vin{h}}\to E|_{\vout{h}},
\end{equation*}
corresponding to an oriented arrow $h$ from $\vout{h}$ to
$\vin{h}$. Note $B^{RL}$ goes to the opposite direction. 
Note also that $\vout{h} = \xl_{\xp,R}$, $\vin{h} = \xl_{\xp+1,L}$
for a wavy line $\xp$. Therefore they are not $\xl$-points, as
\(
\begin{tikzpicture}
    \draw[decorate, decoration = {snake, segment
        length=2mm, amplitude=.4mm}] 
    (0,0) -- (1em,0);
    \draw[->] (1em,0) -- (3.5em,0) 
    node[at end,above] {$\vin{h}$}
    node[at start,above] {$\vout{h}$};
    \draw[decorate, decoration = {snake, segment
        length=2mm, amplitude=.4mm}] 
    (3.5em,0) -- (4.5em,0);
\end{tikzpicture}
.\)
\begin{NB}
When we want to emphasize $h$, we write $B^{LR}_h$, $B^{RL}_h$.
\end{NB}%

  \item A pair of linear maps
\begin{equation*}
    \begin{split}
        & I\colon \WW_\xl\to E|_{\xl},
    \\
    & J\colon E|_\xl \to \WW_\xl
    \end{split}
\end{equation*}
for a point $\xl\in\Lambda$ such that $\Delta R(\xl) = 0$.
\end{enumerate}

We require that $(\nabla,T_1,T_2,T_3,\linebreak[2] B^{LR},
B^{RL},\linebreak[2] I, J)$ satisfy the following conditions:

\begin{aenume}
      \item $\nabla$,$T_1$,$T_2$,$T_3$ extend smoothly at
    $\xl_{\xp,L}$, $\xl_{\xp,R}$. Moreover they together with
    $B^{LR}$, $B^{RL}$ satisfy the hyper-K\"ahler moment map equation
    \begin{equation*}
        \begin{split}
            -\frac{\sqrt{-1}}2\left\{
            B^{LR} (B^{LR})^\dagger - (B^{RL})^\dagger B^{RL}
            \right\} + T_1(\xl_{\xp,L}) &=
            \nu^1_\xp \id,
            \\
            B^{LR} B^{RL} + (T_2 + \sqrt{-1}T_3)(\xl_{\xp,L}) &= 
            (\nu^2_\xp + \sqrt{-1}\nu^3_\xp) \id
        \end{split}
    \end{equation*}
    at $\xl_{\xp,L}$ and
    \begin{equation*}
        \begin{split}
            -\frac{\sqrt{-1}}2 \left\{
              B^{RL} (B^{RL})^\dagger - (B^{LR})^\dagger B^{LR}\right\}
            - T_1(\xl_{\xp,R}) &=
            0, 
            \\
            - B^{RL} B^{LR} - (T_2 + \sqrt{-1} T_3)(\xl_{\xp,R}) &=
            0 
        \end{split}
    \end{equation*}
    at $\xl_{\xp,R}$.

      \item\label{item:matching} Let $\xl\in\Lambda$. Let $\zeta^\pm$
    as before. We suppose $R(\zeta^-) < R(\zeta^+)$. (The condition
    for $R(\zeta^-) > R(\zeta^+)$ is the same after exchanging
    $E_{\zeta^-}$ and $E_{\zeta^+}$.)
    We have $\nabla$, $T_1$, $T_2$,
    $T_3$ defined on $\zeta^-$, $\zeta^+$. Let us denote the former
    with $-$, and the latter with $+$.
    We require $\nabla^{\pm}$, $T_1^{-}$, $T_2^{-}$,
    $T_3^{-}$ extend smoothly at $\xl$. On the other hand,
    we assume $T_j^{+}$ ($j=1,2,3$) behaves as
    \begin{equation*}
        T_j^{+}(s) =
        \begin{pmatrix}
            T_j^{-}(s) + O(s-\xl) &
            O((s-\xl)^{\frac{\Delta R(\xl) - 1}2} \\
            O((s-\xl)^{\frac{\Delta R(\xl) - 1}2} &
            \frac{\rho_j}{s-\xl} + O(1)
        \end{pmatrix}.
    \end{equation*}

      \item\label{item:neutral} Consider a point $\xl$ with $\Delta
    R(\xl) = 0$. Let us put $\pm$ on connections and endomorphisms as
    above. We require that $\nabla^\pm$, $T_1^\pm$, $T_2^\pm$,
    $T_3^\pm$ are smooth at $\xl$, and they satisfy the hyper-K\"ahler
    moment map equation with $I$, $J$:
    \begin{equation*}
        \begin{split}
            -\frac{\sqrt{-1}}2 \left\{
            II^\dagger - J^\dagger J\right) - T_1^-(\xl) + 
          T_1^+(\xl) &= 0,
            \\
            I J - (T_2^-+\sqrt{-1}T_3^-)(\xl)
            + (T_2^++\sqrt{-1}T_3^+)(\xl)
            &= 0.
        \end{split}
    \end{equation*}
    \begin{NB}
        Do we have a possibility to perturb the equation ? Since this
        is not possible for the usual $\SU(N)$-monopoles, it should be
        irrelevant parameters. It is because we just shift $T_j^-(s)$
        by scalars. However $\underline\bw = 0$ is an exceptional
        case, and we have one parameter deformation/resolution. See
        *** below.
    \end{NB}%
\end{aenume}

There is a natural gauge equivalence on these solutions: A gauge
transformation $u$ consists of a unitary gauge transformation
$u_\zeta$ of $E_\zeta$, which is smooth on $[\vout{\zeta},
\vin{\zeta}]$ for each segment $\zeta$. Moreover for \(
\begin{tikzpicture}[baseline=0pt]
    \draw[decorate, decoration = {snake, segment
        length=2mm, amplitude=.4mm},->] 
    (0,0) -- (1,0);
    \node at (0.5,0) {$\boldsymbol\times$};
    \node at (0.2,0.3) {$\zeta^-$};
    \node at (0.95,0.3) {$\zeta^+$};
    \node at (0.5,0.3) {$\xl$};
\end{tikzpicture}
\), we assume $u_{\zeta^-}(\xl) = u_{\zeta^+}(\xl)$ when
$\Delta R(\xl) = 0$,
\begin{equation*}
    u_{\zeta^+}(\xl) =
    \begin{pmatrix}
        u_{\zeta^-}(\xl) & 0\\
        0 & \id_{(E_{\zeta^-}|_{\xl})^\perp}
    \end{pmatrix}
\end{equation*}
when $R(\zeta^-) < R(\zeta^+)$, and similarly when $R(\zeta^-) >
R(\zeta^+)$. The action is given by
\begin{equation*}
    (u_\zeta \circ \nabla \circ u_\zeta^{-1}, u_\zeta T_1 u_\zeta^{-1},
    u_\zeta T_2 u_\zeta^{-1}, u_\zeta T_3 u_\zeta^{-1})
\end{equation*}
for $(\nabla,T_1,T_2,T_3)$ defined on $E_\zeta$. Also by
\begin{equation*}
    (u({\vin{h}}) B^{LR} u({\vout{h}})^{-1}, 
    u({\vout{h}}) B^{RL} u({\vin{h}})^{-1})
\end{equation*}
for $B^{LR}$, $B^{RL}$ corresponding to $h$, where $u({\vin{h}})$, $u({\vout{h}})$ are evaluation of appropriate $u_\zeta$ at points
$\vin{h}$, $\vout{h}$ respectively. And by
\begin{equation*}
    (u(\xl) I, J u(\xl)^{-1})
\end{equation*}
for $\xl\in\Lambda$ such that $\Delta R(\xl) = 0$. Since
$u_{\zeta^-}(\xl) = u_{\zeta^+}(\xl)$, we do not need to
specify a segment for $u(\xl)$.

Let $\cM_{\mathrm{bow}}$ be the space of gauge equivalence classes of
bow solutions. Let $\cM^{\mathrm{reg}}_{\mathrm{bow}}$ be the subset
consisting of classes of solutions which have trivial stabilizer.
One introduces an appropriate Sobolev space to give a topology on
$\cM_{\mathrm{bow}}$ such that $\cM^{\mathrm{reg}}_{\mathrm{bow}}$ is
an open subset with a hyper-K\"ahler structure, as an infinite
dimensional hyper-K\"ahler quotient. See e.g., \cite{MR3300314}. The
dimension of $\cM^{\mathrm{reg}}_{\mathrm{bow}}$ can be computed, but
we omit the formula as we will calculate it in the quiver description
in \propref{prop:dimension}.

\subsection{A quiver description of a bow variety}\label{subsec:quiver}

We assign one dimensional complex vector space $\WW_\xl$ for each
$\xl\in\Lambda$ as before. We do not put hermitian inner products. We
often omit $\xl$ as mentioned above.

In addition, we assign a vector space $V_\zeta$ for each segment
$\zeta$ of a wavy line with $\dim V_\zeta = R(\zeta)$.

\begin{NB}
Let $w_j$ denote the number of arrows between $j$th and $(j+1)$th $\times$, which is equal to the number of segments minus $1$.
\end{NB}%

For a quiver description, data consist of the following:
\begin{enumerate}
      \item A linear endomorphism $B\colon V_\zeta\to V_\zeta$ for
    each segment $\zeta$. When we want to emphasize $\zeta$, we denote
    it by $B_\zeta$.
    \begin{NB}
        We usually do not put the index $\zeta$ on $B$ for brevity.
    \end{NB}%

      \item\label{item:triangle} Let $\xl\in\Lambda$. Let
    $\zeta^-$, $\zeta^+$ be the adjacent segments so that
\(
\begin{tikzpicture}[baseline=0pt]
    \draw[decorate, decoration = {snake, segment
        length=2mm, amplitude=.4mm},->] 
    (0,0) -- (1,0);
    \node at (0.5,0) {$\boldsymbol\times$};
    \node at (0.2,0.3) {$\zeta^-$};
    \node at (0.95,0.3) {$\zeta^+$};
    \node at (0.5,0.3) {$\xl$};
\end{tikzpicture}
\)
is the anticlockwise orientation. We assign triple of linear maps
    \begin{gather*}
        A\colon V_{\zeta^-}\to V_{\zeta^+},
        \\
        a\colon \WW_\xl\to V_{\zeta^+}, \qquad
        b\colon V_{\zeta^-}\to \WW_\xl.
    \end{gather*}
    When we want to emphasize $x$, we denote them by $A_x$, $a_x$, $b_x$.

    \item A pair of linear maps
  \begin{equation*}
      C\colon V_{\vout{h}}\to V_{\vin{h}}, \qquad
      D\colon V_{\vin{h}}\to V_{\vout{h}},
  \end{equation*}
  for each arrow $h$.
  \begin{NB}
      $C$ is $B^{LR}$, $D$ is $B^{RL}$.
  \end{NB}%
  Here we regard $\vout{h}$, $\vin{h}$ as segments containing the
  points and use the notation $V_{\vout{h}}$, $V_{\vin{h}}$ for brevity.

  When we want to emphasize $h$, we denote them by $C_h$, $D_h$.
\end{enumerate}

\begin{NB}
Here is an earlier version of the definition of $\mathbb M$.

Let $\mathbb M$ denote the space of all linear maps $A$, $B$, $a$,
$b$, $C$, $D$, i.e.,
\begin{equation*}
    \begin{split}
        \mathbb M := & \bigoplus_\zeta \End(V_\zeta) \oplus
        \bigoplus_\xl \Hom(V_{\zeta^-},V_{\zeta^+})
        \oplus \Hom(\WW_\xl,V_{\zeta^+})\oplus
        \Hom(V_{\zeta^-},\WW_\xl)
        \\
        & \qquad\qquad \oplus \bigoplus_h
        \Hom(V_{\vout{h}},V_{\vin{h}})\oplus
        \Hom(V_{\vin{h}},V_{\vout{h}}).
    \end{split}
\end{equation*}
\end{NB}%

We require the following conditions:
\begin{aenume}
      \item Let $\xl$, $\zeta^\pm$ as in \ref{item:triangle} above.
    As a linear map $V_{\zeta^-} \to V_{\zeta^+}$ we have
    \begin{equation*}
        B_{\zeta^+} A - A B_{\zeta^-} + ab = 0.
    \end{equation*}
    \begin{NB}
        More precisely, the first $B$ is an endomorphism on
        $B_{\zeta^+}$, while the second $B$ is on $B_{\zeta^-}$.
    \end{NB}%
    They satisfy the two conditions (S1),(S2):
    \begin{description}
          \item[(S1)] There is no nonzero subspace $0\neq S\subset
        V_{\zeta^-}$ with $B_{\zeta^-}(S) \subset S$, $A(S) = 0 = b(S)$.

          \item[(S2)] There is no proper subspace $T\subsetneq
        V_{\zeta^+}$ with $B_{\zeta^+}(T) \subset T$, $\Ima A + \Ima a
        \subset T$.
    \end{description}

    \item As endomorphisms on $V_{\vin{h}}$ and $V_{\vout{h}}$, we have
  \begin{equation*}
      CD + B_{\vin{h}} = \nu^\CC_{\vin{h}},
      \qquad
      -DC - B_{\vout{h}} = 0,
  \end{equation*}
  respectively. We identify $\vin{h}$ with the wavy line containing
  it, and consider the corresponding parameters $\nu$. 

  \begin{NB}
      Here is my convention. Take $\xp$ with $\xp = \vout{h} =
      \vin{h'}$. Then $C_{h'} D_{h'} + B_\xp = \nu^\CC_\xp$, $0
      = - D_{h} C_{h} - B_\xp$. Removing $B_\xp$, we get $C_{h'}
      D_{h'} - D_{h} C_{h} = \nu^\CC_\xp$. Hence \[C_{h'} D_{h'} -
      D_{h} C_{h} = \nu^\CC_\xp.\]
  \end{NB}%

  By these equations, we can safely remove $B_\zeta$ when $\zeta$ is
  connected with an arrow $h$. We will usually remove $B_\zeta$ when
  $\zeta$ is connected with arrows in both sides, i.e., when $\zeta$
  is a wavy line without $\xl$-points. But we keep $B_\zeta$
  otherwise. See $V_0'$ at Figure~\ref{fig:quiverBow}.
\end{aenume}

We introduce the following space of linear maps:
\begin{equation*}
    \begin{split}
        \mathbb M := & 
        \bigoplus_\xl 
        \End(V_{\zeta^-_\xl}) \oplus \End(V_{\zeta^+_\xl}) \oplus
        \Hom(V_{\zeta^-_\xl},V_{\zeta^+_\xl}) \oplus
        \Hom(\WW_\xl,V_{\zeta^+_\xl})\oplus \Hom(V_{\zeta^-_\xl},\WW_\xl)
        \\
        & \qquad\qquad \oplus \bigoplus_h
        \Hom(V_{\vout{h}},V_{\vin{h}})\oplus
        \Hom(V_{\vin{h}},V_{\vout{h}}).
    \end{split}
\end{equation*}
We denote segments in \ref{item:triangle} by $\zeta^\pm_\xl$ and
elements in the component $\End(V_{\zeta^-_\xl})$,
$\End(V_{\zeta^+_\xl})$ by $B_{\zeta^-_\xl}$ and $B_{\zeta^+_\xl}$ as
they also depend on $\xl$. Note that $\zeta^+_\xl = \zeta^-_{\xl'}$ if         
\begin{tikzpicture}[baseline=0pt]
            \draw[decorate, decoration = {snake, segment
              length=2mm, amplitude=.4mm},->] 
            (0,0) -- (1.5,0);
            \node at (0.4,0) {$\boldsymbol\times$};
            \node at (0.4,0.3) {$\xl$};
            \node at (0.75,0.3) {$\zeta$};
            \node at (1.1,0) {$\boldsymbol\times$};
            \node at (1.1,0.35) {$\xl'$};
\end{tikzpicture}.
Hence we may have two $B$ assigned for a single segment $\zeta$. But
we will kill one of them by the relation soon.

We define
\begin{equation*}
    \begin{split}
    & \mu = (\mu_1,\mu_2) \colon\mathbb M\to \mathbb N:= 
    \bigoplus_\xl
        \Hom(V_{\zeta^-_\xl},V_{\zeta^+_\xl})\oplus
        \bigoplus_\zeta \End(V_\zeta)
\\
    & \mu_1\colon (A_\xl,B_{\zeta^+_\xl},B_{\zeta^-_\xl},a_\xl,b_\xl,C_h,D_h) \mapsto
    B_{\zeta^+_\xl} A_\xl - A_\xl B_{\zeta^-_\xl} + a_\xl b_\xl,
    \end{split}
\end{equation*}
and the second component $\mu_2$ is depending on types of two ends of
the wavy line $\zeta$:
\begin{gather*}
    - B_{\zeta^+_\xl} + B_{\zeta^-_{\xl'}} \quad
    \text{if }
        \begin{tikzpicture}[baseline=0pt]
            \draw[decorate, decoration = {snake, segment
              length=2mm, amplitude=.4mm},->] 
            (0,0) -- (1.5,0);
            \node at (0.4,0) {$\boldsymbol\times$};
            \node at (0.4,0.3) {$\xl$};
            \node at (0.75,0.3) {$\zeta$};
            \node at (1.1,0) {$\boldsymbol\times$};
            \node at (1.1,0.35) {$\xl'$};
        \end{tikzpicture},
        \qquad
        C_h D_h - D_{h'} C_{h'} - \nu^\CC_{\vin{h}} \quad\text{if }
        \begin{tikzpicture}[baseline=0pt]
            \draw[->](0,0) -- (0.5,0);
            \draw[decorate, decoration = {snake, segment
              length=2mm, amplitude=.4mm},->] 
            (0.5,0) -- (1,0);
            \draw[->](1,0) -- (1.5,0);
            \node at (0.25,0.3) {$h$};
            \node at (0.75,0.3) {$\zeta$};
            \node at (1.25,0.3) {$h'$};
        \end{tikzpicture},
        \\
        C_h D_h + B_{\zeta^-_\xl} - \nu^\CC_{\vin{h}} \quad\text{if }
        \begin{tikzpicture}[baseline=0pt]
            \draw[->] (0,0) -- (0.5,0);
            \node at (0.25,0.3) {$h$};
            \draw[decorate, decoration = {snake, segment
              length=2mm, amplitude=.4mm},->] 
            (0.5,0) -- (1.5,0);
            \node at (0.8,0.3) {$\zeta$};
            \node at (1.1,0) {$\boldsymbol\times$};
            \node at (1.1,0.3) {$\xl$};
        \end{tikzpicture},
        \qquad
         - B_{\zeta^+_\xl} - D_h C_h \quad\text{if }
        \begin{tikzpicture}[baseline=0pt]
            \draw[->] (1,0) -- (1.5,0);
            \node at (1.25,0.3) {$h$};
            \draw[decorate, decoration = {snake, segment
              length=2mm, amplitude=.4mm},->] 
            (0,0) -- (1,0);
            \node at (0.7,0.3) {$\zeta$};
            \node at (0.4,0) {$\boldsymbol\times$};
            \node at (0.4,0.3) {$\xl$};
        \end{tikzpicture}.
\end{gather*}
We will consider $\mu^{-1}(0)$, and the original $B_{\zeta}$ is given by $B_{\zeta^+_\xl} = B_{\zeta^-_\xl}$. We will use this convention unless we need to distinguish $B_{\zeta^+_\xl}$ and $B_{\zeta^-_\xl}$.

\begin{NB}
This was an earlier definition:

We denote the left hand side of the defining equations as $\mu$, i.e.,
\begin{equation*}
    \begin{split}
    \mu\colon\mathbb M\to \mathbb N:= & 
    \bigoplus_\xl \Hom(V_{\zeta^-},V_{\zeta^+})
    \oplus \bigoplus_h \End(V_{\vout{h}})\oplus \End(V_{\vin{h}})
\\
    & (A,B,a,b,C,D) \mapsto
    \left(B_{\zeta^+} A - A B_{\zeta^-} + ab,
      CD + B_{\vin{h}} - \nu^\CC_{\vin{h}}, - DC - B_{\vout{h}}
      \right).
    \end{split}
\end{equation*}
\end{NB}%

We have a natural group action of $\GV := \prod \GL(V_\zeta)$ by conjugation,
that is given by
\begin{equation*}
    (g_{\zeta^\pm_\xl} B_{\zeta^\pm_\xl} g_{\zeta^\pm_\xl}^{-1}, 
    g_{\zeta^+_\xl} A_\xl g_{\zeta^-_\xl}^{-1},
    g_{\zeta^+_\xl}a_\xl, b_\xl g_{\zeta^-_\xl}^{-1}, 
    g_{\vin{h}} C g_{\vout{h}}^{-1},
    g_{\vout{h}} D g_{\vin{h}}^{-1}).
\end{equation*}

\begin{aenume}
    \setcounter{enumi}{2}
      \item We impose an additional (semi)stability condition with
    respect to the group action of $\GV$ as follows.

    Let $\widetilde\cM$ denote the variety consisting of
    $(A,B,a,b,C,D)\in \mu^{-1}(0)$ satisfying (S1),(S2). As we will
    review \propref{prop:triangle-Nahm} below, it is an affine
    algebraic variety. We assume $\nu^\RR_\xp$ is an integer and
    consider a character
    \begin{equation*}
        \chi\colon \GV 
         \to \CC^\times; (g_\zeta) \mapsto
        \prod (\det g_\zeta)^{-\nu^\RR_\xp},
    \end{equation*}
    where $\zeta$ is the last segment in a wavy line $\xp$,
\(
\begin{tikzpicture}[baseline=0pt]
    \draw[decorate, decoration = {snake, segment
        length=2mm, amplitude=.4mm}] 
    (0,0) -- (2.4em,0);
    \node at (1.8em,.7em) {$\zeta$};
    \node at (1.2em,0) {$\boldsymbol\times$};
    \draw[->] (2.4em,0) -- (3.6em,0);
\end{tikzpicture}
,\) and the product is only over the last segments.
Let $\GV$ act on $\widetilde\cM\times\CC$ by
$g\cdot (m,z) = (g m, \chi(g)^{-1}z)$ for $g\in \GV$,
$m\in\widetilde\cM$, $z\in\CC$.
We say $m \in\widetilde\cM$ is \emph{$\nu^\RR$-semistable} if the
closure of the orbit $\GV(m,z)$ does not intersect with
$\widetilde\cM\times\{0\}$ for $z\neq 0$.
We say $m$ is \emph{$\nu^\RR$-stable} if the orbit $\GV(m,z)$ is closed
and the stabilizer of $(m,z)$ is finite for $z\neq 0$.

It is clear from \propref{prop:numerical} below that the
(semi)stability condition is not effected under the change $\nu^\RR\to
c\nu^\RR$ for $c\in \ZZ_{>0}$. Hence the (semi)stability condition is
well-defined for rational $\nu^\RR$.
\end{aenume}

\begin{NB}
    The differential $d\chi$ of the character $\chi$ is $\xi\mapsto
    - \nu_\xp^\RR \tr \xi$. From a general statement we have a
    homeomorphism
    \[
       \widetilde\cM^{\mathrm{ss}}/\!\!\sim\, \simeq
       \mu_\RR^{-1}(\sqrt{-1}d\chi)/\prod \operatorname{U}(V_\zeta).
    \]
    Here $\mu_\RR$ is defined by
    \[
       \langle \mu_\RR(m),\xi\rangle = \frac12 (\sqrt{-1}\xi m,m).
    \]
    In our setting, $(C, C') = \Re\tr(C (C')^\dagger)$, etc. Hence
    \begin{equation*}
        \begin{split}
        & \langle \mu_\RR(C),\xi\rangle 
        = \frac12 \Re \tr(\sqrt{-1}\xi\cdot C C^\dagger)
        = \tr(\frac{\sqrt{-1}}2 \xi C C^\dagger),
        \\
        & \langle \mu_\RR(D),\xi\rangle
        = \frac12 \Re \tr(\sqrt{-1} \xi\cdot D D^\dagger)
        = \tr(- \frac{\sqrt{-1}}2 D\xi D^\dagger)
        = - \tr(\frac{\sqrt{-1}}2 D^\dagger D\xi).
        \end{split}
    \end{equation*}
    Combining these components, we have
    \begin{equation*}
        \mu_\RR(C,D,T) = 
          -\frac{\sqrt{-1}}2 
        \left\{CC^\dagger - D^\dagger D 
            \right\} + T_1(\xl_{\xp,L}),
    \end{equation*}
    via the identification $\bigoplus \mathfrak u(V_\zeta) \cong
    \bigoplus\mathfrak u(V_\zeta)^*$ via $\eta\mapsto
    -\tr(\eta\bullet)$. Then $\mu_\RR = \sqrt{-1}d\chi$ means
    \begin{equation*}
        - \sqrt{-1} \nu_\xp^\RR \tr\xi
        = - \tr\left( \xi\cdot\left(
        -\frac{\sqrt{-1}}2 
        \left\{CC^\dagger - D^\dagger D 
            \right\} + T_1(\xl_{\xp,L})\right)\right)
    \end{equation*}
    for any $\xi$. Since $\nu^1_\xp = \sqrt{-1}\nu_\xp^\RR$,
    this is equivalent to
    \begin{equation*}
        -\frac{\sqrt{-1}}2 
        \left\{CC^\dagger - D^\dagger D 
            \right\} + T_1(\xl_{\xp,L}) =
            \nu^1_\xp \id.
    \end{equation*}
\end{NB}%

Let $\widetilde\cM^{\mathrm{ss}}$ (resp.\
$\widetilde\cM^{\mathrm{s}}$) denote the set of
$\nu^\RR$-semistable (resp.\ $\nu^\RR$-stable) points.

We introduce the \emph{$S$-equivalence relation} $\sim$ on the
$\widetilde\cM^{\mathrm{ss}}$ by defining $m\sim m'$ if and only if
orbit closures $\overline{\GV m}$ and $\overline{\GV m'}$ intersect in
the $\widetilde\cM^{\mathrm{ss}}$. Let $\cM_{\mathrm{quiver}}$ denote
$\widetilde\cM^{\mathrm{ss}}/\!\!\sim$. It is well-known that
$\widetilde\cM^{\mathrm{ss}}/\!\!\sim$ is a geometric invariant theory
quotient, hence in particular has a structure of a quasi-projective
variety.
Let $\cM^{\mathrm{s}}_{\mathrm{quiver}}$ denote the open subvariety
of $\cM_{\mathrm{quiver}}$ given by $\widetilde\cM^{\mathrm{s}}/\GV$.

Let $\widetilde\cM_{\mathrm{sym}}$ denote the open subvariety of
$\mu_1^{-1}(0)$ consisting of points satisfying (S1),(S2), namely we
only impose the condition $B_{\zeta^+_\xl}A_\xl - A_\xl B_{\zeta^-_\xl} + a_\xl
b_\xl=0$ and (S1),(S2) for each $\xl$. We will review that it is a smooth
holomorphic symplectic manifold so that the second component $\mu_2$
is the moment map with respect to the $\GV$-action. (See
\subsecref{subsec:triangle}.) Therefore $\cM$ is a symplectic
reduction of $\widetilde\cM_{\mathrm{sym}}$ by $\GV$.
In fact, $\widetilde\cM_{\mathrm{sym}}$ is a product of two way parts ($C_h$, $D_h$) and triangle parts ($A_\xl$, $B_{\zeta^\pm_\xl}$, $a_\xl$, $b_\xl$). A two way part is $\Hom(V_{\vout{h}},V_{\vin{h}})\oplus \Hom(V_{\vin{h}},V_{\vout{h}})$, which is a symplectic vector space. Moreover $-C_h D_h$ and $D_h C_h$ are moment maps for the $\GL(V_{\vin{h}})$ and $\GL(V_{\vout{h}})$ action respectively. Therefore it is enough to check the assertion for the triangle part, and we will do in \subsecref{subsec:triangle}.

\begin{NB}
  Let us calculate the moment map from Poisson brackets. Our
  convention is as in \secref{sec:symplectic-structure}, i.e.,
  $\{ \langle\mu,\xi\rangle, f\} = \xi^* f$. Let $e_{ij}$ denote the
  $(i,j)$ matrix unit. We set $C_{ij} = \tr(e_{ij} C)$ as a function
  on
  $\Hom(V_{\vout{h}},
  V_{\vin{h}})\oplus\Hom(V_{\vin{h}},V_{\vout{h}})$. We have
  \begin{equation*}
    \begin{split}
      & \left\{ \langle\mu, e_{ij}\rangle, C_{kl} \right\} = 
    \left\{ -\sum_m C_{mj} D_{im}, C_{kl}\right\} = \delta_{il} C_{kj},\\
      & e_{ij}^* C_{kl} = e_{ij}^* \tr(e_{kl} C) = \tr(e_{kl} e_{ij} C)
      = \delta_{il} C_{kj},
    \end{split}
  \end{equation*}
  and similarly for $D_{kl}$.
\end{NB}%

Now we state an isomorphism between two descriptions.

\begin{Theorem}\label{thm:quiver_desc}
    There is a natural homeomorphism between the space
    $\cM_{\mathrm{bow}}$ of the gauge equivalence classes of bow solutions
    and the space $\cM_{\mathrm{quiver}}$ of the $S$-equivalence
    classes of quiver representations.
    Moreover it is an isomorphism of holomorphic symplectic manifolds
    between the regular locus $\cM^{\mathrm{reg}}_{\mathrm{bow}}$
    and the stable locus $\cM^{\mathrm{s}}_{\mathrm{quiver}}$.
\end{Theorem}

A general theory only guarantees that
$\cM^{\mathrm{s}}_{\mathrm{quiver}}$ is an orbifold, but we will later
show that a point in $\widetilde{\cM}_{\mathrm{quiver}}^{\mathrm{s}}$
has only trivial stabilizer (\lemref{lem:stab}). Therefore
$\cM^{\mathrm{s}}_{\mathrm{quiver}}$ is nonsingular.

This result is, more or less, already known. We first enlarge the
gauge equivalence in $\cM_{\mathrm{bow}}$ allowing complex
gauge not necessarily preserving hermitian inner products. Nahm's
equation is divided into real and complex equations
correspondingly. The real equation is solved by calculus of variation.
Then we solve the complex equation so that the space of complex gauge
equivalence classes can be identified with the space of `boundary
conditions'.
This step first appeared in \cite{MR769355} and subsequently used in
works on Nahm's equations \cite{MR987771} and many others.

The space of boundary conditions depends on the situation, but is
usually described in terms of Lie-theoretic language such as in
\cite{MR1072915,MR1438643}.
For unitary gauge groups in earlier works \cite{MR769355,MR987771} and
in this paper, they are the space of linear maps instead of
Lie-theoretic language.

In \cite{MR769355,MR987771}, the space of linear maps is further
identified with the space of based maps from $\CP^1$ to a flag variety
(of type $A$).
The second named author in \cite{Takayama} observed that the
identification is better understood using the language of
\emph{quiver}, i.e., drawing arrows for linear maps, and applied this
idea for Nahm's equation on the circle. This language is useful and
inevitable to study complicated situations like in this paper.

The remaining detail of the proof of \thmref{thm:quiver_desc} will be
given in \subsecref{subsec:quiver_proof} below.

We regard the bow variety as $\cM_{\mathrm{quiver}}$ equipped with a
structure of a quasiprojective algebraic variety hereafter. We drop
`quiver' from the notation and simply denote it by $\cM$ hereafter.

When we want to emphasize the parameter $\nu$, we write $\cM_\nu$.

\subsection{Factorization map}\label{subsec:factorization-map}

The factorization map $\Psi\colon\cM\to \AAA^{\underline{\bv}}$ is
defined as the collection of spectra of $B_\zeta$ for various
$\zeta$. Although $B_\zeta$ is defined for each segment $\zeta$,
$\Spec(B_\zeta)$ and $\Spec(B_{\zeta'})$ are the same up to constant
depending only on parameter $\nu$ when $\zeta$ and $\zeta'$ are
connected by oriented edges. Therefore we only have $n$ of segments to
consider. Hence the vector $\underline{\bv}$ is
$(\bv_0,\bv_1,\dots,\bv_{n-1})$ where $\bv_i$ is the minimum of
dimensions of vector spaces on segments between the $i$th and $(i+1)$
$\xl$-points.

We set 
\begin{equation*}
    |\underline{\bv}| = \bv_0+\bv_1+\dots+\bv_{n-1},\quad
    \mathfrak S_{\underline{\bv}} 
    = \mathfrak S_{\bv_0}\times\mathfrak S_{\bv_1}\times \dots\times
    \mathfrak S_{\bv_{n-1}}, \quad
    \AAA^{\underline{\bv}} = \AAA^{|\underline{\bv}|}/\mathfrak S_{\underline{\bv}}.
\end{equation*}

\subsection{Stability conditions}

\subsubsection{Characterization of the conditions (S1) and (S2)}

The conditions (S1) and (S2) are not imposed in the original chainsaw
variety in \cite{fra}, but will play important role in various
contexts of ${\cM}$. We first give other characterizations of the
conditions (S1) and (S2).
\begin{Proposition}\label{prop:character_cond1}
    The conditions {\em (S1), (S2)} is equivalent to
    $(1,-1)$-semistability for $\GL(V_-)\times
    \GL(V_+)$-action:
\begin{align*}
&B_{\zeta^\pm}(S_\pm) \subset S_\pm, S_-\subset \Ker b, A(S_-) \subset S_+ \Rightarrow \dim S_- - \dim S_+ \leq 0, \\
&B_{\zeta^\pm}(T_\pm) \subset T_\pm, T_+\supset \Ima a, A(T_-) \subset T_+ \Rightarrow \codim T_- - \codim T_+ \geq 0.
\end{align*}
\end{Proposition}
\begin{proof}
(S1) and (S2) follow from $(1, -1)$-semistability by setting $S_+=0, T_-=V_-$.

Assume $B_{\zeta^-}(S_-) \subset S_-, S_-\subset \Ker b$.
Then we have $B_{\zeta^+}A(S_-) \subset A(S_-)$ from $B_{\zeta^+}A-AB_{\zeta^-}+ab=0$, and each $S_+$ which satisfies $B_{\zeta^+}(S_+) \subset S_+$ includes $A(S_-)$.
Thus the condition (S1) means $\dim S_- \leq \dim S_+$.

Assume $B_{\zeta^+}(T_+) \subset T_+, T_+\supset \Ima a$.
Then we have $B_{\zeta^-}(A^{-1}(T_+)) \subset A^{-1}(T_+)$, where $A^{-1}$ means the inverse image, and each $T_-$ which satisfies $B_{\zeta^-}(T_-) \subset T_-$ includes $A^{-1}(T_+)$.
Thus the condition (S2) means $\codim T_- \geq \codim T_+$.
\end{proof}

\begin{Remark}\label{rem:NY}
  Suppose the bow diagram has one triangle and one two-way
  parts. Therefore in addition to $B_\pm$, $A$, $a$, $b$ above, we
  have additional linear maps $C\colon V_+\to V_-$,
  $D\colon V_-\to V_+$ with extra relations $C D = - B_-$,
  $DC = - B_+$:
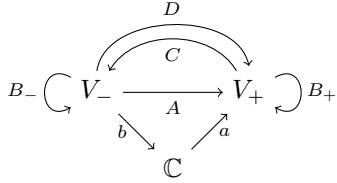
  \begin{equation*}
    \begin{tikzpicture}[
            arcnode/.style 2 args={                
            decoration={
                        raise=#1,             
                        markings,   
                        mark=at position 0.5 with { 
                                    \node[inner sep=0] {#2};
                        }
            },
            postaction={decorate}
}
]
            \node (S0) at (0,1) {$V_-$};
            \node (S1) at (2,1) {$V_+$};
            \node (Wz) at (1,0) {$\CC$};
            \draw[->] (S0) -- (S1) node[midway,below] {$\scriptstyle A$};
            \draw[->] (S0) -- (Wz) node[midway,left] {$\scriptstyle b$};
            \draw[->] (Wz) -- (S1) node[midway,right] {$\scriptstyle a$};
            \path [->] (S1) edge [bend right=60] node [below]
            {$\scriptstyle C$} (S0);
            \path [->] ([yshift=-10pt]S0) edge [bend left=80] node [above]
            {$\scriptstyle D$} ([yshift=10pt]S1);
            \draw[->,arcnode={-8pt}{$\scriptstyle B_{-}$}] (S0) to
    [loop above,out=150,in=210,distance=15pt,looseness=10] (S0);
            \draw[->,arcnode={8pt}{$\scriptstyle B_{+}$}] (S1) to
    [loop above,out=30,in=330,distance=15pt,looseness=10] (S1);
    \end{tikzpicture}.
  \end{equation*}
  As mentioned in Introduction, the variety of quiver representations
  describes the moduli space of framed rank $1$ perverse coherent
  sheaves on the blowup $\widehat{\CP}^2$ \cite{perv1}. In fact, we
  have choices of stability conditions depending on parameters
  $\nu_\pm$: if $S_\pm\subset V_\pm$ satisfies $C(S_+)\subset S_-$,
  $D(S_-)\subset S_+$, $A(S_-) \subset S_+$, $b(S_-) = 0$, we have
  $\nu_- \dim S_- + \nu_+ \dim S_+ \le 0$, and similarly for
  $T_\pm$. ($\nu_+$, $\nu_-$ are $\zeta_0$, $\zeta_1$ in \cite{perv1}
  respectivelly.) Therefore the above conditions (S1),(S2) imply the
  semistability with $\nu_- = 1$, $\nu_+ = -1$. But the converse is
  not true.
  \begin{NB}
    Here is the dictionary between our convention and \cite{perv1}:
    $S_0 = S_+$, $S_1 = S_-$, $\zeta_0 = \nu_+$,
    $\zeta_1 = \nu_-$, $d = C$, $B_1 = A$, $B_2 = D$.
  \end{NB}%
  For example, when $\dim V_+ = \dim V_- = 1$, the moduli space in
  \cite{perv1} with $\nu_- = 1$, $\nu_+ = -1$ is the blowup
  $\widehat{\CC}^2$ of the plane. On the other hand, it is easy to
  check that the bow variety is $\CC^2$. ($a=b=0$ in this case.) The
  difference appears as our conditions imply that $A$ is nonzero,
  while \cite{perv1} only requires $A$, $D$ do not vanish
  simultaneously. Our $\CC^2$ is the complement of the strict
  transform of an axis in $\widehat{\CC}^2$.
\end{Remark}

Next we consider the following complex on $\proj^1$:
\begin{align*}
\begin{CD}
V_-\otimes \mathcal{O}(-1) @>{\alpha =\left[\begin{smallmatrix} x_1 - x_0B_{\zeta^-} \\ x_0A \\ x_0b \end{smallmatrix}\right]}>> (V_-\oplus V_+ \oplus \CC)\otimes \mathcal O @>{\beta = \left[\begin{smallmatrix} -x_0A & x_1-x_0B_{\zeta^+} & x_0a \end{smallmatrix}\right]}>> V_+\otimes \mathcal{O}(1),
\end{CD}
\end{align*}
where $[x_0:x_1]$ is the homogeneous coordinate on $\proj^1$.
The equation $\mu=0$ means $\beta \alpha =0$.
\begin{Proposition}
{\em (i)} The condition {\em (S1)} is equivalent to injectivity of $\alpha$ at any point of $\proj^1$.\\
{\em (ii)} The condition {\em (S2)} is equivalent to surjectivity of $\beta$ at any point of $\proj^1$.\label{prop:character_cond2}
\end{Proposition}
\begin{proof}
For $v \in V_-$, $\alpha v=0$ means $\CC v$ is $B_{\zeta^-}$-invariant and included in $\Ker A \cap \Ker b$.
Thus $\alpha$ is injective if and only if (S1) is satisfied.
\end{proof}
As above, a triangle part is closely related to a monad, and this relation leads to Proposition \ref{prop:chainsaw} (2).

Let us take a wavy line $\xp$, and let $\xl^1$, \dots, $\xl^r$ be the
$\xl$-points on $\xp$ in the anticlockwise order. Let $\zeta^0$,
$\zeta^1$, \dots, $\zeta^{r}$ be the corresponding segments. Using
\propref{prop:character_cond1} successively, we get

\begin{Corollary}
    Suppose $S_{\zeta^i}\subset V_{\zeta^i}$ \textup(resp.\
    $T_{\zeta^i}\subset V_{\zeta^i}$\textup) satisfies
    \begin{gather*}
        B_{\zeta^i}(S_{\zeta^i})\subset S_{\zeta^i} \quad
        (0\le i\le r),\qquad
        A_{\xl^i}(S_{\zeta^{i-1}}) \subset S_{\zeta^i},\quad
        b_{\xl^i}(S_{\zeta^{i-1}}) = 0 \quad
        (1\le i\le r)
        \\
        \left(\text{resp.\ }
        B_{\zeta^i}(T_{\zeta^i})\subset T_{\zeta^i} \quad
        (0\le i\le r),\qquad
        A_{\xl^i}(T_{\zeta^{i-1}}) \subset T_{\zeta^i},\quad
        \Ima a_{\xl^i}\subset T_{\zeta^{i}} \quad
        (1\le i\le r)\right).
    \end{gather*}
    Then $\dim S_{\zeta^0}\le \dim S_{\zeta^1}\le\dots \le \dim
    S_{\zeta^r}$ \textup(resp.\ $\codim T_{\zeta^0}\ge\codim
    T_{\zeta^1}\ge\dots\ge \codim T_{\zeta^r}$\textup).
\end{Corollary}

\begin{Remark}
    Suppose the bow diagram has only one two-way part. The
    corresponding quiver appeared in \cite{MR3134906} and was called
    the dented chainsaw quiver. The above corollary implies that our
    data satisfy the stability condition in \cite{MR3134906}, more
    precisely the $\zeta^\bullet$-semistability therein. A slightly
    stronger stability condition, denoted by $\zeta^-$, is introduced
    so that the dented chainsaw quiver moduli spaces give moduli
    spaces of framed parabolic sheaves on the blowup
    $\widehat{\CP}^2$. It is natural to conjecture that the bow
    variety $\cM$ is the Uhlenbeck-type partial compactification of
    framed parabolic vector bundles on $\widehat{\CP}^2$, as
    $\zeta^\bullet$ is on the wall of a chamber where $\zeta^-$
    lives. However $\cM$ is \emph{not} isomorphic to the space
    $\widehat{\mathfrak Z}_{\underline{d}}$ in \cite{MR3134906} as the
    $\zeta^\bullet$-semistability and ours are different as we saw in
    \remref{rem:NY}. \propref{prop:character_cond2} suggests that our
    partial compactification does not allow singularities on parabolic
    structures.
\end{Remark}

\subsubsection{Numerical criterion}\label{subsubsec:numerical-criterion}

We introduce a numerical criterion for the $\nu^\RR$-(semi)stability.

\begin{Definition}
    Let $(A,B,a,b,C,D)\in \widetilde\cM$.
    \begin{description}
          \item[($\boldsymbol\nu\bf 1$)] Suppose a graded subspace
        $S = \bigoplus S_\zeta \subset \bigoplus V_\zeta$ invariant
        under $A$, $B$, $C$, $D$ with $b(S) = 0$ is given. We further
        assume that the restriction of $A$ induces an isomorphism
        $S_{\zeta^-}\xrightarrow[\cong]{A} S_{\zeta^+}$ for all
        \begin{tikzpicture}[baseline=0pt]
            \draw[decorate, decoration = {snake, segment
              length=2mm, amplitude=.4mm},->] 
            (0,0) -- (1,0);
            \node at (0.5,0) {$\boldsymbol\times$};
            \node at (0.2,0.3) {$\zeta^-$};
            \node at (0.95,0.3) {$\zeta^+$};
            \node at (0.5,0.3) {$\xl$};
        \end{tikzpicture}.
        Then
        \begin{equation*}
            \sum \nu_{\xp}^\RR \dim S_\xp \le 0.
        \end{equation*}
        Here $\xp$ runs over a wavy line, and $\dim S_\xp$ denotes the
        dimension of the vector space over a segment in $\xp$. It is
        independent of the choice of a segment by the assumption.

          \item[($\boldsymbol\nu\bf 2$)] Suppose a graded subspace $T
        = \bigoplus T_\zeta \subset \bigoplus V_\zeta$ invariant under
        $A$, $B$, $C$, $D$ with $\Ima a\subset T$ is given. We further
        assume that the restriction of $A$ induces an isomorphism
        $V_{\zeta^-}/T_{\zeta^-}\xrightarrow[\cong]{A}
        V_{\zeta^+}/T_{\zeta^+}$ for all $\zeta^\pm$ as above.  Then
        \begin{equation*}
            \sum \nu_{\xp}^\RR \codim T_\xp \ge 0.
        \end{equation*}
        The sum and $\codim T_\xp$ are as for ($\nu 1$).
    \end{description}
\end{Definition}

\begin{Proposition}\label{prop:numerical}
    Let $(A,B,a,b,C,D)\in \widetilde\cM$. Then it is
    $\nu^\RR$-semistable if and only if $(\nu 1)$, $(\nu 2)$ are
    satisfied.

    Similarly it is $\nu^\RR$-stable if and only if we have strict
    inequalities in $(\nu 1)$, $(\nu 2)$ unless $S_\zeta = 0$,
    $T_\zeta = V_\zeta$ for all $\zeta$.
\end{Proposition}

\begin{proof}
    Let us follow the argument in \cite[Lem.~3.25]{Lecture}.

    Suppose that $(\nu 1)$ is not satisfied. We take a graded subspace
    $\bigoplus S_\zeta$ satisfying the conditions, but violating the
    inequality. Then we take complement subspaces $S_\zeta^\perp$ to
    $V_\zeta$ and define a $1$-parameter subgroup $\rho(t)$ by acting
    $t$ on $S_\zeta$ and $1$ on $S_\zeta^\perp$. Since $\chi(g(t)) =
    \prod t^{-\nu^\RR_\xp \dim S_\xp}$, $\chi(g(t))^{-1}(z)\to 0$ as
    $t\to 0$ as the inequality is violated.

    By the assumption the limit of $(A,B,a,b,C,D)$ exists in
    $\mu^{-1}(0)$. We have a contradiction to the
    $\nu^\RR$-semistability condition if we check that the conditions
    (S1),(S2) are satisfied for the limit, as it means that the limit
    is in $\widetilde\cM$. An easy way to check them is to switch to
    Hurtubise normal form in \propref{prop:triangle-Nahm}. But let us
    check them in our description:
    The limit is the direct sum
    \begin{equation}\label{eq:12}
        \begin{tikzpicture}[baseline=(current  bounding  box.center),
            arcnode/.style 2 args={                
            decoration={
                        raise=#1,             
                        markings,   
                        mark=at position 0.5 with { 
                                    \node[inner sep=0] {#2};
                        }
            },
            postaction={decorate}
}
]
            \node (S0) at (0,1) {$S_{\zeta^-}$};
            \node (S1) at (2,1) {$S_{\zeta^+}$};
            \node (Wz) at (1,0) {$0$};
            \draw[->] (S0) -- (S1) node[midway,above] {$\scriptstyle A$}
            node[midway,below] {$\scriptstyle \cong$};
            \draw[->,dotted] (S0) -- (Wz);
            \draw[->,dotted] (Wz) -- (S1);
            \draw[->,arcnode={-8pt}{$\scriptstyle B_{\zeta^-}$}] (S0) to
    [loop above,out=60,in=120,distance=15pt,looseness=10] (S0);
            \draw[->,arcnode={-8pt}{$\scriptstyle B_{\zeta^+}$}] (S1) to
    [loop above,out=60,in=120,distance=15pt,looseness=10] (S1);
            \node at (2.8,1) {$\oplus$};
            \node (T0) at (4,1) {$V_{\zeta^-}/S_{\zeta^-}$};
            \node (T1) at (6,1) {$V_{\zeta^+}/S_{\zeta^+}$};
            \node (W) at (5,0) {$\CC$};
            \draw[->] (T0) -- (T1) node[midway,above] {$\scriptstyle A$};
            \draw[->] (T0) -- (W) node[midway,left] {$\scriptstyle b$};;
            \draw[->] (W) -- (T1) node[midway,right] {$\scriptstyle a$};;
            \draw[->,arcnode={-8pt}{$\scriptstyle B_{\zeta^-}$}] (T0) to
    [loop above,out=60,in=120,distance=15pt,looseness=10] (T0);
            \draw[->,arcnode={-8pt}{$\scriptstyle B_{\zeta^+}$}] (T1) to
    [loop above,out=60,in=120,distance=15pt,looseness=10] (T1);
        \end{tikzpicture}
    \end{equation}
    Here we use the same notation for the original linear maps and the
    induced ones.

    Suppose that we are given a subspace $S'\subset (S_{\zeta^-}\oplus
    V_{\zeta^-}/S_{\zeta^-})$ as in (S1). Then its projection to
    $V_{\zeta^-}/S_{\zeta^-}$ also satisfies the conditions in (S1)
    for the second term in \eqref{eq:12}. We take its inverse image
    $S''\subset V_{\zeta^-}$. Then $B_{\zeta^-}(S'')\subset S''$,
    $A(S'')\subset S_{\zeta^+}$ and $b(S'') = 0$. Note that $\dim S''
    \ge \dim S_{\zeta^-} = \dim
    S_{\zeta^+}$. \propref{prop:character_cond1} gives the opposite
    inequality, therefore $S'' = S_{\zeta^-}$, i.e., the projection of
    $S'$ to $V_{\zeta^-}/S_{\zeta^-}$ is zero. Therefore $S'$ is
    contained in $S_{\zeta^-}$. But it means that $S' = 0$ as $A$ is
    invertible on $S_{\zeta^-}$.

    Suppose that we are given a subspace $T'\subset (S_{\zeta^+}\oplus
    V_{\zeta^+}/S_{\zeta^+})$ as in (S2). Then its projection to
    $V_{\zeta^+}/S_{\zeta^+}$ also satisfies the conditions in (S2) in
    the second summand in \eqref{eq:12}. We take its inverse image
    $T''\subset V_{\zeta^+}$. Then it satisfies the conditions in (S2)
    for the original data, hence $T'' = V_{\zeta^+}$. It means that
    the projection of $T'$ to $V_{\zeta^+}/S_{\zeta^+}$ is the whole
    $V_{\zeta^+}/S_{\zeta^+}$. But $T'$ contains $S_{\zeta^+}$ as
    $A\colon S_{\zeta^-}\to S_{\zeta^+}$ is an isomorphism. Therefore
    $T' = (S_{\zeta^+}\oplus V_{\zeta^+}/S_{\zeta^+})$.

    Similarly $(\nu 2)$ follows from the $\nu_\xp$-semistability.
    \begin{NB}
        Let us check it. The point is that the direct sum
    \begin{equation*}
        \begin{tikzpicture}[
            arcnode/.style 2 args={                
            decoration={
                        raise=#1,             
                        markings,   
                        mark=at position 0.5 with { 
                                    \node[inner sep=0] {#2};
                        }
            },
            postaction={decorate}
}
]
            \node (S0) at (0,1) {$V_{\zeta^-}/T_{\zeta^-}$};
            \node (S1) at (2,1) {$V_{\zeta^+}/T_{\zeta^+}$};
            \node (Wz) at (1,0) {$0$};
            \draw[->] (S0) -- (S1) node[midway,above] {$\scriptstyle A$}
            node[midway,below] {$\scriptstyle \cong$};
            \draw[->,dotted] (S0) -- (Wz);
            \draw[->,dotted] (Wz) -- (S1);
            \draw[->,arcnode={-8pt}{$\scriptstyle B_{\zeta^-}$}] (S0) to
    [loop above,out=60,in=120,distance=15pt,looseness=10] (S0);
            \draw[->,arcnode={-8pt}{$\scriptstyle B_{\zeta^+}$}] (S1) to
    [loop above,out=60,in=120,distance=15pt,looseness=10] (S1);
            \node at (3,1) {$\oplus$};
            \node (T0) at (4,1) {$T_{\zeta^-}$};
            \node (T1) at (6,1) {$T_{\zeta^+}$};
            \node (W) at (5,0) {$\CC$};
            \draw[->] (T0) -- (T1) node[midway,above] {$\scriptstyle A$};
            \draw[->] (T0) -- (W) node[midway,left] {$\scriptstyle b$};;
            \draw[->] (W) -- (T1) node[midway,right] {$\scriptstyle a$};;
            \draw[->,arcnode={-8pt}{$\scriptstyle B_{\zeta^-}$}] (T0) to
    [loop above,out=60,in=120,distance=15pt,looseness=10] (T0);
            \draw[->,arcnode={-8pt}{$\scriptstyle B_{\zeta^+}$}] (T1) to
    [loop above,out=60,in=120,distance=15pt,looseness=10] (T1);
        \end{tikzpicture}
    \end{equation*}
    satisfies (S1),(S2). Take a subspace $T'\subset
    (V_{\zeta^+}/T_{\zeta^+}\oplus T_{\zeta^+})$ as in (S2). Then
    $T_{\zeta^-}$ and $T_{\zeta^+}\cap T'$ satisfy the condition in
    \propref{prop:character_cond1}. Therefore
    \begin{equation*}
        \codim T_{\zeta^-} \ge \codim T_{\zeta^+}\cap T'.
    \end{equation*}
    But $\codim T_{\zeta^-} = \codim T_{\zeta^+}$. Hence we must have
    $T_{\zeta^+}\cap T' = T_{\zeta^+}$. On the other hand the
    projection of $T'$ to $V_{\zeta^+}/T_{\zeta^+}$ is the whole
    $V_{\zeta^+}/T_{\zeta^+}$, as $A$ is an isomorphism between
    $V_{\zeta^-}/T_{\zeta^-}$ and $V_{\zeta^+}/T_{\zeta^+}$. Therefore
    $T' = V_{\zeta^+}$.

    Next consider $S'\subset (V_{\zeta^-}/T_{\zeta^-}\oplus
    T_{\zeta^-})$ as in (S1). Then $S'\cap T_{\zeta^-}$ is zero by
    (S1) for the original data. We next consider the projection of
    $S'$ to $V_{\zeta^-}/T_{\zeta^-}$. Since $A\colon
    V_{\zeta^-}/T_{\zeta^-} \to V_{\zeta^+}/T_{\zeta^+}$ is an
    isomorphism, we also have the projection is zero. Therefore $S'=0$.
    \end{NB}%

    Conversely suppose that $(A,B,a,b,C,D)$ is not
    $\nu_\xp$-semistable. We take a $1$-parameter subgroup
    $\rho(t)$ in $\GV$ such that
    $\rho(t)\cdot(A,B,a,b,C,D,z)$ converges to a point in
    $\widetilde\cM\times\{0\}$. Consider the weight space
    decomposition $V_\zeta = \bigoplus_m V_\zeta(m)$, where $\rho(t)$
    acts by $t^m$ on $V_\zeta(m)$. Since $A$, $B$, $C$, $D$ converges,
    the filtration $\cdots \subset F_{\zeta,m} := \bigoplus_{n\ge m}
    V_\zeta(n) \subset F_{\zeta,m-1} := \bigoplus_{n\ge m-1}
    V_\zeta(n)\subset\cdots$ is preserved. The limits are the induced
    maps on the associated graded. Similarly we have $\Ima a \subset
    F_0$, $b(F_1) = 0$. The limits of $a$, $b$ are the induced maps on
    $F_0/F_1$, and zero for other components. Thus $B_{\zeta^+} A = A
    B_{\zeta^-}$ on other components.  Then (S1),(S2) for the limit
    imply
    \(
      A\colon F_{\zeta^-,m}/F_{\zeta^-,m+1} \to 
      F_{\zeta^+,m}/F_{\zeta^+,m+1}
    \)
    is an isomorphism for $m\neq 0$. Hence
    \(
      A\colon F_{\zeta^-,m} \to F_{\zeta^-,m}
    \)
    is also an isomorphism for $m > 0$. Similarly
    \(
      A\colon V_{\zeta^-}/F_{\zeta^-,m} \to V_{\zeta^+}/F_{\zeta^+,m}
    \)
    is an isomorphism for $m\le 0$.

    We take $F_{\zeta,m}$ with $m>0$ as $S_{\zeta}$, $F_{\zeta,m}$
    with $m \le 0$ as $T_\zeta$. The inequalities in $(\nu 1)$, $(\nu
    2)$ are
    \begin{equation*}
        \sum \nu_\xp^\RR \dim F_{\xp,m} \le 0 \quad\text{for $m > 0$},
        \qquad
        \sum \nu_\xp^\RR \codim F_{\xp,m} \ge 0 \quad\text{for $m \le 0$}.
    \end{equation*}
    We take the sum over $m\in\ZZ$ to get
    \begin{equation*}
        \sum_\xp \sum_m \nu_\xp m \dim V_{\xp}(m) \le 0.
    \end{equation*}
    \begin{NB}
        $\codim F_{\xp,m} = \sum_{n < m} \dim V_\xp(n)$.
    \end{NB}%
    But it is in contradiction with $\lim_{t\to 0} \chi(\rho(t))^{-1}
    z = 0$.
    \begin{NB}
        Note $\chi(\rho(t))^{-1} z = t^{\sum\sum \nu_\xp m \dim
          V_\xp(m)} z$.
    \end{NB}%
    
    We leave the proof of the second assertion on the
    $\nu^\RR$-stability condition as an exercise for a reader.
\end{proof}

\begin{Lemma}\label{lem:stab}
    If $(A,B,a,b,C,D)\in\widetilde\cM$ is $\nu^\RR$-stable, it has a
    trivial stabilizer.
\end{Lemma}

\begin{proof}
    Suppose $g\in \GV$ stabilizes $(A,B,a,b,C,D)$. Consider $S_\zeta =
    \Ima(g_\zeta - \operatorname{id}) \subset V_\zeta$ and $T_\zeta =
    \Ker(g_\zeta - \operatorname{id}) \subset V_\zeta$. Then the
    collections $\{ S_\zeta\}$, $\{ T_\zeta\}$ are invariant under
    $A$, $B$, $C$, $D$ with $b(S) = 0$, $\Ima a \subset T$. Assume
    $S\neq 0$, or equivalently $T \neq V$. By the $\nu^\RR$-stability, we have
    \begin{equation*}
        \sum \nu^\RR_\xp \dim S_\xp < 0, \quad
        \sum \nu^\RR_\xp \codim T_\xp > 0.
    \end{equation*}
    But this is impossible as $\codim T_\xp = \dim S_\xp$.
\end{proof}

\subsection{Tangent space}

Let $\mathbb M$ and $\mathbb N$ as in \subsecref{subsec:quiver}.
Let us consider a three term complex
\begin{gather}\label{eq:19}
    \bigoplus_\zeta \End(V_\zeta) \xrightarrow{\alpha}
    \mathbb M \xrightarrow{\beta =\left[
      \begin{smallmatrix}
          \beta_1 \\ \beta_2
      \end{smallmatrix}\right]
} \mathbb N,
\\\notag
    \alpha \colon \bigoplus\xi_\zeta \mapsto
    \begin{pmatrix}
        [\xi_{\zeta^\pm_\xl}, B_{\zeta^\pm_\xl}] \\
        \xi_{\zeta^+} A - A \xi_{\zeta^-} \\
        \xi_{\zeta^+} a \\
        - b\xi_{\zeta^-} \\
        \xi_{\vin{h}} C - C \xi_{\vout{h}} \\
        \xi_{\vout{h}} D - D \xi_{\vin{h}}
    \end{pmatrix}, \quad \beta_1 \colon
    \begin{pmatrix}
        \overset{\bullet}{B} \\
        \overset{\bullet}{A} \\
        \overset{\bullet}{a} \\
        \overset{\bullet}{b} \\
        \overset{\bullet}{C} \\
        \overset{\bullet}{D}
    \end{pmatrix}
    \mapsto
        B_{\zeta^+_\xl} \overset{\bullet}{A} - \overset{\bullet}{A}
        B_{\zeta^-_\xl} + \overset{\bullet}{B}_{\zeta^+_\xl} A - A
        \overset{\bullet}{B}_{\zeta^-_\xl}
        + \overset{\bullet}{a} b + a\overset{\bullet}{b}, 
\end{gather}
and $\beta_2$ is
\begin{gather*}
    -\overset{\bullet}{B}_{\zeta^+_\xl} + \overset{\bullet}{B}_{\zeta^-_{\xl'}},
    \quad
    \overset{\bullet}{C}_h D_h + C_h\overset{\bullet}{D}_h
    - \overset{\bullet}{D}_{h'} C_{h'} - D_{h'}\overset{\bullet}{C}_{h'},
    \\
    \overset{\bullet}{C}_h D_h + C_h\overset{\bullet}{D}_h
    + \overset{\bullet}{B}_{\zeta^-_{\xl}},
    \quad
    - \overset{\bullet}{B}_{\zeta^+_\xl} 
    - \overset{\bullet}{D}_{h} C_{h} - D_{h}\overset{\bullet}{C}_{h},
\end{gather*}
according to two ends of $\zeta$ as before.
Note $\alpha$ (resp.\ $\beta$) is the differential of the $\GV$-action
(resp.\ $\mu$).

\begin{Proposition}\label{prop:differential}
    \textup{(1)}
    $\alpha$ is injective at
    $(A,B,a,b,C,D)\in\widetilde{\cM}^{\mathrm{s}}$.
    
    \textup{(2)} The first component $\beta_1$ of $\beta$ is surjective at
    $(A,B,a,b,C,D)\in\widetilde{\cM}$.

    \textup{(3)} $\beta$ is surjective at
    $(A,B,a,b,C,D)\in\widetilde{\cM}^{\mathrm{s}}$.
\end{Proposition}

\begin{proof}
    (1) 
    Suppose that $\oplus \xi_\zeta$ is in the kernel of $\alpha$. Then
    $S_\pm := \Ima \xi_{\zeta^\pm}$, $T_\pm := \Ker \xi_{\zeta^\pm}$
    satisfy conditions in \propref{prop:character_cond1}. Therefore
    $\dim S_-\le \dim S_+$, $\codim T_- \ge \codim T_+$. But $\codim
    T_\pm = \dim S_\pm$. Therefore both must be equality.

    Next we consider $S := \bigoplus \Ima\xi_\zeta$, $T := \bigoplus
    \Ker\xi_\zeta$. The conditions in $(\nu 1)$, $(\nu 2)$ are
    satisfied, hence the same argument shows $\sum \nu^\RR_\xp \dim
    S_\xp = 0$, and hence $S = 0$, i.e., $\xi_\zeta = 0$ for all $\zeta$.

    (2) Suppose $\oplus \eta_\xl$ is perpendicular to the image of the
    first component of $\beta$ with respect to the trace pairing. Then
    \begin{equation*}
        A \eta_\xl = 0 = \eta_\xl A, \quad
        B_{\zeta^-} \eta_\xl = \eta_\xl B_{\zeta^+}, \quad
        b \eta_\xl =0 = \eta_\xl a.
    \end{equation*}
    Then $S := \Ima\eta_\xl$ is zero by (S1). (Or $T := \Ker\eta_\xl$
    is $V_2$ by (S2).) Therefore $\eta_\xl = 0$.

    (3) Suppose $\oplus \eta_\xl \oplus \xi_\zeta$ is perpendicular to the
    image of $\beta$ with respect to the trace pairing. Then
    \begin{equation*}
        B_{\zeta^-} \eta_\xl = \eta_\xl B_{\zeta^+}, \quad
        b \eta_\xl =0 = \eta_\xl a, \quad
        A \eta_\xl = \xi_{\zeta^+}, \quad
        \eta_{\xl} A = \xi_{\zeta^-},
    \end{equation*}
    and
    \begin{gather*}
        C \xi_{\vout{h}} = \xi_{\vin{h}} C, \quad
        D \xi_{\vin{h}} = \xi_{\vout{h}} D.
    \end{gather*}
    Note that $B_{\zeta^-}\eta_\xl A = \eta_\xl A B_{\zeta^-}$ and
    $B_{\zeta^+}A \eta_\xl = A \eta_\xl B_{\zeta^+}$. Therefore the
    same argument as in (1) shows $\xi_\zeta = 0$ for all
    $\zeta$. Next we consider $S := \Ima \eta_\xl\subset
    V_{\zeta^-}$. Then the condition for (S1) is satisfied. Hence we
    have $S = 0$, i.e., $\eta_\xl = 0$.
\end{proof}

By (3), $\widetilde\cM$ is smooth at
$(A,B,a,b,C,D)\in\widetilde\cM^{\mathrm{s}}$. The tangent space is
isomorphic to $\Ker\beta$. On the stable locus, the quotient map
$\widetilde\cM^{\mathrm{s}}\to \cM^{\mathrm{s}}$ is a principal
$\GV$-bundle. Hence the tangent space of the quotient $\cM^{\mathrm{s}}$
is isomorphic to $\Ker\beta/\Ima\alpha$.
By (2) $\widetilde\cM_{\mathrm{sym}}$ is smooth at
$(A,B,a,b,C,D)\in\widetilde\cM$. The tangent space is isomorphic to
$\Ker\beta_1$. (Recall $\widetilde\cM_{\mathrm{sym}}$ is the open
subvariety of $\mu_1^{-1}(0)$ consisting of points satisfying
(S1),(S2).)

\begin{NB}
    Is it possible to write down the symplectic form on $\Ker\beta_1$
    ? This is a linear version of the isomorphism
    \propref{prop:triangle-Nahm}, hence should be possible.

    For example, if $\bv_1 = \bv_2$, we have
    \begin{equation*}
        \begin{gathered}[m]
        \overset{\bullet}A = \overset{\bullet}u, \quad
        \overset{\bullet}B_1 = - u^{-1} \overset{\bullet}u u^{-1} h u
        + u^{-1} h \overset{\bullet}u + u^{-1} \overset{\bullet}h u
        = [u^{-1}h u, u^{-1}\overset{\bullet}u]
        + u^{-1} \overset{\bullet}h u,
        \\
        \overset{\bullet}B_2 = \overset{\bullet}h - \overset{\bullet}I J
        - I\overset{\bullet}J, \quad
        \overset{\bullet}a = \overset{\bullet}I, \quad
        \overset{\bullet}b = \overset{\bullet}J u + J\overset{\bullet}u.
        \end{gathered}
    \end{equation*}
    Hence a tangent vector is determined by $\overset{\bullet}A$,
    $\overset{\bullet}B_2$, $\overset{\bullet}a$, $\overset{\bullet}b$
    as
    \begin{equation*}
        \overset{\bullet}h = \overset{\bullet}B_2 + \overset{\bullet}a J
        + I (\overset{\bullet}b - J \overset{\bullet}A) A^{-1}
        = \overset{\bullet}B_2 + \overset{\bullet}a b A^{-1}
        + a (\overset{\bullet}b - b A^{-1} \overset{\bullet}A) A^{-1}.
    \end{equation*}

    But I want an obvious description, like quiver variety case. The
    tangent complex is self-dual, as the first and third terms are the
    same. The second component of $\mathbb N'$ is
    $\bigoplus \End(V_\zeta)$, hence is the same as the first term.
\end{NB}%

In particular, we have
\begin{Proposition}\label{prop:dimension}
    The \textup(possibly empty\textup) open locus $\cM^{\mathrm{s}}$
    is smooth of dimension $\dim \mathbb M - \dim \mathbb N - \dim \GV$.
\end{Proposition}

\begin{NB}
\begin{proof}
    This is clear from \propref{prop:differential} above. But let us
    give a different view point that $\cM^{\mathrm s}$ is a symplectic
    reduction of $\widetilde\cM_{\mathrm{sym}}$ by $\GV$: By
    \propref{prop:differential}(2) above we know
    $\widetilde\cM_{\mathrm{sym}}$ is smooth. The tangent space is
    $\Ker\beta_1$. The second component $\mu_2$ is the moment map,
    hence $\beta_2 = d\mu_2$ is surjective on the stable locus by a
    standard result. Then we take the quotient in the sense of the
    geometric invariant theory. The smoothness assertion follows from
    a standard result in the geometric invariant theory, as in the
    case of quiver varieties.

    As for dimension, we have $\dim \cM^{\mathrm s} =
    \dim\widetilde\cM_{\mathrm{sym}} - 2\dim \GV$. On the other hand, we
    have $\dim \widetilde\cM_{\mathrm{sym}} = \dim \mathbb M' - \sum_\xl
    \dim \Hom(V_{\zeta^-_\xl}, V_{\zeta^+_\xl})$ by
    \propref{prop:differential}(2) above. Now note that $\bigoplus
    \Hom(V_{\zeta^-_\xl}, V_{\zeta^+_\xl})$ is the first component of
    $\mathbb N$, and the second component is the Lie algebra of
    $\GV$. Hence we get the assertion.
\end{proof}
\end{NB}%


\subsection{Balanced and cobalanced conditions}

We define $N_h:= \dim V_{\vin{h}} - \dim V_{\vout{h}}$ for an arrow $h$ and $N_{\xl}:= \dim V_{\zeta^-} - \dim V_{\zeta^+}$ for $\xl \in \Lambda$, where $\zeta^{\pm}$ are the adjacent segments 
\begin{NB}
$\zeta^- \rightarrow \xl \rightarrow \zeta^+$.
\end{NB}%
\begin{tikzpicture}[baseline=0pt]
    \draw[decorate, decoration = {snake, segment
        length=2mm, amplitude=.4mm},->] 
    (0,0) -- (1,0);
    \node at (0.5,0) {$\boldsymbol\times$};
    \node at (0.2,0.3) {$\zeta^-$};
    \node at (0.95,0.3) {$\zeta^+$};
    \node at (0.5,0.3) {$\xl$};
\end{tikzpicture}.
Of course, we mean $|N_{\xl}| = \Delta R(\xl)$.

\begin{Definition}
A dimension vector $\underline{\bv}$ is called \emph{balanced} if $N_h=0$ for each $h$.
A dimension vector $\underline{\bv}$ is called \emph{cobalanced} if $N_{\xl} = 0$ for each $\xl$.
\end{Definition}

The balanced condition will play an important role in
\secref{sec:Coulomb}. For the cobalanced condition we have

\begin{Theorem}[\cite{MR2824478}]
    Consider a bow diagram with a cobalanced dimension
    vector. Assume there is at least one arrow. Then the
    corresponding bow variety $\cM$ is isomorphic to a quiver variety
    $\fM$ in \textup{\cite{Na-quiver}} as an affine algebraic variety.
    Moreover the isomorphism respects the symplectic form on regular
    locus.  Conversely, a quiver variety of an affine type $A$ is
    described as a bow variety with a cobalanced dimension
    vector.\label{thm:balanced_bow}
\end{Theorem}

A quiver variety $\fM$ is a symplectic reduction of $\mathbf N\oplus
\mathbf N^*$ by $G = \prod\GL(\bv_i)$, where $\mathbf N$ is as in
Introduction. The level of the moment map is $\nu^\CC$, and the
quotient is the geometric invariant theory quotient with the
$\nu^\RR$-semistability.

\begin{proof}
    This is the case stated in (\ref{item:neutral}) of
    \subsecref{subsec:original}.
In the quiver description, $\dim V_{\zeta^-}=\dim V_{\zeta^+}$. Then $A$ is an isomorphism by \lemref{lem:triangle_fullrank} below.
We eliminate Nahm's data on segments, or $A$ in the quiver description, and $(I, J)$ forms a framing part of a quiver variety.

For example, consider the case when $r_{\xp} = 1, \vin{h} = \xl_{\xp, L}, \zeta^- = [\xl_{\xp, L}, \xl_{\xp}], \zeta^+ = [\xl_{\xp}, \xl_{\xp, R}], \xl_{\xp, R} = \vout{h'}$.
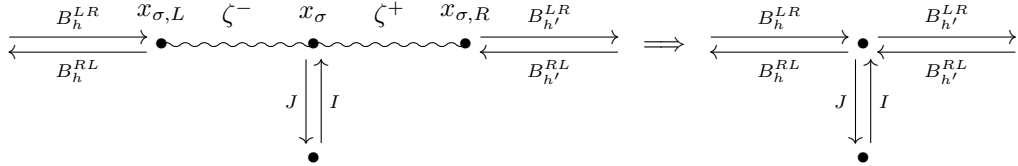
\begin{align*}
\begin{xy}
(20,0)*{\bullet},
(20,4)*{\xl_{\xp,L}},
(30,4)*{\zeta^-},
(40,0)*{\bullet},
(40,4)*{\xl_{\xp}},
(40,-15)*{\bullet},
(50,4)*{\zeta^+},
(60,0)*{\bullet},
(60,4)*{\xl_{\xp,R}},
\ar (0,1);(18,1)^{B^{LR}_h}
\ar (18,-1);(0,-1)^{B^{RL}_h}
\ar @{~} (20,0);(60,0)
\ar (62,1);(80,1)^{B^{LR}_{h'}}
\ar (80,-1);(62,-1)^{B^{RL}_{h'}}
\ar (39,-2);(39,-13)_{J}
\ar (41,-13);(41,-2)_{I}
\end{xy}
\ \ \Longrightarrow \ \ 
\begin{xy}
(20,0)*{\bullet},
(20,-15)*{\bullet},
\ar (0,1);(18,1)^{B^{LR}_h}
\ar (18,-1);(0,-1)^{B^{RL}_h}
\ar (22,1);(40,1)^{B^{LR}_{h'}}
\ar (40,-1);(22,-1)^{B^{RL}_{h'}}
\ar (19,-2);(19,-13)_{J}
\ar (21,-13);(21,-2)_{I}
\end{xy}
\end{align*}
The segments $\zeta^{\pm}$ give the cotangent bundle
$T^*\GL(V_{\zeta^{\pm}})$ by Kronheimer \cite{Kr-cotangent}. 
In the quiver description, we use \propref{prop:triangle-Nahm}(1) below.
Hence we can eliminate them in the symplectic quotient by imposing the equation
\begin{align*}
B^{LR}_hB^{RL}_h - B^{RL}_{h'}B^{LR}_{h'} + IJ = (\nu_{\xp}^2 + \sqrt{-1}\nu_{\xp}^3)\id
\end{align*}
from 3.(a) and (c) in \subsecref{subsec:original}, and so on.
This is nothing but the defining equation in \cite{Na-quiver}.
\end{proof}

\section{Constituents of bow varieties}\label{sec:constituents}
In this section, we study properties of triangle and two-way parts in
the quiver description of a bow variety.  We need these properties in
later sections, in particular in \secref{sec:Coulomb} to prove a
balanced bow variety is a Coulomb branch of quiver gauge theory of
affine type $A$.

\subsection{Triangle part}\label{subsec:triangle}
\subsubsection{Definition}
Consider 
\begin{align*}
\xymatrix@C=1.2em{ V_1 \ar@(ur,ul)_{B_1} \ar[rr]^{A} \ar[dr]_{b} && V_2
  \ar@(ur,ul)_{B_2} 
  \\
  & \CC \ar[ur]_a,&}
\end{align*}
where $\dim V_k = \bv_k$.
\begin{NB}
In Cherkis' description, these spaces are described as follows:
\begin{align*}
\begin{xy}
(5,0)*{\widetilde{\mathfrak{M}}_{\mathrm{tr}} \cong \mathfrak{M}_{\mathrm{N}}:},
(20,3)*{\bv_1},
(25,0)*{\boldsymbol\times},
(30,3)*{\bv_2},
\ar @{-} (20,0);(30,0)
\end{xy}
\end{align*}
\end{NB}%
Set
\begin{align*}
\mathbb{M}&:= \Hom (V_1, V_2) \oplus \End V_1 \oplus \End V_2 \oplus \Hom (\CC, V_2)\oplus \Hom (V_1, \CC),\\
\GL(\underline{\bv}) &:= \GL(V_1)\times \GL(V_2),\\
\widetilde{\cM}&:= \left\{ (A, B_1, B_2, a, b) \in \mu^{-1}(0) \ \Big| \ \text{$(A, B_1, B_2, a, b)$ satisfies (S1) and (S2)} \right\},\\
\cM&:= [\widetilde{\cM}/\GL(\underline{\bv})] \quad(\text{quotient stack}).
\begin{NB}
\text{ a set of $\GL(V_1) \times \GL(V_2)$-equivalence classes of $\widetilde{\cM}$}.
\end{NB}%
\end{align*}
\begin{NB}
\begin{align*}
\cM&:= [\widetilde{\cM} / \GL(V_1) \times \GL(V_2)],
\end{align*}
where the last one is defined as a stack.
\end{NB}%
Here,
\begin{align*}
\mu (A, B_1, B_2, a, b)&:=B_2A-AB_1+ab,\\
\text{(S1)}&: B_1(S_1) \subset S_1, S_1 \subset \Ker A\cap \Ker b \Rightarrow S_1=0,\\
\text{(S2)}&: B_2(T_2) \subset T_2, T_2 \supset \Ima A + \Ima a \Rightarrow T_2=V_2,
\end{align*}
and the $\GL(\underline{\bv})$-action is
\begin{align*}
(A, B_1, B_2, a, b) \mapsto (g_2Ag_1^{-1}, g_1B_1g_1^{-1}, g_2B_2g_2^{-1}, g_2a, bg_1^{-1}), \ \ (g_1, g_2) \in \GL(\underline{\bv}).
\end{align*}
By \propref{prop:differential}(2) $\widetilde{\cM}$ is a $(\bv_1^2+\bv_1+\bv_2^2+\bv_2)$-dimensional smooth variety.
\begin{NB}
Suppose $\eta\in\Hom(V_2,V_1)$ is perpendicular to the image of
$d\mu$ with respect to the natural pairing between $\Hom(V_1,V_2)$
and $\Hom(V_2,V_1)$. We have $\eta a = 0 = b\eta$, $A\eta = 0 = \eta
A$, and $\eta B_2 = B_1\eta$. Then $S_1 := \Ima\eta$ is zero by
(S1). (Or $S_2 := \Ker\eta$ is $V_2$ by (S2).) Therefore $\eta = 0$, i.e., $d\mu$ is surjective.
\end{NB}%
We will use the quotient stack $\cM$ just to reduce notational redundancy in the factorization (\propref{prop:triangle_factor} below). One can simply understand
$[X/G] \cong [Y/H]$ as a $G$-equivariant isomorphism $X \cong Y\times_H G$.
(Here $H$ is a subgroup of $G$.)

For $\widetilde{\cM}$, the following results play important roles.
\begin{Lemma}[\protect{\cite[Lem.~2.18]{Takayama}}]
$A$ has full rank.\label{lem:triangle_fullrank}
\end{Lemma}

We consider solutions of Nahm's equations on an interval, which
contains a single $\xl$-point. The solutions extend smoothly at the
end points.
We consider a gauge transformation $u$ as before, but we impose the
condition that $u$ is the identity at the end points.

\begin{Proposition}
{\em (1) \cite[Prop.~1.15 + \S2]{MR987771}} The moduli space of solutions of Nahm's equations over the interval is isomorphic to the space of matrices
of forms
\begin{equation*}
    \begin{split}
        u\in \GL(n), \eta
    = & \left[ \begin{smallmatrix} h & 0 & g \\ f & 0 & e_0 \\ 0 & \id & e \end{smallmatrix} \right] \in \mathfrak{gl}(n), \\
    &(e_0, e, f, g, h) \in \CC \times \CC^{n-m-1} \times 
    (\CC^m)^* \times \CC^m  \times \End (\CC^m),
    \end{split}
\end{equation*}
for $n = \max(\bv_1, \bv_2), m = \min(\bv_1, \bv_2)$ when
$\bv_1\neq\bv_2$, and
\begin{equation*}
    u\in \GL(n), h\in \mathfrak{gl}(n), I\in \CC^n, J\in (\CC^n)^*
\end{equation*}
when $n = \bv_1 = \bv_2$. The action is given by
\begin{gather*}
        (u,\eta) \mapsto \left(
    \left[\begin{smallmatrix}
          g_m & 0 & 0 \\ 0 & \id & 0 \\ 0 & 0 & \id
    \end{smallmatrix}\right]
    u g_n^{-1},     \left[\begin{smallmatrix}
          g_m & 0 & 0 \\ 0 & \id & 0 \\ 0 & 0 & \id
    \end{smallmatrix}\right]
  \eta
      \left[\begin{smallmatrix}
          g_m^{-1} & 0 & 0 \\ 0 & \id & 0 \\ 0 & 0 & \id
    \end{smallmatrix}\right]\right),
  \quad (g_n,g_m)\in \GL(n)\times \GL(m)
  \qquad \text{if $\bv_1\neq \bv_2$},
  \\
    (u,h,I,J) \mapsto
    (g_2 u g_1^{-1}, g_2 h g_2^{-1}, g_2 I, J g_2^{-1}),
    \quad (g_1,g_2)\in \GL(\underline{\bv})
    \qquad \text{if $\bv_1 = \bv_2$}.
\end{gather*}

{\em (2) \cite[Prop.~2.9]{Takayama}} The above space of matrices is isomorphic to $\widetilde{\cM}$ by\label{prop:triangle-Nahm}

\begin{align*}
(A, B_1, B_2, a, b)&=\begin{cases}
\left( \left[ \begin{smallmatrix} \id & 0 & 0 \end{smallmatrix}\right]u, u^{-1}\eta u, h, g, \left[ \begin{smallmatrix} 0 & 0 & 1 \end{smallmatrix}\right]u \right), & \bv_1 > \bv_2, \\
\left( u, u^{-1}hu, h-IJ, I, Ju \right), & \bv_1 = \bv_2, \\
\left( - u^{-1}\left[ \begin{smallmatrix} \id \\ 0 \\ 0 \end{smallmatrix}\right], - h, - u^{-1}\eta u, u^{-1}\left[ \begin{smallmatrix} 0 \\ 1 \\ 0 \end{smallmatrix} \right],  -f \right), & \bv_1 < \bv_2.\\ 
\end{cases}
\end{align*}
\begin{NB}
    Here is the original version. The sign is wrong in the case $\bv_1
    < \bv_2$
\begin{align*}
(A, B_1, B_2, a, b)&=\begin{cases}
\left( \left[ \begin{smallmatrix} \id & 0 & 0 \end{smallmatrix}\right]u, u^{-1}\eta u, h, g, \left[ \begin{smallmatrix} 0 & 0 & 1 \end{smallmatrix}\right]u \right), & \bv_1 > \bv_2, \\
\left( u, u^{-1}hu, h-IJ, I, Ju \right), & \bv_1 = \bv_2, \\
\left( u^{-1}\left[ \begin{smallmatrix} \id \\ 0 \\ 0 \end{smallmatrix}\right], h, u^{-1}\eta u, u^{-1}\left[ \begin{smallmatrix} 0 \\ 1 \\ 0 \end{smallmatrix} \right],  -f \right), & \bv_1 < \bv_2.\\ 
\end{cases}
\end{align*}
\end{NB}%
\end{Proposition}
Notice that the matrix $\left[ \begin{smallmatrix} 0 & 0 & 0 \\ 0 & 0 & 0 \\ 0 & \id & 0 \end{smallmatrix} \right]$ is nilpotent, and when we regard this nilpotent matrix as $Y$ of an $\mathfrak{sl}(2)$-triple $\{H, X, Y\}$, $\left[ \begin{smallmatrix} 0 & 0 & 1 \end{smallmatrix} \right]$ is the highest weight covector and $\left[ \begin{smallmatrix} 0 \\ 1 \\ 0 \end{smallmatrix} \right]$ is the lowest weight vector.

A pair of matrices $u$, $\eta$ (and with $I$, $J$ when $\bv_1=\bv_2$) as above is called \emph{Hurtubise
  normal form}. As we mentioned after \thmref{thm:quiver_desc}, (1) was
proved in two steps: the first step solves the real equation by
calculus of variation. The second step is the classification of the
solution of the complex equation.

\begin{NB}
   Let us sketch the description of the map in the opposite direction.
   
   Suppose $\bv_1 < \bv_2$. First note that $\Ima A$ is $m$-dimensional as $A$ is injective. We choose a base $e_1$, \dots, $e_m$ of $\Ima A$. Consider vectors $a(1)$, $B_2 a(1)$, \dots, $B_2^k a(1)$ ($0\le k$). If $B_2^{k+1} a(1)$ is in the span of $\Ima A$ and those vectors, the span must be equal to $V_2$ by (S2). Hence $k\ge n-m-1$. Thus we add $e_{m+1} = a(1)$, $e_{m+2} = B_2 a(1)$, \dots, $e_n = B_2^{n-m-1} a(1)$ to get a base of $V_2$. We claim that $B_2$ is in the form of $\eta$ in this base. The middle column is
$\left[
\begin{smallmatrix}
 0 \\ 0\\ \id
\end{smallmatrix}
\right]$ by definition. From the equation $B_2 A - AB_1 + ab = 0$, $B_2(\Ima A)\subset \Ima A + \Ima a$, hence the first column is of a form
$\left[
\begin{smallmatrix}
 h \\ f\\ 0
\end{smallmatrix}
\right]$.

In other words, we define $u$, $\eta$ by
\begin{equation*}
\begin{split}
   u^{-1} &= 
\begin{bmatrix}
   A & a & B_2a & \cdots & B_2^{n-m-1} a
\end{bmatrix},
\\
  \eta &= u B_2 u^{-1} = \begin{bmatrix}
   A & a & B_2a & \cdots & B_2^{n-m-1} a
\end{bmatrix}^{-1} B_2
\begin{bmatrix}
   A & a & B_2a & \cdots & B_2^{n-m-1} a
\end{bmatrix}.
\end{split}
\end{equation*}

Suppose $\bv_1 = \bv_2$. Then $A$ is an isomorphism. We then set $u := A$,
$I := a$, $J := bu^{-1}$, $h := u B_1 u^{-1}$. From the equation $B_2 A - AB_1 + ab = 0$, we have $B_2 = h - IJ$.

Suppose $\bv_1 > \bv_2$. Since $A$ is surjective, $\Ker A$ is $(n-m)$-dimensional. Consider covectors $b$, $bB_1$, \dots, $b B_1^k$ ($0\le k$). If $b B_1^{k+1}|_{\Ker A}$ is in the span of restrictions of these covectors,
$\Ker A\cap \Ker b\cap \Ker bB_1\cap \cdots\cap\Ker bB_1^k$ must be $0$ by (S1). Hence $k\ge n-m-1$. Thus $b|_{\Ker A}$, \dots, $bB_1^{n-m-1}|_{\Ker A}$ are linearly independent and there exist $e_0$, $e_1$, \dots, $e_{n-m-1}$ such that $bB_1^{n-m} = b(e_{n-m-1} B_1^{n-m-1} + \cdots + e_1 B_1 + e_0 \id)$ on $\Ker A$. We then take a base $\tilde f_1$, \dots, $\tilde f_m$ of $V_2^*$, and consider $f_1 := \tilde f_1A$, \dots, $f_m = \tilde f_m A$. We extend them to a base of $V_1^*$ by adding
$f_{m+1} := b (B_1^{n-m-1} - e_{n-m-1}B_1^{n-m-2} - \cdots - e_1 \id)$,
$f_{m+2} := b (B_1^{n-m-2} - e_{n-m-1}B_1^{n-m-3} - \cdots - e_2 \id)$,
\dots,
$f_{n-1} := b (B_1 - e_{n-m-1}\id)$,
$f_n := b$.
That is, we take
\begin{equation*}
    u := 
        \begin{bmatrix}
        A \\ b (B_1^{n-m-1} - e_{n-m-1}B_1^{n-m-2} - \cdots - e_1 \id)
        \\ b (B_1^{n-m-2} - e_{n-m-1}B_1^{n-m-3} - \cdots - e_2 \id)
        \\
        \vdots
        \\
        b (B_1 - e_{n-m-1}\id)
        \\
        b
       \end{bmatrix}.
\end{equation*}
Then $u B_1 u^{-1}$ is in the form of $\eta$. In fact, 
$f_n B_1 = f_{n-1} + e_{n-m-1}f_n$,
$f_{n-1} B_1 = f_{n-2} + e_{n-m-2} f_n$, \dots,
$f_{m+2} B_1 = f_{m+1} + e_1 f_n$. These appear in the bottom row.
Then
$f_{m+1} B_1 - e_0 f_n$ vanishes on $\Ker A$ by the construction. Hence
it is a linear combination of $f_1$, \dots, $f_m$. This explains the middle row.
Next take $f_k$ with $1\le k\le m$. Then $f_k B_1$ vanishes on $\Ker A\cap \Ker b$. In fact, $f_k B_1 = \tilde f_k A B_1 = \tilde f_k(B_2 A + ab)$. Therfore
$f_k B_1$ is written in the linear combination of $f_1$, \dots, $f_m$ together
with one more covector $f_n$. This explains the top row.

Since we determine $e_0$, $e_1$, \dots, $e_{n-m-1}$ implicitly, the formula of the inverse is less explicit. We first define
\begin{equation*}
    u' :=
    \begin{bmatrix}
        A \\ bB_1^{n-m-1} \\ bB_1^{n-m-2} \\ \vdots \\ b
    \end{bmatrix}.
\end{equation*}
We have
\begin{equation*}
    u' B_1 u^{\prime-1} =
    \begin{bmatrix}
        h & 0 & g \\ f & {}^t e & e_0 \\ 0 & \id & 0
    \end{bmatrix},
\end{equation*}
as $B_1(\Ker A\cap \Ker b) \subset \Ker A$. (In fact, the first column corresponds to $\Ker A$, while the third column corresponds to $\Ker b$.)
This is almost desired form of $\eta$, but not quite. So we set
\begin{equation*}
    \begin{bmatrix}
        e_{n-m-1} & e_{n-m-2} & \cdots & e_1 & e_0
    \end{bmatrix}
    :=
    [\begin{matrix}
        0 & \cdots & 0 & \underbrace{1}_{m} & 0 & \cdots & 0
    \end{matrix}]
    u' B_1 u^{\prime-1}
    \begin{bmatrix}
        0 \\ \id_{n-m}
    \end{bmatrix}.
\end{equation*}
Now we define
\begin{equation*}
    u :=
    \begin{bmatrix}
        A \\ b (B_1^{n-m-1} - e_{n-m-1}B_1^{n-m-2} - \cdots - e_1 I)
        \\ b (B_1^{n-m-2} - e_{n-m-1}B_1^{n-m-3} - \cdots - e_2 I)
        \\
        \vdots
        \\
        b (B_1 - e_{n-m-1}I)
        \\
        b
    \end{bmatrix}.
\end{equation*}
\end{NB}%

By combining two isomorphisms, 
we see that  $\widetilde{\cM}$ is isomorphic to a moduli space of solutions of Nahm's equations over an interval, and has a hyper-K\"ahler structure.
Note also that $\widetilde{\cM}$ is an affine algebraic variety, as the space of Hurtubise normal forms is a product of a general linear group and an affine space.

\begin{Corollary}\label{cor:symplectic}
\begin{NB}
$\widetilde{\cM}$ has the holomorphic symplectic structure induced by the hyper-K\"ahler structure of a moduli space of solutions of Nahm's equations over an interval.
It is written as
\end{NB}%
The holomorphic symplectic form on $\widetilde\cM$ is given in the Hurtubise normal form by 
\begin{align*}
\omega = \begin{cases}\tr (d\eta \wedge du u^{-1}+ \eta du u^{-1} \wedge du u^{-1}) & \text{ when }\bv_1 \neq \bv_2,\\
\tr (dh \wedge du u^{-1}+ h du u^{-1} \wedge du u^{-1}+ dI \wedge dJ) & \text{ when }\bv_1 = \bv_2. \end{cases}
\end{align*}
\end{Corollary}

\begin{NB}
The original form needs more explanation. This means the symplectic form is equal to the standard one in Hurtubise normal form. And `It is written as' is confusing, as it is actually given in Hurbutise normal form.
\end{NB}%

In case $\bv_1 = \bv_2$, the above is the standard symplectic form on $T^*\GL(n)\times \CC^n\times (\CC^n)^*$. When $\bv_1\neq\bv_2$, this is a consequence of \cite{MR1438643}. More precisely, we can apply \cite{MR1438643} by the following remark.

\begin{Remark}\label{rem:Slodowy}
When $\bv_1 \neq \bv_2$, we can forget Nahm's data on the interval for the smaller rank.
Then we have $\widetilde{\cM} \cong  \GL(n)\times \mathcal{S}$, where $\mathcal{S}$ is the Slodowy transversal slice of the nilpotent element $\left[ \begin{smallmatrix} 0 & 0 & 0 \\ 0 & 0 & 0 \\ 0 & \id & 0 \end{smallmatrix} \right]$ by \cite{MR1438643}.
However the space of Hurtubise normal forms is \emph{not} the Slodowy slice defined by $Y + \mathfrak z_{\g}(X)$: Consider an $\algsl(2)$-triple
\begin{equation*}
\begin{gathered}
X := \left[
\begin{array}{c|c}
0 & 0 \\ \hline
0 & X_0
\end{array}
\right],\ 
X_0 := \left(
\begin{array}{cccccc}
  0  & n'-1  &&&\bigzero & \\
     &  & 2(n'-2) & & & \\
& & &  \raisebox{5pt}{$\ddots$}  & & \\
  &  & &&& n'-1 \\ 
& \raisebox{10pt}{\bigzero}&&& &0  
\end{array}
\right), \\
Y := \left[
\begin{array}{c|c}
0 & 0 \\ \hline
0 & Y_0
\end{array}
\right],\ 
Y_0 := \left(
\begin{array}{ccccc}
  0  &   & && \\
  1  &   & & \bigzero & \\
  & & \raisebox{5pt}{$\ddots$}  & & \\
  & \raisebox{10pt}{\bigzero} &&1& 0  
\end{array}
\right), \\ 
H := \left[
\begin{array}{c|c}
0 & 0 \\ \hline
0 & H_0
\end{array}
\right],\ 
H_0 := \left(
\begin{array}{ccccc}
   n'-1  &   & &\bigzero & \\
  & n'-3  & & & \\
 & & \raisebox{5pt}{$\ddots$}  & & \\
 \raisebox{10pt}{\bigzero} &&&& 1 - n'  
\end{array}
\right),
\end{gathered}
\end{equation*}
where $n'= n-m$. Then a matrix in $Y + \mathfrak{z}_{\g}(X)$ is of the form
\begin{equation*}
Y + \left(
\begin{array}{c|cccc}
  h & 0 & \cdots  & 0 & g \\ \hline
  f &   & && \\
  0  &   &&\raisebox{-10pt}{\makebox(0,0){\huge $P$}}  &  \\
  \raisebox{10pt}{\vdots} & & & & \\
  0 &&&&  
\end{array}
\right)\quad
(P\in \CC[X_0]).
\end{equation*}
Nevertheless the proof of \cite{MR1438643} works for the space of Hurtubise normal forms: The space of solutions of Nahm's equation is first identified with a symplectic quotient of $T^* \GL(n)$ by a certain nilpotent subgroup of $\GL(n)$. Then the latter is shown to be isomorphic to $\GL(n)\times\mathcal S$. The second statement is \cite[Lem.~3.2]{MR1438643}, but the same proof works also for Hurtubise normal forms.
\begin{NB}
In fact, the slice is more close to one considered in \cite{MR1968260}.
The slice considered in \cite{MR1968260} gives a matrix of the form
\begin{equation*}
X + \left(
\begin{array}{c|cccc}
  h & 0 & \cdots  & 0 & g \\ \hline
  f & e_{n'-1}  & \cdots & e_1 & e_0 \\
  0  &   &&&  \\
  \raisebox{10pt}{\vdots} && \raisebox{10pt}{\bigzero}  & & \\
  0 &&&&  
\end{array}
\right).
\end{equation*}
\end{NB}%
\end{Remark}

\begin{NB}
\begin{Lemma}
$\tr (B_1^k), \tr (B_2^k)$ is Poisson commutating.
\end{Lemma}

  There is no proof of this statement...... Also it is not clear that it means
  $\{ \tr(B_1^k), \tr(B_1^l) \} = 0$, which turns out to be true, but not clear to HN.
\end{NB}%

The moment map for the $\GL(\bv_1)\times \GL(\bv_2)$-action is $(-B_1, B_2)$ by the formula \cite[(2.11)]{MR1438643}.
\begin{NB}
    The formula says the moment map is $(\eta, -
    \operatorname{Ad}(u^{-1})\eta)$, where the first summand is for
    the left multiplication, and the second is the right
    multiplication. See \cite[(2.9)]{MR1438643}.

    Consider $\bv_1 > \bv_2$. Then $n = \bv_1$, $m = \bv_2$. Hence
    $\GL(\bv_1)$ is the right, $\GL(\bv_2)$ is the left
    multiplication. Hence the moment map is $(-u^{-1}\eta u, h) =
    (-B_1,B_2)$.

    Consider $\bv_1 = \bv_2$. Then the left multiplication is
    $\GL(\bv_2)$, the right is $\GL(\bv_1)$, hence the same as above.

    Consider $\bv_1 < \bv_2$. Then $\GL(\bv_1)$ is the left
    multiplication, and $\GL(\bv_2)$ is the right
    multiplication. Therefore the moment map is $(h,-u^{-1}\eta u) =
    (-B_1, B_2)$. The original version was WRONG.

    But we do not fix the sign for $A$, $a$, $b$ yet at this moment.
\end{NB}%

\begin{Corollary}
There exists a $\GL(\bv_1)\times \GL(\bv_2)$-equivariant isomorphism between a triangle and its inverse-orientation:
\begin{align*}
\xymatrix@C=1.2em{ V_1 \ar@(ur,ul)_{B_1} \ar[rr]^{A} \ar[dr]_{b} && V_2
  \ar@(ur,ul)_{B_2} & \cong & V_1 \ar@(ur,ul)_{B_1} && V_2 \ar[ll]_{A'} \ar@(ur,ul)_{B_2} \ar[dl]^{a'} 
  \\
  & \CC \ar[ur]_a& &&  & \CC \ar[ul]^{b'}.&}
\end{align*}\label{cor:triangle_inverse}
\end{Corollary}
\begin{proof}
By Proposition \ref{prop:triangle-Nahm}, triangles in the both sides are isomorphic to the same moduli space of solutions of Nahm's equations over an interval.

More concretely it is given by
\begin{equation*}
    A' = A^{-1},\quad  a' = b A^{-1}, \quad b' = A^{-1} a
\end{equation*}
if $\bv_1 = \bv_2$, and through $(u,\eta)$ if $\bv_1\neq \bv_2$.
\end{proof}

\subsubsection{Factorization}
Let $\Psi$ be a map from a triangle to eigenvalues $w_{1,k}$ of $B_1$ and $w_{2,k}$ of $B_2$:
\begin{align*}
\Psi \colon \mu^{-1}(0) \rightarrow {\AAA}^{\underline{\bv}}.
\end{align*}
Here ${\AAA}^{\underline{\bv}}= {\AAA}^{|\underline{\bv}|}/\mathfrak{S}_{\underline{\bv}}=S^{\bv_1}\CC \times S^{\bv_2}\CC$ as in \subsecref{subsec:factorization-map}.
For $\underline{\bv}= \underline{\bv}' + \underline{\bv}''$, let $({\AAA}^{\underline{\bv}'}\times{\AAA}^{\underline{\bv}''})_{\mathrm{disj}}$ be the open subset of ${\AAA}^{\underline{\bv}'}\times{\AAA}^{\underline{\bv}''}$ consisting of pairs of disjoint configurations.
\begin{Proposition}[cf. {\cite[Proof of Prop.~2.11]{fra}}]
$\cM$ has a factorization isomorphism:\label{prop:triangle_factor}
\begin{align*}
\mathfrak{f}_{\underline{\bv}',\underline{\bv}''} \colon \cM(\underline{\bv})\times_{{\AAA}^{\underline{\bv}}}({\AAA}^{\underline{\bv}'}\times{\AAA}^{\underline{\bv}''})_{\mathrm{disj}} \xrightarrow{\ \sim \ } (\cM(\underline{\bv}')\times \cM(\underline{\bv}'')) \times_{{\AAA}^{\underline{\bv}'}\times{\AAA}^{\underline{\bv}''}}({\AAA}^{\underline{\bv}'}\times{\AAA}^{\underline{\bv}''})_{\mathrm{disj}}.
\end{align*}
\end{Proposition}
We use the next obvious lemma:
\begin{Lemma}
Let $J(w, n)$ be the Jordan block with size $n$ and eigenvalue $w$, and $\phi_{(n, m, w_1, w_2)}$ be a linear map
\begin{align*}
\phi_{(n, m, w_1, w_2)} \colon M(n, m;\CC)\rightarrow M(n, m;\CC), \ \ \ \phi_{(n, m, w_1, w_2)}(A) = J(w_1, n)A - AJ(w_2, m).
\end{align*}
{\em (i)} When $w_1 \neq w_2$, this map is an isomorphism.\\
{\em (ii)} When $w_1=w_2$, $\rank \phi_{(n, m, w, w)} = mn-\min(n,m)$.\label{lem:decomposing}
\end{Lemma}
\begin{proof}[Proof of Proposition \ref{prop:triangle_factor}]
The following is the same as one in \cite{fra}, but we check also that the isomorphism preserves the conditions (S1) and (S2).
Let us decompose $V_i=V_i^{(1)}\oplus V_i^{(2)}$, where $V_i^{(1)}$ is the eigenspace of $B_i$ with eigenvalue $w \in \CC$, and $V_i^{(2)}$ is the sum of the eigenspaces with different eigenvalues.
We take a representative
\begin{align*}
(A, B_1, B_2, a, b) = \left( \left[ \begin{smallmatrix} A^{(11)} & A^{(12)} \\ A^{(21)} & A^{(22)} \end{smallmatrix}\right], \left[ \begin{smallmatrix} B_1^{(1)} & 0 \\ 0 & B_1^{(2)} \end{smallmatrix}\right], \left[ \begin{smallmatrix} B_2^{(1)} & 0 \\ 0 & B_2^{(2)} \end{smallmatrix}\right], \left[ \begin{smallmatrix} a^{(1)} \\ a^{(2)} \end{smallmatrix}\right], \left[ \begin{smallmatrix} b^{(1)} & b^{(2)} \end{smallmatrix}\right]\right).
\end{align*}
Let $S_1^{(1)}$ be a subspace of $V_1^{(1)}$ such that $B_1^{(1)}S_1^{(1)} \subset S_1^{(1)}$ and $S_1^{(1)} \subset \Ker A^{(11)}\cap \Ker b^{(1)}$.
By the equation, for any $v\in S_1^{(1)}$ and $l \in \ZZ_{\geq 0}$, we have
\begin{align*}
(B_2^{(2)}- w \id)^l A^{(21)}v = A^{(21)}(B_1^{(1)}- w \id)^l v.
\end{align*}
On the other hand, we get $(B_1^{(1)}- w \id)^l =0$ for some $l$.
Since eigenvalues of $B_1^{(1)}$ and $B_2^{(2)}$ are disjoint, $(B_2^{(2)}- w \id)^l$ becomes invertible for any $l$.
Thus $A^{(21)}S_1^{(1)}=0$ holds.
Then $S_1^{(1)} =0$ follows from (S1) for a whole triangle, and this means that $(A^{(11)}, B_1^{(1)}, B_2^{(1)}, a^{(1)}, b^{(1)})$ satisfies (S1).
One can check that $(A^{(11)}, B_1^{(1)}, B_2^{(1)}, a^{(1)}, b^{(1)})$ satisfies (S2) by the similar argument.

Furthermore, $A^{(21)}$ satisfies the equation $B_2^{(2)}A^{(21)}-A^{(21)}B_1^{(1)} + a^{(2)}b^{(1)}=0$, and determined as $\phi_{(\dim V_2^{(2)}, \dim V_1^{(1)}, w_1, w_2)}^{-1}(a^{(2)}b^{(1)})$ under this assumption by Lemma \ref{lem:decomposing}.
Thus, we get a map from the right hand side to the left hand side:
\[
\widetilde\cM(\underline{\bv})
\times_{{\AAA}^{\underline{\bv}}}({\AAA}^{\underline{\bv}'}\times{\AAA}^{\underline{\bv}''})_{\mathrm{disj}} \leftarrow
(\widetilde\cM(\underline{\bv}')\times \widetilde\cM(\underline{\bv}''))
\times_{{\AAA}^{\underline{\bv}'}\times{\AAA}^{\underline{\bv}''}}({\AAA}^{\underline{\bv}'}\times{\AAA}^{\underline{\bv}''})_{\mathrm{disj}},
\]
equivariant under $\GL(\underline{\bv}')\times\GL(\underline{\bv}'')$, which induces an isomorphism of stacks.
\end{proof}
Let $\wA^{|\underline{\bv}|}$ be the open subset of $\AAA^{|\underline{\bv}|}$ consisting of distinct configurations, and define $\wA^{\underline{\bv}}:= \wA^{|\underline{\bv}|}/\mathfrak{S}_{\underline{\bv}}$.
\begin{Corollary}
$\widetilde{\cM} \times_{\wA^{\underline{\bv}}} \wA^{|\underline{\bv}|} \cong \wA^{|\underline{\bv}|} \times \GL(\underline{\bv})$.
\end{Corollary}
\begin{proof}
We order eigenvalues of $B_1$ and $B_2$ some way, and denote them by $(w_{1,1}, w_{1,2}, \cdots )$ and $(w_{2,1}, w_{2,2}, \cdots )$ respectively.
Since they are distinct, we have $w_{i,k} \neq w_{j,l}$ if $(i,k) \neq (j, l)$.
By using the above proposition repeatedly, we can take a representative of $(A, B_1, B_2, a, b)$ as
\begin{align*}
A^{(kl)} = (w_{1,l}-w_{2,k})^{-1}, B_1= \diag(w_{1,k}), B_2 = \diag(w_{2,k}), a = {}^t\![1, \cdots, 1], b = [1, \cdots, 1].
\end{align*}
On this representative, $\GL(\underline{\bv})$ acts freely.
\end{proof}
\begin{Corollary}
$\mu^{-1}(0) \times_{\wA^{\underline{\bv}}} \wA^{|\underline{\bv}|} \cong \wA^{|\underline{\bv}|} \times (\GL(\underline{\bv}) \times \CC^{|\underline{\bv}|})/ (\CC^{\times})^{|\underline{\bv}|}$.\label{cor:triangle_factor}
\end{Corollary}

\subsubsection{Dimension estimate for \texorpdfstring{$\Psi^{-1}$}{\textPsi-1}}
Next we calculate the dimension of the fiber of $\Psi$.  Though the
following proposition is in \cite[Prop.~2.11]{fra}, we give a
self-contained argument for the sake of a reader.
\begin{Proposition}
Over each point of $\AAA^{\underline{\bv}}$, the dimension of the fiber of $\Psi$ is $\bv_1^2+\bv_2^2$.\label{prop:triangle_fiber}
\end{Proposition}
\begin{proof}
For $\wA^{\underline{\bv}}$, the assertion follows from Corollary \ref{cor:triangle_factor}.

\begin{NB}
We consider the fiber of $0 \in {\AAA}^{\underline{\bv}}$.
When $B_1$ and $B_2$ are regular nilpotent, $ab$ has $\min(\bv_1, \bv_2)$ relations by Lemma \ref{lem:decomposing}, so $\dim \{(a, b)\} = \bv_1+\bv_2-\min(\bv_1,\bv_2)=\max(\bv_1, \bv_2)$.
And $\{ A \} = \phi_{(\bv_1, \bv_2, 0, 0)}^{-1}(ab)$ has $\min(\bv_1, \bv_2)$-dimension.
Lemma \ref{lem:decomposing} also says that $\dim \GL_{B_k} = \bv_k$.
\end{NB}%

We consider the fiber of $0 \in {\AAA}^{\underline{\bv}}$.  Suppose
$B_1$ and $B_2$ are regular nilpotent. We make both $B_1$ and $B_2$
Jordan normal forms. Then $B_1 A - A B_2$ has the following property:
the sum of all elements in the $i$th diagonal, counting from the lower
left corner, is $0$ for $i=1,\dots,\min(\bv_1,\bv_2)$. Thus $\dim
\{(a, b)\}$ is at most $\bv_1+\bv_2-\min(\bv_1,\bv_2)=\max(\bv_1,
\bv_2)$.  Once $B_1$, $B_2$, $a$, $b$ are fixed, the space of
solutions of the linear equation
$\{ A \} = \phi_{(\bv_1, \bv_2, 0, 0)}^{-1}(ab)$ has
$\min(\bv_1, \bv_2)$-dimension.  Lemma \ref{lem:decomposing} also says
that $\dim \GL_{B_k} = \bv_k$.  Thus we have
\begin{align*}
\dim (\Psi^{-1}(0) \cap \{B_k = J(0, \bv_k) \}) = \max(\bv_1, \bv_2) + \min(\bv_1, \bv_2) + \bv_1^2+\bv_2^2 - \bv_1 - \bv_2 = \bv_1^2 + \bv_2^2.
\end{align*}

When $B_1$ and $B_2$ are not regular nilpotent, we describe them as
length $l$ partition: $B_1=[m_1, \cdots, m_l], B_2=[n_1, \cdots,
n_l]$, $\sum m_i = \bv_1, \sum n_i =\bv_2$. We have the corresponding
block decompositions $(A_{ij})$, $(a_i)$, $(b_i)$ of $A$, $a$,
$b$. From the above argument for the equation for $a_i$, $b_i$, we
have $\dim \{(a, b)\} \leq \sum_i \max(m_i, n_i)$. We also have $\dim
\{A\} =\sum_{i,j}\min(m_i, n_j)$ and $\dim \GL_{B_1} = \sum_{i,j}\min(m_i, m_j)$.
Thus,
\begin{align*}
\dim (\Psi^{-1}(0) \cap &\{B_k \neq J(0, \bv_k) \}) \\
& \leq \sum_i \max(m_i, n_i) + \sum_{i,j}\{\min(m_i, n_j)  - \min(m_i, m_j) - \min(n_i, n_j)\} + \bv_1^2 + \bv_2^2\\
&=\sum_{i \neq j}\{\min(m_i, n_j) - \min(m_i, m_j) -\min(n_i, n_j)\} + \bv_1^2 + \bv_2^2 \leq \bv_1^2 + \bv_2^2.
\end{align*}
\begin{NB}
    Suppose $i < j$. Then $\min(m_i,n_j) \le n_j =
    \min(n_i,n_j)$. Suppose $i > j$. Then $\min(m_i,n_j) \le m_i =
    \min(m_i,m_j)$.
\end{NB}%
Thus we have $\dim (\Psi^{-1}(0)) = \bv_1^2 + \bv_2^2$.
For other fibers, we use the factorization.
\end{proof}
\begin{Corollary}
$\mu^{-1}(0)$ is an irreducible reduced complete intersection in $\mathbb{M}$.
\end{Corollary}
\begin{proof}
$\Psi^{-1}(\wA^{\underline{\bv}})$ is smooth and the dimension of its complement is less than $\bv_1^2+\bv_1+\bv_2^2+\bv_2$.
\end{proof}

\subsubsection{Shift}\label{subsubsec:shift}

Note
\begin{align*}
\mu(A, B_1, B_2, a, b)=(B_2+\zeta\id_{V_2})A-A(B_1+\zeta\id_{V_1})+ab.
\end{align*}
Therefore we have an action of $\CC$ on $\mu^{-1}(0)$ and
$\widetilde\cM$ so that the moment map is shifted as $(-B_1, B_2)
\mapsto (-B_1-\zeta\id_{V_1}, B_2+\zeta\id_{V_2})$.

\begin{NB}
    Since $\mu = \mu_{\zeta}$, the original explanation is confusing.

For $\zeta \in \CC^2$, we define a shift of $\mu$:
\begin{align*}
\mu_{\zeta}(A, B_1, B_2, a, b):=(B_2+\zeta\id_{V_2})A-A(B_1+\zeta\id_{V_1})+ab.
\end{align*}
Then we see $\mu_{\zeta}^{-1}(0) = \mu^{-1}(0)$.
In particular, we can always shift the moment map for a triangle as $(-B_1, B_2) \mapsto (-B_1-\zeta\id_{V_1}, B_2+\zeta\id_{V_2})$.
\end{NB}%
\begin{Remark}
This note corresponds to moving parameters of the hyper-K\"ahler moment map of a moduli space of solutions of Nahm's equations.
See also \cite[Rem.~2.5]{MR3300314}.
\end{Remark}

\subsection{Two-way part}\label{subsec:twoway}
\subsubsection{Definition}
Let $\bw \in \ZZ_{>0}$, $\underline{\bv} \in \ZZ^{\bw + 1}_{\geq 0}$.
We consider the following quiver:
\begin{align*}
\xymatrix@C=1.2em{V^0 \ar@<-.5ex>[rr]_{C_1} && V^1 \ar@<-.5ex>[ll]_{D_1} \ar@<-.5ex>[rr]_{C_2} && \ar@<-.5ex>[ll]_{D_2} \ar@{.}[r] & \ar@<-.5ex>[rr]_{C_{\bw-1}} && V^{\bw-1} \ar@<-.5ex>[ll]_{D_{\bw-1}} \ar@<-.5ex>[rr]_{C_{\bw}} && V^{\bw}, \ar@<-.5ex>[ll]_{D_{\bw}}}
\end{align*}
where $\dim V^k = \bv_k$.
Set
\begin{align*}
\mathbb M \equiv
\mathbb{M}(\underline{\bv}; \bw) := \bigoplus_{k=0}^{\bw-1} \Hom (V^k, V^{k+1})\oplus \Hom (V^{k+1}, V^k).
\end{align*}
It is obvious that $\mathbb{M}$ is a $\sum 2\bv_k\bv_{k+1}$-dimensional hyper-K\"ahler manifold.
We consider this as a holomorphic symplectic manifold.

$\mathbb{M}$ has a $\GL(\underline{\bv})$-hamiltonian action, and its moment map is $\mu_{\bv_k} = C_{k}D_{k} - D_{k+1}C_{k+1}$.
Thus we have a holomorphic symplectic quotient
\begin{align*}
\widetilde{\cM}(\underline{\bv}; \bw):= \bigcap_{k=1}^{\bw-1}\mu_{\bv_k}^{-1}(0) \dslash \prod_{k=1}^{\bw-1} \GL(\bv_k).
\end{align*}
And define $\cM(\underline{\bv}; \bw)$ as the quotient stack
$[\widetilde{\cM}(\underline{\bv}; \bw) / \GL(\bv_0)\times \GL(\bv_\bw)]$.
\begin{NB}
a set of $\GL(\bv)\times \GL(\bv)$-equivalence classes of $\widetilde{\cM}(\underline{\bv}; \bw)$.
 \end{NB}%
\begin{NB} 
\begin{align*}
\cM(\underline{\bv}; \bw):=[\widetilde{\cM}(\underline{\bv}; \bw) / \GL(\bv)\times \GL(\bv)].
\end{align*}
\end{NB}%
Remark that we do not take quotients by the groups at both ends $k=0, \bw$ for $\widetilde{\cM}(\underline{\bv}; \bw)$.
\begin{NB}
In Cherkis' description, two-way part is described as follows:
\begin{align*}
\begin{xy}
(-5,0)*{\widetilde{\cM}(\underline{\bv};\bw):},
(10,3)*{\bv_0},
(15,0)*{\boldsymbol\medcirc},
(20,3)*{\bv_1},
(25,0)*{\boldsymbol\medcirc},
(30,3)*{\cdots},
(35,0)*{\boldsymbol\medcirc},
(40,3)*{\bv_{\bw}},
\ar @{-} (10,0);(40,0)
\end{xy}
\end{align*}
\end{NB}%
As a special case, we often consider $\bv\cdot \underline{1} := (\bv, \cdots, \bv)$.
The coordinate ring of $\widetilde{\cM}(\underline{\bv};\bw)$ is generated by entries of $C_{\bw \cdots 1} := C_{\bw} \cdots C_1, D_{1 \cdots \bw} := D_1 \cdots D_{\bw}$ and $D_1C_1, C_{\bw}D_{\bw}$.
(See e.g.\ \cite{MR958897}). 
Therefore we often suppress the two-way parts as
\begin{align*}
\xymatrix@C=1.2em{V^0 \ar@<-.5ex>[rrr]_{C_{\bw \cdots 1}} \ar@(ur,ul)_{D_1C_1} &&& V^{\bw}. \ar@<-.5ex>[lll]_{D_{1 \cdots \bw}} \ar@(ur,ul)_{C_{\bw}D_{\bw}} }
\end{align*}

\subsubsection{Dimension of \texorpdfstring{$\widetilde{\cM}(\bv\cdot \underline{1}; \bw)$}{M(v\textcdot 1,w)}}
From now on we only consider the case $\underline{\bv} = \bv\cdot \underline{1}$.
\begin{Proposition}
$\bigcap_{k=1}^{\bw-1}\mu_{\bv_k}^{-1}(0)$ has dimension $(\bw+1)\bv^2$.\label{prop:two-way_dim}
\end{Proposition}
\begin{proof}
We use the flatness criterion in \cite[Th.~1]{CB}.
Since we do not take the quotient by the left most and the right most $\GL(V)$, we replace the quiver by a framed quiver of type $A_{\bw-1}$ with dimension vectors $\underline{\bv}' = (\bv, \cdots, \bv), \underline{\bw}' = (\bv, 0, \cdots, 0, \bv)$.
The criterion in \cite[Th.~1]{CB} for a framed quiver is given in \cite[Th.~2.15]{Na-branching}.
(More precisely, we replace $>$ by $\geq$, as the criterion in \cite[Th.~2]{CB} is given there.)
It is easy to check that the criterion is satisfied in this case.
\end{proof}
\begin{NB}
\begin{align*}
\begin{xy}
(-10,0)*{\underline{\bv'}},
(-10,-10)*{\underline{\bw'}},
(0,0)*{\bv},
(0,-10)*{\bv},
(10,0)*{\bv},
(20,0)*{\bv},
(50,0)*{\bv},
(60,0)*{\bv},
(60,-10)*{\bv},
\ar @{-} (2,0);(8,0)
\ar @{-} (0,-2);(0,-8)
\ar @{-} (12,0);(18,0)
\ar @{-} (22,0);(28,0)
\ar @{.} (30,0);(40,0)
\ar @{-} (42,0);(48,0)
\ar @{-} (52,0);(58,0)
\ar @{-} (60,-2);(60,-8)
\end{xy}
\end{align*}
Then we have $\underline{\bw'} - C\underline{\bv'} = 0$ and
\begin{align*}
{}^t \underline{\bv'}\left( \underline{\bw'} - \frac{1}{2}C\underline{\bv'} \right) = \frac{1}{2}{}^t \underline{\bv'} C\underline{\bv'}.
\end{align*}
On the right hand side, as $p(\beta^{(t)}) = 0$ for finite type ADE quiver, we get
\begin{align*}
{}^t \underline{\bv^0}\left( \underline{\bw'} - \frac{1}{2}C\underline{\bv^0} \right) = {}^t \underline{\bv^0}\left( C\underline{\bv'} - \frac{1}{2} C\underline{\bv^0}\right).
\end{align*}
Thus
\begin{align*}
2(\text{LHS} - \text{RHS}) &= {}^t\underline{\bv'}C\underline{\bv'} + {}^t\underline{\bv^0}C\underline{\bv^0} - 2 {}^t\underline{\bv^0}C\underline{\bv'} \\
&={}^t(\underline{\bv^0} + \underline{\bv'})C(\underline{\bv^0} + \underline{\bv'}) > 0,
\end{align*}
because the Cartan matrix is positive definite for ADE.

\begin{Corollary}
$\widetilde{\cM}^{\mathrm{reg}}(\bv\cdot \underline{1}; \bw)$ is nonempty.
\end{Corollary}
\begin{proof}
$\widetilde{\cM}(\bv\cdot \underline{1}; \bw)$ is a type $A$ quiver variety $\fM(\bv', \bw')$, where $\bv' = (\bv, \cdots, \bv) \in \ZZ_{\geq 0}^{\bw-1}$ and $\bw' = (\bv, 0, \cdots, 0, \bv)$.
This dimension vector $(\bv', \bw')$ is a dominant, thus the assertion follows from \cite{Na-quiver}.
\end{proof}
\end{NB}%

\subsubsection{Factorization}
Let $\Psi$ be a map from $\bigcap\mu_{\bv_k}^{-1}(0)$ to eigenvalues of $D_1C_1$:
\begin{align*}
\Psi \colon \bigcap_{k=1}^{\bw-1}\mu_{\bv_k}^{-1}(0) \rightarrow {\AAA}^{\bv}.
\end{align*}
Notice that we can choose any $D_kC_k$ to define $\Psi$ because the characteristic polynomials for $D_kC_k$ are all same.
For $\bv= \bv' + \bv''$, take $({\AAA}^{\bv'}\times{\AAA}^{\bv''})_{\mathrm{disj}} \subset {\AAA}^{\bv'}\times{\AAA}^{\bv''}$.
\begin{Proposition}
$\cM(\bv\cdot \underline{1}; \bw)$ has a factorization morphism:
\begin{align*}
\mathfrak{f}_{\bv',\bv''} \colon \cM(\bv\cdot \underline{1}; \bw)\times_{{\AAA}^{\bv}}({\AAA}^{\bv'}\times{\AAA}^{\bv''})_{\mathrm{disj}} \xrightarrow{\ \sim \ } (\cM(\bv'\cdot \underline{1}; \bw)\times \cM(\bv''\cdot \underline{1}; \bw)) \times_{{\AAA}^{\bv'}\times{\AAA}^{\bv''}}({\AAA}^{\bv'}\times{\AAA}^{\bv''})_{\mathrm{disj}}.
\end{align*}\label{prop:twoway_factor}
\end{Proposition}
\begin{proof}
Put $X_k = D_kC_k$ for $1 \leq k \leq \bw$ and $X_{\bw+1} = C_{\bw}D_{\bw}$.
Let us decompose $V_k = V_k^{(1)} \oplus V_k^{(2)}$, where $V_k^{(1)}$ is the eigenspace of $X_{k+1}$ with eigenvalue $w \in \CC$, and $V_k^{(2)}$ is the sum of the eigenspaces with different eigenvalues.
We take a representative
\begin{align*}
(X_k, C_k, D_k) = \left( \left[ \begin{smallmatrix} X_k^{(1)} & 0 \\ 0 & X_k^{(2)} \end{smallmatrix} \right], \left[ \begin{smallmatrix} C_k^{(11)} & C_k^{(12)} \\ C_k^{(21)} & C_k^{(22)} \end{smallmatrix} \right], \left[ \begin{smallmatrix} D_k^{(11)} & D_k^{(12)} \\ D_k^{(21)} & D_k^{(22)} \end{smallmatrix} \right] \right).
\end{align*}
And the equations $X_{k+1}C_k = C_kD_kC_k = C_k X_k$ and $X_kD_k=D_kC_kD_k = D_kX_{k+1}$ mean
\begin{align*}
X_{k+1}^{(1)}C_k^{(12)} = C_k^{(12)}X_k^{(2)}, \ \ \ X_k^{(1)}D_k^{(12)} = D_k^{(12)}X_{k+1}^{(2)}.
\end{align*}
Like in the proof of Proposition \ref{prop:triangle_factor}, an equation $(X_{k+1}^{(1)}- w \id)^l C^{(12)}_k = C^{(12)}_k(X_k^{(2)}- w \id)^l$ implies $C^{(12)}_k=0$.
Thus we can conclude that the off-diagonal components of $C_k$ and $D_k$ vanish.
This leads to an isomorphism of stacks from the right hand side to the left hand side.
\end{proof}

\begin{NB}
We consider the following quiver:
\begin{align*}
\begin{xy}
(-10,0)*{\bullet},
(-5,0)*{\vdots},
(-5,-6)*{\bv},
(0,0)*{\bullet},
(10,0)*{\bullet},
(20,-4)*{\underbrace{\hspace{36mm}}_{\bw-2}},
(30,0)*{\bullet},
(40,0)*{\bullet},
(50,0)*{\bullet},
(45, 0)*{\vdots},
(45,-6)*{\bv.},
\ar (2,0);(8,0)
\ar (12,0);(18,0)
\ar @{.} (20,0);(28,0)
\ar (32,0);(38,0)
\ar (42,4);(48,4)
\ar (42,2);(48,2)
\ar (-8,-4);(-2,-4)
\ar (-8,4);(-2,4)
\ar (-8,2);(-2,2)
\ar (-8,-4);(-2,-4)
\end{xy}
\end{align*}
Here we identify the both ends.
Let $\mu_{\alpha}$ be the moment map defined for the double of the quiver representation for the dimension vector $\alpha = (1, \bv, \cdots, \bv)$.
Then, by [Crawley-Boevey], $\mu_{\alpha} \colon \mathrm{Rep}(\overline{Q}, \alpha) \rightarrow \End (\alpha)_{0}$ is a flat morphism.
Since $\bigcap\mu_{\bv_k}^{-1}(0) = \mu_{\alpha}^{-1}(0)$, we obtain $\dim \bigcap\mu_{\bv_k}^{-1}(0) = (\bw+1)\cdot 2\bv^2 - \bw\bv^2$.

\subsubsection{Dimension estimate for $\Psi^{-1}$}
Next we calculate the dimension of the fiber of $\Psi$.
\begin{Proposition}
Over each point of $\AAA^{\bv}$, the dimension of the fiber of $\Psi$ is $(\bw+1)\bv^2 -\bv$.
\end{Proposition}
\begin{proof}
We consider the following quiver:
\begin{align*}
\begin{xy}
(0,0)*{\bullet},
(10,0)*{\bullet},
(20,-4)*{\underbrace{\hspace{36mm}}_{\bw-1}},
(30,0)*{\bullet},
(40,0)*{\bullet},
(50,0)*{\bullet},
(45, 0)*{\vdots},
(45,-6)*{\bv.},
\ar (2,0);(8,0)
\ar (12,0);(18,0)
\ar @{.} (20,0);(28,0)
\ar (32,0);(38,0)
\ar (42,4);(48,4)
\ar (42,2);(48,2)
\ar (42,-4);(48,-4)
\end{xy}
\end{align*}
For any $x \in \AAA^{\bv}$, take $X\in \gl(\bv)$ such that $\Psi(X) = x$.
Notice that $\dim \GL / \GL_{X} \leq \bv^2 - \bv$.

Let $\mu_{\alpha}$ be the moment map defined for the double of the quiver representation for the dimension vector $\alpha = (\bv, \cdots, \bv, 1)$.
Then, by [Crawley-Boevey], $\mu_{\alpha} \colon \mathrm{Rep}(\overline{Q}, \alpha) \rightarrow \End (\alpha)_{0}$ is a flat morphism.
On the other hand, we can see that $\Psi^{-1}(x) \cong \mu^{-1}_{\bv}(X, 0, \cdots, 0, \tr X) \times \GL / \GL_{X}$.
Thus we obtain $\dim \Psi^{-1}(x) = \bw\bv^2+ \bv^2-\bv$.
\end{proof}
\end{NB}

\subsubsection{Deformation and resolution}
For $\nu^\CC \in \CC^{\bw-1}$, we can naturally define a deformation:
\begin{align*}
\widetilde{\cM}_{\nu^\CC}(\bv\cdot \underline{1}; \bw):= \bigcap_{k=1}^{\bw-1}\mu_{\bv_k}^{-1}(\nu_k^\CC\id) \dslash \prod_{k=1}^{\bw-1} \GL(\bv_k).
\end{align*}
Since $\mu_{\bv_k}$ is flat, this deformation forms a
$(\bw-1)$-dimensional family.  For a deformed two-way part, the
argument in the proof of Proposition \ref{prop:twoway_factor} is also
valid, so a deformed two-way part has a factorization morphism:
\begin{align*}
\cM_{\nu^\CC}(\bv\cdot \underline{1}; \bw)\times_{{\AAA}^{\bv}}({\AAA}^{\bv'}\times{\AAA}^{\bv''})_{\mathrm{disj}} \xrightarrow{\ \sim \ } (\cM_{\nu^\CC}(\bv'\cdot \underline{1}; \bw)\times \cM_{\nu^\CC}(\bv''\cdot \underline{1}; \bw)) \times_{{\AAA}^{\bv'}\times{\AAA}^{\bv''}}({\AAA}^{\bv'}\times{\AAA}^{\bv''})_{\mathrm{disj}}.
\end{align*}

The factorization is compatible also with the stability parameter
$\nu^\RR\in\RR^{\bw-1}$: A subspace $S$ or $T$ is invariant under $C$,
$D$, hence also under $CD$, $DC$. Therefore $S$, $T$ decompose
according to the decomposition $V_k = V_K^{(1)}\oplus V_k^{(2)}$.
\begin{NB}
    For two-way case, $S$ or $V/T$ appear only in the $0$-eigenspace,
    as we impose $D_1(S^1) = 0 = C_\bw(S^{\bw-1})$, $\Ima C_1 \subset
    T^1$, $\Ima D_\bw \subset T^{\bw-1}$. But our argument works also
    for general bow varieties.
\end{NB}%

\begin{NB}
This part must be rewritten according to the earlier section.

\subsection{Definition of Cherkis bow varieties}
A bow variety is defined as a holomorphic symplectic quotient of some $\mathring{\mathfrak{M}}_{\mathrm{t}}(\underline{\bv})$ and some $\mathfrak{M}_{\mathrm{q}}(\underline{\bv})$.
And it is also described by using $\begin{xy} (5,0)*{\boldsymbol\times}, \ar @{-} (0,0);(10,0) \end{xy}$ and $\begin{xy} (5,0)*{\boldsymbol\medcirc}, \ar @{-}(0,0);(10,0) \end{xy}$.
For example, a diagram
\begin{align*}
\begin{xy}
(5,3)*{\bv_0},
(10,0)*{\boldsymbol\times},
(10,-4)*{\underline{1}},
(15,3)*{\bv_1},
(20,0)*{\boldsymbol\medcirc},
(20, -4)*{\overline{1}},
(25,3)*{\bv_2},
(30,0)*{\boldsymbol\medcirc},
(30,-4)*{\overline{2}},
(35,3)*{\bv_3},
(40,0)*{\boldsymbol\times},
(40,-4)*{\underline{2}},
(45,3)*{\bv_0},
\ar @{-} (5,0);(45,0)
\ar @{-} @(l, l) (5,0);(5,8)
\ar @{-} (5,8);(45,8)
\ar @{-} @(r, r) (45,0);(45,8)
\end{xy}
\end{align*}
means a holomorphic symplectic quotient of $\mathring{\mathfrak{M}}_{\mathrm{t}}(\bv_0, \bv_1) \times \mathfrak{M}_{\mathrm{q}}(\bv_1, \bv_2) \times \mathfrak{M}_{\mathrm{q}}(\bv_2, \bv_3) \times \mathring{\mathfrak{M}}_{\mathrm{t}}(\bv_3, \bv_0)$ by $\GL(\bv_0) \times \GL(\bv_1) \times \GL(\bv_2) \times \GL(\bv_3)$.
We set indexes $\overline{i}$ for $\boldsymbol\medcirc$ and $\underline{i}$ for $\boldsymbol\times$ as above, and refer the dimensions of both sides of $\overline{i}$ and $\underline{i}$ as $\bv_{\overline{i}}^{\pm}$ and $\bv_{\underline{i}}^{\pm}$ respectively.
At the above example, we say $\bv_{\overline{1}}^+=\bv_2, \bv_{\overline{1}}^-=\bv_1$ and $\bv_{\underline{1}}^+=\bv_1, \bv_{\underline{1}}^-=\bv_0$.

\begin{Definition}
A dimension vector $\underline{\bv}$ is called {\em balanced} if $\bv_{\overline{i}}^+ = \bv_{\overline{i}}^-$ for each $\overline{i}$.
A dimension vector $\underline{\bv}$ is called {\em dual balanced} if $\bv_{\underline{i}}^+ = \bv_{\underline{i}}^-$ for each $\underline{i}$.
\end{Definition}
\begin{Theorem}[Cherkis]
A bow variety with a dual balanced dimension vector is a quiver variety.
\end{Theorem}
\begin{proof}
When $\bv$ is dual balanced, Proposition 2.2 implies $\mathring{\mathfrak{M}}_{\mathrm{t}}(\bv_{\underline{i}}^-, \bv_{\underline{i}}^+)$ becomes a framing part of a quiver variety.
\end{proof}

\end{NB}%

\subsection{Proof of \thmref{thm:quiver_desc}}\label{subsec:quiver_proof}

We divide bow solutions into two parts, (1) solutions of Nahm's
equation possibly with $I$, $J$ on an interval containing
$\boldsymbol\times$, and (2) $B^{LR}$, $B^{RL}$. For (1) we separate
the interval $[\xl_{\xp,L},\xl_{\xp,R}]$ at the middle of segments
$(\xl^i_\xp,\xl^{i+1}_\xp)$. Correspondingly we take the quotient by
the gauge equivalence in two steps. We first take the quotient by
gauge transformations which are identity at the end points of
intervals and middle points of segments. Then we next take the
quotient by the residual group, that is $\prod
\operatorname{U}(V_\zeta)$.
Thus $\cM_{\mathrm{bow}}$ is a hyper-K\"ahler quotient of the product
of the space of solutions of Nahm's equation possibly plus $I$, $J$,
and the quaternionic vector space consisting of (2) by the group
$\prod \operatorname{U}(V_\zeta)$.

Now the space of solutions of Nahm's equation possibly plus $I$, $J$
is described by a triple $(A,a,b)$ together with $B$ at $V_j^{w_j}$
and $V_{j+1}^0$ by \propref{prop:triangle-Nahm}.
For (2) we just set $C = B^{LR}$, $D = B^{RL}$. Those form a
quaternionic vector space. Therefore the quotient by the first gauge
equivalence is isomorphic to $\widetilde\cM_{\mathrm{sym}}$. Hence
$\cM_{\mathrm{bow}}$ is a hyper-K\"ahler quotient of
$\widetilde\cM_{\mathrm{sym}}$ by $\prod \operatorname{U}(V_\zeta)$,
while $\cM_{\mathrm{quiver}}$ is a holomorphic symplectic quotient of
$\widetilde\cM_{\mathrm{sym}}$ by $\prod \GL(V_\zeta)$.

The remaining part is a standard identification of GIT quotients and
(real) symplectic quotients due to Kempf-Ness, Kirwan and others, as
$\widetilde\cM_{\mathrm{sym}}$ is finite dimensional. Although we do
not have an explicit (finite-dimensional) description of the
hyper-K\"ahler metric on $\widetilde\cM_{\mathrm{sym}}$, the result
remains true, as is explained in detail in \cite[Prop.~4.7]{MR3300314}
in a closely related situation.

\section{Basic properties}\label{sec:basic-properties}

\subsection{Stratification}\label{subsec:strat}

Let us consider a bow diagram with no triangle part. Then the
corresponding bow variety is a quiver variety of affine type
$A_{\ell-1}$ without framing (i.e., $W=0$) considered in
\cite{Na-quiver}. Let us denote it by $\fM(\underline{\bv})$ to
emphasize the dimension vector $\underline{\bv}$. Let
$\fM^{\mathrm{s}}(\underline{\bv})$ be its open subvariety of
$\nu^\RR$-stable points.

We return back to a general bow diagram.

\begin{Lemma}\label{lem:direct_sum}
    Suppose that a bow diagram contains at least one two-way part, and
    let $x\in\cM$. Then $x$ has a representative
    $(A,B,a,b,C,D)\in\widetilde\cM$ which is a direct sum
    \begin{equation*}
        (A^s,B^s,a^s,b^s,C^s,D^s) \oplus \bigoplus_k
        (A^k = \id,B^k,0,0,C^k,D^k)
    \end{equation*}
    such that
    \begin{enumerate}
          \item $(A^s,B^s,a^s,b^s,C^s,D^s)$ is $\nu^\RR$-stable.
          \item $(C^k,D^k)$ represents a point in
        $\fM^{\mathrm{s}}(\underline{\bv}^k)$, and is extended to
        a bow by setting $A^k = \id$, $B^k$ is an appropriate $C^k
        D^k$ or $D^k C^k$.
    \end{enumerate}
\end{Lemma}

Note that $\fM^{\mathrm{s}}(\underline{\bv}^k)$ is nonempty only if
$\nu\cdot\underline{\bv}^k = (\nu^\RR\cdot\underline{\bv}^k,
\nu^\CC\cdot\underline{\bv}^k) = 0$.

\begin{proof}
    We take $(A,B,a,b,C,D)\in\widetilde\cM^{\mathrm{ss}}$ so that it
    has a closed $\GV = \prod\GL(V_\zeta)$-orbit in
    $\widetilde\cM^{\mathrm{ss}}$.

    If it is not $\nu^\RR$-stable, there is a collection of subspaces
    $S_\zeta$ or $T_\zeta$ as in $(\nu 1)$, $(\nu 2)$, where
    inequalities are equalities. As in the proof of
    \propref{prop:numerical}, we take a $1$-parameter subgroup
    $\rho(t)$ so that the limit is the associated graded.
    \begin{NB}
        Note that it has $\chi(\rho(t)) = 1$, hence 
    \end{NB}%
    But the limit stays in the same orbit as we assume $(A,B,a,b,C,D)$
    has a closed orbit. Therefore $(A,B,a,b,C,D)$ is a direct sum of
    $\nu^\RR$-stable $(A^s,B^s,a^s,b^s,C^s,D^s)$ and the data for $S =
    \{ S_\zeta\}$ or $V/T = \{ V_\zeta/T_\zeta\}$. We also see that $V
    = V/T\oplus T$, hence we do not need to consider the case $V/T$.

    Let us consider the second summand. By the conditions in $(\nu
    1)$, $(\nu 2)$, the induced $a$, $b$ are zero on $S$ or $V/T$, and
    $A$ is an isomorphism. So we can set $A=\id$. From the equation $B
    A = A B$, all $B$ on a wavy line are equal. Hence it is enough to
    determine either end of the wavy line, but it is given by an
    appropriate $CD$ or $DC$. Thus the second summand can be regarded
    as a point in $\fM(\underline{\bv})$ with the dimension vector
    $\underline{\bv}$ given by $(\dim S_\xp)$.
    
    The induced $C$, $D$ on $S$ is not necessary $\nu^\RR$-stable, but
    it decomposes into a direct sum of $\nu^\RR$-stable as it has a
    closed orbit.
\end{proof}

The case when there is no two-way part can be treated in the same way,
but let us postpone it till \propref{prop:chainsaw}.

\begin{Remark}\label{rem:root}
    A description of $\fM^{\mathrm{s}}(\underline{\bv}^k)$ is given in
    \cite[Lem.~3.2]{Na-branching}. There are two types:
    \begin{enumerate}
          \item $\underline{\bv}^k = \delta$, the primitive positive
        imaginary root, and $\fM^{\mathrm{s}}(\delta)$ is the regular
        locus of a partial resolution of the deformation of
        $\CC^2/(\ZZ/\ell\ZZ)$ introduced in \cite{Kr} as a
        hyper-K\"ahler quotient. Note also $\delta = (1,1,\dots,1)$
        since we are considering type $A$.

          \item $\underline{\bv}^k$ is a real positive root,
        indecomposable among roots in $\nu^\perp$, and
        $\fM^{\mathrm{s}}(\underline{\bv}^k)$ is a single point.
    \end{enumerate}

    The proof is given by the reduction to Crawley-Boevey's criterion
    \cite{CB} and the hyper-K\"ahler rotation, hence under a
    restriction on $\nu$. Since Crawley-Boevey's criterion is
    generalized in \cite{2016arXiv160200164B}, the same results remain
    true for any $\nu$.

    Let us consider two extreme cases (a) $\nu$ is not contained in
    any root hyperplanes, (b) $\nu = 0$. In case (a),
    $\fM^{\mathrm{s}}(\underline{\bv})$ is empty, hence we have $\cM =
    \cM^{\mathrm{s}}$. In case (b), the first type of
    $\fM^{\mathrm{s}}(\delta)$ is just $\CC^2\setminus \{0\}/(\ZZ/\ell
    \ZZ)$ ($\ell\neq 1$) or $\CC^2$ ($\ell=1$), and the second type
    means that $\underline{\bv}^k$ is a coordinate vector.
\end{Remark}

This lemma together with the above remark gives us a stratification of
$\cM$. Let us specify the dimension vector $\underline{\bv}$ in the
notation of a bow variety as $\cM(\underline{\bv})$. The underlying
graph, $\boldsymbol\medcirc$ and $\boldsymbol\times$ on the circle, is
fixed, but we will vary dimension vectors.
We will use a similar looking notation $\cM(\underline{\bv},
\underline{\bw})$ in \subsecref{subsec:Coulomb_gauge}, where the
underlying graph is determined by $\underline{\bv}$,
$\underline{\bw}$.
There are no relations among them.

Let us suppose $\nu\cdot\delta = 0$ so that the first type of
$\fM^{\mathrm{s}}(\delta)$ actually occurs. (The case
$\nu\cdot\delta\neq 0$ has a simpler stratification.) The
stratification is
\begin{equation}
    \label{eq:13}
    \cM(\underline{\bv}) = \bigsqcup
    \cM^{\mathrm{s}}(\underline{\bv}')\times 
    S^{\underline{k}}(\fM^{\mathrm{s}}(\delta)),
\end{equation}
where the summation runs over $\underline{\bv}'$ with
$\cM^{\mathrm{s}}(\underline{\bv}')\neq\emptyset$ and a partition
$\underline{k}$ such that $\underline{\bv} - \underline{\bv}' -
|\underline{k}|(1,1,\dots,1)$ is in $\nu^\perp$ and has nonnegative
entries. Here $|\underline{k}|$ is the weight of $\underline{k}$, and
$S^{\underline{k}}(\fM^{\mathrm{s}}(\delta))$ is the stratum of the
symmetric product $S^{|\underline{k}|}(\fM^{\mathrm{s}}(\delta))$
consisting of configurations whose multiplicities are given by
$\underline{k}$.
We do not write the second type
$\fM^{\mathrm{s}}(\underline{\bv}^k)$ explicitly as it is a
point.

This is a general result, but a little imprecise as we do not have a
criterion for $\cM^{\mathrm{s}}(\underline{\bv}')\neq\emptyset$ at
this moment. When the cobalanced condition is satisfied for
$\underline{\bv}$, it is also satisfied for $\underline{\bv}'$, thanks
to the above description of $\underline{\bv}^k$. In fact, the above
coincides with the stratification as a quiver variety
\cite[\S6]{Na-quiver}, \cite[\S3(v)]{Na-alg}. Hence we have the
criterion for $\cM^{\mathrm{s}}(\underline{\bv}')\neq\emptyset$
\cite{CB}.
On the other hand, the balanced condition for $\underline{\bv}$ is not
inherited to $\underline{\bv}'$. Nevertheless we will determine the
stratification for bow varieties with balanced condition in
\S\S\ref{subsec:strat-A},\ref{subsec:strat-affine-A}.

\subsection{Slice to a stratum}\label{subsec:slice}

Let us continue study of the stratification \eqref{eq:13}. As in the
case of quiver varieties (\cite[\S3]{Na-qaff}), we can study the local
(in the complex analytic topology) structure of $\cM$ around a point
in a stratum by using the local normal form of the moment map. The
result used there (\cite[Lem.~3.2.1]{Na-qaff}) is general, and
applicable to the current situation. (See also
\cite[Th.~3]{MR2230091}.)

We use the description of $\cM$ as a symplectic reduction of
$\widetilde\cM_{\mathrm{sym}}$ by $\GV$ as in the proof of
\propref{prop:dimension}.
We take $x\in\cM$ and its representative $m$ in
$\widetilde\cM_{\mathrm{sym}}$ so that it has a closed orbit in
$\widetilde\cM^{\mathrm{ss}}$. Let $\widehat\GV$ be the stabilizer of
$m$. Let $\GV m$ denote the orbit through $m$, and $T_m\GV m$ its tangent
space at $m$. The latter is a subspace of
$T_m\widetilde\cM_{\mathrm{sym}}$ the tangent space of
$\widetilde\cM_{\mathrm{sym}}$ at $m$. We consider the
\emph{symplectic normal} $\widehat\bM$ defined as $(T_m \GV m)^\perp/T_m
\GV m$, where $(T_m \GV m)^\perp$ is the symplectic perpendicular of $T_m
\GV m$ in $T_m \widetilde\cM_{\mathrm{sym}}$.
The stabilizer $\widehat\GV$ naturally acts on $\widehat\bM$.
Let $\widehat\mu\colon\widehat\bM\to\widehat\g^*$ be the moment map
vanishing at the origin. Here $\widehat\g$ is the Lie algebra of
$\widehat\GV$. Then a general result mentioned above says that there is
a local symplectic isomorphism
\begin{equation*}
    (\cM,x) \cong (\widehat\mu^{-1}(0)\dslash\widehat\GV,0).
\end{equation*}

Let us apply this result to our situation. Let $x\in\cM$ and take a
representative $m$. The tangent space $T_m
\widetilde\cM_{\mathrm{sym}}$ is $\Ker\beta_1$ in \eqref{eq:19}. Since
$\alpha$ is the differential of the $\GV$-action, the tangent space $T_m
\GV m$ of the orbit is $\Ima\alpha$. Since $\beta_2$ is the differential
of $\mu_2$, which is the moment map of $\widetilde\cM_{\mathrm{sym}}$
for the $\GV$-action, the symplectic perpendicular $(T_m \GV m)^\perp$ is
$\Ker\beta_2$. Hence
\begin{equation*}
    \widehat\bM \cong \Ker\beta/\Ima\alpha.
\end{equation*}

Let us take the representative $m$ as in \lemref{lem:direct_sum}. We
collect isomorphic $(A^k,B^k,0,0,C^k,D^k)$ and write the second
summand as
\begin{equation*}
    \bigoplus_k (A^k,B^k,0,0,C^k,D^k)^{\oplus n_k}
\end{equation*}
with the multiplicity $n_k\in\ZZ_{>0}$. The stabilizer $\widehat\GV$ of
$m$ is $\prod \GL(n_k)$. In fact, there is no nonzero homomorphism
between different summand thanks to the stability condition. Let us
formally add $k = \infty$ to the index set corresponding to the first
summand in \lemref{lem:direct_sum}. We set $n_\infty = 1$. (Note
however that the stabilizer $\widehat\GV$ does not contain the factor
$\GL(n_\infty)$.)
Let us denote by $V^k_\zeta$ the vector space for a segment $\zeta$
corresponding to the direct summand $(A^k,B^k,0,0,C^k,D^k)$ and also
for $k=\infty$.

We have the corresponding decomposition of $\Ker\beta/\Ima\alpha$ as
\begin{equation*}
    \bigoplus_{k,l} \Hom(\CC^{n_k},\CC^{n_l})
    \otimes E^{kl},
\end{equation*}
where $E^{kl}$ is the middle cohomology of the three term complex
\begin{gather*}
    \bigoplus_\zeta \Hom(V_\zeta^k,V_\zeta^l)
    \xrightarrow{\alpha^{kl}} \mathbb M^{kl} \xrightarrow{\beta^{kl}}
    \mathbb N^{kl} := \bigoplus_x
    \Hom(V_{\zeta_x^-}^k,V_{\zeta_x^+}^l)\oplus \bigoplus_\zeta
    \Hom(V^k_\zeta,V^l_\zeta),
    \\
    \begin{aligned}[t]
    \mathbb M^{kl} := &
            \bigoplus_\xl 
        \bigg(\Hom(V^k_{\zeta^-_\xl}, V^l_{\zeta^-_\xl}) 
        \oplus \Hom(V_{\zeta^+_\xl}^k, V_{\zeta^+_\xl}^l) \oplus
        \Hom(V_{\zeta^-_\xl}^k,V_{\zeta^+_\xl}^l) \oplus
        \Hom(\delta_{l\infty}\WW_\xl,V_{\zeta^+_\xl}^l)
        \\
        & \qquad\qquad 
        \oplus \Hom(V_{\zeta^-_\xl}^k,\delta_{k\infty}\WW_\xl)\bigg)
        \oplus \bigoplus_h\left(
        \Hom(V_{\vout{h}}^k,V_{\vin{h}}^l)\oplus
        \Hom(V_{\vin{h}}^k,V_{\vout{h}}^l)\right).
    \end{aligned}
\end{gather*}
Here $\delta_{k\infty}\WW_\xl$, $\delta_{l\infty}\WW_\xl$ mean that we
set them as $0$ unless $k=\infty$ or $l=\infty$. The definition of
$\alpha^{kl}$, $\beta^{kl}$ are as in \eqref{eq:19}, we replace $A$,
$B$, ... with appropriate $A^k$, $A^l$, $B^k$, $B^l$, ... The action
of the stabilizer $\widehat\GV = \prod\GL(n_k)$ is the standard one on
$\Hom(\CC^{n_k},\CC^{n_l})$ with the trivial one on $E^{kl}$.

The symplectic reduction $\widehat\mu^{-1}(0)\dslash \widehat\GV$ is a
product of a quiver variety $\fM_0$ and an affine space as follows. We
define a quiver by
\begin{itemize}
      \item the set of vertices is $\{ k \} \setminus \{ \infty \}$,
      \item the number of arrows from $k$ to $l$ is $2\delta_{kl} -
    {}^t\!\underline{\bv}^k C\underline{\bv}^l$.
\end{itemize}
Then we put two vector spaces for each vertex $k$:
\begin{itemize}
      \item $\CC^{n_k}$ for the quiver part,
      \item $E^{k\infty}$ for the framing.
\end{itemize}
The affine space is the additional factor $\Hom(\CC^{n_\infty},
\CC^{n_\infty})\otimes E^{\infty\infty} = E^{\infty\infty}$ for
$k=l=\infty$, where $\widehat\GV$ acts trivially. It corresponds to the
tangent space of $\cM^{\mathrm{s}}(\underline{\bv}')$ for the first
factor in \eqref{eq:13}.

The number $2\delta_{kl} - {}^t\!\underline{\bv}^k C\underline{\bv}^l$
is given by computing $\dim E^{kl}$: By the stability as mentioned
above, we have $\dim \Ker\alpha^{kl} = \delta_{kl}$. The first
component of $\beta^{kl}$ is surjective by
\propref{prop:differential}(2). Since the second component of
$\beta^{kl}$ is coming from the differential of the moment map, we
have $\dim\Coker\beta^{kl} = \delta_{kl}$. Therefore $\dim E^{kl} =
\dim \mathbb M^{kl} - \sum_\zeta \dim \Hom(V_\zeta^k,V_\zeta^l) - \dim
\mathbb N^{kl} + 2\delta_{kl}$. We further observe that contributions
from triangular parts cancel for $k$, $l\neq\infty$. Therefore $\dim
E^{kl} = 2\delta_{kl} - {}^t\!\underline{\bv}^k C \underline{\bv}^l$
as in the case of quiver varieties. In particular, the underlying
quiver is a union of Jordan quivers corresponding to
$\underline{\bv}^k = \delta$, and an affine quiver of type $A$
corresponding to other $\underline{\bv}^k$.

For $k = l\neq \infty$, the action of $\widehat\GV$ on the subspace
$\CC\id \otimes E^{kl}\subset \Hom(\CC^{n_k}, \CC^{n_k})\otimes
E^{kk}$ is trivial. It is the tangent space of
$\cM(\underline{\bv}^k)$, which is non trivial only when
$\underline{\bv}^k = \delta$. But this part is contained in the quiver
variety part. When we compute the transversal slice, this factor must
be factored out. See \thmref{thm:strat-affine-type}.

\begin{Proposition}\label{prop:local}
    We have a local isomorphism
    \begin{equation*}
        (\cM,x) \cong (\fM_0\times E^{\infty\infty}, (0,0))
    \end{equation*}
    in complex analytic topology.
\end{Proposition}

\subsection{Semismall morphism}

By \remref{rem:root}, the space $(\sqrt{-1}\RR^3)^{\mathcal I}$ of
parameter $\nu$ have a collection of codimension $3$ subspaces defined
by roots of the affine Lie algebra of type $A_{\ell-1}$. If a
dimension vector is fixed, only finitely many roots contribute as
$\underline{\bv}^k$ cannot have greater entries. We thus have a face
structure on $(\sqrt{-1}\RR^3)^{\mathcal I}$. See
\cite[\S2.3]{Na-branching} for detail.

We choose $\nu^\bullet$ which is contained in a subspace, and $\nu$ so
that 1) it has the same complex part $\nu^{\bullet,\CC} = \nu^\CC$, 2)
it is not contained in any of subspaces, and 3) $\nu^{\bullet,\RR}$ is
contained in the closure of the chamber containing $\nu^{\RR}$. The
most important example is $\nu^\bullet = 0$, $\nu = (\nu^\RR,0)$ with
$\nu^\RR$ not contained in root hyperplanes.

Under the assumption the $\nu^\RR$-stability implies
$\nu^{\bullet,\RR}$-semistability. We have the induced morphism
$\pi\colon \cM_\nu\to \cM_{\nu^\bullet}$. The local description in a
neighborhood of $m\in\cM_{\nu^\bullet}$ can be lifted to $\cM_\nu$:
there are neighborhood $U$ and $V$ of $m$ and $0$ in
$\cM_{\nu^\bullet}$ and $\fM_0$ respectively such that there is a
commutative diagram
\begin{equation*}
    \begin{CD}
        \cM_{\nu} @. \ \supset\  @. \pi^{-1}(U) @>\cong>>
        (\widehat\pi\times\id)^{-1}(V) @. \ \subset\  @. 
        \fM_{(\nu^\RR,0)}\times E^{\infty\infty}
        \\
        @. @. @VVV @VVV @. @.
        \\
        \cM_{\nu^\bullet} @. \supset @. U @>>\cong> V @. \subset
        @. \fM_{0} \times E^{\infty\infty}
    \end{CD}
\end{equation*}
See \cite[\S3]{Na-qaff}, \cite[\S2.7]{Na-branching}. Here $\nu^\RR$ is
regarded as a character of $\widehat\GV$ by restriction. In particular,

\begin{Proposition}
    The image $\pi(\cM_{\nu})$ is a union of strata in \eqref{eq:13}.
    If we replace $\cM_{\nu^\bullet}$ by the image $\pi(\cM_{\nu})$,
    $\pi$ is semismall with respect to the induced
    stratification. Moreover all strata are relevant, i.e., the
    dimension of the fiber is the half of codimenion of a stratum.
\end{Proposition}

See \cite[Rem.~2.24]{Na-branching} for the last statement. It also
follows that the union of smaller strata are the singular points of
$\pi(\cM_\nu)$. In fact, if a point $x$ in a smaller stratum is a
regular point, the symplectic form extends across $x$, as the
codimension of the stratum is greater than or equal to $2$. But as
$\pi$ respect symplectic forms on regular loci, $\pi$ becomes an
isomorphism at $x$. But $\pi^{-1}(x)$ cannot be a point, as the
stratum is relevant. 

In particular, we see that \eqref{eq:13} is the decomposition into
symplectic leaves.

\begin{Remark}
    Since $\pi$ is semismall, $\pi_*(\CC_{\cM_\nu}[\dim \cM_\nu])$ is
    a semisimple perverse sheaf. It is probably true that an
    intersection cohomology complex with \emph{nontrivial} local
    coefficients does not appear in the summand as for the case of
    quiver varieties. It should be possible to prove it by an argument
    in \cite[a paragraph after Rem.~5.13]{Na-branching}. If this is
    true, the multiplicities of IC sheaves in
    $\pi_*(\CC_{\cM_\nu}[\dim \cM_\nu])$ are dimensions of top degree
    homology groups of fibers, and hence are equal to the
    corresponding multiplicities for quiver varieties. They are
    computed in \cite{Na-branching}.
\end{Remark}

\section{Symplectic structure}\label{sec:symplectic-structure}

We study the symplectic structure on bow varieties in this section.

The symplectic structure on the two-way part is
\begin{NB}
  $-\tr(dC\wedge dD)$.
  \begin{NB2}
    This was mistake. See below.
  \end{NB2}
\end{NB}%
$\tr(dC\wedge dD)$.
\begin{NB}
  The symplectic form in \corref{cor:symplectic} was
  $\tr (d\eta \wedge du u^{-1} + \eta du u^{-1}\wedge du u^{-1})$. This
  convention is negative of the convention below in \ref{NB:PB}. The
  convention below gives the Poisson bracket
  $\{ u_{ij}, \eta_{kl}\} = \delta_{jk} u_{il}$. Therefore we have
\end{NB}%
The Poisson bracket is
\begin{NB}
\begin{equation}\label{eq:23}
   \{ C_{ij}, D_{kl}\} = \delta_{il} \delta_{jk}.
\end{equation}
This was wrong. See below.
\end{NB}%
\begin{equation}\label{eq:24}
   \{ C_{ij}, D_{kl}\} = -\delta_{il} \delta_{jk}.
\end{equation}

\begin{NB}
  Note $B = CD$ and $B' = -DC$.
  \begin{equation*}
    \begin{split}
      & \{ B_{ij}, B_{kl}\} = \sum_{m,n} \{ C_{mj}D_{im}, C_{nl}D_{kn}\}
    = \sum_{m,n} C_{mj} D_{kn} \{ D_{im}, C_{nl}\} 
    + D_{im} C_{nl} \{ C_{mj}, D_{kn}\} \\
    & = \sum_{m} \left(-\delta_{il} C_{mj} D_{km} +\delta_{jk} D_{im} C_{ml}\right)
    = - \delta_{il} B_{kj} + \delta_{jk} B_{il}, \\
      & \{ B'_{ij}, B'_{kl}\} = \sum_{m,n} \{ D_{mj}C_{im}, D_{nl}C_{kn}\}
    = \sum_{m,n} D_{mj} C_{kn} \{ C_{im}, D_{nl}\} 
    + C_{im} D_{nl} \{ D_{mj}, C_{kn}\} \\
    & = \sum_{m} \left(\delta_{il} D_{mj} C_{km} - \delta_{jk} C_{im} D_{ml}\right)
    = - \delta_{il} B'_{kj} + \delta_{jk} B'_{il}.
    \end{split}
  \end{equation*}

  \begin{NB2}
    This computation does not have a meaning. We will take products of
    triangle and two-way parts, and perform symplectic
    reductions. The symplectic form on the quotient is well-defined, as
    \[
      \{ \text{components of the moment map}, \text{invariant functions} \} = 0.
    \]
    In fact, hamiltonian vector fields of components of the moment map
    are vector fields generated by Lie algebras, hence the above
    Poisson brackets are obviously $0$.

    In our case, we need to check that $B-CD$, $B'+DC$ are poisson
    commuting with invariant functions. For example,
    \begin{equation*}
      \begin{split}
        0 &= \left\{ B_{ij} - \sum_m C_{mj} D_{im}, \sum_n b_n C_{kn}
        \right\}
        \\
        &= \sum_{n} \left\{ B_{ij}, b_n\right\} C_{kn}
        - \sum_m  C_{mj} \left\{D_{im}, C_{kn}\right\}
        = C_{kj} - \sum_m  C_{mj} \left\{D_{im}, C_{kn}\right\}.
      \end{split}
    \end{equation*}
    Therefore we must have
    \begin{equation*}
      \left\{ D_{im}, C_{kn}\right\} = \delta_{in} \delta_{km},
    \end{equation*}
    contrary to \eqref{eq:23}.

    Let us continue the calculation:
    \begin{equation*}
      \begin{split}
        & \left\{ B_{ij}, \sum_k B_{kk}\right\}
        = \sum_k (-\delta_{ik} B_{kj} + \delta_{jk} B_{ik}) = 0,
        \\
        & \left\{ B_{ij}, \sum_{k,l} B_{kl} B_{lk} \right\}
        = \sum_{k,l} \left(-\delta_{il} B_{kj} + \delta_{jk} B_{il}) B_{lk}
          + B_{kl} (-\delta_{ik} B_{lj} + \delta_{jl} B_{ik})\right) \\
        & \; = - \sum_k B_{kj} B_{ik} + \sum_l B_{il} B_{lj}
        -\sum_l B_{il} B_{lj} + \sum_k B_{kj} B_{ik} = 0, \text{etc}.
      \end{split}
    \end{equation*}
    Therefore this computation does not fix the sign of \eqref{eq:23}.
  \end{NB2}
\end{NB}%

\subsection{Poisson structure on the triangle part}\label{subsec:poiss-triangl}

Let us review the Poisson structure defined on the `triangle' part $\widetilde\cM$ in \subsecref{subsec:triangle} in \cite{fra}.

Let $\mathfrak n(\bv_1,\bv_2) := V_1 \oplus V_2^* \oplus
\Hom(V_2,V_1)$. We equip a Lie algebra structure on $\mathfrak
n(V_1,V_2)$ so that $\Hom(V_2,V_1)$ is central, $[V_1,V_1] = 0 =
[V_2^*,V_2^*]$ and $[v,w^\vee] = v\otimes w^\vee$ for $v\in V_1$,
$w^\vee\in V_2^*$.

We have a natural action of $\gl(\bv_1)\oplus \gl(\bv_2)$ on $\mathfrak n(V_1,V_2)$ induced from the tautological actions of $\gl(\bv_1)$ on $V_1$ and $\gl(\bv_2)$ on $V_2$. It acts by derivations.
\begin{NB}
Let $X_1\in\gl(\bv_1)$, $X_2\in\gl(\bv_2)$. Then $X_1\cdot (v + w^\vee + f) = X_1 v + X_1 f$, $X_2\cdot (v+w^\vee + f) = - w^\vee X_1 - f X_2$. We have
\begin{equation*}
\begin{split}
& X_1 \cdot [v, w^\vee] = X_1 (v\otimes w^\vee) = X_1 v\otimes w^\vee = [X_1\cdot v, w^\vee] = [X_1\cdot v, w^\vee] + [v, X_1\cdot w^\vee],
\\
& X_2 \cdot [v, w^\vee] = X_2 (v\otimes w^\vee) = - v\otimes w^\vee X_2 = [v, X_2\cdot w^\vee] = [X_2\cdot v, w^\vee] + [v, X_2\cdot w^\vee]. 
\end{split}
\end{equation*}
\end{NB}%
We consider the semi-direct product
\begin{equation*}
\mathfrak a := \left(\gl(\bv_1)\oplus \gl(\bv_2)\right) \rtimes \mathfrak n(\bv_1,\bv_2).
\end{equation*}
\begin{NB}
   The Lie algebra structure on the semi-direct product is
\begin{equation*}
   [(X_1,n_1), (X_2, n_2)] = ([X_1, X_2], [n_1,n_2] + X_1\cdot n_2 - X_2 \cdot n_1)
   \qquad
   \text{for $X_i\in\gl(\bv_1)\oplus \gl(\bv_2)$, $n_i\in \mathfrak n(V_1,V_2)$}.
\end{equation*}
\end{NB}%
We give a structure of a Poisson manifold on $\mathfrak a^*$ by
$\{ X, Y\} = [X, Y]$ for $X$, $Y\in\mathfrak a$. 

We have a natural action of $\GL(\bv_1)\times \GL(\bv_2)$ on $\mathfrak
a^*$, which preserves the Poisson structure.
\begin{NB}
    It means that $[gX, gY] = g[X,Y]$ for $X$, $Y\in\mathfrak a$. But
    this is obvious.
\end{NB}%
Recall a $\gl(\bv_1)^*\oplus\gl(\bv_2)^*$-valued map $m$ on the Poisson
manifold $\mathfrak a^*$ with the $\GL(\bv_1)\times\GL(\bv_2)$-action is a
\emph{moment map} if $\{ \langle \xi, m\rangle, f\} = \xi^* f$ for any
function $f$ and $\xi\in\gl(\bv_1)\oplus\gl(\bv_2)$. Here $\xi^*$ is the
vector field generated by $\xi$ and $\xi^* f$ is the differential of
$f$ by $\xi^*$. By our definition of the Poisson bracket, $m$ in our
case is just given by the projection $\mathfrak a^*\to
\gl(\bv_1)^*\oplus\gl(\bv_2)^*$.
\begin{NB}
    It is easier to consider the coadjoint action of $A$ on $\mathfrak
    a^*$. Here $A$ is the Lie group with $\operatorname{Lie}A =
    \mathfrak a$. Then we claim that $\id\colon\mathfrak
    a^*\to\mathfrak a^*$ is a moment map. In fact, $\langle \xi,
    \id\rangle$ is $\xi$ itself regarded as a function on $\mathfrak
    a^*$. Hence $\{ \langle \xi, \id\rangle, X\} = [\xi, X]$ for
    $X\in\mathfrak a$. On the other hand, the differential of $X$ by
    $\xi^*$ is also $[\xi, X]$.
\end{NB}%

We write entries of $\mathfrak a^* = \Hom(V_1,V_2)\oplus \gl(\bv_1)\oplus \gl(\bv_2)\oplus V_2\oplus V_1^*$ by $A$, $B$, $B'$, $a$, $b$. This is almost the same as the notation in \subsecref{subsec:triangle}, but we will set
$B_1 = -B$, $B_2 = B'$ later. 
Let $A_{ij}$ be the $(i,j)$ matrix
unit of $\Hom(V_2,V_1)$. We consider it as a linear function on
$\Hom(V_1,V_2)$ by $A\mapsto \tr(AA_{ij})$, and then extend it to
$\mathfrak a^*$. In other words, $A_{ij}$ is the function on
$\mathfrak a^*$ giving by the $(j,i)$ entry of $A$. We introduce
$B_{ij}$, $B'_{ij}$, $a_i$, $b_j$ in the same way.
\begin{NB}
   So $b_i\in V_1$ has $1$ on the $i$th entry, and $0$ for other entries. It is
   a function on $V_1^*$ taking the $i$th entry of $b$.
\end{NB}%
Then the Poisson brackets are given by (see \cite[(11-15)]{fra})
\begin{equation}
    \label{eq:15}
\begin{gathered}
   \{ A_{ij}, A_{kl} \} = 0,
   \quad   \{ a_i, a_j\} = 0 = \{ b_i, b_j \},
  \\
  \{ B_{ij}, B_{kl} \} = \delta_{jk} B_{il} - \delta_{li} B_{kj}, \quad
  \{ B'_{ij}, B'_{kl} \} = \delta_{jk} B'_{il} - \delta_{li} B'_{kj},
  \\
  \{ B_{ij}, a_k\} = 0 = \{ B'_{ij}, b_k \}, \quad
  \{ B_{ij}, B'_{kl} \} = 0,
  \\
  \{ B_{ij}, b_k \} = \delta_{jk} b_i, \quad
  \{ B'_{ij}, a_k \} = - \delta_{ki} a_j,
  \\
  \{ B_{ij}, A_{kl}\} = \delta_{jk} A_{il}, \quad
  \{ B'_{ij}, A_{kl}\} = - \delta_{li} A_{kj},
  \\
  \{ b_i, a_j\} = A_{ij}, \quad
  \{ A_{ij}, b_k \} = 0 = \{ A_{ij}, a_k\}.
\end{gathered}
\end{equation}

\begin{NB}
    \begin{equation*}
        \{ B_{ij}, \{ b_k, a_l\}\} = \{ \{ B_{ij}, b_k\},a_l \}
        + \{ b_k,\{ B_{ij}, a_l \}\}
        = \delta_{jk} \{ b_i, a_l\} = \delta_{jk} A_{il}.
    \end{equation*}
\end{NB}%

We set
\begin{equation*}
     \mu(A,B,B',a,b) := B' A + A B + ab.
\end{equation*}
\begin{NB}
    Note the difference of sign on $A B_1$ from $\mu$ in
    \subsecref{subsec:triangle}. We should set $B_1 = -B$, $B_2 = B'$.
\end{NB}%

\begin{Proposition}[\protect{\cite[Prop.~3.15]{fra}}]
  The ideal generated by entries of $\mu$ is a Poisson ideal.
\end{Proposition}

\begin{proof}
    Since we have decided to write the $(j,i)$ entry of $A$ by
    $A_{ij}$, the $(j,i)$ entry $\mu_{ij}$ of $\mu$ is
\begin{equation*}
    \mu_{ij}(A,B,B',a,b) = \sum_m (B'_{mj} A_{im} + A_{mj} B_{im}) + a_j b_i.
\end{equation*}
Hence
\begin{equation*}
\begin{split}
  \{ \mu_{ij}, B_{kl} \} &= \left\{ \sum_m (B'_{mj} A_{im} + A_{mj} B_{im})
  + a_j b_i, B_{kl} \right\} \\
  &= \sum_m \left(-B'_{mj}\delta_{il} A_{km} 
  - \delta_{lm} A_{kj} B_{im}
  	+ A_{mj} (\delta_{mk} B_{il} - \delta_{il} B_{km}) \right)
	- a_j \delta_{il} b_k \\
  &= - \delta_{il} \mu_{kj}.
\end{split}
\end{equation*}
Computation of other Poisson brackets are similar. We get
\begin{equation*}
    \begin{gathered}[b]
        \{ \mu_{ij}, B'_{kl} \} = \delta_{jk} \mu_{il}, \quad
        \{ \mu_{ij}, A_{kl} \} = 0, \\
        \{ \mu_{ij}, a_k \} = 0 = \{ \mu_{ij}, b_k\}.
    \end{gathered}
\qedhere
\end{equation*}
\begin{NB}
\begin{equation*}
\begin{split}
  \{ \mu_{ij}, B'_{kl} \} &= \left\{ \sum_m (B'_{mj} A_{im} + A_{mj} B_{im})
  + a_j b_i, B'_{kl} \right\} \\
  &= \sum_m \left(
    (\delta_{jk} B'_{ml} A_{im} - \delta_{lm} B'_{kj}) A_{im}
    + B'_{mj} \delta_{mk} A_{il} 
    + \delta_{jk} A_{ml} B_{im} \right)
  + \delta_{jk} a_l b_i \\
  &= \delta_{jk} \mu_{il},
\\
  \{ \mu_{ij}, A_{kl} \} &= \left\{ \sum_m (B'_{mj} A_{im} + A_{mj} B_{im})
  + a_j b_i, A_{kl} \right\} \\
  &= \sum_m (- \delta_{ml} A_{kj} A_{im}
  + A_{mj} \delta_{mk} A_{il}) = 0,
\\
  \{ \mu_{ij}, a_{k} \} &= \left\{ \sum_m (B'_{mj} A_{im} + A_{mj} B_{im})
  + a_j b_i, a_{k} \right\} \\
  &= - \sum_m \delta_{mk} a_j A_{im} + a_j A_{ik} = 0,
\\
  \{ \mu_{ij}, b_{k} \} &= \left\{ \sum_m (B'_{mj} A_{im} + A_{mj} B_{im})
  + a_j b_i, b_{k} \right\} \\
  &= \sum_m A_{mj} \delta_{mk} b_i - A_{kj} b_i = 0.
\end{split}
\end{equation*}
\end{NB}%
\end{proof}

Therefore a subvariety $\mu = 0$ has the induced Poisson structure. Setting
$B_1 = -B$, $B_2 = B'$, we have the Poisson structure on $\widetilde\cM$.
The $\GL(\bv_1)\times\GL(\bv_2)$-action preserves $\mu=0$, and the moment
map is given by the projection to $\gl(\bv_1)^*\oplus\gl(\bv_2)^*$. Its
restriction coincides with the moment map on $\widetilde\cM$ in
\subsecref{subsec:triangle}.

\begin{Remark}\label{rem:PBcompatible}
  The Poisson brackets \eqref{eq:24} and \eqref{eq:15} are compatible
  in bow varieties. To see this, we consider products of two-way and
  triangle parts, say
  $(\Hom(V_0,V_1)\oplus \Hom(V_1,V_0))\times \widetilde\cM$, and
  check that entries of $B - CD$ are Poisson commuting with functions
  invariant under the $\GL(V_1)$-action. For example,
      \begin{equation*}
      \begin{split}
        & \left\{ B_{ij} - \sum_m C_{mj} D_{im}, \sum_n A_{nl} C_{kn}
        \right\}
        \\
        = \; & \sum_{n} \left\{ B_{ij}, A_{nl} \right\} C_{kn}
        - \sum_m  C_{mj} A_{nl} \left\{D_{im}, C_{kn}\right\}
        = A_{il} C_{kj} - C_{kj} A_{il} = 0.
      \end{split}
    \end{equation*}
\end{Remark}

\subsection{Poisson structure on the slice}

Let $T^*\GL(n) \cong \GL(n)\times \gl(n)$. It has a symplectic form as
a cotangent bundle. We consider the associated Poisson manifold. Let
$u_{ij}$, $\eta_{ij}$ be the $(j,i)$ entries of $u$, $\eta$ for
$(u,\eta)\in\GL(n)\times\gl(n)$ respectively. We consider them as
functions on $\GL(n)\times\gl(n)$. The Poisson brackets are
\begin{gather}\label{eq:17}
    \{ u_{ij}, u_{kl}\} = 0,\quad
    \{ u_{ij}, \eta_{kl}\} = \delta_{jk} u_{il},\quad
    \{ \eta_{ij}, \eta_{kl} \} = \delta_{jk} \eta_{il} - \delta_{li} \eta_{kj}.
\end{gather}
We note
\begin{equation}\label{eq:18}
    \{ \eta_{ij}, u^{kl} \} = \delta_{jk}u^{il},
\end{equation}
where $u^{kl}$ is the $(l,k)$-entry of $u^{-1}$.
\begin{NB}
    We use $0 = \{ \eta_{ij}, u^{kl} u_{lm} \} = \{ \eta_{ij},
    u^{kl}\} u_{lm} + u^{kl} \{ \eta_{ij}, u_{lm} \}$. The second term
    of the right hand side is $-\delta_{mi} u^{kl} u_{lj} = -
    \delta_{mi} \delta_{kj}$. Therefore $\{ \eta_{ij}, u^{kl} \} =
    \delta_{jk} u^{il}$.
\end{NB}%

\begin{NB}\label{NB:PB}
    In the convention in \cite[\S2.3]{MR3300314}, the symplectic form is
    \begin{equation*}
        \Omega = -\tr( d\eta \wedge du u^{-1} + \eta du u^{-1}\wedge du u^{-1}).
    \end{equation*}
    Hence
    \begin{equation*}
        \Omega(\frac{\partial}{\partial\eta_{ij}},\cdot)
        = -\tr(e_{ji}du u^{-1}) = -du_{ki} u^{jk},
    \end{equation*}
    where $e_{ji}$ denotes the matrix unit for the entry $(j,i)$. Hence 
    \begin{equation*}
        \Omega(u_{in} \frac{\partial}{\partial\eta_{jn}},\cdot)
        = - du_{kj} u^{nk} u_{in} = -du_{ij}.
    \end{equation*}
    Therefore
    \begin{equation*}
        \left\{ u_{ij}, \eta_{kl} \right\}
        = \left\langle u_{in} \frac{\partial}{\partial \eta_{jn}},
          d\eta_{kl}\right\rangle
        = \delta_{jk}\delta_{nl} u_{in} = \delta_{jk} u_{il}.
    \end{equation*}

    Next consider
    \begin{equation*}
        \begin{split}
        \Omega(\frac{\partial}{\partial u_{ij}},\cdot)
        & = \tr(d\eta e_{ji} u^{-1} - \eta e_{ji} u^{-1} du u^{-1}
        + \eta du u^{-1} e_{ji} u^{-1})
        \\
        & = u^{ki} d\eta_{jk} - u^{mi} du_{nm} u^{pn} \eta_{jp}
        + u^{mi} \eta_{nm} du_{pn} u^{jp}.
        \end{split}
    \end{equation*}
    We have
    \begin{equation*}
        \begin{split}
        \Omega(u_{qj} \frac{\partial}{\partial u_{qi}},\cdot)
        & = u_{qj} \left( u^{kq} d\eta_{ik} - u^{mq} du_{nm} u^{pn} \eta_{ip}
          + u^{mq} \eta_{nm} du_{pn} u^{ip}\right) \\
        & = d\eta_{ij} - du_{nj} u^{pn} \eta_{ip} +
        \eta_{nj} du_{pn} u^{ip}.
        \end{split}
    \end{equation*}
    Therefore $d\eta_{ij}$ is dual to
    \begin{equation*}
        - u_{qj} \frac{\partial}{\partial u_{qi}} 
        + \eta_{ip} \frac{\partial}{\partial \eta_{jp}}
        - \eta_{nj} \frac{\partial}{\partial \eta_{ni}}.
    \end{equation*}
    Therefore
    \begin{equation*}
        \{ \eta_{ij}, \eta_{kl} \}
        = \left\langle
          - u_{qj} \frac{\partial}{\partial u_{qi}} 
        + \eta_{ip} \frac{\partial}{\partial \eta_{jp}}
        - \eta_{nj} \frac{\partial}{\partial \eta_{ni}},
        d\eta_{kl}
          \right\rangle
          = \delta_{jk}\delta_{pl} \eta_{ip} - \delta_{kn}\delta_{il} \eta_{nj}
          = \delta_{jk} \eta_{il} - \delta_{il} \eta_{kj}.
    \end{equation*}
    We also check
    \begin{equation*}
        \{ \eta_{ij}, u_{kl} \}
        = \left\langle
          - u_{qj} \frac{\partial}{\partial u_{qi}} 
        + \eta_{ip} \frac{\partial}{\partial \eta_{jp}}
        - \eta_{nj} \frac{\partial}{\partial \eta_{ni}},
        d u_{kl}
          \right\rangle
          = - \delta_{qk} \delta_{il} u_{qj} = - \delta_{il} u_{kj}.
    \end{equation*}
\end{NB}%

\begin{NB}
    Here is my convention. We use
    \begin{equation*}
        \GL(V) \times \gl(V)^* \ni (u,\eta) \mapsto
        (u, u^{-1}\eta) \in T^* \GL(V) \cong \GL(V)\times\gl(V)^*,
    \end{equation*}
    where the isomorphism in the right hand side is given by
    considering $\GL(V)$ as an open subset of $\gl(V)$. Let
    $(u^{-1}\eta)_{ij}$ be the $(j,i)$ entry of $u^{-1}\eta$. Then
    \begin{equation*}
        \{ (u^{-1} \eta)_{ij}, u_{kl} \} = \delta_{il} \delta_{jk}.
    \end{equation*}
    Now we use $(u^{-1}\eta)_{ij} = \sum_m u^{mj} \eta_{im}$, where
    $u^{mj}$ denotes $(j,m)$ entry of $u^{-1}$. Therefore the left
    hand side is
    \begin{equation*}
        \sum_m u^{mj} \{ \eta_{im}, u_{kl} \} = \delta_{il} \delta_{jk}.
    \end{equation*}
    Hence $\{ \eta_{im}, u_{kl} \} = \delta_{il} u_{km}$.

    Next consider
    \begin{equation*}
        \begin{split}
        0 &= \{ (u^{-1} \eta)_{ij}, (u^{-1}\eta)_{kl} \} 
        = \sum_{m,n} \{ u^{mj} \eta_{im}, u^{nl} \eta_{kn}\}
\\
          &= \sum_{m,n} \left(u^{mj} \{ \eta_{im}, u^{nl}\} \eta_{kn}
          + \{ u^{mj}, \eta_{kn}\} \eta_{im} u^{nl}
          + u^{mj} u^{nl} \{ \eta_{im}, \eta_{kn}\}
        \right).
        \end{split}
    \end{equation*}
    By \eqref{eq:18}
    \begin{equation*}
        \begin{split}
            0 &= \sum_{m,n} \left( u^{mj} \delta_{mn} u^{il} \eta_{kn}
              - \delta_{nm} u^{kj} \eta_{im} u^{nl} + u^{mj} u^{nl} \{
              \eta_{im}, \eta_{kn} \}
            \right)\\
            &= \sum_m \left(u^{mj} u^{il} \eta_{km} - u^{kj} \eta_{im}
              u^{ml}+ u^{mj} u^{nl} \{ \eta_{im}, \eta_{kn} \} \right).
        \end{split}
    \end{equation*}
    Therefore
    \begin{equation*}
        \{ \eta_{im}, \eta_{kn} \}
        = - \delta_{in} \eta_{km} + \delta_{km} \eta_{in}.
    \end{equation*}
\end{NB}%

We also consider $V\oplus V^*$ as in \subsecref{subsec:triangle}. We write
$I$ for $V$, $J$ for $V^*$, and define functions $I_i$ and $J_j$ on $V$, $V^*$ by taking $i$th and $j$th entries respectively. The Poisson brackets are
\begin{gather*}
    \{ I_i, I_j\} = 0 = \{ J_i, J_j\}, \quad
    \{ I_i, J_j\} = -\delta_{ij}.
\end{gather*}

\begin{NB}
Let us fix the convention.

Let us write $I_i = x_i$, $J_j = y_j$ as usual. The symplectic form is $\omega = \sum dx_i\wedge dy_i$. Then 
\begin{equation*}
  \left\langle dx_i, \frac{\partial}{\partial x_j}\right\rangle
   = \delta_{ij} = -\omega(\frac{\partial}{\partial y_i}, \frac{\partial}{\partial x_j}).
\end{equation*}
Therefore the dual of $dx_i$ is $-\partial/\partial y_i$. Hence  
$\{ x_i, y_j \} = -\langle \partial/\partial y_i,  dy_j\rangle = -\delta_{ij}$.

\begin{NB2}
  This convention is not compatible with NB in \ref{NB:PB}:
  $\Omega = -d\eta\wedge du$ gives $\{ u, \eta\} = 1$ if we drop all
  indices and $u^{-1}$. Therefore the correct symplectic form is
  $-\tr(dI\wedge dJ)$. In the convention in \corref{cor:symplectic},
  we should have
  $\tr(dh\wedge u^{-1}du + h u^{-1}du\wedge u^{-1}du + dI\wedge
  dJ)$. So this is correct.
\end{NB2}%
\end{NB}%

Now we consider Poisson brackets on the slice. One can compute them by using \corref{cor:symplectic}, and find the same formula as above. But it is obvious once we notice that $\GL(n)\times\mathcal S$ is a symplectic reduction of $T^*\GL(n)$ as we have explained in \remref{rem:Slodowy}.

\subsection{Poisson isomorphism for triangle part}

\begin{NB}
     Let us define a morphism $\Phi$ from the space of matrices $u\in
    \GL(n)$, $\eta = \left[ \begin{smallmatrix} h & 0 & g \\ f & 0 &
          e_0 \\ 0 & \id & e \end{smallmatrix} \right] \in
    \mathfrak{gl}(n)$ to $\mathfrak a^*$ by
    \begin{equation*}
        (A,B,B',a,b) =
        \begin{cases}
            \left( \left[ 
                \begin{smallmatrix} \id & 0 & 0 \end{smallmatrix}
              \right]u, -u^{-1}\eta u, h, g, 
              \left[ \begin{smallmatrix} 0 & 0 & 1 \end{smallmatrix}
              \right]u \right), & \bv_1 > \bv_2, \\
            \left( u, - u^{-1}\eta u, \eta-IJ, I, Ju \right), & \bv_1 = \bv_2, \\
            \left( - u^{-1}\left[ \begin{smallmatrix} \id \\ 0 \\
                    0 \end{smallmatrix}\right], h, - u^{-1}\eta u,
              u^{-1}\left[ \begin{smallmatrix} 0 \\ 1 \\
                    0 \end{smallmatrix} \right], -f \right), & \bv_1 <
            \bv_2.
        \end{cases}
    \end{equation*}
\end{NB}%

\begin{Proposition}\label{prop:poisson-triangle}
    The isomorphism between the space of Hurtubise normal forms and $\widetilde\cM$ in \propref{prop:triangle-Nahm} respects the Poisson structures.
\end{Proposition}

\begin{proof}
    Let $\Phi$ be the morphism from the space of Hurtubise normal
    forms to $\widetilde\cM$.

Suppose $\bv_1 = \bv_2$. We have
$A = u$, $B = -u^{-1} \eta u$, $B' = \eta - IJ$, $a = I$, $b = Ju$.
Note $B'_{ij}
\begin{NB}
    = \tr(B' e_{ij}) = \tr((h - IJ) e_{ij})
\end{NB}%
= \eta_{ij} - I_j J_i$ in our convention.

Using \eqref{eq:17}, we compute brackets not involving $B$ as
\begin{gather*}
	\{ \Phi^* a_i, \Phi^* a_j\} = \{ I_i, I_j\} = 0,
	\quad\{ \Phi^* b_i, \Phi^* b_j\} = \sum_{m,n} \{ J_m u_{im}, J_n u_{jn}\} = 0,
\\	
    \{ \Phi^* A_{ij}, \Phi^* A_{kl} \} = \{ u_{ij}, u_{kl} \} = 0,
\\
    \begin{aligned}[t]
    \{ \Phi^* B'_{ij}, \Phi^* B'_{kl} \} &=
    \{ \eta_{ij} - I_j J_i, \eta_{kl} - I_l J_k\}
    = \delta_{jk} \eta_{il} - \delta_{li} \eta_{kj} 
    + \delta_{il} I_j J_k  - \delta_{kj} I_l J_i
\\
    &= \delta_{jk} (\eta_{il} - I_l J_i) - \delta_{li} (\eta_{kj} - I_j J_k)
    = \delta_{jk} \Phi^* B'_{il} - \delta_{li} \Phi^* B'_{kj},
    \end{aligned}
\\
  \{ \Phi^* B'_{ij}, \Phi^* A_{kl} \} = 
    \{ \eta_{ij} - I_j J_i, u_{kl} \} 
    = - \delta_{li} u_{kj}
    = - \delta_{li} \Phi^* A_{kj},
\\    
    \{ \Phi^* b_i, \Phi^* a_j\}
    = \sum_m \{ J_m u_{im}, I_j \} = u_{ij} = \Phi^*A_{ij},
\\
    \{ \Phi^* B'_{ij}, \Phi^*a_k\}
    = \{ \eta_{ij} - I_j J_i, I_k\} = - \delta_{ik} I_{j} = -\delta_{ik} \Phi^* a_j,
\\
    \{ \Phi^* B'_{ij}, \Phi^* b_k \} = \sum_m \{ \eta_{ij} - I_j J_i,   J_m u_{km} \}  
    = \sum_m (- \delta_{mi} u_{kj} J_m + \delta_{jm} J_i u_{km}) = 0,
    \\
\{ \Phi^* A_{ij}, \Phi^* a_k \} = \{ u_{ij}, I_k \} = 0,
\\
\{ \Phi^* A_{ij}, \Phi^* b_k \} = \sum_m \{ u_{ij}, J_m u_{km}\} = 0.
\end{gather*}
These are the same as \eqref{eq:15}. Thus the assertion is checked
for these brackets.

\begin{NB}
    Since $B$ is written by $A$, $B'$, $a$, $b$, as $B = - A^{-1}
    (B'+IJ) A$, we do not need to Poisson brackets for $B$.
\end{NB}%

To compute brackets with $B_{ij}$, we use \eqref{eq:18}. Then
{\allowdisplaybreaks
\begin{gather*}
    \{ \Phi^* B_{ij}, \Phi^* a_k\}
    = -\sum_{m,n} \{ u^{mj} \eta_{nm} u_{in}, I_k \} = 0,
\\    
    \begin{aligned}[t]
    \{ \Phi^*B_{ij}, \Phi^*A_{kl} \} 
    & = - \sum_{m,n} \{ u^{mj} \eta_{nm} u_{in}, u_{kl}  \}
    = - \sum_{m,n} u^{mj} \{ \eta_{nm}, u_{kl} \} u_{in}
\\
    & = \sum_{m,n} u^{mj} \delta_{ln} u_{km} u_{in}
    = \delta_{jk}u_{il} = \delta_{jk} \Phi^*A_{il},
    \end{aligned}
\\
\begin{aligned}[t]
    \{ \Phi^* B_{ij}, \Phi^* b_k \}
    & = - \sum_{m,n,p} \{ u^{mj} \eta_{nm} u_{in}, J_p u_{kp} \}
    = - \sum_{m,n,p} u^{mj} \{ \eta_{nm}, u_{kp}\} u_{in} J_p
\\
    & = \sum_{m,n,p} u^{mj} \delta_{pn} u_{km} u_{in} J_p
    = \delta_{jk} \sum_p u_{ip} J_p = \delta_{jk} \Phi^* b_i,
\end{aligned}
\\
\begin{aligned}[t]
    \{ \Phi^* B_{ij}, \Phi^* B'_{kl} \}
    & = - \sum_{m,n} \{ u^{mj} \eta_{nm} u_{in}, \eta_{kl} - I_lJ_k \}
\\
    & = - \sum_{m,n} \left(\{ u^{mj}, \eta_{kl}\} \eta_{nm} u_{in}
      + u^{mj} \{ \eta_{nm}, \eta_{kl}\} u_{in}
      + u^{mj} \eta_{nm} \{ u_{in}, \eta_{kl}\}
      \right)
\\
    &= \sum_{m,n} \left(\delta_{lm} u^{kj} \eta_{nm} u_{in}
    - u^{mj} (\delta_{mk} \eta_{nl} - \delta_{ln} \eta_{km}) u_{in}
    - u^{mj} \eta_{nm} \delta_{kn} u_{il}\right) = 0.
\end{aligned}
\end{gather*}
For} the remaining computation, let us note
\begin{equation*}
    \{ \Phi^* B_{ij}, \Phi^* B_{kl} \} = \sum_{m,n}
    \{ \Phi^* B_{ij}, u^{ml}\eta_{nm} u_{kn} \}
    = \sum_{m,n} \{ \Phi^* B_{ij}, \Phi^* A^{ml}
    \cdot \Phi^* B'_{nm}\cdot  \Phi^* A_{kn}\}.
\end{equation*}
We substitute the above computation and
\begin{equation*}
    \{ \Phi^* B_{ij}, \Phi^* A^{ml} \} = - \delta_{il} \Phi^* A^{mj}.
\end{equation*}
We get
\begin{equation*}
    \begin{split}
    \{ \Phi^* B_{ij}, \Phi^* B_{kl} \} &= \sum_{m,n}\left(
      -\delta_{il} \Phi^* A^{mj} \cdot \Phi^* B'_{nm}\cdot \Phi^* A_{kn}
      + \delta_{jk} \Phi^* A^{ml} \cdot \Phi^* B'_{nm} \cdot \Phi^* A_{in}
      \right) \\
      &= \delta_{jk} \Phi^* B_{il} - \delta_{il} \Phi^* B_{kj}.
    \end{split}
\end{equation*}
Thus we have checked all relations in \eqref{eq:15}.

For the cases $\bv_1\neq\bv_2$, linear maps $A$, $B$, $B'$, $a$, $b$
are blocks of $u$, $\eta$. Therefore required computation is already
contained above. For example, when $\bv_1 > \bv_2$, we have $a_k =
\eta_{nk}$, where $n = \dim V$.
\begin{NB}
    Remember that our index convention is for the transpose.
\end{NB}%
Hence
\begin{equation*}
    \{ \Phi^* B'_{ij}, \Phi^* a_k \} = \{ \eta_{ij}, \eta_{nk}\}
    = -\delta_{ik} \eta_{nj} = -\delta_{ik} \Phi^* a_j,
\end{equation*}
as $1\le i, j\le m$, hence $j\neq n$. We omit further detail.
\begin{NB}
Suppose $\bv_1 > \bv_2$.
For $A_{ij} = u_{ij}$, we restrict to $1\le j\le m$. For $B_{ij} = -(u^{-1}\eta u)_{ij}$, no restriction. For $B'_{ij} = \eta_{ij}$, $1\le i,j\le m$. For $a_i = \eta_{ni}$, $1\le i \le m$. For $b_i = u_{in}$, no restriction. Hence
\begin{equation*}
\begin{gathered}
  \{ B'_{ij}, A_{kl} \} = \{ \eta_{ij}, u_{kl} \} = -\delta_{li} u_{kj} = -\delta_{li} A_{kj}
  \quad\text{as $1\le j\le m$},
\\
  \{ b_i, a_j\} =  \{ u_{in}, \eta_{nj}\} = u_{ij} = A_{ij}
  \quad\text{as $1\le j\le m$},
\\
  \{ B'_{ij}, b_k\} = \{ \eta_{ij}, u_{kn} \} = -\delta_{ni} u_{kj} = 0
  \quad\text{as $1\le i\le m$},
\\
 \{ A_{ij}, a_k \} = \{ u_{ij}, \eta_{nk} \} = \delta_{jn} u_{ik} = 0
 \quad\text{as $1\le j\le m$},
\\
 \{ A_{ij}, b_k \} = \{ u_{ij}, u_{in} \} = 0.
\end{gathered}
\end{equation*}
We also have
\begin{equation*}
\begin{gathered}
   \{ B_{ij}, a_k\} = -\{ u^{mj}\eta_{nm} u_{in}, \eta_{nk} \} = 0,
   \\
   \{ B_{ij}, A_{kl}\} = -\{ u^{mj}\eta_{nm} u_{in}, u_{kl}\}  = \delta_{jk} u_{il} = \delta_{jk} A_{il}
   \quad\text{as $1\le l\le m$},
   \\
   \{ B_{ij}, b_k\} = -\{ u^{mj}\eta_{nm} u_{in},  u_{kn}\} = \delta_{jk} u_{in} = \delta_{jk} b_i,
   \\
   \{ B_{ij}, B'_{kl}\} = -\{ u^{mj}\eta_{nm} u_{in},  \eta_{kl}\} = 0.
\end{gathered}
\end{equation*}
Other brackets are unnecessary to check.

Next suppose $\bv_1 < \bv_2$. For $A_{ij} = - u^{ij}$, we restrict to
$1\le i\le m$. For $B_{ij} = \eta_{ij}$, $1\le i,j\le m$. For
$B'_{ij} = -(u^{-1}\eta u)_{ij}$, no restriction. For $a_i =
u^{m+1,i}$, no restriction. For $b_i = - \eta_{i,m+1}$, $1\le i\le
m$. Hence
\begin{equation*}
\begin{gathered}
     \{ B_{ij}, b_{k} \} = - \{ \eta_{ij}, \eta_{k,m+1}\}
      = - \delta_{jk} \eta_{i,m+1} = \delta_{jk} b_i,\\
     \{ B_{ij}, A_{kl} \} = -\{ \eta_{ij}, u^{kl}\}
      = - \delta_{jk} u^{il} = \delta_{jk} A_{il},
      \{ B_{ij}, a_k\} = \{ \eta_{ij}, u^{m+1,k} \} = 0 
      \quad\text{as $j\le m$},
\\
      \{ B_{ij}, B'_{kl} \} = - \{\eta_{ij}, u^{ml}\eta_{nm} u_{kn} \} = 0,
\\
      \{ b_i, a_j\} = - \{ \eta_{i,m+1}, u^{m+1,j} \} = - u^{ij} = A_{ij}
      \quad\text{as $1\le i\le m$},
\\
      \{ B'_{ij}, A_{kl} \} = \{ u^{mj}\eta_{nm} u_{in}, u^{kl} \}
      = u^{mj} \{ \eta_{nm}, u^{kl} \} u_{in}
      = \delta_{mk} u^{mj} u^{nl} u_{in} = \delta_{il} u^{kj}
      = - \delta_{il} A_{kj} \quad\text{as $1\le k\le m$},
\\
      \{ B'_{ij}, a_k \} = - \{ u^{mj}\eta_{nm} u_{in}, u^{m+1,k} \}
      = - \delta_{ik} u^{m+1,j} = - \delta_{ik} a_j,
\\
      \{ B'_{ij}, b_k \} = - \{ u^{mj}\eta_{nm} u_{in}, \eta_{k,m+1} \} = 0.\\
\end{gathered}
\end{equation*}

Recall the sign issue on $A$, $a$, $b$ in this case. We have the
equation $B_2 A - A B_1 + ab = 0$, hence the sign for $A$ should be
the sign for $ab$. Looking at the formula, we find this is compatible
with $\{ b_i, a_j\} = A_{ij}$. Hence if this would be preserved, we do
not have a problem.
\end{NB}%
\end{proof}

\subsection{Poisson commutativity}\label{subsec:poiss-comm}

Thanks to \propref{prop:poisson-triangle}, the Poisson structure on
$\cM$ is given by the reduction of the product of one in
\subsecref{subsec:poiss-triangl} (triangle part) and the standard one
on the two way part. Therefore we can apply the computation in
\cite[(8)]{fra}. Here is one example. More examples will be given in
the proof of \propref{prop:Poisson}.

\begin{Proposition}\label{prop:poisson-commute}
    \begin{gather*}
        \{ \tr(B_\zeta^k), \tr(B_{\zeta'}^l) \} = 0\quad
        \text{for segments $\zeta$, $\zeta'$}.
    \end{gather*}
\end{Proposition}

If $\zeta$, $\zeta'$ are not equal or adjacent, the bracket is
obviously $0$. If $\zeta$, $\zeta'$ are adjacent, it is zero by
\eqref{eq:15}. If $\zeta = \zeta'$, it is zero by
\cite[Prop.~3.11]{fra}.

\begin{NB}
    Let $e_{ij}^{(r)} := \sum_{i_1,\dots,i_{r-1}} e_{ii_1} e_{i_1
      i_2}\dots e_{i_{r-1}j}$. This is $(i,j)$-entry of $({}^t\!
    A)^r$. We have
    \begin{equation*}
        \begin{split}
            & \{ e_{kl}, e_{ij}^{(r)}\}
        = \sum \{ e_{kl}, e_{ii_1} e_{i_1 i_2}\dots e_{i_{r-1}j}\} \\
        = \; & \delta_{il}e_{ki_1} \dots e_{i_{r-1}j} 
        - \delta_{i_1k} e_{il} e_{i_1 i_2} \dots e_{i_{r-1}j}
        + \delta_{i_1l} e_{ii_1} e_{k i_2}\dots e_{i_{r-1}j} - \cdots
        = \delta_{il} e^{(r)}_{kj} - \delta_{jk} e_{il}^{(r)}.
        \end{split}
    \end{equation*}
    Then
    \begin{equation*}
        \{ e_{kl}, \sum_i e_{ii}^{(r)}\}
        = \sum_i \delta_{il} e_{ki}^{(r)} - \delta_{ik} e_{il}^{(r)} = 0.
    \end{equation*}
    Now the assertion is clear.
\end{NB}%


\section{Bow varieties as Coulomb branches}\label{sec:Coulomb}
\subsection{Bow varieties associated with quiver gauge theories}\label{subsec:Coulomb_gauge}
Let us consider a framed quiver gauge theory of an affine type $A_{n-1}$ with dimension vectors $\underline{\bv} = (\bv_0, \cdots, \bv_{n-1})$ and $\underline{\bw} = (\bw_0, \cdots, \bw_{n-1})$.
The goal of this section is to show that the Coulomb branch $\cM_C$ is the bow variety associated with
\begin{align}\label{eq:8}
\begin{xy}
(5,3)*{\bv_{n-2}},
(10,0)*{\boldsymbol\times},
(15,3)*{\bv_{n-1}},
(20,0)*{\boldsymbol\medcirc},
(25,3)*{\cdots},
(30,0)*{\boldsymbol\medcirc},
(35,3)*{\bv_{n-1}},
(40,0)*{\boldsymbol\times},
(45,3)*{\bv_0},
(50,0)*{\boldsymbol\medcirc},
(55,3)*{\cdots},
(60,0)*{\boldsymbol\medcirc},
(65,3)*{\bv_0},
(70,0)*{\boldsymbol\times},
(75,3)*{\bv_1},
(25,-5)*{\underbrace{\hspace{15mm}}_{\text{$\bw_{n-1}$}}},
(55,-5)*{\underbrace{\hspace{15mm}}_{\text{$\bw_0$}}},
\ar @{.} (0,0);(5,0)
\ar @{.} (75,0);(80,0)
\ar @{-} (5,0);(75,0)
\end{xy}.
\end{align}
These bow varieties always have balanced dimension vectors.
Conversely, balanced dimension vectors are determined by
$\underline{\bv}, \underline{\bw}$.  We denote the corresponding bow
variety by $\cM(\underline{\bv}, \underline{\bw})$. Unless explicitly
written as $\cM_\nu$, the parameter $\nu$ is set $0$.
Let $\ell$ denote the level of $\underline{\bw}$, i.e., $\ell = \sum_i
\bw_i$. We allow $\ell = 0$, i.e., $\underline{\bw} = 0$ (see
\subsecref{subsec:chainsaw}). But we do not consider the case $n=0$
(no $\boldsymbol\times$).

\begin{NB}
The above bow variety is described as follows:
\begin{align*}
\mathbb{M}_{\mathrm{bow}}(\underline{\bv}, \underline{\bw})&:= \mu^{-1}(0) \subset \prod_{i=0}^{n-1} \mathbb{M}_{\mathrm{q}}(\bv_i; \bw_i) \times \mathfrak{M}_{\mathrm{t}}(\bv_i, \bv_{i+1}),
&\cM(\underline{\bv}, \underline{\bw})&:= \mathbb{M}_{\mathrm{bow}}(\underline{\bv}, \underline{\bw}) \dslash \prod_{i=0}^{n-1} \GL(\bv_i)^{\bw_i+1},\\
\mathring{\mathbb{M}}_{\mathrm{bow}}(\underline{\bv}, \underline{\bw})&:= \mu^{-1}(0) \subset \prod_{i=0}^{n-1} \mathbb{M}_{\mathrm{q}}(\bv_i; \bw_i) \times \mathring{\mathfrak{M}}_{\mathrm{t}}(\bv_i, \bv_{i+1}),
&\mathring{\mathfrak{M}}_{\mathrm{bow}}(\underline{\bv}, \underline{\bw})&:= \mathring{\mathbb{M}}_{\mathrm{bow}}(\underline{\bv}, \underline{\bw}) \dslash \prod_{i=0}^{n-1} \GL(\bv_i)^{\bw_i+1}.
\end{align*}
Here we put $\bv_0=\bv_n$ and $\bw_0=\bw_n$, and $\mu=0$ means equations $-B_{i} + D_{1,i}C_{1,i}=0$ and $-C_{\bw_i,i}D_{\bw_i,i}+B'_{i}=0$ for $0 \leq i \leq n-1$ in addition to that of triangles and quivers:
\end{NB}%

We put subscripts to endomorphisms in the quiver description of a bow
variety as follows:
\begin{align*}
&\xymatrix@C=1.2em{ & V_{i-1}^{\bw_{i-1}} \ar@(ur,ul)_{B_{i-1}} \ar[rr]^{A_{i-1}} \ar[dr]_{b_{i-1}} \ar@{.}[l] && V_{i}^0  \ar@(ur,ul)_{B'_{i}} \ar@<-.5ex>[rr]_{C_{1,i}} && V_{i}^1 \ar@<-.5ex>[ll]_{D_{1,i}} \ar@<-.5ex>[rr]_{C_{2,i}} && \ar@<-.5ex>[ll]_{D_{2,i}} \ar@{.}[r] & \ar@<-.5ex>[rr]_{C_{\bw_i-1,i}} && V_{i}^{\bw_i-1} \ar@<-.5ex>[ll]_{D_{\bw_i-1,i}} \ar@<-.5ex>[rr]_{C_{\bw_i,i}} && V_i^{\bw_i} \ar@<-.5ex>[ll]_{D_{\bw_i,i}} \ar@(ur,ul)_{B_{i}} \ar[rr]^{A_{i}} \ar[dr]_{b_{i}} && V_{i+1}^0  \ar@(ur,ul)_{B'_{i+1}} \ar@{.}[r] & \\
 & & \CC \ar[ur]_{a_{i}} & && && & && && & \CC \ar[ur]_{a_{i+1}} & & }.
\end{align*}

\subsection{Deformation and resolution}\label{subsec:deform-resol}
\subsubsection{\texorpdfstring{$\underline{\bw}=0$}{w=0} case}\label{subsubsec:deform-resol}

When $\underline{\bw} = 0$, we do not have parameters $\nu$ as there
is no two way part. But we can put parameters in a different way,
following \cite[\S3.33]{fra}.

Take $\nu^\CC = (\nu_0^\CC,\nu_1^\CC,\dots,\nu_{n-1}^\CC)\in\CC^n$.
\begin{align*}
\widetilde{\cM}_{\nu^\CC}(\underline{\bv}, 0) = \{ \mu^{-1}(\nu^\CC) \text{ satisfying (S1), (S2)}\}, \ \ \cM_{\nu^\CC}(\underline{\bv}, 0)&= \widetilde{\cM}_{\nu^\CC}(\underline{\bv}, 0) \dslash \prod_{i=0}^{n-1} \GL(\bv_i).
\end{align*}
Here $\mu^{-1}(\nu^\CC)$ means that $B'_i$ is determined by $B_i - B'_i  = \nu_i^\CC\id_{V_i}$.
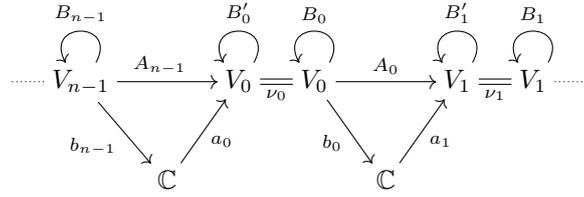
\begin{align*}
&\xymatrix@C=1.2em{ & V_{n-1} \ar@(ur,ul)_{B_{n-1}} \ar[rr]^{A_{n-1}} \ar[dr]_{b_{n-1}} \ar@{.}[l] && V_{0}  \ar@(ur,ul)_{B'_{0}} \ar@{=}[r]_{\nu_0} & V_{0}  \ar@(ur,ul)_{B_{0}} \ar[rr]^{A_{0}} \ar[dr]_{b_{0}} && V_{1}  \ar@(ur,ul)_{B'_{1}} \ar@{=}[r]_{\nu_{1}} & V_{1}  \ar@(ur,ul)_{B_{1}} \ar@{.}[r] & \\
 & & \CC \ar[ur]_{a_{0}} & &  & \CC \ar[ur]_{a_{1}} & & &}
\end{align*}
By the argument in \subsecref{subsubsec:shift}, we get
\begin{align*}
B_{1}'A_{0}-A_{0}B_{0}+a_{1}b_{0} = 0 &\Leftrightarrow (B_{1} - \nu_1^\CC \id_{V_1})A_{0}-A_{0}B_{0}+a_{1}b_{0} = 0,\\
B'_{2}A_{1}-A_{1}B_{1}+a_{2}b_{1} = 0 &\Leftrightarrow (B_{2} - \nu_2^\CC \id_{V_2})A_{1}-A_{1}B_{1}+a_{2}b_{1} = 0\\
&\Leftrightarrow (B_{2} - (\nu_1^\CC +\nu_2^\CC)\id_{V_2})A_{1}-A_{1}(B_{1} - \nu_1^\CC \id_{V_1})+a_{2}b_{1} = 0,
\end{align*}
and so on. Thus only the sum $\sum_k \nu^\CC_k$ matters. It means that
this deformation forms at most a $1$-dimensional family.
\begin{NB}
    In other words,
    \begin{equation*}
        \begin{split}
            & B_1 A_0 - A_0 B_0 - \nu_1^\CC A_0 + a_1 b_0 = 0,\\
            & B_2 A_1 - A_1 B_1 - \nu_2^\CC A_1 + a_2 b_1 = 0,\\
            & \cdots
        \end{split}
    \end{equation*}
\end{NB}%
Therefore we normalize as $\nu_0^\CC = \nu_1^\CC = \cdots = \nu_{n-1}^\CC =
- \nu_*^\CC$ for a $\nu_*^\CC\in\CC$ hereafter.

When $n=1$, the defining equation is
\begin{NB}
  as $B' = B + \nu_*^\CC$, 
\end{NB}%
\begin{equation*}
    [B, A] + \nu_*^\CC A + ab = 0.
\end{equation*}
This is the trigonometric Calogero-Moser space \cite[\S7.3]{MR2077482}.

When some $\bv$ is $0$, this deformation becomes trivial because of
$\id_0=0$.  We expect that it is nontrivial otherwise.
\begin{Remark}
  This deformation is expected to be the `dual' to the
  $\CC^{\times}$-action for $\fM(\underline{\bv},0)$, $B \mapsto tB$
  for $B\in\bN$, $B\mapsto t^{-1}B$ for $B\in\bN^*$ and $t\in
  \CC^*$. See the construction in \subsecref{subsubsec:proof2}. When
  some $\bv$ is $0$, this action also becomes trivial.
\end{Remark}
Like this, $\cM(\underline{\bv}, 0)$ has a $1$-parameter family given
by the real parameter of the hyper-K\"ahler quotient in the original
definition of a bow variety \subsecref{subsec:original}.  This family
corresponds to a resolution in the quiver description of a bow
variety. Concretely depending on the sign of $\sum \nu_k$, we impose
either of (C-S1) or (C-S2) below. This can be checked as in
\subsecref{subsubsec:numerical-criterion}. We leave the details to the
reader as an exercise.
\begin{NB}
    This is because we allow a module with $\dim S_i = 1$ for all $i$,
    $a=b=0$ as either sub or quotient or not.
\end{NB}%

\subsubsection{\texorpdfstring{$\underline{\bw}\neq 0$}{w≠0} case}
In this case, $\cM(\underline{\bv}, \underline{\bw})$ has a $|\underline{\bw}|$-dimensional family of the deformation, which is defined in \subsecref{subsec:quiver}.

For a later purpose, we shift parameters slightly. Let us introduce
complex parameters as $\nu^\CC_{1,i}$, \dots, $\nu^\CC_{\bw_i,i}$
($i=0,1,\dots,n-1$), $\nu^\CC_*$ and consider deformation of the
defining equations given by
\begin{equation}
    \label{eq:21}
\begin{gathered}
B_i'A_{i-1}-A_{i-1}B_{i-1}+a_ib_{i-1} = 0, \quad
-D_{1,i}C_{1,i}-B'_{i} = - \nu^\CC_* -\nu^\CC_{1,i}, \\ 
-D_{k+1,i}C_{k+1,i} + C_{k,i}D_{k,i} = \nu^\CC_{k,i} - \nu^\CC_{k+1,i} 
\quad (k=1,\dots,\bw_i-1),
\quad
C_{\bw_i,i}D_{\bw_i,i} + B_{i} = \nu^\CC_{\bw_i,i}.
\end{gathered}
\end{equation}

\begin{NB}
The followings are not quite correct:

We index complex parameters as $\nu^\CC_{1,i}$, \dots,
$\nu^\CC_{\bw_i,i}$ ($i=0,1,\dots,n-1$). For a later purpose, we shift
parameters slightly, so that the defining equations are
\begin{equation}
\begin{gathered}
B_i'A_{i-1}-A_{i-1}B_{i-1}+a_ib_{i-1} = 0, \quad
-D_{1,i}C_{1,i}-B'_{i} = -\nu^\CC_{1,i}, \\ 
-D_{k+1,i}C_{k+1,i} + C_{k,i}D_{k,i} = \nu^\CC_{k,i} - \nu^\CC_{k+1,i} 
\quad (k=1,\dots,\bw_i-1),
\quad
C_{\bw_i,i}D_{\bw_i,i} + B_{i} = \nu^\CC_{\bw_i,i}.
\end{gathered}
\end{equation}

The deformation considered in \subsecref{subsubsec:deform-resol} can
be absorbed in the $|\underline{\bw}|$-dimensional deformation, hence
does not yield a new direction.

When some $\bv_0$ is $0$ (and $\bw_0 = 0$), the overall shift
$\nu^\CC_{k,i} \mapsto \nu^\CC_{k,i} + s$ for a fixed $s$ can be
absorbed by shifts of $B_i$. Therefore one direction among
the $|\underline{\bw}|$-dimension is trivial.

\begin{NB2}
Consider $\nu_{k,i}^\CC = \nu_i^\CC$ for $i=0,\dots,n-1$. Then
\begin{equation*}
\begin{gathered}
(B_i' - \nu_i^\CC)A_{i-1}-A_{i-1}(B_{i-1}-\nu_{i-1}^\CC)
+(\nu_i^\CC - \nu_{i-1}^\CC)A_{i-1}
+a_ib_{i-1} = 0, \quad
-D_{1,i}C_{1,i}-(B'_{i} - \nu_i^\CC) = 0, \\ 
-D_{k+1,i}C_{k+1,i} + C_{k,i}D_{k,i} = 0 
\quad (k=1,\dots,\bw_i-1),
\quad
C_{\bw_i,i}D_{\bw_i,i} + (B_{i} - \nu^\CC_{i})=0.
\end{gathered}
\end{equation*}
Thus the deformation in \subsecref{subsubsec:deform-resol} cannot be
absorbed. Also the overall shift
$\nu^\CC_{k,i} \mapsto \nu^\CC_{k,i}+s$ yields a trivial deformation.
\end{NB2}
\end{NB}%

If we make shift $\nu^\CC_{k,i} \mapsto \nu^\CC_{k,i} + s$ for a fixed
$s$, keeping $\nu^\CC_*$, it yields a trivial deformation as it is
same as shifts of $B_i$. Therefore we have
$|\underline{\bw}|$-dimensional family.

\subsection{Chainsaw quiver variety}\label{subsec:chainsaw}
We consider the $\underline{\bw}=0$ case.
The corresponding quiver is
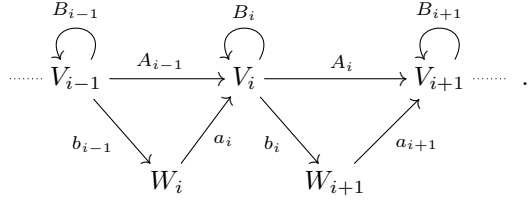
\begin{align*}
&\xymatrix@C=1.2em{ & V_{i-1} \ar@(ur,ul)_{B_{i-1}} \ar[rr]^{A_{i-1}} \ar[dr]_{b_{i-1}} \ar@{.}[l] && V_{i}  \ar@(ur,ul)_{B_{i}} \ar[rr]^{A_{i}} \ar[dr]_{b_{i}} && V_{i+1}  \ar@(ur,ul)_{B_{i+1}} \ar@{.}[r] & \\
 & & W_i \ar[ur]_{a_{i}} &  & W_{i+1} \ar[ur]_{a_{i+1}} & & }.
\end{align*}
This quiver is called a {\em chainsaw quiver}, and studied in \cite{fra}.
$\cM(\underline{\bv}, 0)$ does not coincide with the variety $\mathfrak{Z}_{\underline{d}}$ in \cite{fra} in general, as we impose the conditions (S1), (S2).
Here the variety $\mathfrak{Z}_{\underline{d}}$ is the categorical quotient of $\mu^{-1}(0)$ by $\GV = \prod \GL(\bv_i)$.
Let us denote it by $\overline{\cM}(\underline{\bv}, 0)$.
Our $\cM(\underline{\bv}, 0)$ is an open subvariety of $\overline{\cM}(\underline{\bv}, 0)$, denoted by $\wZ^{\alpha}$ in 
\cite[\S 3(viii)]{blowup}.

Let us assume $\nu_*^\CC = 0$ during the rest of this subsection.
\begin{NB}
    Added on Oct.\ 3.
\end{NB}%

\begin{Proposition}[{\cite[Cor.~3.9]{Takayama}}]
$\cM^{\mathrm{reg}}(\underline{\bv}, 0)$ is the quotient of the open subvariety of $\mu^{-1}(0)$ consisting of points satisfying
\begin{align*}
&\text{\em (C-S1):} \ B_i(S_i) \subset S_i, S_i\subset \Ker b_i, A_i(S_i) \subset S_{i+1} \Rightarrow S_i=0, \\
&\text{\em (C-S2):} \ B_i(T_i) \subset T_i, T_i\supset \Ima a_i, A_i(T_i) \subset T_{i+1} \Rightarrow T_i = V_i,
\end{align*}
by the action of $\GV$.
\end{Proposition}
The conditions (C-S1) and (C-S2) coincides with the $0$-stability condition, in the meaning of King's stability \cite{King}.

A chainsaw quiver variety is related to other important objects: 
\begin{Proposition}
{\em (1) (c.f. \cite[Th.~3.5]{Takayama})} $\cM(\underline{\bv}, 0)$ is isomorphic to a moduli spaces of solutions of Nahm's equations over $S^1$.\\
{\em (2) \cite[Th.~4.2]{Takayama}} $\cM^{\mathrm{reg}}(\underline{\bv}, 0)$ is isomorphic to a framed moduli space of locally free parabolic sheaves over $\proj^1 \times \proj^1$.\\
{\em (3)} $\cM(\underline{\bv}, 0) \cong \bigsqcup_{k} \cM^{\mathrm{reg}}(\underline{\bv}-{k}\cdot\underline{1},0)\times S^{k}(\CC\times \CC^{\times})$.\label{prop:chainsaw}
\end{Proposition}
\begin{proof}
We prove the third assertion. 
Notice that each triangle of $\cM(\underline{\bv}, 0)$ satisfies the conditions (S1), (S2).
Suppose that there exist subspaces $\{S_i\}$ and $\{T_i\}$ as in (C-S1) and (C-S2) respectively.
Proposition \ref{prop:character_cond1} means that $\dim S_i \leq \dim S_{i+1}$ and $\codim T_i \geq \codim T_{i+1}$.
Thus we have $\dim S_i = \dim S_{i+1}$ and $\codim T_i = \codim T_{i+1}$ for any $i$.

\begin{NB}
Take all sets of subspaces $S^k$ and $T^k$ as above, and define
\begin{align*}
S:= \sum_k S^k, \ \ T:= \bigcap_k T^k.
\end{align*}
We introduce some metric on $V_i$s and take orthogonal decomposition
\begin{align*}
V = S^{\bot} \cap T \oplus S^{\bot} \cap T^{\bot} \oplus S\cap T \oplus S\cap T^{\bot}.
\end{align*}
We check that the restriction on $S^{\bot} \cap T$ determines a point in $\cM^{\mathrm{reg}}(\underline{\bv}-{k}\cdot\underline{1},0)$.

By the assumption of $S$ and $T$, each $A, B, a$ and $b$ are described as follows with respect to the above decomposition:
\begin{align*}
A, B = \left[ \begin{smallmatrix} * & * & 0 & 0 \\ 0 & * & 0 & 0 \\ * & * & * & * \\ 0 & * & 0 & * \end{smallmatrix} \right], \ \ a= \left[ \begin{smallmatrix} * \\ 0 \\ * \\ 0 \end{smallmatrix} \right], \ \ b = \left[ \begin{smallmatrix} * & * & 0 & 0 \end{smallmatrix} \right].
\end{align*}
Let $g_t = \diag(\id, t^{-1}\id, t\id, \id) \ (t\in \CC^{\times})$.
By the action of $g_t$ for every $V_i$, we get
\begin{align*}
g_tAg_t^{-1}, g_tBg_t^{-1} = \left[ \begin{smallmatrix} * & t* & 0 & 0 \\ 0 & * & 0 & 0 \\ t* & t^2* & * & t* \\ 0 & t* & 0 & * \end{smallmatrix} \right], \ \ g_ta= \left[ \begin{smallmatrix} * \\ 0 \\ t* \\ 0 \end{smallmatrix} \right], \ \ bg_t^{-1} = \left[ \begin{smallmatrix} * & t* & 0 & 0 \end{smallmatrix} \right].
\end{align*}
Thus the above decomposition is preserved by $A$ and $B$ in the closed orbit. 
Furthermore the restriction on $S^{\bot} \cap T$ induces a map to $\cM^{\mathrm{reg}}(\underline{\bv}-{k}\cdot\underline{1},0)$, and the restriction on $S^{\bot} \cap T^{\bot} \oplus S\cap T \oplus S\cap T^{\bot}$ does a map to $S^{k}(\CC\times \CC^{\times})$.
This is because $A_i|_{S^{\bot} \cap T^{\bot} \oplus S\cap T \oplus S\cap T^{\bot}}$ is isomorphism, and the equation means
\begin{equation*}
[B_0, A_{n-1}\cdots A_0]|_{S^{\bot} \cap T^{\bot} \oplus S\cap T \oplus S\cap T^{\bot}} = 0.
\end{equation*}
\end{NB}%

As in the proof of \lemref{lem:direct_sum}, a point in
$\cM(\underline{\bv},0)$ decomposes into a direct sum of an object
satisfying (C-S1), (C-S2) and those with $a=b=0$ and $A$ is an
isomorphism. A summand of the second type must have $\dim S_i = 1$ for
any $i$ if it is not further decomposable. We can normalize $A$'s to
$\id$ except for one. Then all $B$ are equal. Therefore the data
corresponds to a point in $\CC\times\CC^\times$.
\end{proof}

\subsection{Collapsing morphism to a chainsaw quiver variety}\label{subsec:collapsing-morph}

We construct a \emph{collapsing morphism} $s \colon
\mathbb{M}(\underline{\bv}, \underline{\bw}) \rightarrow
\mathbb{M}(\underline{\bv}, 0)$ as follows:
\begin{align}\label{eq:20}
A_i^{\mathrm{new}} := A_i C_{\bw_i\cdots 1, i}, \ \ B_i^{\mathrm{new}} := B'_i = \nu_{1,i}^\CC + \nu_*^\CC - D_{1,i}C_{1,i}, \ \ a_i^{\mathrm{new}}:=a_i, \ \ b_i^{\mathrm{new}}:= b_iC_{\bw_i \cdots 1, i},
\end{align}
where $C_{\bw_i\cdots 1, i}$ means the composite $C_{\bw_i, i} \cdots C_{1, i}$ as before.
The new data $(A_i^{\mathrm{new}}, B_i^{\mathrm{new}}, a_i^{\mathrm{new}}, b_i^{\mathrm{new}})$ satisfy the equation for $\Hom (V_i, V_{i+1})$ as
\begin{align*}
& B_{i+1}^{\mathrm{new}}A_i^{\mathrm{new}} - A_i^{\mathrm{new}}B_i^{\mathrm{new}} + \nu^\CC_* A_i^{\mathrm{new}}
  + a_{i+1}^{\mathrm{new}}b_i^{\mathrm{new}}\\
&= B_{i+1}' A_i C_{\bw_i \cdots 1, i} + A_i C_{\bw_i \cdots 1, i}(D_{1,i}C_{1,i}  - \nu^\CC_{1,i})
+ a_{i+1}b_iC_{\bw_i \cdots 1, i} \\
& = (B_{i+1}' A_i - A_i B_i + a_{i+1}b_i)C_{\bw_i \cdots 1, i} = 0.
\end{align*}
Obviously, this morphism descends to $s \colon \cM(\underline{\bv}, \underline{\bw}) \rightarrow \overline{\cM}(\underline{\bv}, 0)$, but $s(\cM(\underline{\bv}, \underline{\bw}))$ is not contained in $\cM(\underline{\bv}, 0)$.
This is because (S1) for $(A_i^{\mathrm{new}}, B_i^{\mathrm{new}}, a_i^{\mathrm{new}}, b_i^{\mathrm{new}})$ leads to $\Ker C_{1,i}=0$, which is not true in general.

\begin{Proposition}
    The collapsing morphism $s$ induces an isomorphism
    $s^{-1}(\cM(\underline{\bv},0))\xrightarrow{\cong}
    \cM(\underline{\bv},0)$. Hence we have an open embedding
    $\cM(\underline{\bv},0)\hookrightarrow
    \cM(\underline{\bv},\underline{\bw})$.
\end{Proposition}

\begin{proof}
    Since $B'_i = \nu_{1,i}^\CC + \nu^\CC_* - D_{1,i}C_{1,i}$, we have $C_{1,i} (B'_i - \nu^\CC_*)
    \begin{NB}
        = (\nu_{1,i}^\CC - C_{1,i} D_{1,i}) C_{1,i} 
    \end{NB}%
    = (\nu_{2,i}^\CC - D_{2,i} C_{2,i}) C_{1,i}$, and 
    $C_{2,i} C_{1,i} (B'_i - \nu^\CC_*)
    \begin{NB}
        = C_{2,i} (\nu_{2,i}^\CC - D_{2,i} C_{2,i}) C_{1,i}
    \end{NB}%
    = (\nu_{3,i}^\CC - D_{3,i} C_{3,i})C_{2,i}
    C_{1,i}$, and so on. Finally we get $C_{\bw_i \cdots 1, i} (B'_i - \nu^\CC_*)
    \begin{NB}
        = C_{\bw_i,i}
        (\nu_{\bw_i,i}^\CC - D_{\bw_i,i}C_{\bw_i,i})
        C_{(\bw_i-1)\cdot 1,i}
    \end{NB}%
    = B_i C_{\bw_i\cdot 1,i}$. Therefore $\Ker C_{\bw_i\cdot 1,i}$ is
    contained in $\Ker b_i^\tn$ and invariant under
    $B_i^\tn$. Therefore (S1) for $(A^\tn,B^\tn,a^\tn,b^\tn)$ implies
    the injectivity of $C_{\bw_i\cdot 1,i}$. Since the balanced
    condition is satisfied, all $C_{1,i}$, $C_{2,i}$, \dots are
    isomorphisms. By the action of
    $\GL(V^1_i)\times\cdots\times\GL(V^{\bw_i}_i)$, we normalize
    $C_{1,i}$, $C_{2,i}$, \dots as the identity. Then $D_{1,i}$,
    $D_{2,i}$, \dots are determined by the defining equations $B_i =
    \nu^\CC_{\bw_i,i} - D_{1,i} C_{1,i}$, $D_{2,i} C_{2,i} = C_{1,i}
    D_{1,i} - \nu^\CC_{1,i} + \nu^\CC_{2,i}$, and so on. Hence the
    original data is recovered from the new one.
\end{proof}

\begin{Proposition}\label{prop:collaps-poisson}
    The collapsing morphism $s$ is Poisson.
\end{Proposition}

\begin{proof}
    Recall $\cM(\underline{\bv},\underline{\bw})$ is a reduction of
    $\widetilde\cM_{\mathrm{sym}}$ by $\GV = \prod \GL(V_\zeta)$. On the
    other hand $\overline{\cM}(\underline{\bv},0)$ is also a reduction
    of $\widetilde\cM'$, the product of $\mu^{-1}(0)$ in
    \subsecref{subsec:poiss-triangl} for each triangle. (It is the
    reduction as a Poisson variety.)

    The morphism $s$ is induced from the corresponding morphism
    $\widetilde s\colon \widetilde\cM_{\mathrm{sym}}\to
    \widetilde\cM'$. Therefore it is enough to check that $\widetilde
    s$ is Poisson. But this is clear from the definition of
    $\widetilde s$ given by the formula \eqref{eq:20}.
\end{proof}

\subsection{Coordinate and Local models}\label{subsec:coord-local-models}
We order eigenvalues of $B_i$, and denote them by $(w_{i, 1}, \cdots, w_{i, \bv_i})$.
Define
\begin{gather*}
y_{i, k}:= b_i \prod_{\substack{1 \leq l \leq \bv_i \\ l\neq k}}(B_i - w_{i,l}\id)
C_{\bw_i \cdots 1,i} a_i \quad (n > 1),
\\
y_k := \tr[A_0 \prod_{\substack{1 \leq l \leq \bv_0 \\ l\neq k}}
(B_0 - w_{l}\id)C_{\bw_0 \cdots 1,0}]
 \quad (n = 1).
\end{gather*}
\begin{NB}
The original was wrong:
\begin{gather*}
y_{i, k}:= b_iC_{\bw_i \cdots 1,i} \prod_{\substack{1 \leq l \leq \bv_i \\ l\neq k}}(B_i - w_{i,l}\id)a_i \quad (n > 1),
\\
y_k := \tr[A_0 C_{\bw_0 \cdots 1,0}\prod_{\substack{1 \leq l \leq \bv_0 \\ l\neq k}}(B_0 - w_{l}\id)] \quad (n = 1).
\end{gather*}
\end{NB}%
This definition coincides with that of \cite[\S 3.4]{fra} when $\bw_i=0$.
\begin{NB}
    In fact,
    \begin{equation*}
        \begin{split}
        & b_i \prod_{\substack{1 \leq l \leq \bv_i \\ l\neq k}}(B_i - w_{i,l}\id)a_i
        \\
        = \; & b_i \left( B_i^{\bv_i-1} 
          - e_1(w_{i,1},\dots,\widehat{w_{i,k}},\dots, w_{i,\bv_i}) B_i^{\bv_i-2}
          + \cdots
          + (-1)^{\bv_i - 1} 
          e_{\bv_i-1}(w_{i,1},\dots,\widehat{w_{i,k}},\dots, w_{i,\bv_i}) \id
          \right) a_i.
        \end{split}
    \end{equation*}
\end{NB}%
These are considered as coordinates of
$\cM(\underline{\bv},\underline{\bw})\times_{\AAA^{\underline{\bv}}}\AAA^{|\bv|}$.

\begin{NB}
    It seems the following is not necessary.

Let us write $y_{i,k}$ in Hurtubise normal form (\propref{prop:triangle-Nahm}(1)). Let us put subscript $i$ to $u$, $\eta$, etc, for data on segment between $V_i^{\bw_i}$ and $V_{i+1}^0$. (This is different from the convention for $(A,B,a,b)$ above.) Then $b_i$ is written by data on the $i$th segment, while $a_i$, $B_i$ are ones on the $(i-1)$th segment. We have
\begin{equation*}
b_i = 
\begin{cases}
\begin{bmatrix}
 0 & 0 & 1
\end{bmatrix}
u_i & \text{if $\bv_i > \bv_{i+1}$},\\
J_i u_i & \text{if $\bv_i = \bv_{i+1}$},\\
-f_i & \text{if $\bv_i < \bv_{i+1}$},
\end{cases}
\quad
(a_i,B_i) =
\begin{cases}
(g_{i-1},h_{i-1})  & \text{if $\bv_{i-1} > \bv_i$},\\
(I_{i-1},h_{i-1} - I_{i-1}J_{i-1}) & \text{if $\bv_{i-1} = \bv_{i}$},\\
(u_{i-1}^{-1}
\begin{bmatrix}
 0 \\ 1 \\ 0
\end{bmatrix},
- u_{i-1}^{-1}\eta_{i-1} u_{i-1})
& \text{if $\bv_{i-1} < \bv_{i}$}
\end{cases}
\end{equation*}
\end{NB}%

In the remainder of this subsection, we study some examples of
$\cM(\underline{\bv}, \underline{\bw})$ and
$\overline{\cM}(\underline{\bv}, 0)$. We will be interested in affine algebraic varieties, hence set $\nu^\RR = 0$.

\subsubsection{\texorpdfstring{$n=1, \bv = 1$}{n=1,v=1}}\label{subsubsec:locmodel1}

Let us first consider the case $\nu^\CC = 0$.
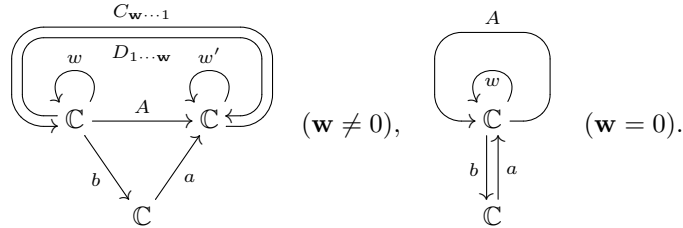
\begin{align*}
    \xymatrix@C=1.2em{
      &&&&
      \\
      & \CC \ar@(ur,ul)_{w} \ar[rr]^{A} \ar[dr]_{b} 
      \ar@<-2pt> `l[lu] `[ur] `[1,3]_{D_{1\cdots\bw}} `_l[0,2] [0,2]
      \ar@{<-} @<+2pt> `l[lu] `[ur] `[1,3]^{C_{\bw\cdots 1}} `_l[0,2] [0,2]
      && \CC \ar@(ur,ul)_{w'} & 
  \\
  && \CC \ar[ur]_a &&
      } \ \ \xymatrix{ \\ (\bw \neq 0), \\ }  \ \ 
  \xymatrix@C=1.2em{
      &&
      \\
      & \CC \ar@(ur,ul)^{w} \ar@<-.5ex>[d]_{b} 
      \ar@{<-} `l[lu] `[ur] `[1,1]^{A} `_l[0,0] [0,0]
      & 
  \\
  & \CC \ar@<-.5ex>[u]_a &
      }  \ \ \xymatrix{ \\ (\bw = 0). \\ }  
\end{align*}
We have $w' = -D_1C_1 = \cdots = -C_{\bw}D_{\bw} = w$.
Then $B_2A-AB_1 + ab = 0$ means $ab=0$, and in closed orbits we can take a representative $a=b=0, A=1$ when $\bw\neq 0$.
We have $y = C_{\bw \cdots 1}$. Let us set $x = (-1)^\bw D_{1 \cdots \bw}$.
Then we get $xy = w^{\bw}$ because $w= -D_1C_1$, so
\begin{align*}
\CC[\cM(1, \bw)] = \CC[w, y, x]/ (xy -w^{\bw}).
\end{align*}
When $\bw=0$, $y:=A$ becomes invertible if and only if (S1) and (S2) are imposed, thus we get
\begin{align*}
\CC[\cM(1, 0)] = \CC[w, y^{\pm 1}], \ \ \CC[\overline{\cM}(1, 0)] = \CC[w, y].
\end{align*}

We consider complex parameter
$(\nu_1^\CC,\dots,\nu_\bw^\CC,\nu^\CC_*)$ as above. Then
$w = \nu^\CC_{\bw} - C_\bw D_\bw = \nu_{\bw-1} ^\CC - C_{\bw -
  1}D_{\bw - 1} = \dots = \nu_1^\CC - C_1 D_1 = w' - \nu^\CC_*$. Thus
we change $ab = 0$ to $\nu_*^\CC A + ab = 0$.

Suppose $\bw\neq 0$. Since $A=0$ contradicts with (S1),(S2), we can
normalize $A=1$. We can furthre normalize $a$, $b$ by
$\CC^\times$-action and the equation $ab=-\nu_*^\CC$.  The defining
relation is perturbed as $xy
\begin{NB}
    = (-C_1 D_1)(- C_2 D_2) \dots (-C_w D_w)
\end{NB}%
= (w - \nu_1^\CC)\dots (w - \nu_{\bw}^\CC)$. The overall shift
$(\nu_1^\CC,\dots,\nu_{\bw}^\CC)\mapsto
(\nu_1^\CC+s,\dots,\nu_{\bw}^\CC+s)$ ($s\in\CC$) and $\nu^\CC_*$ are
trivial deformation, and other directions are nontrivial.

When $\bw=0$, we set $y:=A$ as above. It is invertible thanks to (S1),(S2). Once $y$ is fixed, $a$, $b$ are also fixed by the equation and the $\CC^\times$-action. Hence $\CC[\cM_\nu(1,0)] = \CC[w,y^\pm]$.

\begin{NB}
Suppose $\bw\neq 0$ and we have complex parameter $\nu^\CC =
(\nu_1^\CC,\dots,\nu_\bw^\CC)$. Then $w = \nu^\CC_{\bw} - C_\bw D_\bw
= \nu_{\bw-1} ^\CC - C_{\bw - 1}D_{\bw - 1} = \dots = \nu_1^\CC - C_1
D_1 = w'$. Thus $ab = 0$ as above. The defining relation is perturbed
as $xy
\begin{NB2}
    = (-C_1 D_1)(- C_2 D_2) \dots (-C_w D_w)
\end{NB2}%
= (w - \nu_1^\CC)\dots (w - \nu_{\bw}^\CC)$. The overall shift
$(\nu_1^\CC,\dots,\nu_{\bw}^\CC)\mapsto
(\nu_1^\CC+s,\dots,\nu_{\bw}^\CC+s)$ ($s\in\CC$) is a trivial
deformation, and other directions are nontrivial.

When $\bw = 0$, we change the equation as $ab = 0$ to $\nu_*^\CC A + ab =
0$. As above $A=0$ contradicts with (S1),(S2). Once $y=A$ is fixed, $a$,
$b$ are also fixed by the equation and the $\CC^\times$-action. Hence
$\CC[\cM_\nu(1,0)] = \CC[w,y^\pm]$.
\end{NB}%

\subsubsection{\texorpdfstring{$n=1, \bv = 2$}{n=1,v=2}}\label{subsubsec:locmodel2}

\begin{align*}
    \xymatrix@C=1.2em{
      &&&&
      \\
      & \CC^2 \ar@(ur,ul)_{B} \ar[rr]^{A} \ar[dr]_{b} 
      \ar@<-2pt> `l[lu] `[ur] `[1,3]_{D_{1\cdots\bw}} `_l[0,2] [0,2]
      \ar@{<-} @<+2pt> `l[lu] `[ur] `[1,3]^{C_{\bw\cdots 1}} `_l[0,2] [0,2]
      && \CC^2 \ar@(ur,ul)_{B'} & 
  \\
  && \CC \ar[ur]_a &&
      } \ \ \xymatrix{ \\ (\bw \neq 0), \\ }  \ \ 
  \xymatrix@C=1.2em{
      &&
      \\
      & \CC^2 \ar@(ur,ul)^{B} \ar@<-.5ex>[d]_{b} 
      \ar@{<-} `l[lu] `[ur] `[1,1]^{A} `_l[0,0] [0,0]
      & 
  \\
  & \CC \ar@<-.5ex>[u]_a &
      }  \ \ \xymatrix{ \\ (\bw = 0). \\ }
\end{align*}
If $\bw\neq 0$, the bow variety is isomorphic to the affine type $A_{\bw-1}$ quiver variety $\fM_{\nu^\CC}(\underline{\bv}', \underline{\bw}')$ with dimension vectors $\underline{\bv}' = (2, \cdots, 2), \underline{\bw}' = (1, 0, \cdots, 0)$ by \thmref{thm:balanced_bow}. (The parameter is shifted as
in \eqref{eq:21}.)
Further more, it is well-known
\begin{align*}
\cM(2, \bw) \cong \fM_0(\underline{\bv}', \underline{\bw}') \cong S^2(\CC^2/\ZZ_{\bw})
\end{align*}
for $\nu^\CC = 0$.
(See e.g., \cite[2.10]{Lecture} for $\bw = 1$: We can set $A=\id$,
$a=b=0$, and $C = \diag(y_1,y_2)$, $D = \diag(x_1,x_2)$, $B = DC = \diag(w_1,w_2)$.)
If $\nu^\CC_* = 0$, it is the second symmetric product of $xy =
(w-\nu^\CC_1)(w-\nu^\CC_2)\cdots (w-\nu^\CC_{\bw_1})$. If $\nu^\CC_*\neq 0$, it is
a deformation of the symmetric product.

If $\bw = 0$, we again use \cite[2.10]{Lecture}, and impose that $A$
is invertible. Hence
\begin{align*}
\CC[\cM(2, 0)] \cong \CC[w_1,w_2,y_1^\pm,y_2^\pm]^{\mathfrak S_2}
= \CC[S^2(\CC \times \CC^{\times})], \ 
\CC[\overline{\cM}(2,0)] \cong \CC[w_1,w_2,y_1,y_2]^{\mathfrak S_2}
= \CC[S^2(\CC^2)]
\end{align*}
for $\nu^\CC_* = 0$. For general $\nu_*^\CC$,
\begin{NB}
    it is a deformation of $S^2(\CC\times \CC^{\times})$.
\end{NB}%
the defining equation is
\begin{NB}
  as $B' = B + \nu_*^\CC$, 
\end{NB}%
\begin{equation*}
    [B, A] + \nu_*^\CC A + ab = 0.
\end{equation*}
This is the trigonometric Calogero-Moser space \cite[\S7.3]{MR2077482}.

\subsubsection{\texorpdfstring{$n=2, \underline{\bv} = (1,0)$}{n=2,v=(1,0)}}\label{subsubsec:locmodel3}
\begin{align*}
    \xymatrix@C=1.2em{
      & \CC \ar@(ur,ul)_{w'} \ar@<-.5ex>[rr]_{C_{\bw \cdots 1}} 
      && \CC \ar@(ur,ul)_{w} \ar@<-.5ex>[ll]_{D_{1\cdots \bw}} \ar[dr]_b& 
  \\
\CC \ar[ur]_a && && \CC
      }
\end{align*}
The conditions (S1) and (S2) mean $a, b\in \CC^{\times}$, so we can
take a representation $a=1=b$ when $\bw \neq 0$.  Then we have $y =
C_{\bw \cdots 1}$ by definition and set $x=(-1)^w D_{1 \cdots \bw}$.
Since we again get $xy = (w-\nu_1^\CC)\dots (w-\nu_{\bw}^\CC)$, so
\begin{align*}
\CC[\cM((1,0), (\bw,0))] = \CC[w, y, x]/ 
(xy-(w-\nu_1^\CC)\dots (w-\nu_{\bw}^\CC)).
\end{align*}
When $\bw=0$, we have $y = ba$ by definition and $y\in \CC^{\times}$ if and only if (S1) and (S2) are imposed:
\begin{align*}
\CC[\cM((1,0), 0)] = \CC[w, y^{\pm 1}], \ \ \CC[\overline{\cM}((1,0), 0)] = \CC[w, y].
\end{align*}
Furthermore, the collapsing morphism is described as
\begin{align*}
s \colon \cM((1,0), (\bw,0)) \rightarrow \overline{\cM}((1,0), 0), (w, y, x) \mapsto (w, y).
\end{align*}

\subsubsection{\texorpdfstring{$n=2, \underline{\bv} = (1, 1)$}{n=2,v=(1,1)}}\label{subsubsec:locmodel4}
\begin{align*}
    \xymatrix@C=1.2em{
      &&&&&&&&&
      \\
      & \CC \ar@(ur,ul)_{w_1} \ar[rr]^{A_1} \ar[dr]_{b_1} 
      \ar@<-2pt> `l[lu] `[ur] `[1,8]_{D_{1 \cdots \bw_1,1}} `_l[0,7] [0,7]
      \ar@{<-} @<+2pt> `l[lu] `[ur] `[1,8]^{C_{\bw_1 \cdots 1,1}} `_l[0,7] [0,7]
      && \CC \ar@(ur,ul)_{w'_2}\ar@<-.5ex>[rrr]_{C_{\bw_2 \cdots 1,2}}
      &&& \CC \ar@(ur,ul)_{w_2}\ar@<-.5ex>[lll]_{D_{1 \cdots\bw_2,2}} \ar[rr]^{A_2} \ar[dr]_{b_2}
      && \CC \ar@(ur,ul)_{w'_1}& 
  \\
  && \CC \ar[ur]_{a_2} & &&& & \CC \ar[ur]_{a_1} && 
      }
\end{align*}
We have $y_1 = b_1C_{\bw_1 \cdots 1,1}a_1$ and $y_2=b_2C_{\bw_2 \cdots 1,2}a_2$ by definition.
By Lemma \ref{lem:triangle_fullrank}, $A_1$ and $A_2$ are invertible, so we set $y_{12} = -A_1^{-1}D_{1 \cdots \bw_2,2}A_2^{-1}D_{1 \cdots \bw_1,1}$.
Let $f_{1}(w_1) := (w_1-\nu_{1,1}^\CC)\dots (w_1 -
\nu_{\bw_1,1}^\CC)$, $f_{2}(w_2) := (w_2-\nu_{1,2}^\CC)\dots (w_2 -
\nu_{\bw_2,2}^\CC)$.  Note also $w_1' = w_1 + \nu^\CC_*$, $w_2' = w_2 + \nu^\CC_*$. Since
\begin{align*}
A_1A_2&\left(y_1y_2y_{12}- f_1(w_1) f_2(w_2)
  (w_1-w_2+\nu^\CC_*)(w_1 - w_2 - \nu^\CC_*)\right) \\
      &= -b_1a_1C_{\bw_1 \cdots 1,1}D_{1 \cdots \bw_1,1}b_2a_2C_{\bw_2 \cdots 1,2}D_{1 \cdots \bw_2,2}
        + f_1(w_1)f_2(w_2)A_1(w_2'-w_1)A_2(w_1'-w_2)\\
&=-b_1a_1b_2a_2 f_1(w_1)f_2(w_2) + f_1(w_1) f_2(w_2)(-a_1b_2)(-a_2b_1)=0,
\end{align*}
we have
\begin{multline*}
  \CC[\cM((1,1), (\bw_1,\bw_2))] \\
  = \CC[w_1, w_2, y_1, y_2, y_{12}]/ (y_1y_2y_{12}- f_1(w_1)f_2(w_2)(w_1-w_2+\nu^\CC_*)(w_1-w_2-\nu^\CC_*)).
\end{multline*}
When $\bw_1=\bw_2=0$, $y_{12}^{-1}$ becomes $-A_1A_2$, so we get
\begin{align*}
\CC[\cM((1,1), 0)] &= \CC[w_1, w_2, y_1, y_2, y_{12}^{\pm 1}]/ (y_1y_2-y_{12}^{-1}(w_1-w_2)^2),\\
\CC[\overline{\cM}((1,1), 0)] &= \CC[w_1, w_2, y_1, y_2, y_{12}^{- 1}]/ (y_1y_2-y_{12}^{-1}(w_1-w_2)^2)
\end{align*}
if $\nu_*^\CC = 0$. For general $\nu_*^\CC$, we
impose $w_1 - w_1' = -\nu^\CC_*$, $w_2 - w'_2 = -\nu^\CC_*$. Hence the
defining equation is
\begin{equation*}
    \begin{NB}
        y_1 y_2 - y_{12}^{-1} (w'_1 - w_2)(w_1 - w'_2) = 
    \end{NB}
    y_1 y_2 - y_{12}^{-1} (w_1 - w_2 + \nu_*^\CC)(w_1 - w_2 - \nu_*^\CC) = 0.
\end{equation*}

\begin{NB}
    For $n \ge 3$, $\underline{\bv} = (1,\dots,1)$, we have
    \begin{equation*}
     \CC[\cM(\underline{\bv},0)] = \CC[y_1,\dots,y_n,w_1,\dots,w_n,
    y_{12\dots n}^{-1}]/ (y_1\dots y_n + y_{12\dots n}^{-1}(w_1 - w_2
    - \nu_*^\CC)(w_2 - w_3 - \nu_*^\CC)\dots (w_n - w_1 - \nu_*^\CC)).
    \end{equation*}
\end{NB}%

\subsubsection{\texorpdfstring{$n=2, \underline{\bv} = (2, 0)$}{n=2,v=(2,0)}}\label{subsubsec:locmodel5}
\begin{align*}
    \xymatrix@C=1.2em{
      & \CC^2 \ar@(ur,ul)_{B'} \ar@<-.5ex>[rr]_{C_{\bw \cdots 1}} 
      && \CC^2 \ar@(ur,ul)_{B} \ar@<-.5ex>[ll]_{D_{1\cdots \bw}} \ar[dr]_b& 
  \\
\CC \ar[ur]_a && && \CC
      }
\end{align*}
First we consider the $\bw \neq 0$ case.
Note $\tr(B^{\prime k}) 
\begin{NB}
    = \tr((\nu^\CC_* + \nu_1^\CC - D_1 C_1)^k) 
    = \tr((\nu^\CC_* + \nu_1^\CC - C_1 D_1)^k)
    = \tr((\nu^\CC_* + \nu_2^\CC - D_2 C_2)^k) = \cdots
    = \tr((\nu^\CC_* + \nu_{\bw}^\CC - D_\bw C_\bw)^k)
\end{NB}
 = \tr((\nu^\CC_*+B)^k)$ for $k=1$, $2$ by defining equations.
When $B$ has only one eigenvalue, $B$ (resp.\ $B'$) must be equivalent to the Jordan form with size $2$ because of the condition (S2) (resp.\ (S1)).
So $B+\nu^\CC_k$ and $B'$ are always equivalent.
Using \corref{cor:triangle_inverse} and \propref{prop:triangle-Nahm}(1), we can identify two triangle parts.
Thus the bow variety is isomorphic to affine type $A_{\bw-1}$ quiver variety $\fM_{\nu}(\underline{\bv}', \underline{\bw}')$ (see also \ref{subsubsec:locmodel2}, \cite[Lem.~2.9, Th.~3.24]{Lecture}):
\begin{align*}
    \xymatrix@C=1.2em{
    &&&& \\
      & \CC^2 \ar@(ur,ul)_{B'} \ar@<-.5ex>[rr]_{C_{\bw \cdots 1}} 
      && \CC^2 \ar@(ur,ul)_{B} \ar@<-.5ex>[ll]_{D_{1\cdots \bw}}& 
  \\
\CC \ar[ur]_a && && \CC \ar[ul]^a
      } 
\ \ \xymatrix{ \\ \cong \\ } \ \ 
  \xymatrix@C=1.2em{
      &&
      \\
      & \CC^2 \ar@(ur,ul)^{B} \ar@<-.5ex>[d]_{0} 
      \ar@<-2pt> `l[lu] `[ur] `[1,1]_{D_{1\cdots\bw}} `_l[0,0] [0,0]
      \ar@{<-} @<+2pt> `l[lu] `[ur] `[1,1]^{C_{\bw\cdots 1}} `_l[0,0] [0,0]
      & 
  \\
  & \CC \ar@<-.5ex>[u]_a &
      }    
\end{align*}
And it is a deformation of $S^2(\CC^2/\ZZ_{\bw})$.

When $\bw = 0$, we have $y_1 = b(B-w_2)a$ and $y_2=b(B-w_1)a$ by definition.
Here we can see $y_i$ is invertible if and only if (S1) and (S2) are imposed. 
Set $\xi = ba$.
Since
\begin{align*}
y_1-y_2-\xi(w_1-w_2) &= b(B-w_2)a - b(B-w_1)a - ba(w_1-w_2)=0,
\end{align*}
we have
\begin{align*}
\CC[\cM((2,0), 0)\times_{{\AAA}^{|2|}} {\AAA}^2] &= \CC[w_1, w_2, y_1^{\pm 1}, y_2^{\pm 1}, \xi]/ (y_1-y_2-\xi(w_1-w_2)),\\
\CC[\overline{\cM}((2,0), 0)\times_{{\AAA}^{|2|}} {\AAA}^2] &= \CC[w_1, w_2, y_1, y_2, \xi]/ (y_1-y_2-\xi(w_1-w_2)).
\end{align*}
See also \cite[Ex.~3.20]{main}.

\subsubsection{\texorpdfstring{$n=3, \underline{\bv} = (1, 1, 0)$}{n=3,v=(1,1,0)}}\label{subsubsec:locmodel6}
\begin{align*}
    \xymatrix@C=1.2em{
      &\CC \ar@(ur,ul)_{w_1+\nu^\CC_*}\ar@<-.5ex>[rrr]_{C_{\bw_1 \cdots 1,1}} &&& \CC \ar@(ur,ul)_{w_1} \ar[rr]^{A} \ar[dr]_{b_1} \ar@<-.5ex>[lll]_{D_{1\cdots \bw_1,1}}
      && \CC \ar@(ur,ul)_{w_2+\nu^\CC_*} \ar@<-.5ex>[rrr]_{C_{\bw_2 \cdots 1,2}} &&& \CC \ar@(ur,ul)_{w_2}\ar[dr]_{b_2} \ar@<-.5ex>[lll]_{D_{1\cdots \bw_2,2}}
      & 
  \\
\CC \ar[ur]_{a_1}  &&&&& \CC \ar[ur]_{a_2} &&&&& \CC
      }
\end{align*}
We have $y_1 = b_1C_{\bw_1 \cdots 1,1}a_1$ and $y_2=b_2C_{\bw_2 \cdots 1,2}a_2$ by definition.
Here, (S1) and (S2) mean $a_1, b_2, A \in \CC^{\times}$ and we normalize $a_1=1, b_2=1$ by the action of the leftmost and rightmost $\CC^\times$.
Set $y_{12} = D_{1\cdots \bw_1,1}A^{-1}D_{1\cdots \bw_2,2}$.
Let $f_{1}(w_1) := (w_1-\nu_{1,1}^\CC)\dots (w_1 -
\nu_{\bw_1,1}^\CC)$, $f_{2}(w_2) := (w_2-\nu_{1,2}^\CC)\dots (w_2 -
\nu_{\bw_2,2}^\CC)$ as before.
Since
\begin{align*}
  & A\{y_1y_2y_{12}-f_1(w_1) f_2(w_2)(w_1-w_2-\nu^\CC_*)\}\\
  =\; & b_1C_{\bw_1 \cdots 1,1}C_{\bw_2 \cdots 1,2}a_2D_{1\cdots \bw_1,1}D_{1\cdots \bw_2,2} + f_1(w_1) f_2(w_2)A(w_2+\nu^\CC_*-w_1)\\
=\; & b_1a_2 f_1(w_1) f_2(w_2) +  f_1(w_1) f_2(w_2) (-a_2b_1)=0,
\end{align*}
we have
\begin{align*}
\CC[\cM((1,1,0), \underline{\bw})] = \CC[w_1, w_2, y_1, y_2, y_{12}]/
(y_1y_2y_{12} -f_1(w_1) f_2(w_2)(w_1-w_2-\nu^\CC_*)).
\end{align*}
When $\bw_1=\bw_2=0$, $y_{12}^{-1} = A$, so we get
\begin{align*}
\CC[\cM((1,1,0), 0)] &= \CC[w_1, w_2, y_1, y_2, y_{12}^{\pm 1}]/ (y_1y_2-y_{12}^{-1}(w_1-w_2-\nu_*^\CC)),\\
\CC[\overline{\cM}((1,1,0), 0)] &= \CC[w_1, w_2, y_1, y_2, y_{12}^{-1}]/ (y_1y_2-y_{12}^{-1}(w_1-w_2-\nu_*^\CC)).
\end{align*}
\begin{NB}
  $0 = A(w'_2 - w_1) + a_2 b_1 = A(w_2 + \nu_*^\CC - w_1) + a_2 b_1$.
\end{NB}%
The direction of $\nu^\CC_*$ is a trivial deformation as expected.

\subsection{Factorization property}
We can define $\Psi \colon \mathbb{M}(\underline{\bv}, \underline{\bw}) \rightarrow {\AAA}^{\underline{\bv}}$.
And for $\underline{\bv}= \underline{\bv}' + \underline{\bv}''$, take $({\AAA}^{\underline{\bv}'}\times{\AAA}^{\underline{\bv}''})_{\mathrm{disj}} \subset {\AAA}^{\underline{\bv}'}\times{\AAA}^{\underline{\bv}''}$ as before.
More precisely we consider factorization induced from one for triangle
and two-way parts. For triangle part, we compare eigenvalues of $B_i$
and $B'_{i+1}$, hence second coordinates should be shifted as
$w_{i+1,l} + \nu^\CC_*$. (See (a)' in \subsecref{subsubsec:proof2} below.)
We have factorization of $\cM(\underline{\bv}, \underline{\bw})$:
\begin{Theorem}\label{thm:factorization}
$\cM(\underline{\bv}, \underline{\bw})$ has a factorization morphism:
\begin{align*}
\mathfrak{f}_{\underline{\bv}',\underline{\bv}''} \colon \cM(\underline{\bv}, \underline{\bw})\times_{{\AAA}^{\underline{\bv}}} ({\AAA}^{\underline{\bv}'}\times{\AAA}^{\underline{\bv}''})_{\mathrm{disj}} \xrightarrow{\ \sim \ } (\cM(\underline{\bv}', \underline{\bw})\times \cM(\underline{\bv}'', \underline{\bw})) \times_{ {\AAA}^{\underline{\bv}'}\times {\AAA}^{\underline{\bv}''}} ({\AAA}^{\underline{\bv}'}\times{\AAA}^{\underline{\bv}''})_{\mathrm{disj}}.
\end{align*}
\end{Theorem}
\begin{Corollary}\label{cor:collaps_fact}
    The factorization morphism is compatible with the collapsing
    morphism in \subsecref{subsec:collapsing-morph}.
\end{Corollary}

Consider coordinates $(w_{i,1},\dots,w_{i,\bv_i}, y_{i,1},\dots,
y_{i,\bv_i})$ in \subsecref{subsec:coord-local-models}. Under the
factorization, $w$-coordinates are grouped into two as
$(w_{i,1},\dots,w_{i,\bv'_i})$,
$(w_{i,\bv'_i+1},\dots,w_{i,\bv_i})$. These can be considered as
$w$-coordinates of $\cM(\underline{\bv}',\underline{\bw})$ and
$\cM(\underline{\bv}'',\underline{\bw})$ respectively. Let
$(y'_{i,1},\dots,y'_{i,\bv'_i})$,
$(y''_{i,\bv'_i+1},\dots,y'_{i,\bv_i})$ be $y$-coordinates of
$\cM(\underline{\bv}',\underline{\bw})$ and
$\cM(\underline{\bv}'',\underline{\bw})$ respectively. Then we have

\begin{Lemma}\label{lem:coord_factor}
    The factorization morphism
    $\mathfrak{f}_{\underline{\bv}',\underline{\bv}''}$ sends
    $y_{i,k}$ to
    \begin{equation*}
        \begin{cases}
            y'_{i,k} \prod_{l=\bv'_i+1}^{\bv_i} (w_{i,k} - w_{i,l}) &
            \text{if $1\le k\le \bv'_i$},\\
            y''_{i,k} \prod_{l=1}^{\bv'_i} (w_{i,k} - w_{i,l}) &
            \text{if $\bv'_i + 1\le k\le \bv_i$}.
        \end{cases}
    \end{equation*}
\end{Lemma}

\begin{proof}
    We have block decompositions
    $a_i = \left[
      \begin{smallmatrix}
          a'_i \\ a''_i
      \end{smallmatrix}\right]$,
    $b_i = \left[
      \begin{smallmatrix}
          b'_i & b''_i
      \end{smallmatrix}\right]$,
    $C_{\bw_i\cdots 1,i} = \left[
      \begin{smallmatrix}
          C_{\bw_i\cdots 1,i}' & 0 \\
          0 & C_{\bw_i\cdots 1,i}''
      \end{smallmatrix}\right]$,
    $B_i = \left[
      \begin{smallmatrix}
          B_i' & 0 \\
          0 & B_i''
      \end{smallmatrix}\right]$. Therefore
    \begin{equation*}
        y_{i,k} = b'_i 
        \prod_{\substack{1\le l\le\bv_i \\l\neq k}}
        (B'_i - w_{i,l}\id) C_{\bw_i\cdots 1,i}' a'_i
        + b''_i \prod_{\substack{1\le l\le\bv_i \\l\neq k}}
        (B''_i - w_{i,l}\id) C_{\bw_i\cdots 1,i}'' a''_i.
    \end{equation*}
    \begin{NB}
    \begin{equation*}
        y_{i,k} = b'_i C_{\bw_i\cdots 1,i}' 
        \prod_{\substack{1\le l\le\bv_i \\l\neq k}}
        (B'_i - w_{i,l}\id) a'_i
        + b''_i C_{\bw_i\cdots 1,i}'' \prod_{\substack{1\le l\le\bv_i \\l\neq k}}
        (B''_i - w_{i,l}\id) a''_i.
    \end{equation*}
    \end{NB}%
    Suppose $1\le k\le \bv'_i$. Then the second term vanishes as all
    eigenvalues of $B''_i$ appear in $w_{i,l}$ ($l\neq k$) with
    multiplicities.

    Consider
    \begin{equation*}
        \prod_{\substack{1\le l\le\bv_i \\l\neq k}}
        (B'_i - w_{i,l}\id)
        = \prod_{\substack{1\le l\le\bv'_i \\l\neq k}} (B'_i - w_{i,l}\id)
        \prod_{\bv'_i+1\le l\le \bv_i} (B'_i - w_{i,l}).
    \end{equation*}
    Since $\prod_{1\le l\le\bv'_i} (B'_i - w_{i,l}\id) = 0$, we have
    \begin{equation*}
        \prod_{\substack{1\le l\le\bv'_i \\l\neq k}} (B'_i - w_{i,l}\id)
        B'_i 
        = 
        \prod_{\substack{1\le l\le\bv'_i \\l\neq k}} (B'_i - w_{i,l}\id) w_{i,k}.
    \end{equation*}
    Now the assertion is clear. The case $\bv'_i+1\le k\le\bv_i$ is
    the same.
\end{proof}

\begin{Remark}
    This property for the case $\underline{\bw} = 0$ follows from the
    first part of the proof of \cite[Prop.~3.2]{bdf}, as it is a
    reduction to the $\SL(2)$-case and works also for an affine
    case. The assertion for general $\underline{\bw}$ follows from one
    for $\underline{\bw} = 0$ by \corref{cor:collaps_fact}. We present
    a proof here for the sake of completeness.
\end{Remark}

\subsection{Normality}
\begin{Proposition}
Over each point of $\AAA^{\underline{\bv}}$, the dimension of the fiber of $\Psi$ is $\sum (\bw_i+1)\bv_i^2+\bv_i$.\label{prop:bow_fiber}
\end{Proposition}
\begin{proof}
The equations $B_{i} = D_{1,i}C_{1,i}$ and $C_{\bw_i,i}D_{\bw_i,i} = B'_{i}$ mean $B_i$ and $B'_i$ are always determined by a two-way part.
Thus Propositions \ref{prop:triangle_fiber}, \ref{prop:two-way_dim} mean $\dim \Psi^{-1}(w) = \sum (\bw_i+1)\bv_i^2+\bv_i$ for any $w \in \AAA^{\underline{\bv}}$.
\begin{NB}
For any $x \in \AAA^{\underline{\bv}}$, we have $\dim \left(\Psi_{\mathrm{tr}}^{-1}(x) \times \Psi_{\mathrm{tw}}^{-1}(x)\right) = \sum 2\bv_i^2 + (\bw_i+1)\bv_i^2 -\bv_i$ because of Propositions.
Since characteristic polynomials of $B_{i}, C_{1,i}D_{1,i}$ and $B'_{i}$ are same, the additional equations $B_{i} = D_{1,i}C_{1,i}$ and $C_{w_i,i}D_{w_i,i} = B'_{i}$ deduces the dimension $2(\bv_i^2-\bv_i)$.
Thus we get
\begin{align*}
\sum 2\bv_i^2 + (\bw_i+1)\bv_i^2 -\bv_i - 2(\bv_i^2-\bv_i) = \sum (\bw_i+1)\bv_i^2+\bv_i
\end{align*}
as required.
\end{NB}
\end{proof}
\begin{Corollary}
$\mu^{-1}(0)$ is an irreducible reduced complete intersection in $\mathbb{M}$.
\end{Corollary}

\begin{Theorem}\label{thm:normal}
$\cM(\underline{\bv}, \underline{\bw})$ is normal.
\end{Theorem}
\begin{proof}
    We apply \cite[Cor.~7.2]{CB:normal} in our case, i.e., we
    construct an open subscheme $U$ of $\cM = \cM(\underline{\bv},
    \underline{\bw})$ such that it is normal, the complement in $\cM$
    and the inverse image of the complement in $\widetilde\cM$ both
    have codimension at least two, and $\widetilde\cM$ has the
    property $(S_2)$ of Serre.

\begin{NB}
    Need to introduce $\bA^{\underline{\bv}}$.
\end{NB}%

We consider the open subset $\bA^{\underline{\bv}} \subset
\AAA^{\underline{\bv}}$ consisting of configurations where at most two
points collide. (The same notation $\bA^{\underline{\bv}}$ will be used for a smaller open subset in the proof of \thmref{thm:main}.)
The complement of $\bA^{\underline{\bv}}$ is of codimension $2$, and
the complements of $\Psi^{-1}(\bA^{\underline{\bv}})$ in
$\widetilde\cM$ and $\Psi^{-1}(\bA^{\underline{\bv}}) \dslash \GV$ in
$\cM(\underline{\bv}, \underline{\bw})$ respectively are of
codimension $2$ by Proposition \ref{prop:bow_fiber}.  And
$\widetilde\cM$ is Cohen-Macaulay, hence has the property $(S_2)$.

Thus it is enough to check  $\Psi^{-1}(\bA^{\underline{\bv}}) \dslash \GV$ is normal.
This follows from normality of local models \S \ref{subsubsec:locmodel1}
$\sim$\ref{subsubsec:locmodel6}.
\end{proof}

\subsection{Proof}\label{subsec:proof}

We now prove that the Coulomb branch $\cM_C$ of a framed quiver gauge
of an affine type $A_{n-1}$ is isomorphic to a bow variety in
\subsecref{subsec:Coulomb_gauge}. We will use results in \cite{main}
frequently.

We consider a quiver gauge theory associated with two dimension
vectors $\underline{\bv}$, $\underline{\bw}$. Namely we consider
\[
   \bN := \bigoplus_{i=0}^{n-1}\Hom(\CC^{\bv_i},\CC^{\bv_{i+1}})\oplus
   \Hom(\CC^{\bw_i},\CC^{\bv_i})
\]
as a representation of $G := \GL(\underline{\bv}) = \prod_i
\GL(\bv_i)$. Let $T$ be the product of maximal tori of $\GL(\bv_i)$.
Let $\bN_T$ denote the restriction of $\bN$ to $T$.

\subsubsection{}
We first consider the case $\nu^\CC = 0$.

Let $\cO = \CC[[z]]\subset\cK = \CC((z))$. Let $\Gr_G$ denote the
affine Grassmannian $G_\cK/G_\cO$.
We consider the variety of triples $\cR := \{ ([g],s)\in
\Gr_G\times\bN_\cO \mid g^{-1}s\in\bN_\cO\}$ and its equivariant
Borel-Moore homology group $H^{G_\cO}_*(\cR)$. See
\cite[\S2]{main}. It is equipped with a convolution product, which is
commutative. The Coulomb branch $\cM_C$ is defined as its
spectrum (\cite[\S3]{main}). From its definition we have a morphism
$\intsys\colon \cM_C\to \operatorname{Spec}H^*_{G}(\mathrm{pt})
\cong\AAA^{\underline\bv}$.
It is known that $\intsys^{-1}(\wA^{\underline{\bv}})$ is
isomorphic to $\wA^{|\underline{\bv}|}\times T^\vee/W$ so that
$\intsys$ is the first projection (\cite[Cor.~5.21]{main}).
Here $\wA^{|\underline{\bv}|}$ consists of distinct configurations (for
each $i$th and each pair $(i,i+1)$) as before, but we also require
that $0$ is not contained in a configuration for $i$ when $\bw_i\neq
0$. In the terminology of \cite[\S5]{main}, $\wA^{|\underline{\bv}|}$
is the complement of the union of all generalized root hyperplanes
associated with $(G,\bN)$. Then we define $\wA^{\underline{\bv}}$ as
$\wA^{|\underline{\bv}|}/\mathfrak S_{\underline{\bv}}$.

Let us consider $w_{i,k}$ as a class in $H^2_{\GL(\bv_i)}(\mathrm{pt})$
corresponding to the $k$th coordinate vector of $\CC^{\bv_i}$ so that
the isomorphism $\Spec H^*_G(\mathrm{pt})\cong \AAA^{\underline{\bv}}$
is explicit. We also take coordinates $\sfu_{i,k}$ of $T^\vee$, which
arises as the fundamental class of the point $w_{i,k}^*$ in $\Gr_T$,
the cocharacter of $T$ corresponding to the $k$th coordinate vector. The isomorphism
\(
  \intsys^{-1}(\wA^{\underline{\bv}}) \cong
  \wA^{|\underline{\bv}|}\times T^\vee/W
\)
has been constructed as $\bz^*(\iota_*)^{-1}$ where
\begin{equation*}
    \iota_* \colon H^{T_\cO}_*(\cR_{T,\bN_T})\to 
    H^{G_\cO}_*(\cR)\otimes_{H^*_{G}(\mathrm{pt})} H^*_T(\mathrm{pt})
\end{equation*}
is the pushforward homomorphism given by the inclusion
$\cR_{T,\bN_T}\to \cR$ from the variety of triples for $(T,\bN_T)$ to
one for $(G,\bN)$, and
\begin{equation*}
    \bz^*\colon H^{T_\cO}_*(\cR_{T,\bN_T})\to
    H^{T_\cO}_*(\Gr_T)
\end{equation*}
is the pull-back homomorphism induced from the inclusion $\bz\colon
\Gr_T\to \cR_{T,\bN_T}$ of the zero section (\cite[\S5]{main}). Then
$H^{T_\cO}_*(\Gr_T) \cong \Spec(\AAA^{|\underline{\bv}|}\times
T^\vee)$ (\cite[Prop.5.19]{main}), and the isomorphism respects the
Weyl group action.

On the other hand, we have a morphism $\cM(\underline{\bv},\underline{\bw})\to \AAA^{\underline{\bv}}$ induced from $\Psi$. Let us denote it also by $\Psi$ for brevity. 
By the factorization and models in
\subsecref{subsec:coord-local-models}, we have
$\Psi^{-1}(\wA^{\underline{\bv}})\cong \wA^{\underline{\bv}}\times
T^\vee/W$:
In fact, we first make a reduction to $\underline{\bv}$ has $1$ for an
entry and $0$ for others by the factorization. Then we have either
\ref{subsubsec:locmodel1} or \ref{subsubsec:locmodel3}. When $\bw =
0$, the model is $\CC[w,y^{\pm 1}] = \Spec (\CC\times\CC^\times)$.
When $\bw \neq 0$, the model is $\{ xy = w^\bw\}$, but our
$\wA^{\underline{\bv}}$ excludes $0$, i.e., $w\neq 0$. Therefore
$\{ xy = w^\bw, w\neq 0\} \cong \CC^\times\times\CC^\times$.

We define an isomorphism 
\begin{equation*}
    \Xi^\circ\colon \CC[\cM(\underline{\bv},\underline{\bw})]
    \otimes_{\CC[\AAA^{\underline{\bv}}]}\CC[\wA^{\underline{\bv}}]
    \xrightarrow{\cong}
    \CC[\cM_C] \otimes_{\CC[\AAA^{\underline{\bv}}]}\CC[\wA^{\underline{\bv}}]
\end{equation*}
over $\CC[\wA^{\underline{\bv}}]$ as follows. It is the identity on
$\AAA^{|\underline{\bv}|}$, i.e., $w_{i,k}$ is sent to $w_{i,k}$. We
send $y_{i,k}$ to the homology class $\iota_* \sfy_{i,k}$, where
$\iota_*$ is as above, and
\begin{equation*}
    \sfy_{i,k} = [\pi^{-1}(w_{i,k}^*)], \quad
    \pi\colon \cR_{T,\bN_T}\to \Gr_T; ([g],s)\mapsto [g].
\end{equation*}
Here $[\pi^{-1}(w_{i,k}^*)]$ denotes the fundamental class of
$\pi^{-1}(w_{i,k}^*)$.
By \cite[\S4(vi)]{main} we have
\begin{equation}\label{eq:2}
    \bz^*(\sfy_{i,k}) =
    \sfu_{i,k} \prod_{1\le l\le \bv_{i+1}} (w_{i+1,l} - w_{i,k}),
\end{equation}
if $n\ge 2$ (i.e., not a Jordan quiver) and
\begin{equation}\label{eq:3}
    \bz^*(\sfy_{k}) =
    \sfu_{k} \prod_{\substack{1\le l\le \bv\\ l\neq k}} (w_{l} - w_{k}),
\end{equation}
if $n=1$. (We omit the subscript $i$ for Jordan quiver.) Therefore
$\Xi^\circ$ is indeed a birational isomorphism. Moreover it is
equivariant under Weyl group (i.e., $\mathfrak S_{\underline{\bv}}$)
action, hence it indeed induces a birational isomorphism
$\cM_C\dasharrow\cM(\underline{\bv},\underline{\bw})$.

\begin{Theorem}\label{thm:main}
    $\Xi^\circ$ extends to an isomorphism $\Xi\colon \cM_C
    \to\cM(\underline{\bv},\underline{\bw})$.
\end{Theorem}

\begin{proof}
    We can use \cite[Th.~5.26]{main}, as we have proved that
    $\cM(\underline{\bv},\underline{\bw})$ is normal and all fibers of
    $\Psi$ have the same dimension. (See \cite[Rem.~5.27]{main}.)
    The rest of the argument is exactly the same as one in
    \cite[\S3]{blowup}, but let us repeat it for completeness.

    Let $\bA^{|\underline{\bv}|}$ denote the complement of all pairwise
    intersections of generalized root hyperplanes as in
    \cite[\S5(vi)]{main}. In our case,
    $t\in\bA^{|\underline{\bv}|}\setminus \wA^{|\underline{\bv}|}$
    satisfies one of the following conditions:
    \begin{aenume}
          \item $w_{i,k}(t) = w_{i+1,l}(t)$ for some $i$, $k$, $l$,
        but all others are distinct. Moreover $w_{j,r}(t) \neq 0$ if
        $\dim W_j\neq 0$.

          \item $w_{i,k}(t) = w_{i,l}(t)$ for distinct $k$, $l$ and
        some $i$, but all others are distinct. Moreover
        $w_{j,r}(t) \neq 0$ if $\dim W_j\neq 0$.

          \item All pairs like in (a),(b) are distinct, but
        $w_{i,k}(t) = 0$ for $i$ with $W_i\neq 0$.
    \end{aenume}
    Let us assume $k=1$ and $l=1$ for (a), $l=2$ for (b) for brevity.
    Let $(G',\bN') = (Z_G(t),\bN^t)$. We have $\cR^t_{G,\bN} =
    \cR_{Z_G(t),\bN^t}$ by \cite[Lem.~5.1]{main}.
    Up to isomorphisms, we have $(G',\bN') = (\GL(\bv')\times
    T^{|\underline{\bv}''|}, \bN(\underline{\bv}',\underline{\bw}'))$
    where $\bN(\underline{\bv}',\underline{\bw}')$ is the vector space
    associated with dimension vectors $\underline{\bv}'$,
    $\underline{\bw}'$ given below, and $\underline{\bv}'' =
    \underline{\bv}-\underline{\bv}'$. Here $T^{|\underline{\bv}''|}$
    acts trivially on $\bN' = \bN(\underline{\bv}',\underline{\bw}')$.

    In case (a), $\underline{\bw}' = 0$, ${\bv}'_i = 1 = {\bv}'_{i+1}$
    and other entries are $0$. (We understand $i\neq i+1$, in
    particular $n\ge 2$. The case $n=1$ is treated in case (b).)

    In case (b), $\underline{\bw}'=0$, ${\bv}'_i=2$ and other entries
    are $0$.

    In case (c), ${\bv}'_i = 1$, ${\bw}'_i = \bw_i$ and other entries are
    $0$.

    Let us consider the Coulomb branch $\cM_C(G',\bN')$ of
    $(G',\bN')$.
    It has an isomorphism \[ \bz^{\prime*}(\iota'_*)^{-1}\colon
    \CC[\cM_C(G',\bN')]\otimes_{H^*_{G'}(\mathrm{pt})}
    \CC[\wA^{|\underline{\bv}|}] \cong
    \CC[\wA^{|\underline{\bv}|}\times T^\vee] \] in the same way as
    $\cM$. Here $\bz'$, $\iota'$ are morphisms $\bz$, $\iota$ as above
    for $(G',\bN')$.
    Therefore we have an isomorphism
    \begin{equation}\label{eq:4}
        \Xi^{t\circ}\colon
        \CC[\cM(\underline{\bv},\underline{\bw})]
        \times_{\CC[\AAA^{\underline{\bv}}]}\CC[\wA^{|\underline{\bv}|}]
        \xrightarrow{\cong}
        \CC[\cM_C(G',\bN')]\otimes_{H^*_{G'}(\mathrm{pt})}
        \CC[\wA^{|\underline{\bv}|}]
    \end{equation}
    as the composite $\iota'_* \bz^{\prime*-1}\bz^*(\iota_*)^{-1}
    \Xi^\circ$. If we check that $\Xi^{t\circ}$ extends to an
    isomorphism $\Xi^t$ over $\CC[\AAA^{|\underline{\bv}|}]_t$ for all
    $t\in\bA^{|\underline{\bv}|}\setminus \wA^{|\underline{\bv}|}$,
    $\Xi^\circ$ extends everywhere by \cite[Th.~5.26]{main}.

    In order to the check this assertion, we first replace the left
    hand side of \eqref{eq:4}.
    We have a factorization morphism
    \begin{equation*}
        \mathfrak f_{\underline{\bv}',\underline{\bv}''}\colon
        \cM(\underline{\bv},\underline{\bw})\times_{\AAA^{\underline{\bv}}}
        (\AAA^{\underline{\bv}'}\times\AAA^{|\underline{\bv}''|})
        \dasharrow
        \cM(\underline{\bv}',\underline{\bw})
        \times 
        T^* (T^{|\underline{\bv}''|})^\vee
    \end{equation*}
    defined in an appropriate open subvariety.
    To study $\Xi^{t\circ}$ at $t$ as above, we can further replace
    the right hand side by
    \(
       \cM(\underline{\bv}',\underline{\bw}')
        \times 
        T^* (T^{|\underline{\bv}''|})^\vee,
    \)
    as $\cM(\underline{\bv}',\underline{\bw}')$ is an open subset in
    $\cM(\underline{\bv}',\underline{\bw})$ containing relevant points.

    On the other hand, consider the right hand side of \eqref{eq:4}.
    Since $G' = \GL(\underline{\bv}')\times T^{|\underline{\bv}''|}$
    where the second factor acts trivially on $\bN'$, we have
    $\cM(G',\bN')\cong \cM(\GL(\underline{\bv}'),\bN')\times T^*
    (T^{|\underline{\bv}''|})^{\vee}$
    (\cite[Prop.~5.19]{main}). Considering
    $(\GL(\underline{\bv}'),\bN')$ as a special case of $(G,\bN)$, we
    can define
\begin{equation*}
    \Xi^{\prime\circ}\colon \CC[\cM(\underline{\bv}',\underline{\bw}')]
    \otimes_{\CC[\AAA^{\underline{\bv}'}]}\CC[\wA^{\underline{\bv}'}]
    \xrightarrow{\cong}
    \CC[\cM_C(\GL(\underline{\bv}'),\bN')]
    \otimes_{\CC[\AAA^{\underline{\bv}'}]}\CC[\wA^{\underline{\bv}'}].
\end{equation*}
If we denote $y$-coordinates of
$\cM(\underline{\bv}',\underline{\bw}')$ by $y'_{j,r}$, and the
corresponding homology class by $\sfy'_{j,r}$, $\Xi^{\prime\circ}$
sends $y'_{j,r}$ to $\iota'_*\sfy'_{j,r}$. By
\lemref{lem:coord_factor}, $\mathfrak
f_{\underline{\bv}',\underline{\bv}''}^*(y'_{j,r})$ and $y_{j,r}$
differ by a factor of a regular function which does not vanish at
$t$. Also \eqref{eq:2} in this case implies
    \begin{equation*}
        \bz^{\prime*}(\sfy'_{j,r}) =
    \sfu_{j,r}\prod_{1\le s\le \bv'_{j+1}}(w_{j+1,s} - w_{j,r})
    = \bz^*(\sfy_{j,r})
    \prod_{\bv'_{j+1}+1 \le s\le \bv_{j+1}}(w_{j+1,s} - w_{j,r})^{-1}.
    \end{equation*}
    Similar formula is true for \eqref{eq:3}.
    Thus $\bz^{\prime*}(\sfy'_{j,r})$ and $\bz^*(\sfy_{j,r})$ also
    differ only by a regular function which does not vanish at
    $t$. Therefore $\Xi^{t\circ}$ extends across $t$ if and only if
    $\Xi^{\prime\circ}$ does. Namely it is enough to show the
    assertion for cases (a),(b),(c) above.

    In case (a), $\cM(\underline{\bv}',\underline{\bw}')$ is described
    in \ref{subsubsec:locmodel4} ($n=2$) and \ref{subsubsec:locmodel6}
    ($n\ge 3$). The corresponding Coulomb branch
    $\cM_C(\GL(\underline{\bv}'),\bN')$ is described in
    \cite[Th.~4.1]{main}. If we identify $y_{12}$ in
    \ref{subsubsec:locmodel4}, \ref{subsubsec:locmodel6} with the
    fundamental class of the fiber over the point $-w_{i,k}^*
    -w_{i+1,l}^*$, we can check that $\Xi^{\prime\circ}$ extends to an
    isomorphism $\cM_C(\GL(\underline{\bv}'),\bN')\to
    \cM(\underline{\bv}',\underline{\bw}')$.

    Case (c) is similar. $\cM(\underline{\bv}',\underline{\bw}')$ is
    described in \ref{subsubsec:locmodel1} ($n=1$) and
    \ref{subsubsec:locmodel3} ($n\ge 2$). The corresponding Coulomb
    branch is described again in \cite[Th~4.1]{main}. If we identify
    $x$ in \ref{subsubsec:locmodel1}, \ref{subsubsec:locmodel3} with
    the fundamental class of the fiber over $-w_{i,k}^*$, we can check
    that $\Xi^{\prime\circ}$ extends.

    In case (b), $\cM(\underline{\bv}',\underline{\bw}')$ is described
    in \ref{subsubsec:locmodel2} ($n=1$) and \ref{subsubsec:locmodel5}
    ($n\ge 2$). When $n=1$ (i.e., Jordan quiver),
    $\cM(\underline{\bv}',\underline{\bw}')\cong
    S^2(\CC\times\CC^\times)$ given by $(w_1,w_2,y_1,y_2)/\mathfrak
    S_2$. This is the same as the description of the Coulomb branch in
    \cite[Prop.~6.14]{main}. If $n\ge 2$, we have
    $\GL(\underline{\bv}') = \GL(2)$, $\bN' = 0$. The corresponding
    Coulomb branch, more precisely its ramified $\mathfrak S_2$-cover,
    can be identified with $\CC[w_1,w_2,y_1^\pm, y_2^\pm,\xi]/(y_1 -
    y_2 - \xi(w_1 - w_2)]$ in \ref{subsubsec:locmodel5}, where $\xi$
    is the fundamental class of the closed $\GL(2)_\cO$-orbit through
    cocharacters $w_{1}^*$ and $w_{2}^*$, which is isomorphic to
    $\proj^1$. Moreover $y_1^{-1}$, $y_2^{-1}$ are fundamental classes
    of points $-w_1^*$, $-w_2^*$.
    The computation is similar to \cite[Lem.~6.9]{main},
    hence the further detail is omitted.
\end{proof}

\subsubsection{}\label{subsubsec:proof2}
Next consider the case with parameter $\nu^\CC$. It corresponds to
gauge theories with flavor symmetries. We consider the standard action
of $T^{\bw_i}$ on $\CC^{\bw_i}$, and the induced action of
$T(\underline{\bw}) := \prod_i T^{\bw_i}$ on $\bN$. 
\begin{NB}
    Note that the original explanation was slightly wrong.
\end{NB}%
Note that the action of scalars in $T(\underline{\bw})$ is same as
that of scalars in $G$. Hence we have the induced action of
$(G\times T(\underline{\bw}))/{\CC^\times}$. We also have the
$\CC^\times$-action on the component
$\bigoplus_{i=0}^{n-1}\Hom(\CC^{\bv_i},\CC^{\bv_{i+1}})$ of $\bN$ by
scaling. The latter is defined also for $\underline{\bw} = 0$.
Let
$\tilde G = (G\times T(\underline{\bw}))/{\CC^\times} \times
\CC^\times$ when $\underline{\bw}\neq 0$, and $G\times \CC^\times$
when $\underline{\bw}=0$. We have the induced action of $\tilde G_\cO$
on $\cR$, and consider a larger equivariant Borel-Moore homology group
$H^{\tilde G_\cO}_*(\cR)$.
It is equipped with a convolution product, which is commutative. We
have a morphism $\varpi\colon \Spec H^{\tilde G_\cO}_*(\cR)\to \Spec H^*_{\tilde G}(\mathrm{pt}) \cong \AAA^{\underline{\bv}}\times \AAA^{|\underline{\bw}|}$ or $\AAA^{\underline{\bv}}\times\AAA$.
We denote additional coordinates by $\nu^\CC_{k,i}$ ($i=0,\dots, n-1$,
$k=1, \dots, \bw_i$), $\nu^\CC_*$ when $\bw\neq 0$ and $\nu^\CC_*$
when $\bw = 0$. These are identified with parameters $\nu^\CC_{k,i}$,
$\nu^\CC_*$ of bow varieties.

We define $\wA^{|\underline{\bv}|+|\underline{\bw}|}$,
$\bA^{|\underline{\bv}|+|\underline{\bw}|}$ and
$\wA^{|\underline{\bv}|+1}$, $\bA^{|\underline{\bv}|+1}$ as above, but
we consider differences $w_{i+1,l}(t) + \nu^\CC_*(t) - w_{i,k}(t)$,
$w_{i,k}(t) - w_{i,l}(t)$, $w_{i,r}(t) - \nu^\CC_{j,p}(t)$, and
conditions are
\begin{itemize}
      \item[(a)'] $w_{i+1,l}(t) - w_{i,k}(t) + \nu^\CC_*(t) = 0$ for
    some $i$, $k$, $l$, but all others are differences are
    nonzero. ($n=1$ case is included unlike undeformed case.)
    \begin{NB}
    Moreover $w_{j,r}(t) \neq \nu^\CC_{j,p}(t)$ for any $j$
    and pairs $(p,r)$ if $\dim W_j\neq 0$.
    \end{NB}%

      \item[(b)'] $w_{i,k}(t) - w_{i,l}(t) = 0$ for distinct $k$, $l$
    and some $i$, but all other differences are nonzero.
    \begin{NB}
        Moreover $w_{j,r}(t) \neq \nu^\CC_{j,p}(t)$ for any $j$ and
        pairs $(p,r)$ if $\dim W_j\neq 0$.
    \end{NB}%

  \item[(c)'] $w_{i,k}(t) - \nu^\CC_{i,p}(t) + \nu^\CC_*(t) = 0$ for some
    $i$, $k$, $p$, but all other differences like in (a)',(b)' are
    nonzero.
\end{itemize}

We define $\Xi^\circ$ over $\wA^{|\underline{\bv}|+|\underline{\bw}|}$
or $\wA^{|\underline{\bv}|+1}$ as above,
\begin{NB}
    Note that we compare eigenvalues of $B'_{i+1}$, $B_i$ for the
    $i$th triangle. Since eigenvalues of $B'_{i+1}$ are those of
    $B_{i+1}$ plus $\nu^\CC_*$, the shift $+\nu^\CC_*(t)$ appears in
    the condition (a)' appears.
\end{NB}%
and check that $\Xi^{t\circ}$ extends for
$t\in \bA^{|\underline{\bv}|+|\underline{\bw}|}\setminus
\wA^{|\underline{\bv}|+|\underline{\bw}|}$
or $\bA^{|\underline{\bv}|+1}\setminus\wA^{|\underline{\bv}|+1}$. It
is reduced to checks for local models as above. For (c)', we use
\cite[\S 4(iii)]{main}.

For (a)', let us consider $\bN^t$. It is the space for finite type
$A_2$ gauge theory with $\underline\bv =(1,1)$. Note that
endomorphisms coming back from $w_{i+1,l}^*$ to $w_{i,k}^*$ (when
$n=2$) is not fixed by $t$, as the difference is $\nu^\CC_*(t)$, not
$-\nu^\CC_*(t)$. Similary we do not have coming back endomorphisms
even for $n=1$. By the same reason, the factorization of
$\cM_{\nu_*^\CC}(\underline\bv,0)$ yields a finite type $A_2$ bow
variety. For example, only one of either $w_1 - w_2 + \nu_*^\CC$ or
$w_1 - w_2 - \nu_*^\CC$ vanishes in
\ref{subsubsec:locmodel4}. Therefore the local model which we should
use is one in \ref{subsubsec:locmodel6} with $\underline{\bw}=0$. 

For (b)', we have
$(G',\bN') = (\GL(2)\times T^{|\underline{\bv}'|}, 0)$. This is clear
for $n\ge 2$, but it is also true for $n=1$, as
$\Hom(\CC^{\bv_i},\CC^{\bv_{i+1}}) = \Hom(\CC^{\bv_0},\CC^{\bv_0})$ is
multiplied by scalar corresponding to $\nu^\CC_*$. Similarly the
factorization implies $A = 0$ in the relevant summand even for $n=1$,
hence the local model is \ref{subsubsec:locmodel5} with $\bw=0$.

The remaining argument is the same as the $\nu^\CC = 0$ case.

\subsection{Identifications of additional structures}

\subsubsection{Poisson structures}

Recall $H^{G_\cO}_*(\cR)$ carries a Poisson structure from the
noncommutative deformation $H^{G_\cO\rtimes\CC^\times}_*(\cR)$
(\cite[\S3(iv)]{main}). On the other hand,
$\cM(\underline{\bv},\underline{\bw})$ has the Poisson structure
compatible with the symplectic structure on
$\cM^{\mathrm{s}}(\underline{\bv},\underline{\bw})$, as a reduction.

\begin{Proposition}\label{prop:Poisson}
    The isomorphism $\Xi$ in \thmref{thm:main} respects the Poisson
    structure if we multiply $-1$ on the bracket for the bow variety.
\end{Proposition}

\begin{proof}
    It is enough to check that the symplectic form is respected over
    $\wA^{|\underline{\bv}|}$. Moreover we can replace
    $\cM(\underline{\bv},\underline{\bw})$ by the chainsaw quiver
    variety $\cM(\underline{\bv},0)$ thanks to
    \propref{prop:collaps-poisson}.

    We compare Poisson brackets among $\sfy_{i,k}$, $w_{j,l}$ and
    $y_{i,k}$, $w_{j,l}$.
    \begin{NB}
    $\{ \sfu_{i,k}, w_{j,l} \}$ and $\{ u_{i,k}, w_{j,l}\}$ where
    \begin{equation*}
        u_{i,k} := y_{i,k} \prod_{1\le l\le \bv_{i+1}} (w_{i,k} - w_{i+1,l})^{-1}.
    \end{equation*}
    By \eqref{eq:2} $\sfu_{i,k}$ is mapped to $u_{i,k}$.
    \end{NB}%
    We have $\{ w_{i,k}, w_{j,l} \} = 0$ by \cite[\S3(vi)]{main} for
    the Coulomb branch, and by \propref{prop:poisson-commute} for the bow
    variety.

    Let us first suppose $n\neq 1$. 
    By \cite[Cor.~5.21(2)]{main} we have
    \(
       \{ \sfu_{i,k}, w_{j,l} \} = \delta_{ij}\delta_{kl} \sfu_{i,k}
    \)
    and other brackets are $0$. Therefore
    \begin{equation*}
        \begin{gathered}[t]
            \{ \sfy_{i,k}, w_{j,l}\} = \delta_{ij} \delta_{kl} \sfy_{i,k},
            \qquad
            \{ \sfy_{i,k}, \sfy_{j,l}\} = 0 \quad\text{if $|i - j| \neq 1$},
            \\
            \{ \sfy_{i,k}, \sfy_{i+1,l}\} = 
            \frac{\sfy_{i,k} \sfy_{i+1,l}}{w_{i,k} - w_{i+1,l}}
        \end{gathered}
    \end{equation*}
    by \eqref{eq:2}. Here we assume $n \ge 3$. For $n=2$, there is an
    extra contribution as $i+2 \equiv i$ modulo $2$. We have
    \begin{equation*}
        \{ \sfy_{i,k}, \sfy_{i+1,l}\} = 
            \frac{2\sfy_{i,k} \sfy_{i+1,l}}{w_{i,k} - w_{i+1,l}}
    \end{equation*}
    in the last equality. The coefficient $0$, $1$ or $2$ for $\{
    \sfy_{i,k}, \sfy_{j,l}\}$ is uniformly written as $2\delta_{ij} -
    c_{ij}$ by the Cartan matrix $(c_{ij})$. Comparing it with
    \cite[(24),(25)]{fra}, we find this Poisson bracket is negative of
    one for the chainsaw quiver variety. (The indices $i$, $k$ are
    swapped, and our $w_{i,k}$ (resp.\ $y_{i,k}$) is $x_{k,i}$ (resp.\
    $y_{k,i}$) in \cite{fra}. Also the sign for \cite[(25)]{fra} is
    wrong, and missed in the last displayed formula in the
    proof. Compare it with \cite[Prop.~3.5]{fra}.)

    Next suppose $n=1$. By \eqref{eq:3} we have
    \begin{equation*}
        \{ \sfy_k, w_l \} = \delta_{kl} \sfy_k, \qquad
        \{ \sfy_k, \sfy_l \} = \frac{2(1-\delta_{kl})\sfy_k\sfy_l}{w_k - w_l}.
    \end{equation*}
    \begin{NB}
        \begin{equation*}
            \begin{split}
                & \{ \sfy_k, \sfy_k \} = \{ \sfu_k \prod_{m\neq k} (w_m - w_k),
           \sfu_k \prod_{n\neq k} (w_n - w_k)\} \\
           = \; &
           \{ \sfu_k \prod_{m\neq k} (w_m - w_k), \sfu_k\}
           \prod_{n\neq k} (w_n - w_k)
           + 
           \{ \sfu_k \prod_{m\neq k} (w_m - w_k), 
           \prod_{n\neq k} (w_n - w_k)\}
           \sfu_k
           \\
           =\; &
           \{ \prod_{m\neq k} (w_m - w_k), \sfu_k\} \sfy_k
           +
           \{ \sfu_k, \prod_{n\neq k} (w_n - w_k) \} \sfy_k = 0,
            \end{split}
        \end{equation*}
        (this is clear from the skew commutativity)
        \begin{equation*}
            \begin{split}
                & \{ \sfy_k, \sfy_l \} = \{ \sfu_k \prod_{m\neq k} (w_m - w_k),
           \sfu_l \prod_{n\neq l} (w_n - w_l)\} \\
           = \; &
           \{ \sfu_k \prod_{m\neq k} (w_m - w_k), \sfu_l\}
           \prod_{n\neq l} (w_n - w_l)
           + 
           \{ \sfu_k \prod_{m\neq k} (w_m - w_k), 
           \prod_{n\neq l} (w_n - w_l)\}
           \sfu_l
           \\
           =\; &
           \{\prod_{m\neq k} (w_m - w_k), \sfu_l\}
            \sfu_k \prod_{n\neq l} (w_n - w_l)
           + 
           \{ \sfu_k, 
           \prod_{n\neq l} (w_n - w_l)\}
           \sfu_l\prod_{m\neq k} (w_m - w_k)
           \\
           =\; &
           - \frac{\sfy_k \sfy_l}{w_l - w_k}
           + \frac{\sfy_k \sfy_l}{w_k - w_l}
           = \frac{2\sfy_k \sfy_l}{w_k - w_l}.
            \end{split}
        \end{equation*}
    \end{NB}%
    
    On the other hand, $\cM(\bv,0)$ is the reduction of
    $\widetilde\cM$ in \subsecref{subsec:triangle} (triangle) by the
    diagonal $\GL(V) \subset \GL(V)\times \GL(V)$, where we take $V_1
    = V_2 = V$. We calculate the Poisson bracket as in
    \cite[(24),(25)]{fra} to deduce the same assertion. We have $2$ in
    the coefficient, as both $p_{k,i}$ and $q_{k,i}$ contribute in
    contrast to the calculation in \cite{fra}.
\end{proof}

\subsubsection{Hamiltonian torus action}\label{subsubsec:hamilt-torus-acti}

Recall $\cR$ has connected components parametrized by $\pi_1(\Gr_G)$,
and hence $H^{G_\cO}_*(\cR)$ is graded by $\pi_1(G)$
(\cite[\S3(v)]{main}). Therefore $\cM_C$ has an action of
$\pi_1(G)^\vee$, where $\pi_1(G)^\vee$ is the Pontryagin dual of
$\pi_1(G)$. Since $G = \prod_i \GL(\bv_i)$, $\pi_1(G)^\vee =
(\ZZ^n)^\vee$ is the $n$-dimensional torus $(\CC^\times)^n$. Let us
denote coordinates as $(s_0,s_1,\dots,s_{n-1})$ so that the $i$-th
vertex gives $s_i$.

On the other hand, the bow variety
$\cM(\underline{\bv},\underline{\bw})$ has the action of the
$n$-dimensional torus by scalar multiplication on $\WW_\xl$ for each
$\xl\in\Lambda$. Using the numbering on $\xl$, we have coordinates
$(t_0,t_1,\dots,t_{n-1})$ on the torus.
Note that the action of the $1$-dimensional diagonal subgroup
$(t,t,\dots,t)$ can be absorbed to $\prod\GL(V_\zeta)$, hence is
trivial on the quotient space
$\cM(\underline{\bv},\underline{\bw})$. 
\begin{NB}
    We have $y_{i,k} \mapsto t_{i+1} t_i^{-1} y_{i,k}$.
\end{NB}%

We have an extra $\CC^\times$-action given by $(A,b) \mapsto (t_\delta
A, t_\delta b)$ at the vertex $0$, and other entries are
unchanged. This action can be absorbed to $\prod \GL(V_\zeta)$ if some
$V_\zeta = 0$, but cannot in general. It corresponds to the central
extension of the affine Lie algebra.

Let us choose an isomorphism $\{
(t_0,\dots,t_{n-1},t_\delta) \}/\{ (t,t,\dots,t,1)\} \cong
(\CC^\times)^n$ by $s_i = t_{i+1} t_i^{-1}$ for $i\neq n-1$, $s_{n-1} = t_0 t_{n-1}^{-1} t_\delta$.
\begin{NB}
    The inverse is $t_0 = 1$, $t_1 = s_0$, \dots, $t_{n-1} = s_0
    s_1\dots s_{n-2}$, $t_\delta = s_0 s_1 \dots s_{n-1}$.
\end{NB}%

\begin{Proposition}
    The isomorphism $\Xi$ in \thmref{thm:main} respects the
    $(\CC^\times)^n$-action.
\end{Proposition}

\begin{proof}
    It is enough to check the assertion over $\wA^{|\underline{\bv}|}$. 

    Let us take $s = (s_i)\in(\CC^\times)^n$. Since $w_{i,k}$ is an
    equivariant variable, $s\cdot w_{i,k} = w_{i,k}$ in the Coulomb
    branch. On the other hand, $w_{i,k}$ is an eigenvalue of $B_i$, we
    also have $s\cdot w_{i,k} = w_{i,k}$ in the bow variety. 

    Since $\sfy_{i,k}$ is the fundamental class of
    $\pi^{-1}(w_{i,k}^*)$, $s\cdot \sfy_{i,k} = s_i \sfy_{i,k}$. On
    the other hand, by the definition of $y_{i,k}$ we have $t\cdot
    y_{i,k} = t_{i+1} t_i^{-1} y_{i,k}$ if $i\neq n-1$ and $t\cdot
    y_{n-1,k} = t_{0} t_{n-1}^{-1} t_\delta y_{n-1,k}$. Now the
    assertion is checked.
\end{proof}

Note that the $(\CC^\times)^n$-action preserves the symplectic form,
which can be checked easily in either point of view.

\subsubsection{\texorpdfstring{$\CC^\times$}{C\^*}-action}\label{subsubsec:homological-degree}

Let $\deg_h$ denote the cohomological degree for the Coulomb branch
$H^{G_\cO}_*(\cR) = \CC[\cM]$. Since $w_{i,k}$ is a class in
$H^2_{\GL(\bv_i)}(\mathrm{pt})$, we have $\deg_h w_{i,k} = 2$. On the
other hand, $\sfy_{i,k}$ is the fundamental class of
$\pi^{-1}(w_{i,k}^*)$. This fiber is an infinite dimensional vector
space, but its degree is defined as a codimension in relative to
another infinite dimensional vector space, the fiber of $\mathcal T$
in \cite[\S2]{main}. The codimension is given by the formula in \cite[Lem.~2.2]{main}. We have
\begin{equation*}
    \deg_h \sfy_{i,k} = 2d_{w_{i,k}^*} =
    2 \sum_{\chi} \max(-\langle\chi,w_{i,k}^*\rangle,0) \dim\bN(\chi)
    = 2 \bv_{i+1},
\end{equation*}
where $\bN(\chi)$ is the weight space of $\bN$ with weight $\chi$. The
cohomological degree must be corrected in the monopole formula, as noted in \cite[Rem.~2.8(2)]{main}. The correction term is
\begin{equation*}
    \sum \langle \chi,w_{i,k}^*\rangle \dim\bN(\chi)
    = -\bv_{i+1} + \bv_{i-1} + \bw_i.
\end{equation*}
Let $\deg_m$ denote the corrected degree in the monopole formula. We
thus have
\begin{equation*}
    \deg_m w_{i,k} = 2, \quad \deg_m \sfy_{i,k} = \bv_{i+1} + \bv_{i-1} + \bw_i.
\end{equation*}
We have the corresponding $\CC^\times$-action on $\cM$ so that
$\deg_m$ gives weight.

A general expectation is that the corresponding $S^1$-action, the
restriction of the $\CC^\times$-action, extends to an $\SU(2)$-action
on $\cM$ which rotates the complex structures. (See
\cite[Rem.~2.8(2)]{main}.)

Let us construct an $\SU(2)$-action on $\cM$. We use the original
definition of $\cM$ in \subsecref{subsec:original}. For
$(B^{LR},B^{RL})$, we define the quaternion module structure by
$J(B^{LR},B^{RL}) = (- (B^{RL})^\dagger, (B^{LR})^\dagger)$, and
consider the induced $\SU(2) = \grpSp(1)$-action.
Here ${}^\dagger$ denote the hermitian adjoint. The $S^1$-action has
weight $1$ on both $B^{LR}$, $B^{RL}$. The same construction applies
to $I$, $J$.

Let $(\nabla,T_1,T_2,T_3)$ be a solution of Nahm's equations on an
interval with a $\xl$-point in the middle, corresponding to the
triangle part. We consider $T := T_1 i + T_2 j + T_3 k$ as an
imaginary quaternion valued function, and consider the adjoint
$\grpSp(1)$-action $q T q^{-1}$ ($q\in\grpSp(1)$). This action does
not preserve the condition (\ref{item:matching}) in
\subsecref{subsec:original} as $(\rho_1,\rho_2,\rho_3)$ is also
changed by the adjoint action. Therefore we compose the gauge
transformation which is $\rho(q)^{-1}$ on $(E_{\zeta^\pm}|_\xl)^\perp$
at $\xl$. Here $\rho$ is a homomorphism $\grpSp(1)\to
\mathrm{U}(\Delta R(\xl))$ associated with $(\rho_1,\rho_2,\rho_3)$.

Let us restrict to the $S^1$-action. The first adjoint action is
restricted to weight $2$-action on $T_2 + T_3 i$. Since $\eta$ in
\propref{prop:triangle-Nahm} is the value of $T_2 + T_3 i$ at a
certain point, it has weight $2$. On the other hand, the second action
is restricted to the conjugation by
\begin{equation*}
    \operatorname{diag}(\underbrace{1,\dots,1}_{m},t^{n-m-1},t^{n-m-3},
    \dots,t^{1-n+m}) \qquad t\in S^1.
\end{equation*}
\begin{NB}
    Note
    \begin{equation*}
        \left(
\begin{array}{ccccc}
   t^{n'-1}  &   & &\bigzero & \\
  & t^{n'-3}  & & & \\
 & & \raisebox{5pt}{$\ddots$}  & & \\
 \raisebox{10pt}{\bigzero} &&&& t^{1 - n'}
\end{array}
\right)
        \left(
\begin{array}{ccccc}
  0  &   & && \\
  1  &   & & \bigzero & \\
  & & \raisebox{5pt}{$\ddots$}  & & \\
  & \raisebox{10pt}{\bigzero} &&1& 0  
\end{array}
\right)
\left(
\begin{array}{ccccc}
   t^{1-n'}  &   & &\bigzero & \\
  & t^{3-n'}  & & & \\
 & & \raisebox{5pt}{$\ddots$}  & & \\
 \raisebox{10pt}{\bigzero} &&&& t^{n'-1}
\end{array}
\right) 
= t^{-2} \left(
\begin{array}{ccccc}
  0  &   & && \\
  1  &   & & \bigzero & \\
  & & \raisebox{5pt}{$\ddots$}  & & \\
  & \raisebox{10pt}{\bigzero} &&1& 0  
\end{array}
\right)
    \end{equation*}
with $n' = n-m$.
\end{NB}%
Therefore in sum, weights on $(A,B_1,B_2,a,b)$ of the triangle part
are
\begin{equation*}
    \operatorname{wt}(A,B_1,B_2,a,b) =
    (0,2,2,1+\bv_1 - \bv_2,1+\bv_2 - \bv_1).
\end{equation*}

From the definition of $w_{i,k}$, $y_{i,k}$, we have
\begin{equation*}
    \operatorname{wt} w_{i,k} = 2, \quad
    \operatorname{wt} y_{i,k} = \bv_{i-1} + \bv_{i+1} + \bw_i.
\end{equation*}
Thus
\begin{Proposition}
    The $\CC^\times$-action on $\cM$ given by the corrected degree is
    coming from the $\SU(2)$-action defined as above.
\end{Proposition}


\section{Hanany-Witten transition and its applications}
\label{sec:hanany-witten}
As explained in \secref{sec:bow-varieties}, a bow diagram comes from a
brane configuration in type IIB string theory \cite{MR1451054}. It was
observed that a new $3$-brane is created when a $D5-$ and an
$NS5$-brane pass through each other. This is called Hanany-Witten
transition. We interpret it as an isomorphism of two bow varieties and
study its applications in this section.


\subsection{\texorpdfstring{$3d$}{3d} Mirror symmetry}\label{subsec:HC_corresp}

Before explaing Hanany-Witten transition, we review the explanation of
the $3d$ mirror symmetry via $\SL(2,\ZZ)$-symmetry of the string
theory, as a warming up example.

In order to obtain a bow diagram from a brane configuration, it is
enough to replace a Dirichlet 5-brane with $\boldsymbol\times$, and a
NS 5-brane with $\boldsymbol\medcirc$.  And by exchanging NS 5-branes
and Dirichlet 5-branes in a brane configuration, we get the dual
theory $B$ of the original theory $A$, which is given by the original
brane configuration.  In physics, it is known that it gives $3d$
mirror symmetry, hence the Coulomb branch of theory $A$ coincides with
the Higgs branch of theory $B$ as holomorphic symplectic manifolds:
\begin{align*}
\mathcal{M}^A_C \cong \mathcal{M}^B_H, \quad
\mathcal{M}^A_H \cong \mathcal{M}^B_C.
\end{align*}
By using these relations, we can easily get the Coulomb branch from the Higgs branch of the same theory:
\begin{align*}
\begin{xy}
(0,0)*{\text{theory A}},
(50,0)*{\mathcal{M}^A_C},
(0,-15)*{\text{theory B}},
(55,-15)*{\mathcal{M}^B_C = \mathcal{M}^A_H},
(25,6)*{\text{D5}\mapsto \boldsymbol\times},
(25,2)*{\text{NS5}\mapsto \boldsymbol\medcirc},
(25,-9)*{\text{D5}\mapsto \boldsymbol\times},
(25,-13)*{\text{NS5}\mapsto \boldsymbol\medcirc},
(-10,-7.5)*{\text{NS5}\leftrightarrow\text{D5}},
(60,-7.5)*{\boldsymbol\medcirc\leftrightarrow\boldsymbol\times},
\ar (10,0);(45,0)
\ar @{<->} (0,-4);(0,-11)
\ar @{<->} (50,-4);(50,-11)
\ar (10,-15);(45,-15)
\end{xy}
\end{align*}

In the quiver gauge theory, a Higgs branch is given as a quiver variety $\fM_0(\underline{\bv}, \underline{\bw})$, so we can get the corresponding Coulomb branch as the bow variety which is given by the following procedure:\\
(i) Rewrite a quiver variety as a bow variety associated with a bow diagram with cobalanced dimension vector (see Theorem \ref{thm:balanced_bow}).\\
(ii) Exchange $\boldsymbol\medcirc$ and $\boldsymbol\times$.\\
(iii) Construct a bow variety from the new bow diagram.\\
\begin{align*}
\begin{xy}
(10,0)*{\text{Higgs branch }\fM_0(\underline{\bv},\underline{\bw}):},
(45,3)*{\bv_0},
(50,0)*{\boldsymbol\medcirc},
(55,3)*{\bv_1},
(60,0)*{\boldsymbol\times},
(65,3)*{\cdots},
(70,0)*{\boldsymbol\times},
(75,3)*{\bv_1},
(80,0)*{\boldsymbol\medcirc},
(85,3)*{\bv_2},
(90,0)*{\boldsymbol\times},
(95,3)*{\cdots},
(100,0)*{\boldsymbol\times},
(105,3)*{\bv_2},
(110,0)*{\boldsymbol\medcirc},
(115,3)*{\bv_3},
(65,-5)*{\underbrace{\hspace{15mm}}_{\text{$\bw_1$}}},
(95,-5)*{\underbrace{\hspace{15mm}}_{\text{$\bw_2$}}},
(10,-20)*{\text{Coulomb branch }\cM(\underline{\bv},\underline{\bw}):},
(45,-17)*{\bv_0},
(50,-20)*{\boldsymbol\times},
(55,-17)*{\bv_1},
(60,-20)*{\boldsymbol\medcirc},
(65,-17)*{\cdots},
(70,-20)*{\boldsymbol\medcirc},
(75,-17)*{\bv_1},
(80,-20)*{\boldsymbol\times},
(85,-17)*{\bv_2},
(90,-20)*{\boldsymbol\medcirc},
(95,-17)*{\cdots},
(100,-20)*{\boldsymbol\medcirc},
(105,-17)*{\bv_2},
(110,-20)*{\boldsymbol\times},
(115,-17)*{\bv_3},
(65,-25)*{\underbrace{\hspace{15mm}}_{\text{$\bw_1$}}},
(95,-25)*{\underbrace{\hspace{15mm}}_{\text{$\bw_2$}}},
\ar @{.} (40,0);(45,0)
\ar @{.} (115,0);(120,0)
\ar @{-} (45,0);(115,0)
\ar @{.} (40,-20);(45,-20)
\ar @{.} (115,-20);(120,-20)
\ar @{-} (45,-20);(115,-20)
\end{xy}
\end{align*}
In fact, the Coulomb branch defined in \S \ref{sec:Coulomb} is the same one obtained by this procedure.

\subsection{Hanany-Witten transition as an isomorphism of bow varieties}
In a brane configuration, exchange of positions of a D5-brane and its next NS5-brane generates new D3-branes \cite{MR1451054}.
\begin{align*}
\begin{xy}
(10,-14)*{z},
(20,-14)*{t},
(40,-14)*{(z<t)},
(60,-14)*{t},
(70,-14)*{z},
(90,-14)*{(t<z)},
\ar @{.} (-5,0);(0,0)
\ar @{-} (0,0);(10,0)^{\bv_0}
\ar @{.} (10,10);(10,-10)
\ar @{-} (10,0);(20,0)^{\bv}
\ar @{-} (20,10);(20,-10)
\ar @{-} (20,0);(30,0)^{\bv_1}
\ar @{.} (30,0);(35,0)
\ar @{.} (45,0);(50,0)
\ar @{-} (50,0);(60,0)^{\bv_0}
\ar @{-} (60,10);(60,-10)
\ar @{-} (60,0);(70,0)^{\bv'}
\ar @{.} (70,10);(70,-10)
\ar @{-} (70,0);(80,0)^{\bv_1}
\ar @{.} (80,0);(85,0)
\end{xy}
\end{align*}
\begin{NB}
\begin{align*}
\begin{xy}
(10,-14)*{z},
(20,-14)*{t},
(60,-14)*{z=t},
(100,-14)*{t},
(110,-14)*{z},
\ar @{-} (0,0);(10,0)^{v_0}
\ar @{.} (10,10);(10,-10)
\ar (8,11);(12,11)
\ar @{-} (10,0);(20,0)^{v}
\ar @{-} (20,10);(20,-10)
\ar @{-} (20,0);(30,0)^{v_1}
\ar (35,0);(45,0)^{s_v}_{z\rightarrow t - 0}
\ar @{-} (50,0);(60,0)^{v_0}
\ar @{.} (59.5,10);(59.5,-10)
\ar @{-} (60.5,10);(60.5,-10)
\ar @{-} (60,0);(70,0)^{v_1}
\ar (85,0);(75,0)_{s'_{v'}}^{t + 0 \leftarrow z}
\ar @{-} (90,0);(100,0)^{v_0}
\ar @{-} (100,10);(100,-10)
\ar @{-} (100,0);(110,0)^{v'}
\ar @{.} (110,10);(110,-10)
\ar (112,11);(108,11)
\ar @{-} (110,0);(120,0)^{v_1}
\end{xy}
\end{align*}
\end{NB}
We interpret this transition in a language of bow variety.

\begin{Proposition}[Hanany-Witten transition]
Set $m + m'=l+n+1 \ (m, m', l, n \geq 0)$.
Then there exists a $\GL(l)\times \GL(n)$-equivariant isomorphism, preserving holomorphic symplectic structures;
\begin{align*}
\begin{xy}
(10,3)*{l},
(15,0)*{\boldsymbol\medcirc},
(20,3)*{m},
(25,0)*{\boldsymbol\times},
(30,3)*{n},
(40,0)*{\cong},
(50,3)*{l},
(55,0)*{\boldsymbol\times},
(60,3)*{m'},
(65,0)*{\boldsymbol\medcirc},
(70,3)*{n},
\ar @{.} (5,0);(10,0)
\ar @{-} (10,0);(30,0)
\ar @{.} (30,0);(35,0)
\ar @{.} (45,0);(50,0)
\ar @{-} (50,0);(70,0)
\ar @{.} (70,0);(75,0)
\end{xy}
\end{align*}
Here quotients are taken at only $m$ and $m'$.
Moreover both $\nu^\RR$-stability and $\nu^\RR$-semistability
conditions are preserved under the isomorphism.
\label{prop:HW-trans}
\end{Proposition}

By abuse of terminology, the corresponding transition of bow diagrams
is also called Hanany-Witten transition. (Recall dimension vectors are
parts of bow diagrams.)

\begin{proof}
Consider the following part of bow data:
\begin{equation*}
\xymatrix@C=1.2em{ V_1 \ar@(ur,ul)_{B_1} \ar@<-.5ex>[rr]_{C} 
  && V_2 \ar@(ur,ul)_{B_2}
  \ar@<-.5ex>[ll]_{D} \ar[rr]^{A} \ar[dr]_{b} && V_3 \ar@(ur,ul)_{B_3}
  \\
  &&& \CC \ar[ur]_a &}
\qquad
\begin{aligned}[m]
    & DC + B_1 = 0,
    \begin{NB}
        B_1 = DC, 
    \end{NB}%
    \quad CD + B_2 = \nu^\CC,
    \begin{NB}
        B_2 = CD
    \end{NB}%
    \\
    & B_3A - AB_2 + ab = 0.
\end{aligned}
\end{equation*}
Like Proposition \ref{prop:character_cond2}, we consider a three term complex
\begin{equation*}
    \begin{CD}
        V_2 @>\alpha = \left[
          \begin{smallmatrix}
              D \\ A \\ b
          \end{smallmatrix}\right]>> 
        V_1\oplus V_3 \oplus \CC @>{\beta = \left[
          \begin{smallmatrix}
              AC & (B_3 - \nu^\CC)& a 
          \end{smallmatrix}\right]}>>
        V_3,
    \end{CD}
\end{equation*}
where $\beta\alpha = 0$ and $\alpha$ is injective.
\begin{NB}
    $ACD + (B_3 -\nu^\CC)A + ab = A(\nu^\CC - B_2) + (B_3 - \nu^\CC) A + ab = 
    B_3 A - A B_2 + ab = 0$.
\end{NB}%
Let 
\begin{align*}
  V_2^\tn := \Coker\alpha.
\end{align*}
We define new bow data by
\begin{equation*}
\xymatrix@C=1.2em{ V_1 \ar@(ur,ul)_{B_1} \ar[rr]^{A^\tn} \ar[dr]_{b^\tn} &&
  V_2^\tn
  \ar@(ur,ul)_{B_2^\tn} \ar@<-.5ex>[rr]_{C^\tn} && V_3 \ar@(ur,ul)_{B_3}
  \ar@<-.5ex>[ll]_{D^\tn}
  \\
  & \CC \ar[ur]_{a^\tn} &&&}
\end{equation*}
where $A^\tn$, $a^\tn$ are composition of inclusions of $V_1$, $\CC$,
to $V_1\oplus V_3\oplus\CC$ and the projection $V_1\oplus V_3\oplus\CC\twoheadrightarrow V_2^\tn$, and $D^\tn$ is the corresponding
composition for $V_3$, but we change the sign.
The remaining are as follows: $b^\tn = bC$, $C^\tn$ is a homomorphism
induced from $\beta$, and $B_2^\tn = - D^\tn C^\tn$.

Let us check the defining equations. Recall $-D^\tn$ is the composition
of $V_3\to V_1\oplus V_3\oplus \CC\twoheadrightarrow V_2^\tn$. Since
$C^\tn$ is induced from $\beta$, $-C^\tn D^\tn$ is nothing but the
composition of the inclusion $V_3\to V_1\oplus V_3\oplus\CC$ and
$\beta$, which is equal to $B_3 - \nu^\CC$.

Next consider
\begin{equation*}
    \begin{CD}
        V_1 @>\alpha^\tn =\left[
          \begin{smallmatrix}
              -B_1 \\ C^\tn A^\tn \\ b^\tn
          \end{smallmatrix}\right]>>
        V_1 \oplus V_3\oplus \CC @>{\beta^\tn
        = \left[
          \begin{smallmatrix}
              A^\tn & -D^\tn & a^\tn
          \end{smallmatrix}\right]}>> V_2^\tn,
    \end{CD}
\end{equation*}
where $\beta^\tn$ is nothing but the natural projection. 
\begin{NB}
    $- A^\tn B_1 - D^\tn C^\tn A^\tn + a^\tn b^\tn
    = B_2^\tn A^\tn - A^\tn B_1 + a^\tn b^\tn = 0$.
\end{NB}%
We have
\begin{equation}\label{eq:1}
    \alpha^\tn =
    \begin{bmatrix}
        -B_1 \\ AC \\ bC
    \end{bmatrix}
    \begin{NB}
    =
    \begin{bmatrix}
        D \\ A \\ b
    \end{bmatrix}
    C         
    \end{NB}%
    = \alpha C.
\end{equation}
Therefore $\beta^\tn\alpha^\tn = 0$, which is nothing but the remaining defining equation.

\begin{NB}
    In summary, we have
    \begin{equation*}
        \left[V_1\xrightarrow{\alpha^\tn} V_1\oplus V_3\oplus\CC
        \xrightarrow{\beta^\tn} V_2^\tn\right]
      = \left[V_1\xrightarrow{\alpha C} V_1\oplus V_3\oplus\CC
        \xrightarrow{\text{projection}} \Coker\alpha\right].
    \end{equation*}
\end{NB}%

Let us check the condition (S1). Take a subspace $S\subset V_1$ such
that $B_1(S)\subset S$, $A^\tn(S) = 0 = b^\tn(S)$. Observe that
$A^\tn(S) = 0$ means $S \oplus 0\oplus 0\subset \Ima\alpha$. Let us
consider $\tilde S = \alpha^{-1}(S \oplus 0\oplus 0)$. Then $D(\tilde S) = S$ and $A(\tilde
S) = 0 = b(\tilde S)$. Therefore 
\begin{equation*}
    \alpha (\nu^\CC - B_2)(\tilde S) = \alpha CD(\tilde S) = \alpha C(S) 
    = \alpha^\tn (S) =
    \begin{bmatrix}
        B_1(S) \\ 0 \\ 0
    \end{bmatrix}.
\end{equation*}
The condition $B_1(S)\subset S$ implies $B_2(\tilde S)\subset \tilde S$. Hence $\tilde S = 0$ thanks to (S1) for the original data. We have $S=0$ as well.

Let us check the condition (S2). Suppose we have a subspace $T\subset
V_2^\tn$ such that $B_2^\tn(T)\subset T$, $\Ima A^\tn + \Ima a^\tn
\subset T$. We take its inverse image $\tilde T = (\beta^\tn)^{-1}(T)$
in $V_1\oplus V_3\oplus \CC$. By the second assumption, it contains
$V_1\oplus \{0\}\oplus \CC$. Hence $\tilde T$ is a form of $V_1 \oplus \bar
T\oplus\CC$ for $\bar T\subset V_3$.
We also have $\Ima\alpha\subset\tilde T$.
\begin{NB}
    as $\beta^\tn\alpha = 0$ by the definition of $\beta^\tn$.
\end{NB}%
Hence $A(V_2)\subset \bar T$.  As $B_2^\tn = -D^\tn C^\tn$, 
\begin{NB}
    and $D^\tn C^\tn$ is the composite
    $\Coker\alpha\xrightarrow{\beta} V_3 \xrightarrow{
      \begin{bmatrix}
          0\\ \id \\ 0
      \end{bmatrix}
      } V_1\oplus V_3\oplus \CC \xrightarrow{\beta^\tn} \Coker\alpha$.
\end{NB}%
the condition $B_2^\tn(T)\subset T$ implies $0 \oplus \beta(\tilde
T)\oplus 0\subset \tilde T$, i.e., $AC(V_1)+B_3(\bar T)+a(\CC)\subset
\bar T$.
Hence $\bar T = V_3$ thanks to (S2) for the original data. We have
$T=V_2^\tn$ as well.

The inverse construction is clear. The original vector space $V_2$ is
recovered from the new data as $\Ker\beta^\tn$. Note also $\beta^\tn$ is surjective thanks to (S2).
Then $A$, $b$, $D$ are given as restrictions of projections $V_1\oplus
V_3\oplus\CC$ to $V_3$, $\CC$, $V_1$ to $\Ker\beta^\tn$ (up to sign),
$a$ is $C^\tn a^\tn$, and $C$ is $\alpha^\tn$ by \eqref{eq:1}. Finally
we set $B_2 = \nu^\CC - CD$. The conditions (S1),(S2) for
$(A,B_2,B_3,a,b)$ follow from the conditions (S1),(S2) for
$(A^\tn,B_1,B_2^\tn,a^\tn,b^\tn)$. We leave the details to the reader
as an exercise.
\begin{NB}
    Suppose we have a subspace $S\subset V_2 = \Ker\beta^\tn$ such
    that $B_2(S)\subset S$ and $A(S) = 0 = b(S)$. Since $A$, $b$ are
    projections of $\Ker\beta^\tn\to V_3$, $\CC$ respectively, the
    second condition means $S\subset V_1$. Therefore $\beta^\tn(S) =
    0$ means $A^\tn(S) = 0$. Since $\nu^\CC - B_2 = CD =
    \alpha^\tn\circ (\Ker\beta^\tn\to V_1)$, $B_2(S)\subset S$ implies
    $B_1(S)\subset S$, $b^\tn(S) = 0$. Therefore the condition (S1)
    for the `new' data implies $S=0$.

    Suppose we have a subspace $T\subset V_3$ such that $B_3(T)\subset
    T$, $\Ima A + \Ima a \subset T$. The second condition means $\Ima
    (\Ker\beta^\tn\to V_3) + \Ima C^\tn a^\tn \subset T$.
    We consider $\tilde T := \beta^\tn(V_1\oplus T\oplus \CC)$. It
    contains $\Ima A^\tn + \Ima a^\tn$. Let us check $B_2^\tn(\tilde
    T) \subset \tilde T$. Note $B_2^\tn\beta^\tn(V_1\oplus T\oplus\CC)
    = B_2^\tn A^\tn(V_1) + B_2^\tn D^\tn(T) + B_2^\tn a^\tn(\CC)$. We
    have $B_2^\tn A^\tn = A^\tn B_1 - a^\tn b^\tn$. Therefore $B_2^\tn
    A^\tn(V_1) \subset \Ima A^\tn + \Ima a^\tn \subset \widetilde
    T$. We also have 
    \[ B_2^\tn D^\tn(T) + B_2^\tn a^\tn(\CC)
    = D^\tn C^\tn D^\tn(T) + D^\tn C^\tn a^\tn(\CC)
    = D^\tn (B_3(T) + \Ima C^\tn a^\tn) \subset D^\tn T
    \]
    By the definition of $\beta^\tn$, we have $D^\tn T\subset \tilde
    T$. Thus $\tilde T = V_2^\tn$ by (S2) for the `new' data.

    Take $v_3\in V_3$. Then there exists $v_1\oplus v_3'\oplus \lambda
    \in V_1\oplus T\oplus\CC$ such that $\beta(v_1\oplus
    v_3'\oplus\lambda) = \beta(0\oplus v_3 \oplus 0)$. Therefore
    $v_1\oplus (v_3'-v_3)\oplus \lambda\in\Ker\beta^\tn$. Since
    $T\supset \Ima (\Ker\beta^\tn\to V_3)$, we have $v_3' - v_3\in
    T$. Therefore $v_3\in T$ as well. It means $T = V_3$.
\end{NB}%

Let us check the stability condition. Note that the isomorphism
respects the group action by $\GL(V_1)\times \GL(V_3)$. As we can
choose any segment $\zeta$ from a wavy line $\xp$ in the numerical
criterion, we could also choose any $\zeta$ in the definition of
$\chi$ for the stability condition. In particular, if we choose
$\GL(V_1)$ and $\GL(V_3)$ for $\chi$, it is clear that the stability
condition is unchanged under the transition.

The proof of the compatibility of symplectic forms will be given in
\secref{sec:appendix1}.
\end{proof}

\begin{NB}
  Suppose $\dim V_2 = \dim V_3$. In this case, the triangle part is
  $T^*\GL(V_2)\times T^*V_3$, hence the quotient is
  $T^* \Hom(V_1,V_3)\times T^*V_3$ we have coordinates
  $\overline{C} \defeq AC$, $\overline{D} \defeq DA^{-1}$, $a$ and
  $\overline{b} \defeq bA^{-1}$. On the other hand,
  \(
    V_1^\tn\oplus \CC\xrightarrow{
    \left[\begin{smallmatrix}
      A^\tn & a^\tn
    \end{smallmatrix}\right]}
  V_2^\tn
  \)
  is an isomorphism, and we have coordinates
  $\overline{C}^\tn \defeq C^\tn \left[\begin{smallmatrix}
      A^\tn & a^\tn
    \end{smallmatrix}\right]$ and
  $\overline{D}^\tn \defeq \left[\begin{smallmatrix}
      A^\tn & a^\tn
    \end{smallmatrix}\right]^{-1} D^\tn$.

  Looking at definitions of new maps, we find
  \begin{equation*}
    \overline{C}^\tn = \left[\begin{matrix}
      \overline{C} & a
    \end{matrix}\right], \qquad
    \overline{D}^\tn = \left[\begin{matrix}
      \overline{D} \\ \overline{b}
    \end{matrix}\right].
\end{equation*}
This coincides with the isomorphism in \secref{sec:appendix1}.

Let us compute Poisson brackets:
\begin{equation*}
  \{ \overline{C}_{ij}, \overline{D}_{kl}\} = -\delta_{il} \delta_{jk}, \qquad
  \{ \overline{C}^\tn_{ij}, \overline{D}^\tn_{kl}\} = -\delta_{il} \delta_{jk},
  \qquad
  \{ a_i, \overline{b}_j\} = \{ a_i, b_m A^{jm}\}
  = - A_{mi} A^{jm} = -\delta_{ij}.
\end{equation*}
Since $a_i$, $\overline{b}_j$ are part of $\overline{C}^\tn$,
$\overline{D}^\tn$, Poisson brackets are compatible. The problem was
that the symplectic forms $\tr(dI\wedge dJ)$ and $-\tr(dC\wedge dD)$
have different sign, but it is fixed now.
\end{NB}%

\begin{Remark}
    Let us sketch another proof of the last step checking $(\nu 1)$,
    $(\nu 2)$ directly.

Let $S := \bigoplus S_i\subset \bigoplus V_i$ be a graded subspace
invariant under $A$, $B$, $C$, $D$ with $b(S)=0$ such that
$A|_{S_2}\colon S_2\to S_3$ is an isomorphism.
We define $S_2^\tn := A^\tn(S_1)$, i.e., the image of $S_1\oplus 0
\oplus 0$ under the projection $V_1\oplus
V_3\oplus\CC\xrightarrow{\beta^\tn} V_2^\tn$.
One can check that
    $S_1\oplus S_2^\tn\oplus S_3$ satisfies the conditions in $(\nu 1)$.

\begin{NB}
\begin{proof}
    We can check $A^\tn(S_1) \subset S_2^\tn$, $b^\tn(S_1) = 0$,
    $C^\tn(S_2^\tn) \subset S_3$ directly. Let us check
    $B_2^\tn(S_2^\tn)\subset S_2^\tn$. From the definition,
    $B_2^\tn(S_2^\tn)$ is the image of $0\oplus AC(S_1)\oplus 0$ under
    the projection $V_1\oplus V_3\oplus\CC\xrightarrow{\beta^\tn}
    V_2^\tn$. Let $s\in S_1$. Then $0\oplus AC(s)\oplus 0 \equiv
    -DC(s)\oplus 0\oplus 0$ modulo $\Ima\alpha$. Since $DC(s)\in S_1$,
    the assertion follows. Next let us check $D^\tn(S_3)\subset
    S_2^\tn$. Since $S_3 = A(S_2)$, $D^\tn(S_3)$ is the image of
    $0\oplus A(S_2)\oplus 0$ under the projection $V_1\oplus
    V_3\oplus\CC\xrightarrow{\beta^\tn} V_2^\tn$. As above, we change
    representatives to the image of $D(S_2)\oplus 0\oplus 0$. The
    assertion follows. Let us check that the restriction of $A^\tn$ to
    $S_1$ gives an isomorphism $S_1\xrightarrow{\cong} S_2^\tn$. By
    the definition, it is surjective. We also have $\dim S_1 \le \dim
    S_2^\tn$ by \propref{prop:character_cond1}. Hence the map is an
    isomorphism.
\begin{NB2}
This is nothing but $\Ima \alpha\cap (S_1\oplus 0\oplus 0) =
0$. Consider $S' := \alpha^{-1}(S_1\oplus 0\oplus 0)$. It is a
subspace of $V_2$ such that $A(S') = 0 = b(S')$. Moreover
$(B_2-\nu^\CC)(S') = - CD(S') \subset S_2$. Hence $D(B_2-\nu^\CC)(S')
\subset S_1$, $b(B_2-\nu^\CC)(S') = 0$, and also $A(B_2-\nu^\CC)(S') =
(B_3A + ab)(S') = 0$. Therefore $\alpha(B_2-\nu^\CC)(S')\subset
S_1\oplus 0\oplus 0$, i.e., $(B_2-\nu^\CC)(S')\subset S'$. We have
$S'=0$ by (S1). Thus the claim is proved.
\end{NB2}%
\end{proof}
\end{NB}%

Conversely suppose a graded subspace $S^\tn := S_1\oplus S^\tn_2\oplus
S_3\subset V_1\oplus V^\tn_2\oplus V_3$ is given. We assume $S^\tn$ is
invariant under $A^\tn$, $B^\tn$, $C^\tn$, $D^\tn$ with $b^\tn(S^\tn)
= 0$ and that $A^\tn|_{S_1}\colon S_1\to S_2^\tn$ is an
isomorphism. We define $S_2\subset V_1\oplus V_3\oplus \CC$ as 
\[
   \{ (A^\tn|_{S_1})^{-1} D^\tn(s)\oplus s\oplus 0 \mid s\in S_3\}.
\]
It is contained in $\Ker\beta^\tn$, hence regarded as a subspace of
$V_2$.
One can check that
    $S_1\oplus S_2\oplus S_3$ satisfies the conditions in $(\nu 1)$.
\begin{NB}
    \begin{proof}
        By construction, we have $A_2|_{S_2}\colon S_2\to S_3$ is an
        isomorphism.
        As $b$ (resp.\ $D$) is the projection to $\CC$ (resp.\ $V_1$),
        $b(S_2) = 0$ (resp.\ $D(S_2)\subset S_1$).
        Let $s\in S_1$. As $C=-\alpha^\tn$, $C(s) = 
        (-B_1s)\oplus C^\tn A^\tn(s)\oplus 0$. Since
        $\beta^\tn\alpha^\tn = 0$, $C(s)$ is in $S_2$. 
        Consider $(\nu^\CC - B_2)(S_2) = CD(S_2)$. Note that
        $D(S_2)\subset S_1$, hence $CD(S_2)\subset S_2$ as we have
        just checked. Therefore $S_2$ is invariant under $B_2$.
    \end{proof}
\end{NB}%

Let $T := \bigoplus T_i\subset \bigoplus V_i$ be a graded subspace invariant under $A$, $B$, $C$, $D$ with $\Ima a\subset T$ such that
$A\colon V_2/T_2 \to V_3/T_3$ is an isomorphism. (Here the induced map is denoted by the same symbol $A$.)
We define $\widetilde T_2^\tn\subset V_1\oplus V_3\oplus\CC$ as
\begin{equation*}
    \widetilde T_2^\tn :=
    \{ v_1\oplus v_3\oplus z \mid v_1 \bmod T_1 = D A^{-1}(v_3 \bmod T_3) \}.
\end{equation*}
Then we define $T_2^\tn$ as the image of $\widetilde T_2^\tn$ under
$\beta^\tn$. One can check that $T_1\oplus T_2^\tn\oplus T_3$ satisfies the conditions in $(\nu 2)$.

\begin{NB}
\begin{proof}
    It is clear that $\Ima A^\tn + \Ima a^\tn\subset T_2^\tn$ and also
    $D^\tn(T_3)\subset T_2^\tn$. Next recall $C^\tn = \beta$. Let $v_1\oplus v_3\oplus z\in \widetilde T_2^\tn$. Then
    \begin{equation*}
        \begin{split}
            & \beta(v_1\oplus v_3\oplus z) \bmod T_3
        = \left( AC (v_1) + (B_3 - \nu^\CC) v_3 \right) \bmod T_3
\\
        = \; & (ACD A^{-1} + B_3 - \nu^\CC) (v_3 \bmod T_3)
        = (- A B_2 A^{-1} + B_3)(v_3 \bmod T_3) = 0.
        \end{split}
    \end{equation*}
    Therefore $C^\tn(T_2^\tn) \subset T_3$. Since $B_2^\tn = - D^\tn
    C^\tn$, we also have $B_2^\tn(T_2^\tn)\subset T_2^\tn$.

    Next consider the induced map $V_1/T_1
    \xrightarrow{\beta\circ(\id\oplus 0\oplus 0)}
    V_2^\tn/T_2^\tn$. Since $V_1\oplus 0\oplus 0\cap \widetilde
    T_2^\tn = T_1\oplus 0 \oplus 0$, it is injective. On the other
    hand, take $v_1\oplus v_3\oplus z$. If we take $v'_1\in V_1$ such
    that $v'_1\bmod T_1 = D A^{-1}(v_3 \bmod T_3)$, we have $v_1\oplus
    v_3 \oplus z \equiv (v_1 - v'_1)\oplus 0 \oplus 0$ modulo
    $\widetilde T_2^\tn$. Therefore the map is also surjective.

\end{proof}
\end{NB}%

Let $T^\tn := T_1\oplus T_2^\tn \oplus T_3\subset V_1\oplus
V_2^\tn\oplus V_3$ be a graded subspace invariant under $A^\tn$,
$B^\tn$, $C^\tn$, $D^\tn$ with $\Ima a^\tn \subset T^\tn$ such that
$A^\tn\colon V_1/T_1\to V_2^\tn/T_2^\tn$ is an isomorphism. We define
$T_2 := A^{-1}(T_3)$. One can check that $T_1\oplus T_2\oplus T_3$
satisfies the conditions in $(\nu 2)$.
\begin{NB}
    We have $A(T_2)\subset T_3$ by definition. Under the isomorphism
    $V_2\cong \Ker\beta^\tn$, $A$ is the projection to $V_3$, hence
    $T_2 = \Ker\beta^\tn\cap V_1\oplus T_3\oplus\CC$.

    Let us check $C(T_1)\subset T_2$. Recall $C = \alpha^\tn = \left[
    \begin{smallmatrix}
        - B_1 \\ C^\tn A^\tn \\ b^\tn
    \end{smallmatrix}
    \right]$. Then $C(T_1) \subset V_1\oplus T_3\subset \CC$. Hence
    the assertion follows.
    Next let us check $B_2(T_2)\subset T_2$. Let $v_2\in T_2$. We have
    $A B_2 v_2 \bmod T_3 = B_3 A v_2 \bmod T_3 = 0$. The assertion
    follows.
    Next let us check $D(T_2)\subset T_1$. Suppose $(v_1,v_3,z)\in T_2
    = V_1\oplus T_3\oplus \CC\cap \Ker\beta^\tn$. Then $D(v_1,v_3,z) =
    v_1$. Therefore
    \begin{equation*}
        A^\tn D(v_1,v_3,z) = A^\tn v_1 = D^\tn v_3 - a^\tn z.
    \end{equation*}
    Since $v_3\in T_3$, this is in $T_2^\tn$ by the conditions. Since
    $A^\tn$ induces an isomorphism $V_1/T_1\xrightarrow{\cong}
    V_2^\tn/T_2^\tn$, $D(v_1,v_3,z) \in T_1$ as required.

    Finally let us check that $A\colon V_2/T_2\to V_3/T_3$ is an
    isomorphism. By the definition, $V_2/T_2\to V_3/T_3$ is injective.
    We consider $\Ima A + T_3\subset V_3$. It contains $\Ima A + \Ima
    a$ and invariant under $B_3$. Therefore $\Ima A + T_3 = V_3$ by
    (S2). It means $V_2/T_2\to V_3/T_3$ is surjective. The assertion
    follows.
\end{NB}%

Since $S_1$, $S_3$, $T_1$, $T_3$ are unchanged, the original data
satisfy $(\nu 1)$ (resp.\ $(\nu 2)$) if and only if the new data
satisfy $(\nu 1)$ (resp.\ $(\nu 2)$).
\end{Remark}

\subsection{Invariants}

In this subsection we introduce invariants which are preserved under
Hanany-Witten transition.

Recall two numbers defined in \S \ref{subsec:quiver};
\begin{align*}
N_{h}= \bv_{\vin{h}} - \bv_{\vout{h}}, \ \ N_{\xl}= \bv_{\zeta^-} - \bv_{\zeta^+}.
\end{align*}
\begin{Lemma}
Suppose $\zeta^- = [\vin{h}, \xl]$.
Then $(N_h, N_{\xl})=(m, -n)$ is transferred to $(m-1,-n-1)$\begin{NB}$(m-1,-n+1)$5/28\end{NB} under 
Hanany-Witten transition~\ref{prop:HW-trans}.\label{lem:HW_number}
\end{Lemma}
\begin{proof}\
This is because that Hanany-Witten transition means
\begin{align*}
\begin{xy}
(0,3)*{l},
(5,0)*{\boldsymbol\medcirc},
(5,-4)*{h},
(12.5,3)*{l+m},
(20,0)*{\boldsymbol\times},
(20,-4)*{\xl},
(30,3)*{l+m+n},
(45,0)*{\cong},
(55,3)*{l},
(60,0)*{\boldsymbol\times},
(60,-4)*{\xl},
(70,3)*{l+n+1},
(80,0)*{\boldsymbol\medcirc},
(80,-4)*{h},
(90,3)*{l+m+n},
\ar @{-} (0,0);(35,0)
\ar @{-} (55,0);(95,0)
\end{xy}
\end{align*}
\end{proof}

\begin{NB}
Consider a bow diagram of affine type. We do not assume the dimension
vector $\underline{\bv}$ satisfies either balanced or dual balanced
condition. 
\end{NB}%

We define two numbers:
\begin{align*}
N(h_i, h_{i+1})&:=N_{h_{i}}-N_{h_{i+1}} + \{\text{the number of $\boldsymbol\times$ between $h_{i+1} \to h_i$}\}\\
N(\xl_i, \xl_{i+1})&:=N_{\xl_{i}}-N_{\xl_{i+1}} + \{\text{the number of $\boldsymbol\medcirc$ between $\xl_{i} \to \xl_{i+1}$}\},
\end{align*}
where we number $h_i$ clockwise, while $\xl_i$ anti-clockwise. Here
`between $h_{i+1}\to h_{i}$' means that on the arc starting from
$h_{i+1}$ towards $h_{i}$ in the anti-clockwise direction.
For example, we have
\begin{align*}
\begin{xy}
(5,3)*{\bv_0},
(10,0)*{\boldsymbol\times},
(10,-4)*{\xl_{1}},
(15,3)*{\bv_1},
(20,0)*{\boldsymbol\medcirc},
(20,-4)*{h_{2}},
(25,3)*{\bv_2},
(30,0)*{\boldsymbol\medcirc},
(30,-4)*{h_{1}},
(35,3)*{\bv_3},
(40,0)*{\boldsymbol\times},
(40,-4)*{\xl_{2}},
(45,3)*{\bv_4},
(80,5)*{N(h_1, h_2) = (\bv_3-\bv_2) - (\bv_2-\bv_1) + 0},
(78,0)*{= \bv_1 -2\bv_2+\bv_3},
(80,-5)*{N(\xl_1, \xl_2) = (\bv_0-\bv_1) - (\bv_3-\bv_4) + 2},
(84,-10)*{= \bv_0 - \bv_1 - \bv_3 + \bv_4 + 2.},
\ar @{-} (5,0);(45,0)
\end{xy}
\end{align*}

We suppose the numbers $\ell$, $n$ of $\boldsymbol\medcirc$ and $\boldsymbol\times$ are
greater than $1$ so that ${h_i} \neq h_{i+1}$, $\xl_i\neq
\xl_{i+1}$. Here we understand $i$, $i+1$ modulo either $n$ or
$\ell$. 
\begin{NB}
For finite case, we understand $i$ runs from $0$ to $n$ or
$\ell$, where $h_0$, $x_0$, $h_{\ell+1}$, $x_{n+1}$ are hypothetical
points added at the left and right ends of the diagram.
\end{NB}%

\begin{Proposition}\label{prop:NN}
$N(h_i, h_{i+1})$ and $N(\xl_i, \xl_{i+1})$ are invariant under Hanany-Witten transition.
\end{Proposition}
\begin{proof}
Recall that Lemma \ref{lem:HW_number} states that $(N_{h_i}, N_{\xl_j})=(m, -n)$ is transferred to $(m-1,-n-1)$ \begin{NB}$(N_{h_i}, N_{\xl_j})=(-m, n)$ is transferred to $(-m+1,n+1)$5/28\end{NB}.
\begin{align*}
\begin{xy}
(5,0)*{\boldsymbol\medcirc},
(5,-4)*{h_i},
(5,4)*{N_{h_i}=m},
(30,0)*{\boldsymbol\times},
(30,-4)*{\xl_j},
(30,4)*{N_{\xl_j}=-n},
(45,0)*{\rightarrow},
(60,0)*{\boldsymbol\times},
(60,-4)*{\xl_j},
(60,4)*{N_{\xl_j}=-n-1},
(85,0)*{\boldsymbol\medcirc},
(85,-4)*{h_i},
(85,4)*{N_{h_i}=m-1},
\ar @{-} (0,0);(35,0)
\ar @{-} (55,0);(90,0)
\end{xy}
\end{align*}
\begin{NB}
\begin{align*}
\begin{xy}
(5,0)*{\boldsymbol\medcirc},
(5,-4)*{h_i},
(5,4)*{N_{h_i}=-m},
(30,0)*{\boldsymbol\times},
(30,-4)*{\xl_j},
(30,4)*{N_{\xl_j}=n},
(45,0)*{\rightarrow},
(60,0)*{\boldsymbol\times},
(60,-4)*{\xl_j},
(60,4)*{N_{\xl_j}=n+1},
(85,0)*{\boldsymbol\medcirc},
(85,-4)*{h_i},
(85,4)*{N_{h_i}=-m+1},
\ar @{-} (0,0);(35,0)
\ar @{-} (55,0);(90,0)
\end{xy}
\end{align*}
5/28\end{NB}
In the transition, the number of $\boldsymbol\times$ between $h_{i}$ and
$h_{i-1}$ decreases one and the number of $\boldsymbol\times$ between $h_{i+1}$
and $h_{i}$ increases one. (This remains true even if $h_{i+1} =
h_{i-1}$ since we are counting $\boldsymbol\times$ on \emph{different}
intervals.) This implies $N(h_i, h_{i+1})$ and $N(h_{i-1}, h_{i})$ are
invariant.  This is the same for $N(\xl_i, \xl_{i+1})$ and
$N(\xl_{i-1}, \xl_{i})$.
\end{proof}

\begin{Lemma}\label{lem:dim_unchange}
    $-N_h^2 + \bv_{\zeta^+} + \bv_{\zeta^-}$ and
    $-N_x^2 + \bv_{\vin{h}} + \bv_{\vout{h}}$ are invariant under Hanany-Witten transition~\ref{prop:HW-trans}.
\end{Lemma}

\begin{proof}
    Let $l$, $m$, $n$ as in the proof of \lemref{lem:HW_number}. Then $-N_h^2 + \bv_{\zeta^+} + \bv_{\zeta^-}$ is equal to
\begin{equation*}
   -m^2 + 2(l+m) + n
\end{equation*}
before the transition, and
\begin{equation*}
  -(m-1)^2 + 2l + n + 1
\end{equation*}
after the transition. They are equal.
\begin{NB}
$-N_x^2 + \bv_{\vin{h}} + \bv_{\vout{h}}$ is $-n^2+ 2l+m$ before, and
$-(n+1)^2 + 2(l+n)+m+1$ after. They are equal.
\end{NB}%
Since changes of dimensions remain the same even if we exchange $\boldsymbol\medcirc$ and $\boldsymbol\times$, the same is true for $-N_x^2 + \bv_{\vin{h}} + \bv_{\vout{h}}$.
\end{proof}

The number $-N_h^2 + \bv_{\zeta^+} + \bv_{\zeta^-}$ is the contribution of these parts to the dimension of the bow variety. Since Hanany-Witten transition is an isomorphism of bow varieties, this result is natural.

\subsection{Coulomb branches of type \texorpdfstring{$A$}{A} and nilpotent orbits}
\label{subsec:nilpotent}

Suppose that a framed quiver gauge theory of type $A_{n-1}$ with
dimension vectors $\underline{\bv}$, $\underline{\bw}$ is given. We
assign two $\algsl(n)$-weights $\lambda := \sum_{i=1}^{n-1} \bw_i
\Lambda_i$, $\mu := \sum_{i=1}^{n-1} \bw_i \Lambda_i - \bv_i\alpha_i$.
Here $\Lambda_i$ is the $i$th fundamental weight, and $\alpha_i$ is the $i$th simple root.
This is the same assignment as one appeared in quiver varieties.
Let $\ell = \sum_i \bw_i$ be the level of $\underline{\bw}$.

We consider the corresponding bow diagram satisfying the balanced
condition as in \eqref{eq:8} 
so that the associated bow variety
$\cM(\underline{\bv},\underline{\bw})$ is the Coulomb branch of the
framed quiver gauge theory of finite type $A$. We number $\boldsymbol\times$ and
$\boldsymbol\medcirc$ as:
\begin{align*}
\begin{xy}
(3,3)*{\bv_0=0},
(10,0)*{\boldsymbol\times},
(10,-4)*{\xl_1},
(15,3)*{\bv_1},
(20,0)*{\boldsymbol\medcirc},
(21,-4)*{h_\ell},
(27.5,3)*{\cdots},
(27,-4)*{\cdots},
(35,0)*{\boldsymbol\medcirc},
(36,-4)*{h_{\ell-\bw_1+1}},
(40,3)*{\bv_1},
(45,0)*{\boldsymbol\times},
(45,-4)*{\xl_2},
(50,3)*{\bv_2},
(55,0)*{\boldsymbol\medcirc},
(55,-4)*{h_{\ell-\bw_1}},
(62.5,3)*{\cdots},
(62.5,-4)*{\cdots},
(73,0)*{\boldsymbol\medcirc},
(74,-4)*{h_{\ell-\bw_1-\bw_2+1}},
(79,3)*{\bv_2},
(85,0)*{\boldsymbol\times},
(87,-4)*{\xl_3},
(90,3)*{\bv_3},
(110,3)*{\bv_{n-1}},
(115,0)*{\boldsymbol\times},
(115,-4)*{\xl_n},
(123,3)*{\bv_n = 0},
\ar @{-} (0,0);(95,0)
\ar @{.} (95,0);(105,0)
\ar @{-} (105,0);(125,0)
\end{xy}
\end{align*}

We set $\Lambda_i = [\underbrace{1,\dots,1}_{i},0,\dots,0]$, $\alpha_i
= [0,\dots,0,\underset{i}{1},\underset{i+1}{-1},0,\dots,0]\in\ZZ^n$ as
usual, and regard $\lambda$, $\mu$ as elements in $\ZZ^n$ (or
$\gl_n$-weights). Concretely
\begin{equation}\label{eq:5}
\begin{gathered}[m]
\lambda = [\lambda_1, \lambda_2,\cdots,\lambda_n],
\quad \mu = [\mu_1, \mu_2, \cdots, \mu_{n}],
\\
\lambda_i = \sum_{j\ge i} \bw_j,
\ 
\mu_i = \bv_{n-1} + \sum_{j\geq i}\bu_j,\quad \text{for }i=1, 2, \cdots, n-1,
\ \lambda_n = 0, \ \mu_{n} = \bv_{n-1}, 
\end{gathered}
\end{equation}
where $\underline{\bu} = \underline{\bw} - C\underline{\bv} = (\bw_i +
\bv_{i-1} + \bv_{i+1} - 2\bv_i)_{i=1}^{n-1}$. (We set $\bv_n = 0$.)
Then $\underline{\bu}\in\ZZ_{\geq 0}^{n-1}$ if and only if $\mu$ is
dominant, i.e., $\mu_1\ge\mu_2\ge\cdots\ge\mu_n$. When $\mu$ is
dominant, it is regarded as a partition with at most $n$ rows. On the
other hand, $\lambda$ is always dominant, and is equal to
${}^t(1^{w_1}2^{w_2}\dots (n-1)^{w_{n-1}})$ where ${}^t$ is the
transposed partition.
Note also $|\lambda| = |\mu|$, where $|\ |$ denote the sum of
entries. (For a partition corresponding to a Young diagram, it is the
total number of boxes.)
\begin{NB}
    It is the inner product with $(\bullet, [1,1,\dots,1])$. Note that
    $(\alpha_i,[1,1,\dots,1]) = 0$ for $1\le i\le n-1$. Therefore
    $|\lambda| =|\mu|$ is obvious.
\end{NB}%
We have $\lambda_1 - \mu_1 
\begin{NB}
   = \sum_{j=1}^{n-1} (\bw_j - \bu_j)- \bv_{n-1} = 
   \sum_{j=1}^{n-1} (-\bv_{j-1} - \bv_{j+1} +2\bv_j) - \bv_{n-1}
\end{NB}%
 = \bv_1$, $\lambda_1 + \lambda_2 - (\mu_1 + \mu_2)
\begin{NB} 
  = (\lambda_1 - \mu_1) + (\lambda_2 - \mu_2) = \bv_1 + (\bv_2 - \bv_1)
\end{NB}%
= \bv_2$, \dots, $(\lambda_1 + \dots + \lambda_{n-1}) - (\mu_1 + \dots + \mu_{n-1}) = \bv_{n-1}$. Thus $\lambda \ge \mu$ in the dominance order if and only if $\sum_{j=1}^i \lambda_j \ge \sum_{j=1}^i \mu_j$ for any $i$.

Let us write ${}^t\!\lambda$ as
$[{}^t\!\lambda_1,{}^t\!\lambda_2,\dots,{}^t\!\lambda_\ell]$ with
$\ell = \lambda_1 = \sum_i \bw_i$. For example, 
${}^t\!\lambda_1$ is the last $j$ with $\bw_j \neq 0$, etc.

\begin{Lemma}\label{lem:bowN}
    Consider the bow diagram with the balanced condition associated
    with $\underline{\bv}$, $\underline{\bw}$ as above.

    \textup{(1)} $N(h_i,h_{i+1}) = {}^t\!\lambda_{i} - {}^t\!\lambda_{i +
      1}$. In particular, $N(h_i,h_{i+1})$ is always nonnegative.

    \textup{(2)} $N(x_i,x_{i+1})$ is the $i$th entry of
    $\underline{\bu} = \underline{\bw} - C\underline{\bv}$. In
    particular, $N(x_i,x_{i+1}) \ge 0$ for any $i$ if and only if
    $\mu = \sum\bw_i\Lambda_i - \bv_i\alpha_i$ is dominant.
\end{Lemma}

\begin{proof}
    (1) Since $N_{h_i} = 0$ by the balanced condition,
    $N(h_i,h_{i+1})$ is the number of $\boldsymbol\times$ between $h_{i+1}$ and
    $h_{i}$\begin{NB}$h_i$ and $h_{i+1}$5/28\end{NB}. For example $N(h_1, h_2) = j_1 - j_2$ if $h_1$ (resp.\
    $h_2$) is between $x_{j_1}$ and $x_{j_1+1}$ (resp.\ $x_{j_2}$ and
    $x_{j_2+1}$). Now $j_1 = {}^t\!\lambda_1$, $j_2 = {}^t\!\lambda_2$
    by definition. Hence $N(h_1,h_2) = {}^t\!\lambda_1 -
    {}^t\!\lambda_2$. General cases are the same.

    (2) We calculate $N(x_i,x_{i+1})$ as
    \begin{align*}
        N(\xl_i, \xl_{i+1}) &= (\bv_{i-1} - \bv_{i}) - (\bv_i -
        \bv_{i+1}) + \bw_i
        \\
        &= \bv_{i-1} - 2\bv_i + \bv_{i+1} + \bw_i = \bu_i. \qedhere
\end{align*}
\end{proof}

Let us apply Hanany-Witten transition to move $\boldsymbol\medcirc$ to the left
and $\boldsymbol\times$ to the right. Then we get
\begin{equation}\label{eq:7}
    \begin{xy}
(0,6)*{0},
(5,0)*{\boldsymbol\medcirc},
(5,-4)*{h_\ell},
(10,6)*{{}^t\!\lambda_\ell},
(15,0)*{\boldsymbol\medcirc},
(15,-4)*{h_{\ell-1}},
(20,6)*{\substack{{}^t\!\lambda_{\ell-1}\\ + \\ {}^t\lambda_\ell}},
(25,0)*{\boldsymbol\medcirc},
(25,-4)*{h_{\ell-2}},
(35,6)*{\cdots},
(45,0)*{\boldsymbol\medcirc},
(45,-4)*{h_2},
(52,6)*{\displaystyle\sum_{j\ge 2}{}^t\!\lambda_j},
(58,0)*{\boldsymbol\medcirc},
(58,-4)*{h_1},
(72,6)*{\displaystyle\sum_{j\ge 1}{}^t\!\lambda_j
=\displaystyle\sum_{j\ge 1}\mu_j},
(84,0)*{\boldsymbol\times},
(84,-4)*{\xl_1},
(91,6)*{\displaystyle\sum_{j\ge 2}\mu_j},
(97,0)*{\boldsymbol\times},
(97,-4)*{\xl_2},
(107,6)*{\cdots},
(117,0)*{\boldsymbol\times},
(117,-4)*{\xl_{n-2}},
(122,6)*{\substack{\mu_{n-1}\\ + \\ \mu_{n}}},
(127,0)*{\boldsymbol\times},
(127,-4)*{\xl_{n-1}},
(132,6)*{\mu_{n}},
(137,0)*{\boldsymbol\times},
(137,-4)*{\xl_{n}},
(142,6)*{0},
\ar @{-} (0,0);(142,0)
\end{xy}.
\end{equation}
We have
\begin{gather*}
    N_{h_1} = {}^t\!\lambda_1, \quad
    N_{h_2} = {}^t\!\lambda_2, \quad \dots,\quad
    N_{h_\ell} = {}^t\!\lambda_\ell, \\
    N_{\xl_1} = \mu_1, \quad
    N_{\xl_2} = \mu_2, \quad \dots,\quad
    N_{\xl_n} = \mu_n
\end{gather*}
for this bow diagram.
Since $\lambda$ is dominant, we have $0 \le N_{h_1} \le N_{h_2} \le
\dots \le N_{h_\ell}$.

It is convenient for us to write $N_{h_\xp}$ and $N_{\xl_i}$, rather than dimensions themselves, on $\boldsymbol\medcirc$ and $\boldsymbol\times$ in the bow diagram. If we write a dimension of a vector space on a particular segment (like $0$ in the above diagram), other dimensions are calculated from $N_{h_\xp}$ and $N_{\xl_i}$. We will use such a notation often hereafter.

%
For a finite type $A$, a quiver variety
$\fM_0(\underline{\bv},\underline{\bw})$ with dominant $\mu$ is
isomorphic to an intersection of nilpotent orbit and transversal slice:
\begin{Theorem}[{\cite[Theorem 8.4]{Na-quiver}}]
Suppose $\mu$ is dominant. We have
\begin{equation*}
\fM_0(\underline{\bv}, \underline{\bw}) \cong 
\overline{\mathcal{N}}({}^t\!\mu)
\cap \mathcal{S}({}^t\!\lambda),
\end{equation*}
where
$\mathcal{N}$ is a nilpotent orbit, $\overline{\mathcal{N}}$ is its
closure, $\mathcal{S}$ is a transversal slice, and ${}^t\!\mu$ is a
partition which is represented by a transpose of Young diagram of a
partition $\mu$.
\label{thm:finite_dominant1}
\end{Theorem}
And it is known what is the corresponding Coulomb branch by
\cite{blowup}:
\begin{Theorem}
The Coulomb branch $\cM(\underline{\bv},\underline{\bw})$ of a quiver gauge theory with dominant $\mu$
is $\overline{\mathcal{N}}(\lambda)\cap \mathcal{S}(\mu)$, where $\lambda$ and $\mu$ are given by \eqref{eq:5}.\label{thm:finite_dominant2}
\end{Theorem}

\begin{NB}
    The total number of boxes is given by the inner product with
    $[1,1,\dots,1]$. Therefore $|\lambda| = \sum_i i \bw_i = |\mu|$.
\end{NB}%

Note that $\lambda$, $\mu$ are swapped and taken transpose for $\fM$
and $\cM$.

Let us give proofs of both \thmref{thm:finite_dominant1}, and
\thmref{thm:finite_dominant2} by the procedure in
\subsecref{subsec:HC_corresp}.

\begin{NB}
When we describe the above $\fM(\underline{\bv}, \underline{\bw})$ as a bow variety, we have $N_{\xl_i}=0$ for any $i$ and $N_{h_i} = \bv_i-\bv_{i-1}$.
\begin{align*}
\begin{xy}
(0,3)*{\bv_0=0},
(10,0)*{\boldsymbol\medcirc},
(10,-4)*{h_1},
(15,3)*{\bv_1},
(20,0)*{\boldsymbol\times},
(21,-4)*{\xl_1},
(27.5,3)*{\cdots},
(27.5,-4)*{\cdots},
(35,0)*{\boldsymbol\times},
(36,-4)*{\xl_{\bw_1}},
(40,3)*{\bv_1},
(45,0)*{\boldsymbol\medcirc},
(45,-4)*{h_2},
(50,3)*{\bv_2},
(55,0)*{\boldsymbol\times},
(57,-4)*{\xl_{\bw_1+1}},
(62.5,3)*{\cdots},
(65,-4)*{\cdots},
(70,0)*{\boldsymbol\times},
(73,-4)*{\xl_{\bw_1+\bw_2}},
(75,3)*{\bv_2},
(80,0)*{\boldsymbol\medcirc},
(81,-4)*{h_3},
(85,3)*{\bv_3},
\ar @{-} (0,0);(85,0)
\ar @{.} (85,0);(100,0)
\end{xy}
\end{align*}
By Hanany-Witten transition, we can make all $\boldsymbol\times$ to be left and all $\boldsymbol\medcirc$ to be right.
Then, because of Lemma \ref{lem:HW_number}, we obtain
\begin{align*}
\begin{xy}
(5,3)*{0},
(10,0)*{\boldsymbol\times},
(10,-4)*{\xl_1},
(15,3)*{N_{\xl_1}},
(20,0)*{\boldsymbol\times},
(20,-4)*{\xl_2},
(30,3)*{N_{\xl_1}+N_{\xl_2}},
(40,0)*{\boldsymbol\times},
(40,-4)*{\xl_3},
(45,3)*{\cdots},
(50,0)*{\boldsymbol\times},
(52,-4)*{\xl_{|\underline{\bw}|}},
(65,3)*{\sum N_{\xl_i}=-\sum N_{h_i}},
(80,0)*{\boldsymbol\medcirc},
(80,-4)*{h_1},
(85,3)*{\cdots},
(90,0)*{\boldsymbol\medcirc},
(92,-4)*{h_{n-2}},
(102.5,3)*{-N_{h_{n-1}}-N_{h_{n}}},
(115,0)*{\boldsymbol\medcirc},
(116,-4)*{h_{n-1}},
(122.5,3)*{-N_{h_{n}}},
(130,0)*{\boldsymbol\medcirc},
(132,-4)*{h_{n}},
(135,3)*{0},
\ar @{-} (5,0);(135,0)
\end{xy}
\end{align*}
where
\begin{align*}
N_{\xl_i} = \begin{cases}
1 & 1 \leq i \leq \bw_1, \\
2 & \bw_1+1 \leq i \leq \bw_1+\bw_2, \\
\vdots & \vdots \\
n -1 & \sum_{j=1}^{n-2}\bw_{j} + 1 \leq i \leq |\underline{\bw}|,
\end{cases} \ \ \  
N_{h_i} = (\bv_i - \bv_{i-1}) - \sum_{j=i}^{n-1}\bw_j.
\end{align*}
\end{NB}%

A quiver variety $\fM_0(\underline{\bv},\underline{\bw})$ corresponds to
a bow diagram with the cobalanced condition
\thmref{thm:balanced_bow}. It is obtained from the bow diagram for
$\cM(\underline{\bv},\underline{\bw})$ by exchanging $\boldsymbol\times$ and
$\boldsymbol\medcirc$ by \subsecref{subsec:HC_corresp}. Looking at \eqref{eq:7},
we find that it exchanges two partitions $\lambda$ and $\mu$ after
taking their transposes. Thus it is enough to prove
\thmref{thm:finite_dominant2}.

We divide \eqref{eq:7} at the middle and consider
$\cM(\underline{\bv},\underline{\bw})$ as the symplectic reduction of
the product of varieties corresponding to the left and right sides by
the action of $\GL(r)$ with $r = \sum {}^t\!\lambda_{j}$.
The left hand side, as a quiver variety of type $A$, has nonzero $W$
at the rightmost end. It is the variety originally appeared in
\cite{MR549399}. Note that the dominance condition is automatically
satisfied, and $\lambda$ is a partition. The variety is isomorphic to
$\overline{\mathcal N}(\lambda)$.
The right hand side, viewed as the moduli space of solutions of Nahm's
equations on an interval, gives $\GL(r)\times \mathcal S(\mu)$ by
\cite{MR1438643}.
Thus, we obtain \thmref{thm:finite_dominant2}.

\begin{Remark}\label{rem:deform}
    We can also generalize these results to cases of deformation and
    resolutions. Let us state the results for only special cases for
    brevity.

    Let $\nu = (\nu^\CC,0)$ be a parameter with vanishing real
    part. The left hand side of \eqref{eq:7} gives the closure
    $\overline{\mathcal O}$ of a conjugacy class $\mathcal O$
    determined by $\nu^\CC$ and $\lambda$. (See e.g.,
    \cite{MR1980997}.) The above argument show that
    $\cM_\nu(\underline{\bv},\underline{\bw}) \cong \overline{\mathcal
      O}\cap \mathcal S(\mu)$.
    
    Let $\nu = (0,\nu^\RR)$ be a parameter with vanishing complex
    part. We further assume $\nu^\RR_\xp > 0$ for all $\xp$. Then the
    left hand side of \eqref{eq:7} gives the cotangent bundle
    $T^*\mathcal F$ of a partial flag variety $\mathcal F$ of type
    $A$. We have a projective morphism $\pi\colon T^*\mathcal F\to
    \overline{\mathcal N}(\lambda)$. (See \cite[Th.~7.3]{Na-quiver}.)
    Then $\cM_\nu(\underline{\bv},\underline{\bw}) \cong
    \pi^{-1}(\mathcal S(\mu))$.

    We also have the corresponding results for quiver varieties. The
    second case $\nu = (0,\nu^\RR)$ recovers \cite[Th.~8]{MR2130242},
    which was conjectured in \cite[Conj.~8.6]{Na-quiver}.
\end{Remark}

\subsection{Stratification of Coulomb branches of type \texorpdfstring{$A$}{A}}\label{subsec:strat-A}

We continue to work on the Coulomb branch of a framed quiver gauge
theory of finite type $A_{n-1}$.
If $\mu$ is dominant, we have $\cM_C \cong \overline{\mathcal
  N}(\lambda)\cap \mathcal S(\mu)$, and we have a stratification
\(
   \cM_C = \bigsqcup_{\kappa}
   \mathcal N(\kappa)\cap \mathcal S(\mu),
\)
where $\kappa$ runs dominant weights between $\mu$ and $\lambda$, i.e.,\
$\mu\le\kappa\le\lambda$. We give a stratification in general in this
subsection.

It is easier to state results in terms of $\algsl(n)$-weights
$\lambda := \sum_{i=1}^{n-1} \bw_i\Lambda_i$,
$\mu := \lambda - \sum_{i=1}^{n-1}\bv_i\alpha_i$. Let us write
$\cM(\mu,\lambda)$ instead of $\cM(\underline{\bv},\underline{\bw})$.

\begin{Theorem}\label{thm:strat-finite-type}
    We have a stratification
    \begin{equation*}
        \cM(\mu,\lambda)
        = \bigsqcup_{\mu\le\kappa\le\lambda} \cM^{\mathrm{s}}(\mu,\kappa),
    \end{equation*}
    where $\kappa$ runs dominant weights between $\mu$ and $\lambda$. We have
    $\cM^{\mathrm{s}}(\mu,\kappa)\neq\emptyset$ for any $\kappa$.
%
%
    Moreover a transversal slice to $\cM^{\mathrm{s}}(\mu,\kappa)$ in
    $\cM(\mu,\lambda)$ is isomorphic to $\cM(\kappa,\lambda)$.
\end{Theorem}

These properties were originally proved in \cite{2015arXiv151003908N}
based on the description of Coulomb branches as moduli spaces of
singular monopoles, which is not justified yet. Our argument gives a
rigorous proof for type $A$. (See also \cite[Rem.~3.19]{blowup}.)

\begin{proof}
    Let us regard $\lambda$, $\mu\in\ZZ^n$ as in \eqref{eq:5}.

    As in \subsecref{subsec:nilpotent}, we start with a bow variety
    satisfying the balanced condition and apply Hanany-Witten
    transition to move $\boldsymbol\medcirc$ to the left and $\boldsymbol\times$ to the
    right.
    We divide it at the middle and consider $\cM(\mu,\lambda)$ as the
    symplectic reduction of the product of varieties from the left and
    right by $\GL(r)$ with $r= \sum \lambda_i$. The left side
    is $\overline{\mathcal N}(\lambda)$ as before. The right side is
    another bow variety $\widetilde{\cM}(\mu)$ associated with
\begin{equation*}
\begin{tikzpicture}[baseline=(current  bounding  box.center)]
  \node[label=above:$r$] at (0.4,0) {$\vphantom{j^X}$};
  \node[label=below:$\xl_1$,label=above:$\mu_1$] at (1,0)
        {$\boldsymbol\times$};
  \node[label=below:$\xl_2$,label=above:$\mu_2$] at (2,0)
        {$\boldsymbol\times$};
  \node[label=below:$\cdots$,label=above:$\cdots$] at (3,0)
        {$\vphantom{j^X}$};
  \node[label=below:$\xl_{n-2}$,label=above:$\mu_{n-2}$] at (4,0)
        {$\boldsymbol\times$};
  \node[label=below:$\xl_{n-1}$,label=above:$\mu_{n-1}$] at (5,0)
        {$\boldsymbol\times$};
  \node[label=below:$\xl_{n}$,label=above:$\mu_{n}$] at (6,0)
        {$\boldsymbol\times$}; 
  \node[label=above:$0$] at (6.7,0) {$\vphantom{j^X}$};
  \draw[-] (0,0) -- (7,0);
\end{tikzpicture}
\end{equation*}
\begin{NB}
    \begin{align*}
        \begin{xy}
(3,6)*{r=\displaystyle\sum_{j\ge 1}\mu_j},
(12,0)*{\boldsymbol\times},
(12,-4)*{\xl_1},
(19,6)*{\displaystyle\sum_{j\ge 2}\mu_j},
(25,0)*{\boldsymbol\times},
(25,-4)*{\xl_2},
(35,6)*{\cdots},
(45,0)*{\boldsymbol\times},
(45,-4)*{\xl_{n-2}},
(50,6)*{\substack{\mu_{n-1}\\ + \\ \mu_{n}}},
(55,0)*{\boldsymbol\times},
(55,-4)*{\xl_{n-1}},
(60,6)*{\mu_{n}},
(65,0)*{\boldsymbol\times},
(65,-4)*{\xl_{n}},
(70,6)*{0},
\ar @{-} (0,0);(70,0)
\end{xy}
\begin{NB2}
\qquad
    \sum_{j\ge i} \mu_j = \bv_{i-1} + \sum_{j=i}^{n-1} (j - i + 1)\bw_j
    \quad \text{with $\bv_0 = 0$},
\end{NB2}%
\end{align*}
\end{NB}%
where we do not take the quotient by the leftmost $\GL(r)$. Here we denote dimensions on the leftmost and rightmost segments, and $N_{\xl_i}$.

The stratification
\begin{equation}\label{eq:14}
    \overline{\mathcal N}(\lambda)
    = \bigsqcup_{\kappa\le\lambda} \mathcal N(\kappa)
\end{equation}
by orbits induces a stratification of
$\cM(\mu,\lambda)$ as
\begin{equation*}
    \cM(\mu,\lambda) = 
    \bigsqcup \left(\mathcal N(\kappa)\times \widetilde\cM(\mu)\right)
    /\!\!/\!\!/ \GL(r) =
    \bigsqcup \Phi^{-1}(\mathcal N(\kappa))\dslash \GL(r),
\end{equation*}
where $\Phi\colon\widetilde{\cM}(\mu)\to\mathfrak{gl}(r)$ is
the moment map. Moreover
\(
    \Phi^{-1}(\overline{\mathcal N}(\kappa))\dslash \GL(r)
\)
is a bow variety associated with the bow diagram \eqref{eq:7} with
$\lambda$ replaced by $\kappa$, that is
\begin{equation*}
\begin{tikzpicture}[baseline=(current  bounding  box.center)]
  \node[label=above:$r$] at (0,0) {$\vphantom{j^X}$};
  \node[label=below:$\xl_1$,label=above:$\mu_1$] at (1,0)
        {$\boldsymbol\times$};
  \node[label=below:$\xl_2$,label=above:$\mu_2$] at (2,0)
        {$\boldsymbol\times$};
  \node[label=below:$\cdots$,label=above:$\cdots$] at (3,0)
        {$\vphantom{j^X}$};
  \node[label=below:$\xl_{n-2}$,label=above:$\mu_{n-2}$] at (4,0)
        {$\boldsymbol\times$};
  \node[label=below:$\xl_{n-1}$,label=above:$\mu_{n-1}$] at (5,0)
        {$\boldsymbol\times$};
  \node[label=below:$\xl_{n}$,label=above:$\mu_{n}$] at (6,0)
        {$\boldsymbol\times$}; 
  \node[label=above:$0$] at (6.7,0) {$\vphantom{j^X}$};
  \node[label=below:$h_1$,label=above:${}^t\!\kappa_1$] at (-1,0)
        {$\boldsymbol\medcirc$};
  \node[label=below:$h_2$,label=above:${}^t\!\kappa_2$] at (-2,0)
        {$\boldsymbol\medcirc$};
  \node[label=below:$\cdots$,label=above:$\cdots$] at (-3,0)
        {$\vphantom{j^X}$};
  \node[label=below:$h_{\ell-2}$,label=above:${}^t\!\kappa_{\ell-2}$] at (-4,0)
        {$\boldsymbol\medcirc$};
  \node[label=below:$h_{\ell-1}$,label=above:${}^t\!\kappa_{\ell-1}$] at (-5,0)
        {$\boldsymbol\medcirc$};
  \node[label=below:$h_{\ell}$,label=above:${}^t\!\kappa_{\ell}$] at (-6,0)
        {$\boldsymbol\medcirc$}; 
  \node[label=above:$0$] at (-6.7,0) {$\vphantom{j^X}$};
  \draw[-] (-7,0) -- (7,0);
\end{tikzpicture}
\end{equation*}
\begin{NB}
\begin{equation*}
            \begin{xy}
(0,6)*{0},
(5,0)*{\boldsymbol\medcirc},
(5,-4)*{h_\ell},
(10,6)*{{}^t\!\kappa_\ell},
(15,0)*{\boldsymbol\medcirc},
(15,-4)*{h_{\ell-1}},
(20,6)*{\substack{{}^t\!\kappa_{\ell-1}\\ + \\ {}^t\kappa_\ell}},
(25,0)*{\boldsymbol\medcirc},
(25,-4)*{h_{\ell-2}},
(35,6)*{\cdots},
(45,0)*{\boldsymbol\medcirc},
(45,-4)*{h_2},
(52,6)*{\displaystyle\sum_{j\ge 2}{}^t\!\kappa_j},
(58,0)*{\boldsymbol\medcirc},
(58,-4)*{h_1},
(75,6)*{\displaystyle\sum_{j\ge 1}{}^t\!\kappa_j=
r=\displaystyle\sum_{j\ge 1}\mu_j},
(92,0)*{\boldsymbol\times},
(92,-4)*{\xl_1},
(99,6)*{\displaystyle\sum_{j\ge 2}\mu_j},
(105,0)*{\boldsymbol\times},
(105,-4)*{\xl_2},
(115,6)*{\cdots},
(125,0)*{\boldsymbol\times},
(125,-4)*{\xl_{n-2}},
(130,6)*{\substack{\mu_{n-1}\\ + \\ \mu_{n}}},
(135,0)*{\boldsymbol\times},
(135,-4)*{\xl_{n-1}},
(140,6)*{\mu_{n}},
(145,0)*{\boldsymbol\times},
(145,-4)*{\xl_{n}},
(150,6)*{0},
\ar @{-} (0,0);(150,0)
\end{xy}.
\end{equation*}
\end{NB}%
\begin{NB}
Since $\kappa$ is dominant, we have $0 \le N_{h_1} \le N_{h_2} \le
\dots \le N_{h_\ell}$.
\begin{NB2}
    $N_{h_1} = {}^t\!\kappa_\ell$, $N_{h_2} = {}^t\!\kappa_{\ell-1}$, \dots,
    $N_{h_\ell} = {}^t\!\kappa_1$.
\end{NB2}%
\end{NB}%

Let us also note that the stratification \eqref{eq:14} is a special
case of the stratification for quiver varieties (see e.g.,
\cite[Prop.~6.7]{Na-quiver}):
\begin{equation*}
    \fM_0(\underline{\bv}',\underline{\bw}')
    = \bigsqcup \fM_0^{\mathrm{reg}}(\underline{\bv}'_0,\underline{\bw}'),
\end{equation*}
where $\underline{\bw}' = (r,0,\dots,0)$, $\underline{\bv}' =
(\sum_{j\ge i} {}^t\!\lambda_j)_{i=1}^{\ell}$, $\underline{\bv}'_0 =
(\sum_{j\ge i} {}^t\!\kappa_j)_{i=1}^{\ell}$. The dominance condition
${}^t\!\kappa_1\ge {}^t\!\kappa_2\ge \dots \ge {}^t\!\kappa_\ell$ is a special
case of the necessary condition for
$\fM_0^{\mathrm{reg}}(\underline{\bv}'_0,\underline{\bw}')\neq\emptyset$
in \cite[Lem.~8.1]{Na-quiver}, \cite[Lem.~4.7]{Na-alg}.

\begin{Claim}
    We have 
    $n\ge {}^t\!\kappa_1$ if $\Phi^{-1}(\mathcal N(\kappa))\neq\emptyset$.
\end{Claim}

We add a new $h_0$ at the right of $\xl_n$. Then we have $N_{h_0} =
0$, and $N(h_0,h_1) = n - {}^t\!\kappa_1$. We move $h_0$ to the right of
$h_1$ by Hanany-Witten transitions. Since $N(h_0,h_1)$ is preserved,
we have $N_{h_0} = n$ in the resulted bow diagram. Now we consider the
left half of the resulted bow diagram as a quiver variety. We derive
$n\ge {}^t\!\kappa_1$ as a necessary condition for $\fM_0^{\mathrm{reg}}\neq\emptyset$ as above. The claim is proved.

Let us apply Hanany-Witten transition to move $\boldsymbol\times$ to the left so
that $N_{h_\xp} = 0$ for $1\le\xp\le \ell$. Namely we move
${}^t\!\kappa_{\ell}$ of $\boldsymbol\times$'s to the left of $h_\ell$, move
$({}^t\!\kappa_{\ell-1}-{}^t\!\kappa_\ell)$ of $\boldsymbol\times$'s between $h_\ell$
and $h_{\ell-1}$, and so on. Since $N(h_\xp,h_{\xp+1})$ does not
change, it means all $N_{h_\xp}$ are equal, but $N_{h_\ell}=0$ by the
first move.
\begin{NB}
    We formally add $\boldsymbol\medcirc$ ($h_{\ell+1}$). Then $N_{h_{\ell+1}} =
    0$. Since $N(h_\ell,h_{\ell+1}) = N_{h_\ell} = {}^t\kappa_\ell$ is
    unchanged, $N_{h_\ell}$ must decrease $N_{h_\ell}={}^t\!\kappa_\ell$
    if we move ${}^t\!\kappa_\ell$ of $\boldsymbol\times$ to the left of $h_\ell$.
\end{NB}%
This procedure is possible as $n\ge {}^t\!\kappa_1$.
\begin{NB}
This procedure is possible if we have enough $\boldsymbol\times$'s,
i.e., more than or equal to $N_{h_\ell}$. Let us perform this
procedure after possibly adding extra $\boldsymbol\times$'s ($\xl_{n+1}$,
$\xl_{n+2}$, \dots) to the right of $\xl_n$. Then it is possible to
make ${}^t\!\kappa_1 = 0$. Since $N(\xl_{n+1},\xl_{n+2}) =
N(\xl_{n+2},\xl_{n+3}) = \cdots = 0$ is preserved, we cannot move
$\boldsymbol\medcirc$ to the right of $\xl_{n+1}$. In fact, suppose that the
rightmost $\boldsymbol\medcirc$ is between $\xl_{n+N-1}$ and $\xl_{n+N}$. Then
$\boldsymbol\medcirc$ does not cross $\xl_{n+N}$, hence dimensions on segments on
the left and right of $\xl_{n+N}$ are both $0$. The balanced condition
implies that the dimensions on all segments between $\xl_{n+N-1}$ and
$\xl_{n+N}$ are also $0$. Therefore $N(\xl_{n+N-1},\xl_{n+N})$ is the
sum of the number of $\boldsymbol\medcirc$ between $\xl_{n+N-1}$ and $\xl_{n+N}$
and the dimension of the left of $\xl_{n+N-1}$. This is positive by
the assumption, but it is a contradiction as
$N(\xl_{n+N-1},\xl_{n+N})$ is unchanged.  More precisely, we use the
condition that $\Phi^{-1}(\mathcal N(\kappa))\neq \emptyset$, hence the
dimension on the segment on the left of $\xl_{n+N-1}$ must be
nonnegative. We also see that $n\ge {}^t\!\kappa_1$.
\begin{NB2}
    Let us add $h_{0}$ on the right of $x_n$. Then we have
    $N(h_{0},h_{1}) = n - {}^t\!\kappa_1 \ge 0$. Hence we can
    apply \remref{rem:finite_move}.
\end{NB2}%
\end{NB}%

If $n = {}^t\!\kappa_1$, the last $\xl_n$ is moved between $h_1$ and
$h_2$. But the the balanced condition is satisfied, and the dimension
on the rightmost segment is $0$. Therefore dimensions on two segments,
left and right of $h_1$ are both $0$. Thus we can remove $h_1$, and
the resulted bow variety is associated with a framed quiver gauge
theory of type $A_{n-1}$.
The corresponding dimension vectors $\bv'$, $\bw'$ are determined by
\begin{equation}\label{eq:6}
    \kappa - \kappa_n [1,1,\dots,1] = \sum_i \bw'_i\Lambda_i,
    \quad
    \mu - \kappa_n [1,1,\dots,1]
    = \sum_i \bw'_i\Lambda_i - \bv'_i\alpha_i.
\end{equation}
\begin{NB}
    The final removal is just taken the ${}^t\kappa_1$ from the dual
    partition ${}^t\kappa_1$. In the original partition, it means
    subtracting $\kappa_n [1,1,\dots,1]$. But if we would know that it
    should correspond to a gauge theory, it must be this form as the
    last entry is $0$ in the right hand side. We can check the formula
    for $\mu$ in the same way, but it is also clear that we need to
    have $|\kappa - \kappa_n [1,1,\dots,1]| = |\mu - \kappa_n [1,1,\dots,1]|$.
\end{NB}%
Since we must have $\bv'_i\ge 0$, the strata are labeled by
$\kappa$ with $\mu\le\kappa\le\lambda$.
More precisely our $\kappa\in\ZZ^n$ is a $\mathfrak{gl}(n)$-weight, so
must be changed to an $\algsl(n)$-weight by \eqref{eq:6}.

We claim that the following condition is satisfied:
\begin{equation*}
    \text{(H-S1):} \ B_i(S_i) \subset S_i, S_i\subset \Ker b_i, A_i(S_i) \subset S_{i+1} \Rightarrow S_i=0.
\end{equation*}
In fact, by \propref{prop:character_cond1}, we have $\dim S_1\le \dim
S_2\le\dots$. Hence $S_i = 0$ as $\dim S_{n+1} = 0$. Now this
condition implies that the $\GL(r)$-action on $\widetilde{\cM}(\mu)$
is free. Therefore $\Phi^{-1}(\mathcal N(\kappa))\dslash \GL(r)$ is
smooth and symplectic. 

Alternatively we can check the smoothness as follows. The stability
conditions $(\nu 1)$, $(\nu 2)$ (with $\nu = 0$) are equivalent to the
condition $\fM_0^{\mathrm{reg}}$ on the left side, as
$S_{\zeta^-}\xrightarrow{\cong} S_{\zeta^+}$, $V_{\zeta^-}/T_{\zeta^-}
\xrightarrow{\cong} V_{\zeta^+}/T_{\zeta^+}$ mean that $S_\zeta = 0$,
$T_\zeta = V_\zeta$ on segments right of $h_1$.
Therefore $\Phi^{-1}(\mathcal N(\kappa))\dslash \GL(r)$ coincides with
$\cM^{\mathrm{s}}(\mu,\kappa)$, hence is smooth.
It also implies that this stratification coincides with \eqref{eq:13},
where we do not have the factor $S^k(\fM^{\mathrm{s}}(\delta))$ as we are considering finite type $A$ bow varieties.

Let us prove $\lambda\ge\kappa$ from \eqref{eq:13}. Recall that the difference of dimension vectors $\underline{\bv} - \underline{\bv}'$ in \eqref{eq:13} has nonnegative entries. Thus $\sum_{j\ge i} {}^t\! \lambda_j \ge \sum_{j\ge i} {}^t\! \kappa_j$. This is equivalent to ${}^t \lambda\le {}^t\!\kappa \Leftrightarrow \lambda\ge\kappa$.

Let us remark we also see that
$\cM^{\mathrm{s}}(\mu,\lambda)\neq\emptyset$ for $\mu\le\lambda$ with
dominant $\lambda$. In fact, we have $\cM(\mu,\lambda)\neq\emptyset$
as it is a Coulomb branch. We have just observed $\cM(\mu,\kappa)$ is a
Coulomb branch, hence its dimension is $2\sum \bv'_i$. This is
strictly smaller than $2\sum \bv_i$ if $\kappa\neq\lambda$ since $\lambda
> \kappa$. Therefore $\cM^{\mathrm{s}}(\mu,\lambda)\neq\emptyset$.

Finally let us consider a local structure around a point $p$ in
$\cM^{\mathrm{s}}(\mu,\kappa)$.
\begin{NB}
We choose a lift of $p$ in
$\Phi^{-1}(\mathcal N(\kappa))$ and suppose $\Phi(p) \in \mathcal
N(\kappa)\cap\mathcal S(\kappa)$. Consider the orbit $\GL(r)p$ through $p$
and take the symplectic normal space $\widehat{\mathbf M}$ at $p$,
i.e., $\widehat{\mathbf M} = (T_p \GL(r)p)^\perp/T_p\GL(r)p$ where
$T_p\GL(r)p$ is the tangent space to $\GL(r)p$ at $p$, and $(T_p
\GL(r)p)^\perp$ is its symplectic perpendicular. Since the
$\GL(r)$-action is free, $\widetilde{\cM}(\mu)$ is locally
$G$-equivariantly hamiltonian isomorphic to $\gl(r)^*\times \GL(r)
\times \widehat{\mathbf M}$ in complex analytic topology. (See e.g.,
\cite[Lem.~3.2.1]{Na-qaff} or \cite[Th.~3]{MR2230091}.) Thus
$\Phi^{-1}(\overline{\mathcal N}(\lambda))\dslash \GL(r)$ is locally
isomorphic to $\overline{\mathcal N}(\lambda)\times\widehat{\mathbf
  M}$ so that $\Phi^{-1}({\mathcal N}(\kappa))\dslash \GL(r)$ corresponds
to ${\mathcal N}(\kappa)\times\widehat{\mathbf M}$. The assertion follows
as $\overline{\mathcal N}(\lambda)\cap \mathcal S(\kappa) \cong
\cM(\kappa,\lambda)$ by \thmref{thm:finite_dominant2}.
\end{NB}%
By \propref{prop:local} a transversal slice is given a quiver variety,
which is of type $A$. It is clear that the slice is locally isomorphic
to $\overline{\mathcal N}(\lambda)\cap \mathcal S(\kappa)$, which is
$\cM(\kappa,\lambda)$ by \thmref{thm:finite_dominant2}.
\end{proof}

\subsection{Coulomb branches and quiver varieties of affine type \texorpdfstring{$A$}{A}}

We now turn to affine type $A$. We take a framed quiver gauge theory
of type affine $A_{n-1}$ with dimension vectors $\underline{\bv}$,
$\underline{\bw}$. Let $\ell = \sum_i \bw_i$ be the level of
$\underline{\bw}$. 

We number $\boldsymbol\times$ and $\boldsymbol\medcirc$ as
\begin{align}\label{eq:10}
\begin{xy}
(5,3)*{\bv_{n-2}},
(10,0)*{\boldsymbol\times},
(10,-4)*{\xl_{n-1}},
(15,3)*{\bv_{n-1}},
(20,0)*{\boldsymbol\medcirc},
(21,-4)*{h_{\bw_0}},
(27.5,3)*{\cdots},
(27.5,-4)*{\cdots},
(35,0)*{\boldsymbol\medcirc},
(36,-4)*{h_1},
(40,3)*{\bv_{n-1}},
(45,0)*{\boldsymbol\times},
(45,-4)*{\xl_0},
(50,3)*{\bv_0},
(55,0)*{\boldsymbol\medcirc},
(55,-4)*{h_\ell},
(62.5,3)*{\cdots},
(62.5,-4)*{\cdots},
(70,0)*{\boldsymbol\medcirc},
(72.5,-4)*{h_{\ell-\bw_1+1}},
(75,3)*{\bv_0},
(80,0)*{\boldsymbol\times},
(81,-4)*{\xl_1},
(85,3)*{\bv_1},
\ar @{.} (0,0);(5,0)
\ar @{-} (5,0);(85,0)
\ar @{.} (85,0);(90,0)
\end{xy}
\end{align}
Note that $x_0$, $x_1$, \dots are in the anticlockwise order, while
$h_1$, $h_2$, \dots are in the clockwise order.

Before discussing Coulomb branches, we review generalized Young diagrams and their relations to weights of affine Lie algebras.

Let $\algsl(n)_\aff$ (resp.\ $\widehat{\algsl}(n)$) be the affine Lie
algebra of type $A_{n-1}$, which includes (resp.\ does not include)
the degree operator $d$ in the Cartan subalgebra.
We assign two $\algsl(n)_\aff$ weights $\lambda := \sum_{i=0}^{n-1}
\bw_i \Lambda_i + \bv_0\delta$, $\mu := \sum_{i=0}^{n-1} \bw_i \Lambda_i -
\bv_i\alpha_i + \bv_0\delta$ as in finite case. Our convention is $\langle
d,\Lambda_i\rangle = 0$, $\langle d,\alpha_i\rangle = \delta_{i0}$, $\delta = \alpha_0+\dots+\alpha_{n-1}$. We add $\bv_0\delta$ so that $\langle d, \mu\rangle = 0$. This is a different convention from the quiver variety setting, but simultaneous shifts of $\lambda$, $\mu$ by an element in $\ZZ\delta$ does not affect anything in practice. 
The projection 
\(
   \text{($\algsl(n)_\aff$-weight lattice)}\to
   \text{($\widehat{\algsl}(n)$-weight lattice)}
\)
is given by the quotient modulo $\ZZ\delta$.

We use the notion of a \emph{generalized Young diagram with the level
  $\ell$ constraint} as in \cite[App.~A]{Na-branching}. A sequence of
integers $[\lambda_1,\dots,\lambda_n]$ is a generalized Young diagram
with the level $\ell$ constraint if it satisfies the inequalities
\begin{equation*}
    \begin{gathered}[m]
    \lambda_1 \ge \lambda_2 \ge \dots \ge \lambda_n,
    \\
    \lambda_1 - \lambda_n
    \begin{NB}
        = \sum_{j=1}^{n-1} \bw_i = \ell - \bw_0
    \end{NB}
    \le \ell.
    \end{gathered}
\end{equation*}

As in \eqref{eq:5} we introduce two integral vectors
$[\lambda_1,\lambda_2,\dots,\lambda_n]$,
$[\mu_1,\mu_2,\dots,\mu_n]\in\ZZ^n$ by
\begin{equation}
    \label{eq:9}
        \lambda_i = \sum_{j=i}^{n-1} \bw_j, \qquad
        \mu_i = \bv_{n-1} - \bv_0 + \sum_{j=i}^{n-1} \bu_j,\quad \text{for }i=1,
        2, \cdots, n-1, \ \lambda_n = 0, \ \mu_{n} = \bv_{n-1} - \bv_0,
\end{equation}
where $\underline{\bu} = \underline{\bw} - C\underline{\bv} = (\bw_i +
\bv_{i-1} + \bv_{i+1} - 2\bv_i)_{i=0}^{n-1}$. (The index is modulo
$n$.) Then $[\lambda_1,\dots,\lambda_n]$ is a generalized Young
diagram with the level $\ell$ constraint.
On the other hand,
\begin{equation*}
    \begin{gathered}[m]
    \mu_1 - \mu_2 = \bu_1, \quad \mu_2 - \mu_3 = \bu_2, \quad \dots,\quad
    \mu_{n-1} - \mu_n = \bu_{n-1},\\
    \mu_1 - \mu_n = \sum_{j=1}^{n-1} \bu_j = \ell - \bu_0.
    \end{gathered}
\end{equation*}
Therefore $\mu$ is dominant if and only if $[\mu_1,\dots,\mu_n]$ is a
generalized Young diagram with the level $\ell$ constraint.
Note also
\begin{equation*}
    |\lambda| := \sum_{i=1}^n \lambda_i = \sum_{i=1}^n \mu_i =: |\mu|
\end{equation*}
from the definition.

Remark that $\mu$ is not recovered from $[\mu_1,\dots,\mu_n]$, as
there is an ambiguity of $\ZZ\delta$. In fact, the set of generalized
Young diagrams with the level $\ell$ constraint is bijective to the
set of dominant $\widehat{\gl}(n)$-weights of level $\ell$. The
projection 
\(
   \text{$\widehat{\gl}(n)$-weights}\to
   \text{$\widehat{\algsl}(n)$-weights}
\)
is identified with the quotient modulo \emph{shifts}
\(
   [\mu_1,\dots,\mu_n] \mapsto [\mu_1 + k,\dots,\mu_n + k]
\)
for $k\in\ZZ$.
The above $|\lambda|$ is the \emph{charge}, which is the eigenvalue of
$J(0)$, the $0$th Fourier mode of the Heisenberg algebra added to
$\widehat{\algsl}(n)$ to get $\widehat{\gl}(n)$. Hence the constraint
$|\lambda|=|\mu|$ is natural. A generalized Young diagram together with an integer $\langle d,\lambda\rangle$ bijectively corresponds to a dominant $\gl_\aff(n)$-weights.

We have $\lambda_1 - \mu_1 + \bv_0
\begin{NB}
   = \sum_{j=1}^{n-1} (\bw_j - \bu_j)- \bv_{n-1} + 2\bv_0 = 
   \sum_{j=1}^{n-1} (-\bv_{j-1} - \bv_{j+1} +2\bv_j) - \bv_{n-1} + 2\bv_0 
\end{NB}%
 = \bv_1$, $\lambda_1 + \lambda_2 - (\mu_1 + \mu_2) + \bv_0
\begin{NB} 
  = (\lambda_1 - \mu_1 + \bv_0) + (\lambda_2 - \mu_2) = \bv_1 + (\bv_2 - \bv_1)
\end{NB}%
= \bv_2$, \dots, $(\lambda_1 + \dots + \lambda_{n-1}) - (\mu_1 + \dots + \mu_{n-1}) + \bv_0 = \bv_{n-1}$. Thus $\lambda \ge \mu$ in the dominance order if and only if $\sum_{j=1}^i \lambda_j + \langle d, \lambda-\mu\rangle \ge \sum_{j=1}^i \mu_j$ for any $i=1,\dots,n$ and $\langle d,\lambda-\mu\rangle\ge 0$.

We draw a generalized Young diagram as
\begin{equation*}
   \text{\scriptsize $n$ boxes }\Big\{
   \Yvcentermath1
   \cdots
   \underbrace{\overset{-3/2}{\young(\rf\rf\rf,\rf\rf\rf)}}_{\text{$\ell$ boxes}}
   \,
   \overset{-1/2}{\young(\rf\rf\rf,\rf\rf\hf)}
   \,
   \overset{1/2}{\young(\rf\rf\hf,\hf\hf\hf)}
   \,
   \overset{3/2}{\young(\hf\hf\hf,\hf\hf\hf)}
   \cdots,
\end{equation*}
where a box is indexed by $(i,\xp,N)$ with $1\le i\le n$, $1\le \xp\le
\ell$, $N\in\ZZ+1/2$ and we put a gray box $\graysquare$ if
$\ell(N-1/2)+\xp\le \lambda_i$. The above figure is
$[\lambda_1,\lambda_2] = [2, -1]$ for $n=2$, $\ell = 3$.

We define the \emph{transpose} of $\lambda$ by flipping each
$n\times\ell$ rectangle along its diagonal. It has $\ell$ entries and
satisfies the level $n$ constraint. The transpose describes the
correspondence of dominant weights under the level-rank duality (see
\cite[App.~A]{Na-branching}). Let us denote the transpose as
$[{}^t\!\lambda_1,{}^t\!\lambda_2,\dots,{}^t\!\lambda_\ell]$. The
above example gives $[{}^t\!\lambda_1,{}^t\!\lambda_2,{}^t\!\lambda_3]
= [1,1,-1]$. Note that the transpose of a shifted diagram involves a
`vertical shift' of the transposed diagram. The vertical shift
corresponds to a cyclic rotation of the affine Dynkin diagram of
type $A_{n-1}$.
Thus the transpose should be understood as a map from the set of level $\ell$ dominant $\widehat{\algsl}(n)$ weights to the set of level $n$ dominant $\widehat{\algsl}(\ell)$ weights modulo cyclic rotations.


Let us return back to the bow diagram with the balanced condition.

\begin{Lemma}
    Consider the bow diagram with the balanced condition associated
    with dimension vectors $\underline{\bv}$,
    $\underline{\bw}\in\ZZ_{\ge 0}^n$ as in \eqref{eq:8}.
    \begin{NB}
    Consider a bow variety satisfying the dual balanced condition. We
    define dimension vectors $\underline{\bv}$, $\underline{\bw}$ as
    in \subsecref{subsec:Coulomb_gauge} after exchanging $\boldsymbol\times$ and
    $\boldsymbol\medcirc$, i.e., the bow variety is isomorphic to a quiver
    variety $\fM(\underline{\bv},\underline{\bw})$.
    \end{NB}%
    Then the followings hold\textup:

    \textup{(1)}
    \(
        N(h_\xp,h_{\xp+1}) = {}^t\!\lambda_{\xp} - {}^t\!\lambda_{\xp+1}
    \)
    for $1\le\xp\le \ell-1$. And
    \(
        N(h_\ell,h_1) = n - ({}^t\!\lambda_1 - {}^t\!\lambda_{\ell}).
    \)
    In particular, $N(h_\xp,h_{\xp+1})$ \textup($1\le \xp\le
    \ell-1$\textup), $N(h_\ell,h_1)$ are always nonnegative.

    \textup{(2)} $N(\xl_i, \xl_{i+1})$ is equal to the $i$th entry of
    $\underline{\bu} = \underline{\bw} - C\underline{\bv}$. Therefore
    $N(\xl_i,\xl_{i+1})\ge 0$ for any $i$ if and only if $\sum
    \bw_i\Lambda_i - \bv_i \alpha_i$ is dominant.
\label{lem:dominant-bow}
\end{Lemma}
\begin{proof}
    The proof of (2) is exactly the same as in \lemref{lem:bowN}. The
    proof of (1) is also almost the same, as the nontrivial part of
    our diagram sits in the rectangle at $N=1/2$, and regarded as a
    usual Young diagram. Let us check the formula for
    $N(h_\ell,h_1)$. First note that $h_\ell$ (resp.\ $h_1$) is
    between $\xl_i$ and $\xl_{i+1}$ (resp.\ $\xl_j$ and $\xl_{j+1}$) with
    $i={}^t\!\lambda_\ell$, $j = {}^t\!\lambda_1$. Since we count the
    number of $\boldsymbol\times$ in the arc $h_1 \to h_\ell$, it is equal to $n
    - (j-i)$.
\end{proof}

\begin{Proposition}\label{prop:unique}
    A bow diagram satisfying the balanced condition is determined by
    $N(\xl_i,\xl_{i+1})$ \textup($0\le i\le n-1$\textup),
    $N(h_\xp,h_{\xp+1})$ \textup($1\le \xp\le \ell$\textup) plus either of
    \begin{equation*}
        - \sum_{\xp=1}^\ell N_{h_\xp}^2
        + \sum_{i=0}^{n-1} (\bv_{\zeta_i^-} + \bv_{\zeta_i^+})
        \quad\text{or}\quad
        - \sum_{i=0}^{n-1} N_{\xl_i}^2 + 
        \sum_{\xp=1}^{\ell} (\bv_{\vin{h_\xp}} + \bv_{\vout{h_\xp}}).
    \end{equation*}
    Here $\zeta_i^- \to \xl_i \to \zeta_i^+$.
    In particular, there is at most one bow diagram satisfying the
    balanced condition among those obtained by successive application
    of Hanany-Witten transitions.
\end{Proposition}

\begin{proof}
    Since $N(\xl_i,\xl_{i+1})$, $N(h_\xp,h_{\xp+1})$ and two invariants
    above are preserved under Hanany-Witten transitions, the last
    assertion is a consequence of the first assertion. 

    Let us prove the first assertion. Consider a balanced bow diagram for
    $\underline{\bw}$, $\underline{\bv}$. 
    Since $N(h_\xp,h_{\xp+1})$ is the number of $\boldsymbol\times$ between $h_{\xp+1}$ and
    $h_{\xp}$\begin{NB}$h_\xp$ and $h_{\xp+1}$5/28\end{NB}, the collection $\{N(h_\xp,h_{\xp+1})\}_{\xp=1}^\ell$
    determines distribution of $\boldsymbol\medcirc$ and $\boldsymbol\times$, and hence
    $\underline{\bw}$.
    
    By \lemref{lem:dominant-bow}(2), $\{ N(\xl_i,\xl_{i+1})\}$ determines
    $\underline{\bw} - C \underline{\bv}$. It defines $\underline{\bv}$
    up to an addition of a multiple of ${}^t(1,1,\dots,1)$.
    
    If we add ${}^t(1,1,\dots,1)$ to $\underline{\bv}$, the above two
    invariants increase by $2n$ and $2\ell$ respectively. Therefore
    either of two invariants fixes the ambiguity.
\end{proof}

\begin{Proposition}
    If $N(\xl_i, \xl_{i+1}) \geq 0$ for any $i$, the bow diagram is
    transformed to one with a cobalanced dimension vector by
    successive application of Hanany-Witten transitions of
    Proposition~\ref{prop:HW-trans}. \textup(It is unique by
    \propref{prop:unique}.\textup) Here we assume that the bow diagram
    has at least one $\boldsymbol\medcirc$.
    \label{prop:balanced_bow}
\end{Proposition}

\begin{NB}
\begin{proof}
We prove that all $N_{\xl_i}$ can be made simultaneously to $0$ by applying Hanany-Witten transitions several times.
When $N_{\xl_{i+1}} - N_{\xl_{i}} < 0$, we can use Hanany-Witten transition for $\xl_{i+1}$-th $\boldsymbol\times$ and its interior $\boldsymbol\medcirc$ some times in order to make $N_{\xl_{i+1}} - N_{\xl_i}$ to be $0$ because of $N(\xl_i, \xl_{i+1}) \geq 0$.
\begin{align*}
\begin{xy}
(0,3)*{5},
(5,0)*{\boldsymbol\times},
(5,-4)*{\xl_{i}},
(0,-8)*{N_{\xl_{i}}=-3},
(10,3)*{2},
(15,0)*{\boldsymbol\medcirc},
(20,3)*{\cdots},
(20,-5)*{\underbrace{\hspace{15mm}}_{\text{at least $2$}}},
(25,0)*{\boldsymbol\medcirc},
(30,3)*{8},
(35,0)*{\boldsymbol\times},
(35,-4)*{\xl_{i+1}},
(40,-8)*{N_{\xl_{i+1}}=-5},
(40,3)*{3},
(50,0)*{\Rightarrow},
(60,3)*{5},
(65,0)*{\boldsymbol\times},
(65,-4)*{\xl_{i}},
(70,3)*{2},
(75,0)*{\boldsymbol\medcirc},
(80,3)*{\cdots},
(80,-5)*{\underbrace{\hspace{15mm}}_{\text{at least $0$}}},
(85,0)*{\boldsymbol\medcirc},
(90,3)*{?},
(95,0)*{\boldsymbol\times},
(95,-4)*{\xl_{i+1}},
(100,-8)*{N_{\xl_{i+1}}=-3},
(100,3)*{?},
(105,0)*{\boldsymbol\medcirc},
(110,3)*{?},
(115,0)*{\boldsymbol\medcirc},
(120,3)*{3},
\ar @{.} (-5,0);(0,0)
\ar @{-} (0,0);(40,0)
\ar @{.} (40,0);(45,0)
\ar @{.} (55,0);(60,0)
\ar @{-} (60,0);(120,0)
\ar @{.} (120,0);(125,0)
\end{xy}
\end{align*}
Thus we get $N_{\xl_{i+1}} \geq N_{\xl_{i}}$ for all $i$, and cyclicity of the diagram implies that $N_{\xl_{i+1}} = N_{\xl_{i}}$ for all $i$ by finite times Hanany-Witten transition.
And then we use Hanany-Witten transition $N_{\xl_0}$ times for all $\boldsymbol\times$ simultaneously.
Therefore, we can get a bow variety with $N_{\xl_i}=0$ for all $i$.
\begin{align*}
\begin{xy}
(-5,4)*{N_{\xl}=},
(5,0)*{\boldsymbol\medcirc},
(15,0)*{\boldsymbol\times},
(15,4)*{-1},
(25,0)*{\boldsymbol\medcirc},
(35,0)*{\boldsymbol\medcirc},
(45,0)*{\boldsymbol\times},
(45,4)*{-1},
(55,0)*{\boldsymbol\times},
(55,4)*{-1},
(65,0)*{\boldsymbol\medcirc},
(-5,-14)*{N_{\xl}=},
(5,-10)*{\boldsymbol\times},
(5,-14)*{0},
(15,-10)*{\boldsymbol\medcirc},
(25,-10)*{\boldsymbol\medcirc},
(35,-10)*{\boldsymbol\times},
(35,-14)*{0},
(45,-10)*{\boldsymbol\times},
(45,-14)*{0},
(55,-10)*{\boldsymbol\medcirc},
(65,-10)*{\boldsymbol\medcirc},
\ar @{-} (0,0);(70,0)
\ar @{-} (0,-10);(70,-10)
\ar (14,-2);(6,-8)
\ar (44,-2);(36,-8)
\ar (54,-2);(46,-8)
\ar @{.} (25,-3);(25,-7)
\ar @{.} (65,-3);(65,-7)
\end{xy}
\end{align*}
\end{proof}

I do not understand why we can make $N_{\xl_{i+1}}\ge N_{\xl_i}$ for
all $i$, if I consider $i$ modulo $\ell$. We start from $i=0$ and move
$\boldsymbol\medcirc$ to the right. But from $\xl_{\ell-1}$, $\boldsymbol\medcirc$ goes back
to the right of $\xl = \xl_0$.....
\end{NB}%

\begin{proof}
    We prove that all $N_{\xl_i}$ can be made simultaneously to $0$ by
    applying Hanany-Witten transitions several times.

    We first apply Hanany-Witten transition so that $\boldsymbol\medcirc$ appear
    only between $\xl_{0}$ and $\xl_1$:
    
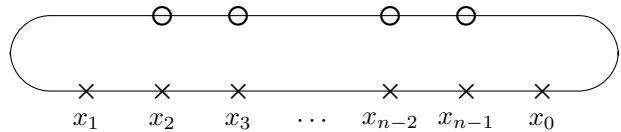
\begin{equation*}
        \begin{tikzpicture}
            \draw[rounded corners=15pt] (0,0) rectangle ++(8,1);
            \node at (1,0) {$\boldsymbol\times$};
            \node at (1,-.4) {$\xl_1$};
            \node at (2,0) {$\boldsymbol\times$};
            \node at (2,-.4) {$\xl_2$};
            \node at (2,1) {$\boldsymbol\medcirc$};
            \node at (3,0) {$\boldsymbol\times$};
            \node at (3,1) {$\boldsymbol\medcirc$};
            \node at (3,-.4) {$\xl_3$};
            \node at (4,-.4) {$\cdots$};
            \node at (5,0) {$\boldsymbol\times$};
            \node at (5,-.4) {$\xl_{n-2}$};
            \node at (5,1) {$\boldsymbol\medcirc$};
            \node at (6,0) {$\boldsymbol\times$};
            \node at (6,-.4) {$\xl_{n-1}$};
            \node at (6,1) {$\boldsymbol\medcirc$};
            \node at (7,0) {$\boldsymbol\times$};
            \node at (7,-.4) {$\xl_{0}$};
        \end{tikzpicture}
    \end{equation*}
    Then $N(\xl_1,\xl_2)$, $N(\xl_2,\xl_3)$, \dots,
    $N(\xl_{n-2},\xl_{n-1})$, $N(\xl_{n-1},\xl_0)$, $N(\xl_0,\xl_1)$
    are all nonnegative and the sum is equal to $\ell$.

    By Hanany-Witten transition, we move $N(\xl_1,\xl_2)$ of
    $\boldsymbol\medcirc$'s between $\xl_1$ and $\xl_2$, $N(\xl_2,\xl_3)$ of
    $\boldsymbol\medcirc$'s between $\xl_2$ and $\xl_3$, and so on until
    $N(\xl_{n-1},\xl_{0})$ of $\boldsymbol\medcirc$'s between $\xl_{n-1}$ and
    $\xl_{0}$. Then $N(\xl_{0},\xl_1)$ of $\boldsymbol\medcirc$'s are
    remained between $\xl_{0}$ and $\xl_1$. Since
    $N(\xl_i,\xl_{i+1})$ is unchanged, the resulted bow diagram has
    $N_{\xl_0} = N_{\xl_1} = \cdots = N_{\xl_{n-1}} = N_{\xl_0}$.

And then we use Hanany-Witten transition $N_{\xl_0}$ times for all $\boldsymbol\times$ simultaneously.
Therefore, we can get a bow variety with $N_{\xl_i}=0$ for all $i$.
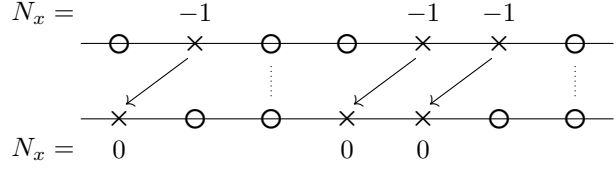
\begin{align*}
\begin{xy}
(-5,4)*{N_{\xl}=},
(5,0)*{\boldsymbol\medcirc},
(15,0)*{\boldsymbol\times},
(15,4)*{-1},
(25,0)*{\boldsymbol\medcirc},
(35,0)*{\boldsymbol\medcirc},
(45,0)*{\boldsymbol\times},
(45,4)*{-1},
(55,0)*{\boldsymbol\times},
(55,4)*{-1},
(65,0)*{\boldsymbol\medcirc},
(-5,-14)*{N_{\xl}=},
(5,-10)*{\boldsymbol\times},
(5,-14)*{0},
(15,-10)*{\boldsymbol\medcirc},
(25,-10)*{\boldsymbol\medcirc},
(35,-10)*{\boldsymbol\times},
(35,-14)*{0},
(45,-10)*{\boldsymbol\times},
(45,-14)*{0},
(55,-10)*{\boldsymbol\medcirc},
(65,-10)*{\boldsymbol\medcirc},
\ar @{-} (0,0);(70,0)
\ar @{-} (0,-10);(70,-10)
\ar (14,-2);(6,-8)
\ar (44,-2);(36,-8)
\ar (54,-2);(46,-8)
\ar @{.} (25,-3);(25,-7)
\ar @{.} (65,-3);(65,-7)
\end{xy}
\end{align*}
This final step is possible as the bow diagram has at least one $\boldsymbol\medcirc$.
\end{proof}

\begin{Corollary}\label{cor:coulomb-quiver}
    Consider a framed quiver gauge theory of affine type $A_{n-1}$
    with dimension vectors $\underline{\bv}$, $\underline{\bw}$. Let
    $\ell$ be the level of $\underline{\bw}$, which is assumed to be
    positive.
    Assume that $\underline{\bu} = \underline{\bw} -
    C\underline{\bv}\in\ZZ_{\ge 0}^n$. Then the Coulomb branch
    $\cM(\underline{\bv},\underline{\bw})$ is isomorphic to a quiver
    variety $\fM_0(\underline{\bv}',\underline{\bw}')$ of affine type
    $A_{\ell-1}$ of level $n$.
%

    Moreover $\underline{\bv}'$, $\underline{\bw}'$ are determined by
    two conditions: \textup{(a)} Two generalized Young diagrams are
    swapped and taken transpose. \textup{(b)} It satisfies
    \begin{equation}\label{eq:11}
        2\sum_{i=0}^{n-1} \bv_i 
        \begin{NB}
        = - \sum_p N_{h_\xp}^2
        + \sum_{i=0}^{n-1} (\bv_{\zeta_i^-} + \bv_{\zeta_i^+})
        \end{NB}%
        = \sum_{\xp=0}^{\ell - 1} 2 \bv'_\xp \bw'_\xp 
        - (\bv'_\xp - \bv'_{\xp-1})^2.
    \end{equation}
\end{Corollary}

Note that $N(h_\xp,h_{\xp+1})\ge 0$ for any $\xp$, as it comes from a framed
quiver gauge theory. Hence we have $\underline{\bw}' -
C\underline{\bv}'\in\ZZ_{\ge 0}^\ell$, i.e., the weight $\sum_\xp
\bw'_\xp\Lambda'_\xp - \bv'_\xp \alpha'_\xp$ is dominant. Here
$\Lambda'_\xp$, $\alpha'_\xp$ are fundamental and simple roots for
$\algsl(\ell)_\aff$.

\begin{proof}
    If $n=1$, the cobalanced condition is automatically satisfied
    for a bow diagram arising from a framed quiver gauge
    theory.

    If $n > 1$, we get $N(\xl_i, \xl_{i+1}) \geq 0$ for any $i$ by
    Lemma~\ref{lem:dominant-bow}. We now use
    Proposition~\ref{prop:balanced_bow} to get a bow variety with a
    cobalanced dimension vector. It is a quiver variety by Theorem
    \ref{thm:balanced_bow}.

Let us determine $\underline{\bv}'$, $\underline{\bw}'$
explicitly. Let us start from the bow diagram in \eqref{eq:10} and
move all $\boldsymbol\medcirc$ between $\xl_0$ and $\xl_1$ as above. We suppose
that no $\boldsymbol\medcirc$ cross $\xl_0$ so that dimensions $\bv_0$,
$\bv_{n-1}$ adjacent to $\xl_0$ are unchanged.
We have $N_{x_0} = \bv_{n-1} - \bv_0 = \mu_n$. Then $N_{x_i} = \mu_i$
follows by \lemref{lem:dominant-bow}. 

Since $h_\ell$ is between $\xl_{{}^t\!\lambda_\ell}$ and
$\xl_{{}^t\!\lambda_{\ell}+1}$ in the original diagram, it crosses
${}^t\lambda_\ell$ of $\boldsymbol\times$ during the transitions. Therefore the
dimension of the segment to the left of $h_\ell$ is
${}^t\lambda_\ell+\bv_0$. Hence $N_{h_\ell} = {}^t\!\lambda_\ell$. It
follows that $N_{h_\xp} = {}^t\!\lambda_p$ by
\lemref{lem:dominant-bow}.
Thus we obtain
\begin{equation}\label{eq:16}
    \begin{tikzpicture}[baseline=(current  bounding  box.center),
        label distance=1pt]
        \draw[rounded corners=15pt] (0,0) rectangle ++(8,2);
        \node[label=below:$\xl_1$,label=above:$\mu_1$] (x1) at (1,0) 
        {$\boldsymbol\times$};
        \node[label=below:$\xl_2$,label=above:$\mu_2$] at (2,0)
        {$\boldsymbol\times$};
        \node[label=below:$h_1$,label=above:${}^t\!\lambda_1$] at (2,2)
        {$\boldsymbol\medcirc$};
        \node[label=below:$\xl_3$,label=above:$\mu_3$] at (3,0)
        {$\boldsymbol\times$};
        \node[label=below:$h_2$,label=above:${}^t\!\lambda_2$] at (3,2)
        {$\boldsymbol\medcirc$};
        \node at (4,-.5) {$\cdots$};
        \node at (4,1.4) {$\cdots$};
        \node[label=below:$\xl_{n-2}$,label=above:$\mu_{n-2}$] at (5,0)
        {$\boldsymbol\times$};
        \node[label=below:$h_{\ell-1}$,label=above:${}^t\!\lambda_{\ell-1}$]
        at (5,2) {$\boldsymbol\medcirc$};
        \node[label=below:$\xl_{n-1}$,label=above:$\mu_{n-1}$] at
        (6,0) {$\boldsymbol\times$};
        \node[label=below:$h_{\ell}$,label=above:${}^t\!\lambda_{\ell}$] at
        (6,2) {$\boldsymbol\medcirc$};
        \node[label=below:$\xl_{0}$,label=above:$\mu_{n}$] at (7,0)
        {$\boldsymbol\times$};
        \node at (8.5,1) {$\bv_{0}$};
        \node at (-1.2,1) {$\bv_{0}+\sum_i i\bw_i$};
    \end{tikzpicture}
\end{equation}
Here numbers ${}^t\!\lambda_\xp$, $\mu_i$ above $\boldsymbol\medcirc$, $\boldsymbol\times$ indicate
the values of $N_{h_\xp}$, $N_{\xl_i}$ respectively. Two numbers $\bv_0$ and
$\bv_0+ \sum_i i\bw_i$ are dimensions of vector spaces on two segments,
between $\xl_0$ and $h_\ell$, $h_1$ and $\xl_1$ respectively.

Note that the dimension on the segment between $x_1$ and $h_1$ are
calculated from $\bv_0$ by adding either $N_{x_i}$ or $N_{h_\xp}$. The
answer is $\bv_0 + \sum_i i\bw_i$, as $|\lambda| = |\mu| = \sum_i i\bw_i$.

We take a bow diagram with the cobalanced condition associated with
dimension vectors $(\underline{\bv}',\underline{\bw}')$. The result
should be the same as above.
Then it is clear that the pair of two generalized Young diagrams
associated with $\underline{\bw}'$, $\underline{\bv}'$ is
$({}^t\!\mu,{}^t\!\lambda)$.
The remaining constraint is
\begin{equation*}
    2\sum_{i=0}^{n-1} \bv_i = - \sum_p N_{h_\xp}^2
    + \sum_{i=0}^{n-1} (\bv_{\zeta_i^-} + \bv_{\zeta_i^+})
    = \sum_{\xp=0}^{\ell - 1} 2 \bv'_\xp \bw'_\xp - (\bv'_\xp - \bv'_{\xp-1})^2
\end{equation*}
by \lemref{lem:dim_unchange}. It can be checked also directly: The
left hand side is the dimension of the Coulomb branch for
$(\underline{\bv},\underline{\bw})$. The right hand side is the
dimension of the Higgs branch for
$(\underline{\bv}',\underline{\bw}')$.
\end{proof}

\remref{rem:deform} applies also to this case: We have an isomorphism
$\cM_\nu(\underline{\bv},\underline{\bw}) \cong \fM_\nu(\underline{\bv}',\underline{\bw}')$ for any parameter $\nu$.

\begin{Remark}
    As mentioned in Introduction, \corref{cor:coulomb-quiver} gives a
    proof of the conjecture that the Coulomb branch $\mathcal
    M_C(\underline{\bv},\underline{\bw})$ is a quiver variety of an
    affine type $A_{\ell-1}$ introduced in \cite{Na-quiver} when
    $\sum_i \bw_i \Lambda_i - \bv_i \alpha_i$ is dominant
    \cite{2015arXiv150303676N}.
    Its dimension vectors are given by the same way as above up to a
    diagram automorphism. It is compatible with the level rank duality
    and the conjectural geometric Satake correspondence for affine
    Kac-Moody groups \cite{braverman-2007}. It is because the
    conjecture was based on \cite{Na-branching}, where a proposal of
    \cite{braverman-2007} was checked for type $A$ using the level
    rank duality.
    \begin{NB}
        More precisely, the equality \eqref{eq:11} is contained in
        \cite[Th.~4.8]{braverman-2007}: the dimension of
        $\fM_0(\underline{\bv}',\underline{\bw}')$ is $2|\lambda-\mu|$
        in the notation in \cite{braverman-2007}, which is equal to
        our $\sum \bv_i$.
    \end{NB}%
\end{Remark}

\begin{Remark}\label{rem:finite_move}
    Consider a bow diagram of finite type $A_{n-1}$. We consider it
    as affine type $A_{n-1}$ with $\dim V_{\zeta^-} = 0 = \dim
    V_{\zeta^+}$ where $\zeta^-\to 0 \to \zeta^+$, e.g., $\bv_0 = 0 =
    \bv'_{n-1}$ in Figure~\ref{fig:Bow}. Then the above results can be
    applied to a finite type bow diagram:

    \begin{itemize}
          \item In \lemref{lem:dominant-bow}, the condition
        $N(x_0,x_1) \ge 0$ is always satisfied as $N(x_0,x_1) = \bv_1
        + \bv_{n-1}$. Also $N(x_{n-1},x_0) = \bv_{n-2} + \bw_{n-1} -
        2\bv_{n-1} = \bu_{n-1}$. Therefore $N(x_{n-1},x_0)\ge 0$ is a
        part of the dominance condition.

          \item As for \propref{prop:balanced_bow}, we perform the
        movement of $\boldsymbol\medcirc$ in the proof in the finite
        part, that is $\boldsymbol\medcirc$ does not cross $x_0$. Then
        the condition $\bv_0 = 0 = \bv'_{n-1}$ is preserved. Moreover,
        $\boldsymbol\medcirc$ remained between $\xl_0$ and $\xl_{n-1}$
        are irrelevant: Since $\bv'_{n-1} = 0$ and the cobalanced
        condition is satisfied, dimensions between $\xl_0$ and
        $\xl_{n-1}$ are all $0$. In particular $\bv_{n-1} = 0$, and we
        can remove $\boldsymbol\medcirc$.
\end{itemize}
\end{Remark}

\subsection{Stratification for affine type \texorpdfstring{$A$}{A}}\label{subsec:strat-affine-A}

Now we determine the stratification \eqref{eq:13} for the affine
case. Consider the Coulomb branch of the framed quiver gauge theory
for dimension vectors $\underline{\bv}$, $\underline{\bw}$. We assign
$\lambda := \sum_{i=0}^{n-1} \bw_i\Lambda_i + \bv_0\delta$, $\mu := \sum_{i=0}^{n-1}
\bw_i\Lambda_i - \bv_i\alpha_i + \bv_0\delta$ as before. We have a similar
result as in finite case.

\begin{Theorem}\label{thm:strat-affine-type}
    \textup{(1)} Suppose $\ell\neq 1$. We have a stratification
    \begin{equation*}
        \cM(\mu,\lambda)
        = \bigsqcup_{\kappa,\underline{k}} \cM^{\mathrm{s}}(\mu,\kappa)
        \times S^{\underline{k}}(\CC^2\setminus\{0\}/(\ZZ/\ell\ZZ)),
    \end{equation*}
    where $\underline{k} = [k_1,k_2,\dots]$ is a partition, and $\kappa$
    is a dominant weight between $\mu$ and
    $\lambda-|\underline{k}|\delta$.
    We have 
    \begin{equation*}
        \cM^{\mathrm{s}}(\mu,\kappa)
        \begin{cases}
            \neq\emptyset & \text{if $n\neq 1$ or $\mu = \kappa$},
            \\
            = \emptyset & \text{if $n=1$ and $\mu \neq \kappa$}.
        \end{cases}
    \end{equation*}
%
%
    Moreover a transversal slice to $\cM^{\mathrm{s}}(\mu,\kappa)\times
    S^{\underline{k}}(\CC^2\setminus\{0\}/(\ZZ/\ell\ZZ))$ in
    $\cM(\mu,\lambda)$ is isomorphic to
    $\cM(\kappa,\lambda-|\underline{k}|\delta)\times
    {}^{\mathrm{c}}\mathcal U^{k_1}_n\times {}^{\mathrm{c}}\mathcal
    U^{k_2}_n\times\cdots$.

    \textup{(2)} Suppose $\ell = 1$. The same is true if we replace
    $\CC^2\setminus\{0\}/(\ZZ/\ell\ZZ)$ by $\CC^2$ and we only allow
    $\kappa = \lambda - |\underline{k}|\delta$.
    \begin{NB}
        In particular, the first factor
        $\cM(\kappa,\lambda-|\underline{k}|\delta)$ is just $\cM(\kappa,\kappa)
        = \text{point}$ for the transversal slice.
    \end{NB}%
\end{Theorem}

Here ${}^{\mathrm{c}}\mathcal U^{k}_n$ denotes the centered Uhlenbeck
space for rank $n$, instanton charge $k$. In other words, it is the
factor of the quiver variety $\fM_0(k,n)$ associated with the Jordan
quiver with dimension vectors $(k,n)$ such that $\fM_0(k,n) =
{}^{\mathrm{c}}\mathcal U^{k}_n \times \AAA^2$.

\begin{NB}
    We have $\dim {}^{\mathrm{c}}\mathcal U^{k}_n =
    2(kn-1)$. Therefore 
    \(
      \dim {}^{\mathrm{c}}\mathcal U^{k_1}_n\times {}^{\mathrm{c}}\mathcal
      U^{k_2}_n\times\cdots + \dim
      S^{\underline{k}}(\CC^2\setminus\{0\}/(\ZZ/\ell \ZZ)) 
      = 2|\underline{k}|n.
    \)
    We now check that
    \begin{equation*}
        \begin{split}
            & \frac12{\dim \cM(\mu,\lambda)} = |\lambda - \mu|
        = |\kappa - \mu| + |\lambda - |\underline{k}|\delta - \kappa| +  
        2|\underline{k}|n
        \\
        =\; & \frac12 \dim \cM(\mu,\kappa)
        + \frac12 \dim \cM(\kappa, \lambda - |\underline{k}|\delta)
        + \frac12 \dim {}^{\mathrm{c}}\mathcal U^{k_1}_n\times
        {}^{\mathrm{c}}\mathcal U^{k_2}_n\times\cdots + 
        \frac12 \dim S^{\underline{k}}(\CC^2\setminus\{0\}/(\ZZ/\ell \ZZ)).
        \end{split}
    \end{equation*}
\end{NB}%

\begin{proof}
    The proof goes as in \thmref{thm:strat-finite-type}, but we
    cannot use nilpotent orbits, so we will use quiver varieties
    instead.

    We start with a bow diagram with the balanced condition, and apply
    Hanany-Witten transition to arrive at \eqref{eq:16}. 
    Let us consider the stratification \eqref{eq:13} for
    \eqref{eq:16}. Since $\nu=0$, $\underline{\bv}^\alpha$ is a
    coordinate vector of the affine Dynkin diagram of type
    $A_{\ell-1}$. For segments between $h_1$ and $h_2$, \dots,
    $h_{\ell-1}$ and $h_{\ell}$,
    $\fM^{\mathrm{s}}(\underline{\bv}^\alpha)$ is just given by one
    dimension vector space on the segment and $0$ elsewhere (and hence
    linear maps are all $0$). However we have $\boldsymbol\times$ between
    $h_\ell$ and $h_1$, hence the corresponding
    $\fM^{\mathrm{s}}(\underline{\bv}^\alpha)$ is given by the data
    for $\CC$ on all segments between $h_1$ and $h_\ell$ through
    $x_1$, $x_2$, \dots, $x_0$. Therefore the remaining vector
    $\underline{\bv}'$ has the following form:
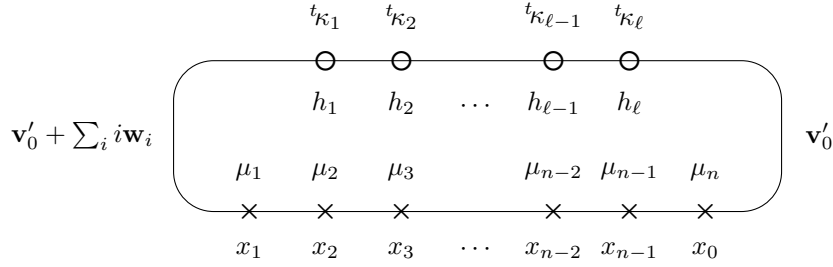
\begin{equation*}
    \begin{tikzpicture}[baseline=(current  bounding  box.center),
        label distance=1pt]
        \draw[rounded corners=15pt] (0,0) rectangle ++(8,2);
        \node[label=below:$\xl_1$,label=above:$\mu_1$] (x1) at (1,0) 
        {$\boldsymbol\times$};
        \node[label=below:$\xl_2$,label=above:$\mu_2$] at (2,0)
        {$\boldsymbol\times$};
        \node[label=below:$h_1$,label=above:${}^t\!\kappa_1$] at (2,2)
        {$\boldsymbol\medcirc$};
        \node[label=below:$\xl_3$,label=above:$\mu_3$] at (3,0)
        {$\boldsymbol\times$};
        \node[label=below:$h_2$,label=above:${}^t\!\kappa_2$] at (3,2)
        {$\boldsymbol\medcirc$};
        \node at (4,-.5) {$\cdots$};
        \node at (4,1.4) {$\cdots$};
        \node[label=below:$\xl_{n-2}$,label=above:$\mu_{n-2}$] at (5,0)
        {$\boldsymbol\times$};
        \node[label=below:$h_{\ell-1}$,label=above:${}^t\!\kappa_{\ell-1}$]
        at (5,2) {$\boldsymbol\medcirc$};
        \node[label=below:$\xl_{n-1}$,label=above:$\mu_{n-1}$] at
        (6,0) {$\boldsymbol\times$};
        \node[label=below:$h_{\ell}$,label=above:${}^t\!\kappa_{\ell}$] at
        (6,2) {$\boldsymbol\medcirc$};
        \node[label=below:$\xl_{0}$,label=above:$\mu_{n}$] at (7,0)
        {$\boldsymbol\times$};
        \node at (8.5,1) {$\bv'_{0}$};
        \node at (-1.2,1) {$\bv'_{0}+\sum_i i\bw_i$};
    \end{tikzpicture}
\end{equation*}
The point is that $N_{\xl_i}$ remains $\mu_i$ ($\mu_n$ for
$N_{\xl_0}$). Vector spaces on segments 
\begin{tikzpicture}[baseline=0pt]
    \draw (0.1,0) circle (0.1);
    \node at (1,0) {$\boldsymbol\times$};
    \node at (0,0.35) {$h_1$};
    \draw[-] (0.2,0) -- (1,0);
    \node at (1,0.35) {$\xl_1$};
\end{tikzpicture}
and
\begin{tikzpicture}[baseline=0pt]
    \node at (0,0) {$\boldsymbol\times$};
    \node at (0,0.35) {$\xl_0$};
    \draw (1,0) circle (0.1);
    \draw[-] (0,0) -- (0.9,0);
    \node at (1,0.35) {$h_\ell$};
\end{tikzpicture}
decrease their dimensions in the same amount, hence they are $\bv_0'$
and $\bv_0' + \sum i\bw_i$. There is a constraint $|{}^t\!\kappa| := \sum
{}^t\!\kappa_\xp = |\mu|$. In fact, starting from $\bv_0'$, calculate
dimensions of vector spaces going the lower and upper halves. Two
answers must match at $\bv_0' + \sum i\bw_i$.

\begin{Claim}
    $[{}^t\!\kappa_1,\dots, {}^t\!\kappa_\ell]$ is a generalized Young
    diagram with the level $n$ constraint.
\end{Claim}

We consider the upper half as a quiver variety as in the proof of
\thmref{thm:strat-finite-type}. By the necessary condition for
$\fM_0^{\mathrm{reg}}\neq\emptyset$ in \cite[Lem.~8.1]{Na-quiver},
\cite[Lem.~4.7]{Na-alg}, we have ${}^t\!\kappa_1\ge{}^t\!\kappa_2\ge\dots
\ge{}^t\!\kappa_\ell$. Next we move $h_\ell$ to the left of $h_1$ by
Hanany-Witten transitions going the lower half. Then ${}^t\!\kappa_\ell$
is changed to ${}^t\!\kappa_\ell + n$, as $N(h_\ell, h_1)$ must be
preserved. Then the necessary condition says ${}^t\!\kappa_\ell + n \ge
{}^t\kappa_1$, which is nothing but the level $n$ constraint. The claim
is proved.

Hence we have $N(h_\xp,h_{\xp+1})$, $N(h_\ell,h_1)\ge 0$, thus
corresponding to a unique balanced bow diagram by
\propref{prop:balanced_bow}. Thus we have a pair $(\kappa,\mu)$ of $\gl(n)_\aff$-weights such that $\kappa$ is dominant and $\kappa\ge\mu$. Note that $\mu$ does not change even as a $\gl(n)_\aff$-weight thanks to our convention $\langle d,\mu\rangle = 0$.
The associated generalized Young diagram $[\kappa_1,\dots,\kappa_n]$ with the level $\ell$ constraint is given by the transpose of the above vector $[{}^t\!\kappa_1,\dots,{}^t\!\kappa_\ell]$. We have $|\kappa| = |{}^t\!\kappa| = |\mu|$. We also have $\langle d,\kappa\rangle = \bv_0'$.

Since entries of $\underline{\bv} - |\underline{k}|(1,1,\dots,1) -
\underline{\bv}'$ in \eqref{eq:13} are nonnegative, we have
\begin{equation*}
   \bv'_0 \le \bv_0 -  |\underline{k}|, \quad 
   \bv'_0 + {}^t\!\kappa_\ell \le \bv_0 - |\underline{k}| + {}^t\!\lambda_\ell,
   \quad \dots, \quad
   \bv'_0 + {}^t\!\kappa_1 + \cdots + {}^t\!\kappa_\ell \le 
   \bv_0 - |\underline{k}| + {}^t\!\lambda_1 + \cdots + {}^t\!\lambda_\ell.
\end{equation*}
This is equivalent to $\kappa\le\lambda-|\underline{k}|\delta$.
\begin{NB}
 This should be correct, but I need checks.
\end{NB}%

Let us study the nonemptyness of $\cM^{\mathrm{s}}(\mu,\kappa)$. The
argument is almost the same as for finite type, but we need to
consider an extra factor
$S^{\underline{k}}(\CC^2\setminus\{0\}/(\ZZ/\ell \ZZ))$ or
$S^{\underline{k}}(\CC^2)$. It has dimension $2l(\underline{k})$. On
the other hand, if we remove $|\underline{k}|\delta$ from $\kappa$, the
dimension of $\cM^{\mathrm{s}}(\mu,\kappa)$ drops by
$2|\underline{k}|n$. Therefore the extra factor is smaller dimensional
if $n\neq 1$, and the same argument as the finite case works.

Next suppose $n=1$. The balanced condition implies the cobalanced
one in this case. Therefore $\cM(\mu,\kappa)$ is a quiver variety of
affine type $A_{\ell-1}$ with level $1$. By
\cite[Prop.~2.10]{Lecture}, we have $a = b = 0$ automatically, hence
$\cM^{\mathrm{s}}(\mu,\kappa)=\emptyset$ unless $\mu = \kappa$. (One can
also use the criterion in \cite{CB}.)

Finally the description of the transversal slice follows from
\propref{prop:local}.
\end{proof}

\begin{Remark}\label{rem:non_balanced}
    The same argument shows that arbitrary bow variety not necessarily
    satisfying the balanced condition has a stratification of the same
    form. In fact, we move to \eqref{eq:16} in the beginning of the
    proof, which is possible without assuming $\lambda$ is
    dominant. In the remaining argument, we consider $\kappa$ instead
    of $\lambda$, and the dominance of $\lambda$ is not used. In
    particular, any bow variety is homeomorphic to one with the
    balanced condition.
\end{Remark}

\subsection{Torus fixed points}

Let us consider the fixed point set of the $(\CC^\times)^n$-action
introduced in \subsecref{subsubsec:hamilt-torus-acti}. We do not
consider its scheme theoretic structure, and are interested only in
the underlying topological space. 

For study of quiver varieties, it is useful to analyze torus fixed
points. See e.g., \cite{MR1285530}. However situation for bow
varieties is different as the following analysis indicates.

Let us start with an example.

\begin{Example}
    Consider the bow variety of type $A_1$ with dimension vectors
    $(\bv_1,\bw_1)$ with $\bw_1 = 1$:
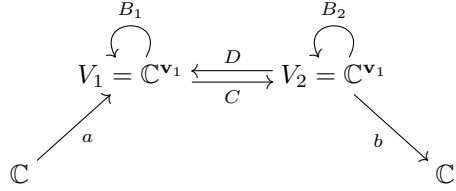
\begin{align*}
    \xymatrix@C=1.2em{
      & V_1 = \CC^{\bv_1} \ar@(ur,ul)_{B_1} \ar@<-.5ex>[rr]_{C} 
      && V_2 = \CC^{\bv_1} \ar@(ur,ul)_{B_2} \ar@<-.5ex>[ll]_{D} \ar[dr]_b& 
  \\
\CC \ar[ur]_a && && \CC
      }
\end{align*}
The $\CC^\times$-action is given by $b\mapsto t b$ and other maps are
preserved.
There is no possibility of the second summand in
\lemref{lem:direct_sum}, hence $\cM = \cM^{\mathrm{s}}$. Then a
$\CC^\times$-fixed point has a representative $(B_1,B_2,a,b,C,D)$ such
that there are $1$-parameter subgroups $\rho_1\colon \CC^\times \to
\GL(V_1)$, $\rho_2\colon\CC^\times\to \GL(V_2)$ such that
\begin{equation*}
    (B_1,B_2,a,tb,C,D) = (\rho_1(t)B_1 \rho_1(t)^{-1}, \rho_2(t)B_2 \rho_2(t)^{-1}, \rho_1(t)a, b \rho_2(t)^{-1}, 
    \rho_2(t) C \rho_1(t)^{-1}, \rho_1(t) D \rho_2(t)^{-1}).
\end{equation*}
We decompose $V_1 = \bigoplus V_1(m)$, $V_2 = \bigoplus V_2(m)$ into
weight spaces. Then $\Ima a\subset V_1(0)$,
$B_1(V_1(m)) \subset V_1(m)$. By (S2) we must have $V_1 =
V_1(0)$.
Similarly $V_2 = V_2(-1)$. Then we must have $C = 0 = D$, hence
$B_1 = DC = 0$, $B_2 = CD = 0$. Therefore (S1) and (S2) imply
$\bv_1 = 1$. In other words, $\cM^{\CC^\times}$ is a single point if
$\bv_1 = 1$, and is empty if $\bv_1 > 1$.

\begin{NB}
    Let us apply the Hanany-Witten transition. We get
    \begin{align*}
    \xymatrix@C=1.2em{
      & V_1 = \CC^{\bv_1} \ar@(ur,ul)_{B_1} \ar[rr]_{A} \ar[dr]_b
      && V_2^\tn = \CC \ar@(ur,ul)_{B_2} & 
  \\
\CC \ar[ur]_{a_1} && \CC \ar[ur]_{a_2} &&
      }
\end{align*}
Since $C$, $D$ automatically disappear (hence are not drawn), we have
$B_2 = 0$ also.
Consider the $\CC^\times$-fixed point. We have
\begin{equation*}
    (A, B_1,B_2,a_1,tb,t^{-1} a_2) = 
    (\rho_1(t) A \rho_2(t)^{-1}, \rho_1(t)B_1 \rho_1(t)^{-1}, 
    \rho_2(t)B_2 \rho_2(t)^{-1}, \rho_1(t)a_1, b \rho_1(t)^{-1}, 
    \rho_2(t)a_2).
\end{equation*}
We have $V_1 = V_1(0)$ as above. Thus $b = 0$. If $V_2^\tn =
V_2^\tn(n)$ with $n\neq 0$, then $A$ is also $0$. But $S_1 = V_1$
violates (S1). Therefore $V_2^\tn = V_2^\tn(0)$, hence $a_2=0$. Since
$0 = B_2 A = A B_1$, $\Ker A$ is invariant under $B_1$. Hence $\Ker A
= 0$ by (S1), but this is possible only when $\bv_1 = 1$. If $\bv_1 =
1$, $A$ is an isomorphism, and $B_1 = 0$ also. Thus the fixed point
set is a single point.
\end{NB}%

This is compatible with Hikita conjecture \cite{2015arXiv150102430H}
(see also \cite[\S1(viii)]{2015arXiv150303676N}): the cohomology ring
of Higgs branch with generic moment map level is isomorphic to the
coordinate ring of the torus fixed point subscheme. In this example,
Higgs branch is a quiver variety of type $A_1$ with dimension vectors
$(\bv_1,\bw_1=1)$. It is a single point if $\bv_1 = 1$, and is empty
if $\bv_1 > 1$.

Note that the coordinate ring of the empty set is the $0$-dimensional
vector space. And the cohomology ring of the empty set is also the
$0$-dimensional vector space.

On the other hand, it is not clear how to formulate Hikita conjecture
when Higgs and Coulomb branches are swapped: Suppose $\bv_1 >
1$. Since $\cM = \cM^{\mathrm{s}}$, it cannot be changed by the
perturbation of the stability condition. (Any symplectic resolution
must be $\cM$ itself.)
Since the Euler number of $\cM$ is $0$, the cohomology group of $\cM$
is nontrivial at least in two degrees, the degree $0$ and an odd
degree. In the usual formulation, we implicitly assume there is no odd
cohomology group, as the coordinate ring side is always commutative,
\emph{not} super commutative.
Also, a naive definition of the corresponding Higgs branch is just the
categorical quotient $\fM_0(\underline{\bv},\underline{\bw})$, which is
a single point $\{0\}$ in the reduced structure.
\end{Example}

Let us consider the case $\underline{\bw} = 0$. Let us denote
$(\CC^\times)^n$ by $T$ hereafter.

\begin{Proposition}\label{prop:torus-fixed-points}
    The fixed point set $\cM(\underline{\bv},0)^{T}$ is
    empty unless $\underline{\bv} = 0$.
\end{Proposition}

\begin{proof}
    Consider the stratification in \propref{prop:chainsaw}(3). Let us
    first consider the factor $S^k(\CC\times\CC^\times)$. The action
    of $(t_0,\dots,t_{n-1})$ is trivial on the factor. The extra
    $\CC^\times$ acts on $S^k(\CC\times\CC^\times)$ by the
    multiplication on $\CC^\times$ by its definition. Therefore we do
    not have fixed points in $S^k(\CC\times\CC^\times)$ unless $k=0$.

    Next consider $\cM^{\mathrm{reg}}(\underline{\bv},0)$. Let us take
    a representative $(A_i,B_i,a_i,b_i)$ ($0\le i\le n-1$) of a fixed
    point. We have a (unique) homomorphism $\rho\colon T
    \to \prod \GL(V_i)$ such that
    \begin{gather*}
        t_\delta A_{n-1} = \rho_0(t) A_{n-1} \rho_{n-1}(t)^{-1},\quad
        A_i = \rho_{i+1}(t) A_i \rho_i(t)^{-1} (i\neq n-1), \quad
        B_i = \rho_i(t) B_i \rho_i(t)^{-1}, \\
        t_i^{-1} a_i = \rho_i(t) a_i, \quad
        t_\delta t_0 b_{n-1} = b_{n-1} \rho_{n-1}(t)^{-1}, \quad
        t_{i+1} b_i = b_i \rho_i(t)^{-1} (i\neq n-1),
    \end{gather*}
    where $\rho_i(t)$ is the $\GL(V_i)$-component of $\rho$.

    We consider the weight space decomposition of $V_i$ with respect
    to $\rho_i$. By the first equation, $A_{n-1}$ and $b_{n-1}$ increase
    weights for $t_\delta$ by $1$, and other maps preserve them. (We
    understand that $W_i$ has weight $0$.)  Therefore we can unwind
    the chainsaw quiver to a handsaw quiver as
    \begin{align*}
        &
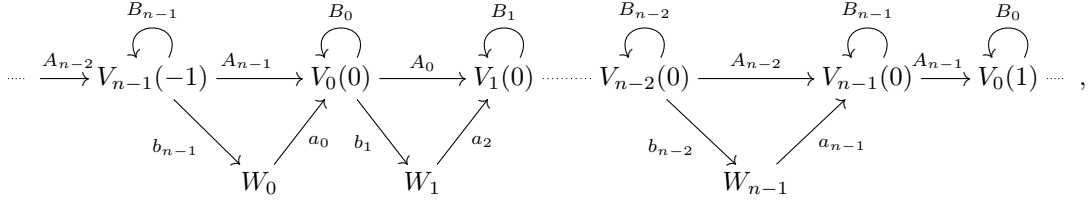
\xymatrix@C=0.6em{ 
          & 
          \ar[rr]^(.3){A_{n-2}} \ar@{.}[l] 
          && V_{n-1}(-1) \ar@(ur,ul)_{B_{n-1}} \ar[rr]^{A_{n-1}} \ar[dr]_{b_{n-1}} && V_{0}(0)  \ar@(ur,ul)_{B_{0}} \ar[rr]^{A_{0}} \ar[dr]_{b_{1}} && V_{1}(0)  \ar@(ur,ul)_{B_{1}} \ar@{.}[rr] && V_{n-2}(0) \ar@(ur,ul)_{B_{n-2}} \ar[rr]^{A_{n-2}} \ar[dr]_{b_{n-2}} && V_{n-1}(0) \ar@(ur,ul)_{B_{n-1}}
          \ar[rr]^{A_{n-1}} && V_0(1) \ar@(ur,ul)_{B_{0}} \ar@{.}[r] &
          \\
          && & & W_0 \ar[ur]_{a_{0}} & & W_{1} \ar[ur]_{a_{2}} & &&& W_{n-1} 
          \ar[ur]_{a_{n-1}} 
          &},
\end{align*}
where the number in $(\ )$ is the $t_\delta$-weight.
For a triangle without $W$, we have $B_j A_{j-1} = A_{j-1}
B_{j-1}$. Then the conditions (S1),(S2) imply that $A_{j-1}$ is an
isomorphism. But $V_j(m) = 0$ if $|m|$ is sufficiently
large. Therefore $V_j(m) = 0$ if $m\neq 0$ or $j=n-1$, $m=0$.

Next consider the action of $(t_0,\dots,t_{n-1})$. The representative
is a direct sum of $n$ summands, $(a_j,b_{j-1})$ is zero except
possibly for $j = i$ ($0\le i\le n-1$). If $a_j = 0 = b_{j-1}$, we
have $A_{j-1}$ is an isomorphism as above. Then $V_{n-1}(-1) = 0 =
V_{n-1}(0)$ imply any $V_i(0)$ also vanishes.
\end{proof}

Next consider a general case.

\begin{Proposition}\label{prop:torus-fixed-points2}
  Suppose $\nu = 0$.
    The fixed point set
    $\cM(\mu,\lambda)^{T}$ is either
    a single point or empty.
\end{Proposition}

\begin{proof}
    We prove $\cM^{\mathrm{s}}(\mu,\kappa)^{T}$ is a single
    point with a specific $\kappa$ determined by $\mu$, and is
    $\emptyset$ otherwise.
    \begin{NB}
        I should determine $\kappa$.
    \end{NB}%

    As in the proof of \propref{prop:torus-fixed-points}, we choose a
    representative of a fixed point and take a homomorphism
    $\rho\colon T\to \prod \GL(V_\zeta)$. We unwind the
    bow diagram as above, and have a direct sum decomposition into $n$
    summands, where $(a_j,b_{j-1})$ is zero except possibly for $j=i$
    ($0\le i\le n-1$) as above. Since $A_{j-1}$ is an isomorphism if
    $j\neq i$, we can remove the corresponding segment from the bow
    diagram so that the corresponding bow variety does not
    change. Therefore the $i$th summand corresponds to
\begin{equation}\label{eq:22}
\begin{tikzpicture}[baseline=(current  bounding  box.center)]
  \node[label=below:$\xl_i$,label=above:$\mu_i$] at (0,0) 
  {$\boldsymbol\times$};
  \node[label=below:$h_0$] at (1,0)
        {$\boldsymbol\medcirc$};
  \node[label=below:$h_{-1}$] at (2,0)
        {$\boldsymbol\medcirc$};
  \node[label=below:$\cdots$,label=above:$\cdots$] at (3,0)
  {$\vphantom{j^X}$};
  \node[label=below:$h_1$] at (-1,0)
        {$\boldsymbol\medcirc$};
  \node[label=below:$h_2$] at (-2,0)
        {$\boldsymbol\medcirc$};
  \node[label=below:$\cdots$,label=above:$\cdots$] at (-3,0)
        {$\vphantom{j^X}$};
  \draw[-] (-4,0) -- (4,0);
\end{tikzpicture}.
\end{equation}
Note that $N_{\xl_i} = \mu_i$ as $A_{i-1}$ is an isomorphism for other
summands. Let us number $\boldsymbol\medcirc$ as above.

As in the proof of \thmref{thm:strat-affine-type}, we have
$N(h_\xp,h_{\xp+1}) \ge 0$ for any $\xp$ by the necessary condition
for $\fM_0^{\mathrm{reg}}\neq\emptyset$. On the other hand, $\sum
N(h_\xp,h_{\xp+1}) = 1$ by definition. Therefore $N(h_\xp,h_{\xp+1}) =
1$ for some $\xp_0$ and $=0$ otherwise.
Suppose $\xp_0 > 0$. Then $N(h_\xp,h_{\xp+1}) = 0$ for $\xp < \xp_0$.
Consider vector spaces on segments which are right of $\xl_i$. As
$N_{h_{\xp}} = 0$ for sufficiently small $\xp$, we have $N_{h_{\xp}} =
0$ for $\xp\le 0$. Thus vector spaces which are right of $\xl_i$
vanish. Since $N(h_0,h_1)$ is also $0$ by assumption, we have $N_{h_1}
= 1$. Moving to the left of $\xl_i$, we get $N_{h_1} = N_{h_2} =
\cdots = N_{h_{\xp_0}} = 1$, $N_{h_{\xp_0+1}} = N_{h_{\xp_0+2}} =
\cdots = 0$.  Hence the vector space on the segment $[h_{\xp_0},
h_{\xp_0-1}]$ is $1$-dimensional, the next one on
$[h_{\xp_0-1},h_{\xp_0-2}]$ is $2$-dimensional, and so on. The vector
space over $[h_1,\xl_i]$ is $\xp_0$-dimensional. In particular, $\xp_0
= \mu_i$. The corresponding bow variety is isomorphic to
$\overline{\mathcal N}(\lambda)\cap \mathcal S(\mu)$ with $\lambda =
(\xp_0)$, ${}^t\!\mu = (1^{\xp_0})$ by the argument in
\subsecref{subsec:nilpotent}. Then $\lambda = \mu$, hence it is a
single point.

Similarly $\xp_0 < 0$ implies $-1 = N_{h_0} = N_{h_{-1}} = \cdots =
N_{h_{\xp_0+1}}$, $0 = N_{h_{\xp_0}} = N_{h_{\xp_0}-1} = \cdots$. We
get $\xp_0 = \mu_i$ also in this case. The corresponding bow variety
is again a single point. If $\xp_0 = 0$, all vector spaces vanish, and
$\mu_i = 0 = \xp_0$.

Note that dimension of vector spaces are given by $\mu_i$, hence $\kappa$
is determined by $\mu$.
\end{proof}

According to Hikita conjecture mentioned above, we expect
$\cM(\mu,\lambda)^{T}\neq\emptyset$ if and only if the
corresponding quiver variety $\fM(\mu,\lambda)$ is nonempty with
generic moment map level. The latter is equivalent that the weight
space $L(\lambda)_\mu$ is nonzero, where $L(\lambda)$ is the
integrable representation of the highest weight $\lambda$.

\begin{Proposition}
  Suppose $\nu$ is arbitrary.
  The fixed point set $\cM(\mu,\lambda)^{T}$ is finite.
\end{Proposition}

\begin{proof}
  We argue as in \propref{prop:torus-fixed-points2}. We have a direct
  sum decomposition of a representative of a fixed point as in
  \eqref{eq:22}.

  According to the sign of $\mu_i$, we apply Hanany-Witten transitions
  $|\mu_i|$ times to move $\boldsymbol\times$ to the left or right. We
  claim that it is possible, i.e., the number of $\boldsymbol\medcirc$
  to the left or right of $\xl_i$ is more than or equal to
  $|\mu_i|$. In fact, if this is not possible, the last step is
\begin{equation*}
\begin{tikzpicture}[baseline=(current  bounding  box.center)]
  \node[label=below:$\xl_i$] at (0,0) 
  {$\boldsymbol\times$};
  \node[label=below:$h_N$] at (1,0)
        {$\boldsymbol\medcirc$};
  \node[label=below:$h_{N-1}$] at (2,0)
        {$\boldsymbol\medcirc$};
  \node[label=below:$\cdots$,label=above:$\cdots$] at (3,0)
  {$\vphantom{j^X}$};
  \draw[-] (-1,0) -- (3,0);
\end{tikzpicture},
\end{equation*}
when $\mu_i > 0$. Since the dimension on the left segment to $\xl_i$ is $0$,
$N_{\xl_i} > 0$ means the dimension on the right segment is
negative. This is a contradiction. The same is true if $\mu_i < 0$.

Thus the cobalanced condition is met by Hanany-Witten transitions,
hence the bow variety containing the direct summand is isomorphic to a
quiver variety of type $A$. Since the framed vector space $W$ has
dimension $1$, the quiver variety is a single point.

We have only finitely many ways to decompose dimension vectors, the
fixed point set is finite.
\end{proof}

Let us return to the case $\nu = 0$ so that $\cM(\mu,\lambda)$ is
affine. We further assume $\cM(\mu,\lambda)^{T}\neq\emptyset$, so
$\cM(\mu,\lambda)^{T}$ is a single point by
\propref{prop:torus-fixed-points2}.  Take a generic $1$-parameter
subgroup $\rho\colon\CC^\times\to T$ such that
$\cM(\mu,\lambda)^{\rho(\CC^\times)} = \cM(\mu,\lambda)^{T}$. We have
a diagram
\begin{equation*}
  \mathrm{pt} = \cM(\mu,\lambda)^{T} \xleftarrow{p}
  \mathcal A \xrightarrow{j} \cM(\mu,\lambda),
\end{equation*}
where $\mathcal A$ is the attracting set consisting of points $x$ such
that $\lim_{t\to 0} \rho(t)x$ exists. The map $p$ is given by the
limit, and $j$ is the inclusion. We have Braden's hyperbolic
restriction functor $p_* j^!$ on the $T$-equivariant derived category
of constructible sheaves on $\cM(\mu,\lambda)$ \cite{Braden}. Let
$\mathrm{IC}(\cM(\mu,\lambda))$ be the intersection cohomology complex
associated with the trivial local system on $\cM(\mu,\lambda)$. 
We consider $p_* j^! \mathrm{IC}(\cM(\mu,\lambda))$. In \cite[\S
3]{2014arXiv1406.2381B} the notion of \emph{hyperbolic semi-smallness}
is introduced, which implies that
$p_* j^! \mathrm{IC}(\cM(\mu,\lambda))$ is a vector space (instead of
complex of vector spaces).

\begin{Proposition}
  $p_* j^!$ is hyperbolic semi-small. Therefore
  $p_* j^! \mathrm{IC}(\cM(\mu,\lambda))$ is a vector space.
\end{Proposition}

\begin{proof}
  Let us assume $\ell\neq 1$ for notational simplicity.

  Recall the stratification
  $\cM(\mu,\lambda) = \bigsqcup \cM^{\mathrm{s}}(\mu,\kappa)\times
  S^{\underline{k}}(\CC^2\setminus \{0\}/(\ZZ/\ell \ZZ))$
  in \thmref{thm:strat-affine-type}. Then
  $\mathrm{IC}(\cM(\mu,\lambda))$ is a locally constant sheaf up to
  shifts on each stratum. The hyperbolic semi-smallness in this case
  is
  \begin{equation*}
    \dim \mathcal A\cap X_{\kappa,\underline{k}} \le
    \frac12 \dim X_{\kappa,\underline{k}}, \qquad
    X_{\kappa,\underline{k}} := \cM^{\mathrm{s}}(\mu,\kappa)\times
  S^{\underline{k}}(\CC^2\setminus \{0\}/(\ZZ/\ell \ZZ)),
  \end{equation*}
  and also the same estimate holds for the repelling set
  $\mathcal R = \{ x\mid \text{$\lim_{t\to\infty}\rho(t)x$ exists}\}$.

  On the factor
  $S^{\underline{k}}(\CC^2\setminus \{0\}/(\ZZ/\ell \ZZ))$, the
  $T$-action is induced from $t\cdot (z_1,z_2) = (tz_1, t^{-1} z_2)$
  for $(z_1,z_2)\in\CC^2$. Therefore the attracting and repelling sets
  are $z_2 = 0$ and $z_1 = 0$ respectively. They are half-dimensional
  in $\CC^2$. Hence we may assume $\underline{k} = \emptyset$.

  We take a generic parameter $\nu^\RR$ so that
  $\pi\colon \cM_{(0,\nu^\RR)}(\mu,\kappa)\to \cM(\mu,\kappa)$ is a
  resolution. Since
  $\pi^{-1}(\mathcal A\cap \cM^{\mathrm{s}}(\mu,\kappa))$ (resp.\
  $\pi^{-1}(\mathcal R\cap \cM^{\mathrm{s}}(\mu,\kappa))$) is
  contained in the attracting set $\widetilde{\mathcal A}$ (resp.\
  repelling set $\widetilde{\mathcal R}$) for
  $\cM_{(0,\nu^\RR)}(\mu,\kappa)$, it is enough to estimate the
  dimension of $\widetilde{\mathcal A}$ and $\widetilde{\mathcal R}$.

  Since $\cM_{(0,\nu^\RR)}(\mu,\kappa)$ is smooth, we can use the
  Bialynicki-Birula decomposition. The tangent space at a fixed point
  set decomposes into weight spaces with respect to $\rho(\CC^\times)$
  and the dimension of $\widetilde{\mathcal A}$ (resp.\
  $\widetilde{\mathcal R}$) is sum of dimensions of positive (resp.\
  negative) weight spaces. Since $T$ preserves the symplectic form,
  positive and negative weight spaces are dual to each
  other. Therefore we have
  $\dim \widetilde{\mathcal A} = \dim\widetilde{\mathcal R} = \frac12
  \dim \cM_{(0,\nu^\RR)}(\mu,\kappa)$.
\end{proof}

We have a base of $p_* j^! \mathrm{IC}(\cM(\mu,\lambda))$ given by
$\frac12\dim \cM(\mu,\lambda)$-dimensional irreducible components of
$\mathcal A$. (See \cite[Th.~3.23]{2014arXiv1406.2381B}.)

It is expected that the hyperbolic restriction functor and the
pushforward by the semi-small resolution are `dual' in Higgs and
Coulomb branches. Therefore

\begin{Conjecture}
  $\dim p_* j^! \mathrm{IC}(\cM(\mu,\lambda)) = \dim L(\lambda)_\mu$.
\end{Conjecture}

If $\mu$ is dominant, $\cM(\mu,\lambda)$ is a quiver variety by
\corref{cor:coulomb-quiver}. Then the attracting set is the so-called
tensor product variety. In particular, the number of irreducible
components is given by a tensor product multiplicity (\cite[\S
6]{Na-branching}). The conjecture follows from the level-rank duality.
Recall that \cite{Na-branching} is motivated by the conjectural
geometric Satake correspondence for affine Kac-Moody groups
\cite{braverman-2007}. The above conjecture is a refinement of the
geometric Satake correspondence, as $\mu$ is not necessarily dominant,
or it could be even not a weight of $L(\lambda)$, if we understand
$\dim L(\lambda)_\mu$ as $0$. See also \cite{MR3134906} for another
(closely related but different) conjectural geometric realization of
$L(\lambda)_\mu$ for arbitrary $\mu$.

Recall also irreducible components of $\mathcal A$ for dominant $\mu$
are regarded as Mirkovi\'c-Vilonen (MV) cycles for affine Kac-Moody groups
of type $A$ in \cite[\S 6]{Na-branching}. We now remove the assumption
that $\mu$ is dominant, hence we get MV cycles for arbitrary weights. It is natural to expect
\begin{Conjecture}
  The union of sets of $\frac12 \dim \cM(\mu,\lambda)$-dimensional
  irreducible components of $\mathcal A$ has a structure of a crystal,
  isomorphic to the crystal of the integrable highest weight
  representation $L(\lambda)$.
\end{Conjecture}

\appendix
\section*{Appendix}

\section{Symplectic structures under Hanany-Witten transition}
\label{sec:appendix1}

In this section, we describe Hanany-Witten transition more explicitly
in terms of matrices, and show that symplectic structures are
preserved under Hanany-Witten transition.

We consider four cases depending on the relation between $m$ and $n$.\\
(i) $m=n$\\
(ii) $m=n+1$\\
(iii) $m < n$\\
(iv) $m>n+1$\\
Now, when $m\leq n$, we have $m' \geq l+1$, and we can see that the assertions (i) and (ii), and (iii) and (iv) are equivalent respectively:
\begin{align*}
\begin{xy}
(0,0)*{(m\leq n)},
(10,3)*{l},
(15,0)*{\boldsymbol\medcirc},
(20,3)*{m},
(25,0)*{\boldsymbol\times},
(30,3)*{n},
(40,0)*{\cong},
(50,3)*{l},
(55,0)*{\boldsymbol\times},
(60,3)*{m'},
(65,0)*{\boldsymbol\medcirc},
(70,3)*{n},
(85,0)*{(m' \geq l+1)},
(-20,-10)*{\Longleftrightarrow},
(-3,-10)*{(m' \geq l+1)},
(10,-7)*{n},
(15,-10)*{\boldsymbol\medcirc},
(20,-7)*{m'},
(25,-10)*{\boldsymbol\times},
(30,-7)*{l},
(40,-10)*{\cong},
(50,-7)*{n},
(55,-10)*{\boldsymbol\times},
(60,-7)*{m},
(65,-10)*{\boldsymbol\medcirc},
(70,-7)*{l},
(82,-10)*{(m \leq n)},
\ar @{-} (10,0);(30,0)
\ar @{-} (50,0);(70,0)
\ar @{-} (10,-10);(30,-10)
\ar @{-} (50,-10);(70,-10)
\end{xy}
\end{align*}
Thus it is enough to show the assertion (i) and (iii).

(i) Let us consider the symplectic reduction by $\GL(V_2)$ of the
first space in the proof of \propref{prop:HW-trans}.
Recall that the triangle part $\widetilde\cM$ is
$T^*\GL(V_2)\times\Hom(\CC, V_3)\times\Hom(V_2,\CC)$ by
\propref{prop:triangle-Nahm} as $\dim V_2 = \dim V_3$. Then the
symplectic reduction is
\begin{equation*}
  \Hom(V_1, V_3)\times\Hom(V_3,V_1)\times \Hom(V_1,\CC)\times\Hom(\CC,V_1)
\end{equation*}
by setting
\begin{equation*}
  \tilde{C} \defeq AC, \quad \tilde{D} \defeq DA^{-1},
  \quad \tilde{a} = a, \quad \tilde{b} \defeq bA^{-1}. 
\end{equation*}
The symplectic form is given by
\begin{equation*}
  \tr(d\tilde{C}\wedge d\tilde{D}) + \tr(d\tilde{a}\wedge d\tilde{b}).
\end{equation*}

On the other hand, consider the second space in the proof of
\propref{prop:HW-trans}. Then
\(
\left[\begin{smallmatrix}
  A^\tn & a^\tn
\end{smallmatrix}\right]
\colon V_1\oplus\CC\to V_2^\tn
\)
is an isomorphism, and the symplectic reduction by $\GL(V^\tn)$ is
\begin{equation*}
  \Hom(V_1\oplus\CC, V_3)\times \Hom(V_3,V_1\oplus \CC)
\end{equation*}
by setting
\begin{equation*}
  \tilde{C}^\tn \defeq C^\tn \left[\begin{matrix}
      A^\tn & a^\tn
    \end{matrix}\right], \quad
  \tilde{D}^\tn \defeq \left[\begin{matrix}
      A^\tn & a^\tn
    \end{matrix}\right]^{-1} D^\tn.
\end{equation*}
The symplectic form is given by
\begin{equation*}
  \tr(d\tilde{C}^\tn\wedge d\tilde{D}^\tn).
\end{equation*}

Looking at definitions of Hanany-Witten transition, we find
\begin{equation*}
    \tilde{C}^\tn = \left[\begin{matrix}
      \tilde{C} & a
    \end{matrix}\right], \qquad
    \tilde{D}^\tn = \left[\begin{matrix}
      \tilde{D} \\ \tilde{b}
    \end{matrix}\right].
\end{equation*}
Therefore the symplectic forms are the same.


Data and Hanany-Witten transition are illustrated as follows:
\begin{align*}
\xymatrix@C=1.2em{
l \ar@<-.5ex>[rr]_{C} && n \ar@(ur,ul)_{-CD} \ar[rr]^{\id} \ar[dr]_{b} \ar@<-.5ex>[ll]_{D} && n \ar@(ur,ul)_{-CD-ab} \\
&&& 1 \ar[ur]_{a} & }
\ \xymatrix{\rightarrow \\ } \
\xymatrix@C=1.2em{
l \ar@(ur,ul)_{-DC} \ar[rr]^{\left[ \begin{smallmatrix}\id \\ 0 \end{smallmatrix}\right]} \ar[dr]_{bC} && l+1 \ar@(ur,ul)_{-D^\tn C^\tn} \ar@<-.5ex>[rr]_{C^\tn} && n \ar@<-.5ex>[ll]_{D^\tn} \\
& 1 \ar[ur]_{\left[ \begin{smallmatrix} 0 \\ 1 \end{smallmatrix} \right]} &&&  }
\end{align*}

\begin{NB}
\begin{align*}
\begin{xy}
(0,0)*{l},
(15,0)*{n},
(35,0)*{n},
(25,-10)*{1},
(60,0)*{l},
(80,0)*{l+1},
(70,-10)*{1},
(95,0)*{n},
\ar (3,1);(12,1)^{C}
\ar (12,-1);(3,-1)^{D}
\ar @(ur, ul) (17,2);(13,2)_{CD}
\ar @(ur, ul) (37,2);(33,2)_{CD+vw}
\ar (19,0);(32,0)^{\id}
\ar (17,-2);(23,-8)_{w}
\ar (27,-8);(33,-2)_{v}
\ar (45,0);(50,0)
\ar @(ur, ul) (62,2);(58,2)_{DC}
\ar @(ur, ul) (82,2);(78,2)_{\eta}
\ar (64,0);(75,0)^{\left( \begin{smallmatrix} \id \\ 0 \end{smallmatrix} \right)}
\ar (62,-2);(68,-8)_{-wC}
\ar (72,-8);(78,-2)_{\left( \begin{smallmatrix} 0 \\ 1 \end{smallmatrix} \right)}
\ar (85,1);(92,1)^{Y}
\ar (92,-1);(85,-1)^{X}
\end{xy}
\end{align*}
\end{NB}

(iii) 
Let us consider the second space in the proof of
\propref{prop:HW-trans}. By \propref{prop:triangle-Nahm} the
symplectic reduction by $\GL(V_2^\tn) = \GL(m')$ is isomorphic to
\begin{align*}
  \left\{ (\eta, C^\tn, D^\tn) \in \mathcal{S}(l, m') \times \Hom(\CC^{m'},\CC^n)
  \times \Hom(\CC^n, \CC^{m'}) \,\middle|\,  \eta = D^\tn C^\tn \right\},
\end{align*}
where we write dimensions for the slice explicitly. The symplectic
form is
\(
   \tr(d C^\tn \wedge d D^\tn).
\)

Let us consider the first space in the proof of \propref{prop:HW-trans}.
It is the product of $\Hom(\CC^l,\CC^m)\times\Hom(\CC^m,\CC^l)$ and
$\GL(n)\times\mathcal S(m,n)$ by \propref{prop:triangle-Nahm}.
The symplectic reduction by $\GL(m)$ is described as
\begin{align*}
             \left\{ (C, D, \eta', u) \in
             \Hom(\CC^l,\CC^m)\times \Hom(\CC^m,\CC^l)
             \times
             \mathcal{S}(m,n)\times \GL(n)
             \,\middle|\, CD = h'
             \right\}/\GL(m),\\
\end{align*}
where $\eta'$ has the block decompositon as
\begin{equation*}
  \eta' = \left[ \begin{smallmatrix} h' & 0 & g \\ f & 0 & e_0 \\ 0 & \id & e \end{smallmatrix} \right].
\end{equation*}
The symplectic form is
\(
\tr( d\eta' \wedge duu^{-1} + \eta' duu^{-1}\wedge duu^{-1}) + \tr( dC  
\wedge dD ).
\)
See
\begin{align*}
\xymatrix@C=1.8em{
l \ar@<-.5ex>[rr]_{C} && m \ar@(ur,ul)_{-CD} \ar[rr]^{-u^{-1}\left[\begin{smallmatrix}\id \\ 0 \\ 0 \end{smallmatrix} \right]} \ar[dr]_{-f} \ar@<-.5ex>[ll]_{D} && n \ar@(ur,ul)_{-u^{-1}\eta'u} \\
&&& 1 \ar[ur]_{u^{-1}\left[ \begin{smallmatrix} 0 \\ 1 \\ 0 \end{smallmatrix} \right]} & }
\ \xymatrix{\rightarrow \\ } \
\xymatrix@C=1.2em{
l \ar@(ur,ul)_{-DC} \ar[rr]^{-\left[ \begin{smallmatrix}\id \\ 0 \\ 0 \end{smallmatrix}\right]} \ar[dr]_{-fC} && m' \ar@(ur,ul)_{-\eta} \ar@<-.5ex>[rr]_{C^\tn} && n \ar@<-.5ex>[ll]_{D^\tn} \\
& 1 \ar[ur]_{\left[ \begin{smallmatrix} 0 \\ 1 \\ 0\end{smallmatrix} \right]} &&&  }
\end{align*}

The decomposition $\CC^n = \CC^m\oplus \CC^{n-m-1}\oplus \CC$ in the
Hurtubise normal form induces an isomorphism
$V_2^\tn \cong \CC^l\oplus \CC^{n-m} \oplus \CC$. Then we find
\begin{equation*}
  C^\tn = u^{-1}\left[ \begin{smallmatrix} C & 0 & g \\ 0 & \id & e^+
    \end{smallmatrix} \right], \qquad
  D^\tn = \left[ \begin{smallmatrix} D & 0 \\ f & 0 \\ 0 & \id
    \end{smallmatrix} \right]u,
  \begin{NB}
    \qquad
    \eta = D^\tn C^\tn = \left[ \begin{smallmatrix} DC & 0 & Dg \\ fC & 0 & fg \\ 0 & \id & e^+ \end{smallmatrix} \right]
  \end{NB}%
\end{equation*}
where $e^+ = \left[ \begin{smallmatrix} e_0 \\ e\end{smallmatrix}\right]$.

\begin{NB}
we construct the following morphism:
\begin{align*}
(C, D, \eta', u) \mapsto (\eta, X, Y) = \left( \left[ \begin{smallmatrix} DC & 0 & Dg \\ fC & 0 & fg \\ 0 & \id & e^+ \end{smallmatrix} \right], u^{-1}\left[ \begin{smallmatrix} C & 0 & g \\ 0 & \id & e^+ \end{smallmatrix} \right], \left[ \begin{smallmatrix} D & 0 \\ f & 0 \\ 0 & \id \end{smallmatrix} \right]u \right),\\
\text{where, }\eta' =\left[ \begin{smallmatrix} CD & 0 & g \\ f & 0 & e_0 \\ 0 & \id & e \end{smallmatrix} \right], \ \ e^+ = \left[ \begin{smallmatrix} e_0 \\ e\end{smallmatrix}\right]
\end{align*}
$g'\in \GL(m)$ acts as, 
\begin{align*}
(C, D, \eta', u) \mapsto \left( g'C, D{g'}^{-1}, \left[ \begin{smallmatrix} g' & 0 \\ 0 & \id \end{smallmatrix} \right]\eta'\left[ \begin{smallmatrix} g' & 0 \\ 0 & \id \end{smallmatrix} \right]^{-1}, \left[ \begin{smallmatrix} g' & 0 \\ 0 & \id \end{smallmatrix} \right]u \right)
\end{align*}
so the above morphism is well-defined.
And we can also check that this morphism satisfies required conditions.
\end{NB}%

Now
\begin{equation*}
  \begin{split}
    & \tr(d C^\tn\wedge dD^\tn) \\
    = \; &
  \tr( (-u^{-1}duu^{-1}\left[
    \begin{smallmatrix}C & 0 & g \\ 0 & \id & e^+\end{smallmatrix}\right]
  + u^{-1}\left[
    \begin{smallmatrix}dC & 0 & dg \\ 0 & 0 & de+\end{smallmatrix}\right])  
  \wedge (\left[
    \begin{smallmatrix}dD & 0 \\ df & 0 \\ 0 & 0\end{smallmatrix}\right]u
  + \left[
    \begin{smallmatrix}D & 0 \\ f & 0 \\ 0 & \id\end{smallmatrix}\right]du) )
  \\
= \; &\tr( \left[\begin{smallmatrix}CdD & 0 & 0 \\ df &0 &0 \\ 0 &0 &0\end{smallmatrix}\right] \wedge duu^{-1} + \eta' duu^{-1}\wedge  
duu^{-1} + \left[\begin{smallmatrix}dC \wedge dD & 0 & 0 \\ 0 &0 &0 \\ 0 &0 &0\end{smallmatrix}\right] + \left[\begin{smallmatrix}dCD &0 &dg \\ 0 &0 &de_0  
    \\ 0 &0 &de\end{smallmatrix}\right] \wedge duu^{-1})
\\
=\; & \tr( d\eta' \wedge duu^{-1} + \eta' duu^{-1}\wedge duu^{-1}) + \tr( dC  
\wedge dD ).
  \end{split}
\end{equation*}
Therefore the symplectic forms are the same.

\begin{NB}
\begin{align*}
\begin{xy}
(0,0)*{l},
(15,0)*{m},
(35,0)*{n},
(25,-10)*{1},
(60,0)*{l},
(80,0)*{m'},
(70,-10)*{1},
(95,0)*{n},
\ar (3,1);(12,1)^{B}
\ar (12,-1);(3,-1)^{A}
\ar @(ur, ul) (17,2);(13,2)_{BA}
\ar @(ur, ul) (37,2);(33,2)_{u^{-1}\eta' u}
\ar (19,0);(32,0)^{u^{-1}\left( \begin{smallmatrix}\id \\ 0 \\ 0 \end{smallmatrix} \right)}
\ar (17,-2);(23,-8)_{-f}
\ar (27,-8);(33,-2)_{u^{-1}\left( \begin{smallmatrix} 0 \\ 1 \\ 0 \end{smallmatrix} \right)}
\ar (45,0);(50,0)
\ar @(ur, ul) (62,2);(58,2)_{AB}
\ar @(ur, ul) (82,2);(78,2)_{\eta}
\ar (64,0);(77,0)^{\left( \begin{smallmatrix}\id \\ 0 \\ 0 \end{smallmatrix} \right)}
\ar (62,-2);(68,-8)_{-fB}
\ar (72,-8);(78,-2)_{\left( \begin{smallmatrix} 0 \\ 1 \\ 0 \end{smallmatrix} \right)}
\ar (83,1);(92,1)^{Y}
\ar (92,-1);(83,-1)^{X}
\end{xy}
\end{align*}
\end{NB}

\begin{NB}
\section{\texorpdfstring{$D$}{D}-type ALF space}\label{sec:appendix2}
The bow varieties given in \S \ref{subsubsec:locmodel5} is closely related to $D$-type ALF spaces.
When $\bw \geq 4$, Proposition \ref{prop:HW-trans} means
\begin{align*}
\cM((2,0), (\bw,0)) \cong \mathcal{N}({}^t\!\mu)\cap \mathcal{S}(\lambda),
\end{align*}
where $\mu = [1, \cdots,1] \ (|\mu| = \bw)$ and $\lambda = [\bw-2,2]$.
\begin{align*}
\begin{xy}
(5,3)*{0},
(10,0)*{\boldsymbol\times},
(15,3)*{2},
(20,0)*{\boldsymbol\medcirc},
(25,3)*{\cdots},
(30,0)*{\boldsymbol\medcirc},
(35,3)*{2},
(40,0)*{\boldsymbol\times},
(45,3)*{0},
(25,-5)*{\underbrace{\hspace{15mm}}_{\text{$\bw$}}},
\ar @{-} (5,0);(45,0)
\end{xy} \ & \cong \ \begin{xy}
(5,3)*{0},
(10,0)*{\boldsymbol\medcirc},
(15,3)*{1},
(20,0)*{\boldsymbol\times},
(25,3)*{2},
(30,0)*{\boldsymbol\medcirc},
(35,3)*{2},
(40,0)*{\boldsymbol\medcirc},
(45,3)*{\cdots},
(50,0)*{\boldsymbol\medcirc},
(55,3)*{2},
(60,0)*{\boldsymbol\medcirc},
(65,3)*{2},
(70,0)*{\boldsymbol\times},
(75,3)*{1},
(80,0)*{\boldsymbol\medcirc},
(85,3)*{0},
(45,-5)*{\underbrace{\hspace{15mm}}_{\text{$\bw$-4}}},
\ar @{-} (5,0);(85,0)
\end{xy}\\
& \cong \ \begin{xy}
(5,3)*{0},
(10,0)*{\boldsymbol\medcirc},
(15,3)*{1},
(20,0)*{\boldsymbol\medcirc},
(25,3)*{2},
(30,0)*{\boldsymbol\times},
(35,3)*{2},
(40,0)*{\boldsymbol\medcirc},
(45,3)*{\cdots},
(50,0)*{\boldsymbol\medcirc},
(55,3)*{2},
(60,0)*{\boldsymbol\times},
(65,3)*{2},
(70,0)*{\boldsymbol\medcirc},
(75,3)*{1},
(80,0)*{\boldsymbol\medcirc},
(85,3)*{0},
(45,-5)*{\underbrace{\hspace{15mm}}_{\text{$\bw$-4}}},
\ar @{-} (5,0);(85,0)
\end{xy}\\
&= \xymatrix@=.8em{
\CC \ar@<-.5ex>[rr] && \CC^2 \ar@<-.5ex>[ll] \ar@<-.5ex>[rr] \ar@<-.5ex>[d] && \CC^2 \ar@<-.5ex>[ll] \ar@<-.5ex>[rr] && \cdots \ar@<-.5ex>[ll] \ar@<-.5ex>[rr] && \CC^2 \ar@<-.5ex>[ll] \ar@<-.5ex>[rr] && \CC^2 \ar@<-.5ex>[ll] \ar@<-.5ex>[rr] \ar@<-.5ex>[d] && \CC \ar@<-.5ex>[ll]\\
&& \CC \ar@<-.5ex>[u] && && && && \CC \ar@<-.5ex>[u] && }.
\end{align*}
We define an additional $\CC^{\times}$-action on $\cM((2,0), (\bw,0))$ as a $\CC^{\times}$-action on the leftmost triangle:
\begin{align*}
(a, B)\mapsto (t^{-1}a, B), \ \ t\in \CC^{\times}.
\end{align*}
\begin{Proposition}
\begin{align*}
\Psi^{-1}(\{ \tr B=0 \})\dslash (\CC^{\times} \times \prod \GL) \cong X_{D_{\bw}},
\end{align*}
where $X_{D_n} = \{ (x, y, z) \mid w^2+yx^2+y^{1-n} =0\} \cong \CC^2/ D_n$.
\end{Proposition}
\begin{proof}
We describe the isomorphism explicitly by using \S \ref{sec:appendix1}:
\begin{align*}
\begin{xy}
(5,3)*{0},
(10,0)*{\boldsymbol\medcirc},
(15,3)*{1},
(20,0)*{\boldsymbol\medcirc},
(25,3)*{2},
(30,0)*{\boldsymbol\times},
(35,3)*{2},
\ar @{-} (5,0);(35,0)
\ar @{.} (35,0);(40,0)
\end{xy} \ \cong \ \begin{xy}
(5,3)*{0},
(10,0)*{\boldsymbol\medcirc},
(15,3)*{1},
(20,0)*{\boldsymbol\times},
(25,3)*{2},
(30,0)*{\boldsymbol\medcirc},
(35,3)*{2},
\ar @{-} (5,0);(35,0)
\ar @{.} (35,0);(40,0)
\end{xy} \ \cong \ \begin{xy}
(5,3)*{0},
(10,0)*{\boldsymbol\times},
(15,3)*{2},
(20,0)*{\boldsymbol\medcirc},
(25,3)*{2},
(30,0)*{\boldsymbol\medcirc},
(35,3)*{2},
\ar @{-} (5,0);(35,0)
\ar @{.} (35,0);(40,0)
\end{xy}.
\end{align*}
Up to $\prod \GL$-action, we get
\begin{align*}
\xymatrix@C=1.2em{
\CC \ar@(ur,ul)_{DC=0} \ar@<-.5ex>[rr]_{C} && \CC^2 \ar@(ur,ul)_{CD} \ar[rr]^{\id} \ar[dr]_{J} \ar@<-.5ex>[ll]_{D} && \CC^2 \ar@(ur,ul)_{CD+IJ} \\
&&& \CC \ar[ur]_{I} & }
\ \xymatrix{\rightarrow \\ } \
\xymatrix@C=1.2em{
\CC \ar@(ur,ul)_{0} \ar[rr]^{\left[ \begin{smallmatrix}1 \\ 0 \end{smallmatrix}\right]} \ar[dr]_{-JC} && \CC^2 \ar@(ur,ul)_{\left[\begin{smallmatrix}0 & DI \\ JC & JI \end{smallmatrix} \right]} \ar@<-.5ex>[rr]_{\left[\begin{smallmatrix} C & I \end{smallmatrix}\right]} && \CC^2 \ar@<-.5ex>[ll]_{\left[\begin{smallmatrix} D \\ J \end{smallmatrix} \right]} \\
& \CC \ar[ur]_{\left[ \begin{smallmatrix} 0 \\ 1 \end{smallmatrix} \right]} &&&  }
\ \xymatrix{\rightarrow \\ } \
\xymatrix@C=1.2em{
& \CC^2 \ar@(ur,ul)_{B=\left[\begin{smallmatrix}0 & JCDI \\ 1 & JI \end{smallmatrix} \right]} \ar@<-.5ex>[rr]_{\left[\begin{smallmatrix} 0 & DI \\ 1 & JI \end{smallmatrix}\right]} && \CC^2 \ar@<-.5ex>[ll]_{\left[\begin{smallmatrix} JC & 0 \\ 0 & 1 \end{smallmatrix} \right]} \ar@<-.5ex>[rr]_{\left[\begin{smallmatrix} C & I \end{smallmatrix}\right]} && \CC^2 \ar@<-.5ex>[ll]_{\left[\begin{smallmatrix} D \\ J \end{smallmatrix} \right]} \\
\CC \ar[ur]_{a = \left[ \begin{smallmatrix} 1 \\ 0 \end{smallmatrix} \right]} &&& && }.
\end{align*}
Now $\tr B=0$ means $JI=0$, and the $\CC^{\times}$-action for $a$ is equivalent to the $\CC^{\times}$-action $(I, J) \mapsto (t^{-1}I, tJ)$.
Thus $\Psi^{-1}(\{ \tr B=0 \})\dslash (\CC^{\times} \times \prod \GL)$ is isomorphic to the quiver variety of affine type $D$ given by Kronheimer \cite{Kr}. 
\end{proof}
Furthermore, it is shown that this isomorphism also holds for $0 \leq \bw \leq 3$ \cite[\S 6, 7]{MR1636383}.
Especially, it is known that $\cM((2,0), 0)$ is a product of $\CC\times \CC^{\times}$ and Atiyah-Hitchin manifold \cite{MR934202}.
Here Atiyah-Hitchin manifold is $X_{D_0}$ and its double covering is $X_{D_1}$.
It is partially known that these spaces have ALF metric, we conjecture that these spaces are all ALF.
\end{NB}%


\bibliographystyle{myamsalpha}
\bibliography{nakajima,mybib,bow}

\setcounter{section}{5}

\section*{Errata}

\subsection{}

In the definition of $\mathbb M^{kl}$ in \secref{subsec:slice},
$\Hom(\delta_{l\infty}\WW_\xl,V_{\zeta^+_\xl}^l)$ and
$\Hom(V_{\zeta^-_\xl}^k,\delta_{k\infty}\WW_\xl)$ are typos of
$\Hom(\delta_{k\infty}\WW_\xl,V_{\zeta^+_\xl}^l)$ and
$\Hom(V_{\zeta^-_\xl}^k,\delta_{l\infty}\WW_\xl)$ respectively.

\subsection{}

\ref{subsubsec:locmodel6} is not correct. A correct computation is
given in \cite[4(ii)]{BFN18}.

\subsection{}

The proof of \propref{prop:bow_fiber} is incomplete. Let us give a
corrected proof.

By the factorization \thmref{thm:factorization}, it is enough to
consider the case when all eigenvalues are the same.

Let $\psi\colon \mu^{-1}(0)\to \End(V_1)\times \End(V_2)$ given by
$(A,B_1,B_2,a,b)\mapsto (B_1,B_2)$.
We can shift both $B_1$, $B_2$ by the same scalar. Therefore we may
assume the eigenvalue is zero.
Let us take nilpotent conjugacy
classes $O_1$, $O_2$ corresponding to partitions $[m_1,\dots,m_l]$,
$[n_1,\dots,n_l]$. Then during the proof of
\propref{prop:triangle_fiber} we obtained an equality
\begin{equation*}
    \dim \psi^{-1}(O_1,O_2) \le
    \sum_{i\neq j} \left( \min(m_i,n_j) - \min(m_i,m_j) - \min(n_i,n_j)
      \right) + \bv_1^2 + \bv_2^2.
\end{equation*}
We rewrite the right hand side as
\begin{multline*}
    \sum_{i < j} \left( \min(m_i,n_j) - \min(n_i,n_j) \right)
    + \sum_{i > j} \left( \min(m_i,n_j) - \min(m_i,m_j) \right)\\
    + \bv_1^2 + \bv_2^2
    - \frac12 \sum_{i\neq j} \left(\min(m_i,m_j) + \min (n_i,n_j)\right)\\
    = 
    \sum_{i < j} \left( \min(m_i,n_j) - \min(n_i,n_j) \right)
    + \sum_{i > j} \left( \min(m_i,n_j) - \min(m_i,m_j) \right)\\
    + \frac12(\dim O_1 + \dim O_2) + 
    \frac12 (\bv_1^2 + \bv_2^2 + \bv_1 + \bv_2).
\end{multline*}
In the first sum, the second term is $n_j$ as $n_1$, $n_2$, \dots are
non-increasing. Therefore the first sum is nonpositive. Similarly the
second sum is nonpositive. Therefore
\begin{equation}\label{eq:25}
    \dim \psi^{-1}(O_1,O_2)
    \le \frac12 (\dim O_1 + \dim O_2) + 
    \frac12 (\bv_1^2 + \bv_2^2 + \bv_1 + \bv_2).
\end{equation}

Consider the two-way part
\begin{equation*}
\xymatrix@C=1.2em{V^1 \ar@<-.5ex>[rr]_{C} && V^2 \ar@<-.5ex>[ll]_{D}} \end{equation*}
and the projection
\begin{equation*}
    \psi \colon \Hom(V^1,V^2)\oplus \Hom(V^2,V^1) \to
    \End(V^1)\times \End(V^2);
    (C,D) \mapsto (DC,CD).
\end{equation*}
We assume $\dim V_1 = \dim V_2 = \bv_1$.  
Let $O_1$, $O_2$ be nilpotent conjugacy classes in $\End(V_1)$,
$\End(V_2)$ as above, and consider $\psi^{-1}(O_1,O_2)$. Then \cite[Prop.~5.3]{MR549399} says
\begin{equation}\label{eq:26}
    \dim \psi^{-1}(O_1,O_2)
    \le \frac12 (\dim O_1 + \dim O_2) + \bv_1^2.
\end{equation}

When $CD$ is in a conjugacy class $O$ with a single \emph{nonzero}
eigenvalue, $C$ and $D$ are invertible. Therefore 
$DC = C^{-1} CD C$ is also in $O$, and $\psi^{-1}(O,O) \cong O\times\{C\in
\Hom(V^1,V^2)\mid \text{$C$ is isomorphism}\}$ as $D$ is determined from
$CD$ and $C$ as $C^{-1} CD$. Therefore we have the same estimate \eqref{eq:26}.

Let us combine a triangle part for $V_1$, $V_2$ and a two-way part for
$V_2$, $V_3$ in the middle vector space $V_2$ as in
\propref{prop:HW-trans}. We take nilpotent conjugacy classes $O_1$,
$O_2$, $O_3$ and consider $\psi^{-1}(O_1,O_2,O_3)$, where $\psi$ is a
map to $\End(V_1)\times \End(V_2)\times \End(V_3)$ defined in the same
way as above. Combining \eqref{eq:25}, \eqref{eq:26} we get
\begin{equation*}
    \dim\psi^{-1}(O_1,O_2,O_3) \le
    \frac12 (\dim O_1 + \dim O_3)
    + \frac12 (\bv_1^2 + \bv_2^2 + \bv_1 + \bv_2) + \bv_2^2.
\end{equation*}

Similarly when we combine two triangle parts or two-way parts, we have
a similar dimension estimate where $\dim O_2$ drops out. Therefore for a
bow diagram satisfying the balanced condition, we have
\begin{equation*}
    \dim (\text{fiber of $\Psi$})
    \le \sum_i \bv_i^2 + \bv_i + \bv_i^2 \bw_i
    = \sum_i (\bw_i+1) \bv_i^2 + \bv_i.
\end{equation*}
\begin{NB}
    Note that $\bv_i^2 + \bv_i$ appears twice for each $i$ from the
    triangle parts. On the other hand, we have $\bv_i^2 \bw_i$ from
    the two-way parts between $i$ and $i+1$.
\end{NB}%
This completes the proof of \propref{prop:bow_fiber}.

\subsection{}

In the proof of \thmref{thm:normal}, the dimension estimate of the
complement of $\Psi^{-1}(\bA^{\underline{\bv}}) \dslash \GV$ is given
as follows. We first use the stratification in
\subsecref{subsec:strat} (or one in \thmref{thm:strat-affine-type}) to
check that all strata except the open one have codimension $\ge 2$.
Therefore it is enough to consider the intersection of
$\Psi^{-1}(\bA^{\underline{\bv}}) \dslash \GV$ with the open
stratum. Its inverse image in $\Psi^{-1}(\bA^{\underline{\bv}})$
consists of $0$-stable points, and the action of $\GV$ is
free. Therefore we can use \propref{prop:bow_fiber} after
subtracting $\dim \GV = \sum (\bw_i + 1)\bv_i^2$.

\subsection{}

In the last part of the proof of
Theorems~\ref{thm:finite_dominant1},~\ref{thm:finite_dominant2}, we
identify the right hand side of \eqref{eq:7} with the moduli space of
solutions of Nahm's equations, and then with
$\GL(r)\times \mathcal S(\mu)$ by \cite{MR1438643}. Here
$\mathcal S(\mu)$ is the Slodowy slice. However, solutions in this
paper have poles at $n$ points corresponding to $\xl_1$, \dots,
$\xl_n$ with residue $\mu_1$, \dots, $\mu_n$, while ones in
\cite{MR1438643} have pole only at the right end of the interval with
residue $(\mu_1,\dots,\mu_n)$. Hence we need to show that two moduli
spaces are isomorphic in order to use \cite{MR1438643}. Moreover, the
slice in Proposition~\ref{prop:triangle-Nahm} is different from the
Slodowy slice as explained in Remark~\ref{rem:Slodowy}.
The first assertion might be shown directly by analyzing the limiting behavior
of solutions when $\xl_1$, \dots, $\xl_n$ collide.
Instead, we use Proposition~\ref{prop:triangle-Nahm} successively to show
that the right hand side is isomorphic to $\GL(r)\times\mathcal S$,
where $\mathcal S$ is the space of matrices of form
\setcounter{MaxMatrixCols}{50}
\begin{equation*}
   \begin{bmatrix}
    0 & \cdots & 0 & * & \rvline &&&& * & \rvline &
    &&&& * & \rvline &\\
    1 && \raisebox{-1ex}[0ex][0ex]{\text{\rm\LARGE0}} & * & \rvline
    &&&& * &\rvline &
    &&&& * &\rvline & 
    \\
    & \ddots && \vdots & \rvline &&\raisebox{1ex}[0ex][0ex]{\text{\rm\Huge0}}
    && \vdots & \rvline &
    &&\raisebox{1ex}[0ex][0ex]{\text{\rm\Huge0}}
    && \vdots & \rvline &
    \raisebox{1ex}[0ex][0ex]{\text{\Huge$\cdots$}} \\
    \raisebox{1ex}[0ex][0ex]{\text{\rm\LARGE0}} && 1 & * & \rvline
    &&&& * & \rvline &&&&& * & \rvline &
    \\
    \cline{1-16}
    * & \cdots & * & * & \rvline &
    0 & \cdots & 0 & * & \rvline &
    &&&& * & \rvline &
    \\
    &&&& \rvline &
    1 && \raisebox{-1ex}[0ex][0ex]{\text{\rm\LARGE0}} & * & \rvline &
    &&&& * &\rvline &
    \\
    &&
    && \rvline &
    & \ddots && \vdots & \rvline &
    && \raisebox{1ex}[0ex][0ex]{\text{\rm\Huge0}} && \vdots & \rvline &
    \raisebox{1ex}[0ex][0ex]{\text{\Huge$\cdots$}}
    \\
    & \raisebox{2ex}[0ex][0ex]{\text{\rm\Huge0}} &&& \rvline &
    \raisebox{1ex}[0ex][0ex]{\text{\rm\LARGE0}} && 1 & * & \rvline
    &&&&& * & \rvline & \\
    \cline{1-16}
    * & \cdots & * & * & \rvline &
    * & \cdots & * & * & \rvline &&&&&&& \\
    &&&& \rvline &
    &&&& \rvline & && \multirow{2}{*}{$\Huge\ddots$}
    &&&&& \\
    &&&& \rvline &
    &&&& \rvline & &&&&&&& \\
    & \raisebox{2ex}[0ex][0ex]{\text{\rm\Huge0}} &&& \rvline &
    & \raisebox{2ex}[0ex][0ex]{\text{\rm\Huge0}} &&& \rvline & &&&&&&& \\
    \cline{1-11}
    & \multirow{2}{*}{\Huge\vdots} &&& \rvline && \multirow{2}{*}{\Huge\vdots} &&& \rvline \\
    &&&&\rvline &&&&& \rvline
        \end{bmatrix},
\end{equation*}
where block sizes are $\mu_n$, $\mu_{n-1}$, \dots
for both rows and columns. When $\mu$ is a partition, i.e.,
$\mu_1\ge \mu_2\ge\dots$, the space $\mathcal S$ satisfies
the conditions of the second lemma in \cite[\S3.2]{MV22}. Therefore
it is a transversal slice to the orbit corresponding to $\mu$.
We now use $\GL(r)\times\mathcal S$ in the argument of
\cite[Lem.~3.2]{MR1438643} as in Remark~\ref{rem:Slodowy} to see that
the moduli space in \cite{MR1438643} is isomorphic to
$\GL(r)\times\mathcal S$. As a byproduct, we also see
$\GL(r)\times\mathcal S(\mu)\cong\GL(r)\times\mathcal S$.

\subsection{}

The proof of the second assertion in \propref{prop:unique} is not
correct. It is corrected in Prop.~3.10 in \cite{Na18}.

\subsection{}

Proposition~\ref{prop:balanced_bow} states a bow diagram with
$N(\xl_i, \xl_{i+1}) \geq 0$ for any $i$ can be transformed to a
cobalanced one. The result may have negative dimension on a
segment. In such a case, the cobalanced bow variety is empty. Since
Proposition~\ref{prop:HW-trans} is an isomorphism, the original bow
variey is also empty.

\subsection{}

In the last sentence in \remref{rem:non_balanced}, it was claimed that
any bow variety is homeomorphic to one with the balanced
condition. The reason was because any bow variety has a stratification
of the same form as in \thmref{thm:strat-affine-type}. However, the
index set $\{ (\kappa,\underline{k})\}$ of the stratification depends
on $\lambda$ as $\kappa + |\underline{k}| \le \lambda$ is imposed. It
is not clear to the authors whether this index set can be replaced by
one given by a dominant $\lambda'$. Therefore we could only say that
each stratum of an arbitrary bow variety is isomorphic to one of a
balanced bow variety.

\subsection*{}

We thank Tiziano Gaibisso for discussion on the last two items.


\end{document}